\documentclass[11pt]{article}
\usepackage{amssymb,amsmath,amsthm,hyperref,graphicx,color,array,fullpage,mathrsfs}
\usepackage{dsfont}
\usepackage[notref,notcite,color,final]{showkeys}
\usepackage{diagrams}
\usepackage{makeidx}
\makeindex
\makeatletter

\@addtoreset{equation}{section}

\makeatother

\newtheorem{theorem}{Theorem}[section]
\newtheorem{lemma}[theorem]{Lemma}
\newtheorem{corollary}[theorem]{Corollary}
\newtheorem{conjecture}[theorem]{Conjecture}

\newtheorem{example}[theorem]{Example}
\newtheorem{proposition}[theorem]{Proposition}
\theoremstyle{remark}
\newtheorem{remark}[theorem]{Remark}

\newtheoremstyle{definition}
{3pt}
{3pt}
{\upshape}
{}
{\bfseries}
{:}
{.5em}
{}
\theoremstyle{definition}
\newtheorem{definition}[theorem]{Definition}

\begin{document}
\title{\textsc{Iwasawa Main Conjecture for Rankin-Selberg $p$-adic $L$-functions}}

\author{\textsc{Xin Wan}}
\date{}

\maketitle
\begin{abstract}
In this paper we prove that the $p$-adic $L$-function that interpolates the Rankin-Selberg product of a general modular form and a CM form of higher weight divides the characteristic ideal of the corresponding Selmer group. This is one divisibility of the Iwasawa main conjecture for this $p$-adic $L$-function. We prove this conjecture using congruences between Klingen-Eisenstein series and cusp forms on the group $\mathrm{GU}(3,1)$, following the strategy of recent work by C. Skinner and E. Urban. The actual argument is, however, more complicated due to the need to work with general Fourier-Jacobi expansions. This theorem is used to deduce a converse of the Gross-Zagier-Kolyvagin theorem and the $p$-adic part of the precise BSD formula in the rank one case.
\end{abstract}

\section{Introduction}
Let $p$ be an odd prime. An important problem in number theory is studying the relations between special values of $L$-functions and arithmetic objects in $p$-adic families. The first case was studied by Iwasawa in the 1950's for class groups of number fields, resulting in the asymptotic formula for class numbers in cyclotomic towers of field extensions. Later on, Mazur realized that the idea of Iwasawa's theory can be applied to elliptic curves which provides a powerful way to study its arithmetic (e.g. the BSD conjecture). Such idea was further generalized to many kinds of Galois representations. In 2004, Kato \cite{Kato} formulated the Iwasawa main conjecture for modular form on $\mathrm{GL}_2/\mathbb{Q}$ and proved one divisibility by constructing an Euler system. Later on, Skinner-Urban \cite{SU} proved the other side of divisibility in the case when the modular form is ordinary at $p$, using Eisenstein congruences on the unitary group $\mathrm{U}(2,2)$.

Our paper can be viewed as an extension of Skinner-Urban's approach to the Rankin-Selberg product of a modular form $f$ and an ordinary CM form whose weight is higher than $f$, using another rank $4$ unitary group $\mathrm{U}(3,1)$. Surprisingly, such result has lots of arithmetic applications which cannot be seen by previous techniques, including the proof by Skinner \cite{Skinner} of the converse of a theorem of Gross-Zagier and Kolyvagin, and the $p$-part of the precise BSD formula in the analytic rank one case \cite{JSW}. It is also the starting point and a key ingredient of the author's work proving the Iwasawa main conjecture for supersingular elliptic curves \cite{WAN2}. Now we describe the context of our main results.

Let $\mathcal{K}\subset \mathbb{C}$ be an imaginary quadratic field such that $p$ splits in $\mathcal{K}$ as $(p)=v_0\bar{v}_0$. We fix an isomorphism $\iota:\mathbb{C}_p\simeq \mathbb{C}$ and suppose $v_0$ is determined by $\iota$. There is a unique $\mathbb{Z}_p^2$-extension $\mathcal{K}_\infty/\mathcal{K}$ unramified outside $p$. Let $\Gamma_\mathcal{K}:=\mathrm{Gal}(\mathcal{K}_\infty/\mathcal{K})$ \index{$\Gamma_\mathcal{K}$}. Suppose $\mathbf{f}$ \index{$\mathbf{f}$} is a Hida family of ordinary cuspidal eigenforms new outside $p$ with coefficient ring $\mathbb{I}$, a normal finite extension of the power series ring $\mathbb{Z}_p[[W]]$ of one variable $W$. Let $L$ be a finite extension of $\mathbb{Q}_p$ with integer ring $\mathcal{O}_L$. Suppose $\xi$ \index{$\xi$} is an $L$-valued Hecke character of $\mathbb{A}_\mathcal{K}^\times/\mathcal{K}^\times$ whose infinity type is $(\frac{\kappa}{2},-\frac{\kappa}{2})$ for some even integer $\kappa\geq 6$ and such that $\mathrm{ord}_{v_0}(\mathrm{cond}(\xi_{v_0}))\leq 1$ and $\mathrm{ord}_{\bar{v}_0}(\mathrm{cond}(\xi_{\bar{v_0}}))\leq 1$. Denote by $\boldsymbol{\xi}$ the $\mathcal{O}_L[[\Gamma_\mathcal{K}]]$-adic family of Hecke characters containing $\xi$ as some specialization (we make this precise in Section \ref{EiDatum}). We write $\hat{\mathcal{O}}_L^{ur}$ for the completion of the maximal unramified extension of $\mathcal{O}_L$ and $\hat{\mathbb{I}}^{ur}$ \index{$\mathbb{I}^{ur}$} for the normalization of the ring corresponding to an irreducible component of $\mathbb{I}\hat{\otimes}_{\mathcal{O}_L}\hat{\mathcal{O}}_L^{ur}$.

In Section \ref{Formulation}, we associate with $\mathbf{f}$, $\mathcal{K}$ and $\boldsymbol{\xi}$ a dual Selmer group $X_{\mathbf{f},\mathcal{K},\xi}$, which is a finite module over the ring $\mathbb{I}[[\Gamma_\mathcal{K}]]$. On the analytic side, for a finite set of primes $\Sigma$ containing all bad primes, we construct in this paper using a doubling method the ``$\Sigma$-primitive'' $p$-adic $L$-functions $\mathcal{L}_{\mathbf{f},\xi,\mathcal{K}}^\Sigma\in\hat{\mathbb{I}}^{ur}[[\Gamma_\mathcal{K}]]$, $\mathcal{L}_{f,\xi,\mathcal{K}}^\Sigma\in\hat{\mathcal{O}}_L^{ur}[[\Gamma_\mathcal{K}]]$, interpolating the algebraic parts of the special $L$-values $L_\mathcal{K}(f_\phi,\xi_\phi,\frac{\kappa}{2})$, where $f_\phi$ and $\xi_\phi$ are specializations of the families $\mathbf{f}$ and $\boldsymbol{\xi}$ ($f_\phi$ has weight $2$ and $\xi_\phi$ has infinity type $(\kappa/2,-\kappa/2)$). The general case when $\Sigma$ does not necessarily contain all bad primes, is obtained by putting back the local Euler factors at primes omitted. We let $\mathcal{L}_{f_0,\xi,\mathcal{K}}$ be the specialization of $\mathcal{L}_{\mathbf{f},\xi,\mathcal{K}}$ to a single form $f_0$ of weight 2 and trivial character in the family $\mathbf{f}$, which we assume is defined over $L$. In Section \ref{6.4} we also recall closely related $p$-adic $L$-functions $\mathcal{L}_{\mathbf{f},\xi,\mathcal{K}}^{\Sigma,Hida}$ and $\mathcal{L}_{\mathbf{f},\xi,\mathcal{K}}^{Hida}$ constructed by Hida. We also associate with $f=f_0$, $\mathcal{K}$, $\xi$ a dual Selmer group $X_{f,\mathcal{K},\xi}$ over $\hat{\mathcal{O}}^{ur}_L[[\Gamma_\mathcal{K}]]$. The Iwasawa-Greenberg main conjecture says that the characteristic ideal of $X_{\mathbf{f},\mathcal{K},\xi}$ (resp. $X_{f,\mathcal{K},\xi}$) is generated by $\mathcal{L}_{\mathbf{f},\xi,\mathcal{K}}$ (resp. $\mathcal{L}_{f,\xi,\mathcal{K}}$).

Let $\bar{\mathbb{Q}}\subset\mathbb{C}$ be the algebraic closure of $\mathbb{Q}$ and let $G_\mathbb{Q}=\mathrm{Gal}(\bar{\mathbb{Q}}/\mathbb{Q})$ be the Galois group. Let $G_p\subset G_\mathbb{Q}$ be the decomposition group determined by the inclusion $\bar{\mathbb{Q}}\subset \bar{\mathbb{Q}}_p$ coming from $\iota$. We write $\epsilon$ for the cyclotomic character and $\omega$ for the Techim{\"u}ller character of $G_\mathbb{Q}$.
Let $g$ be a cuspidal eigenform on $\mathrm{GL}_2/\mathbb{Q}$ with the associated $p$-adic Galois representation $\rho_g:G_\mathbb{Q}\rightarrow \mathrm{GL}_2(\mathcal{O}_L)$. We say $g$ satisfies $\mathbf{(irred)}$ \index{$\mathbf{(irred)}$} if:
\begin{itemize}
\item The residual representation $\bar{\rho}_g$ is absolutely irreducible.
\end{itemize}
\noindent If $g$ is nearly ordinary at $p$, then $\rho_g|_{G_p}$ is equivalent to an upper triangular representation and we say it satisfies $\mathbf{(dist)}$ \index{$\mathbf{(dist)}$} if:
\begin{itemize}
\item The characters of $\rho_g|_{G_p}$ on the diagonal are distinct modulo the maximal ideal of $\mathcal{O}_L$.
\end{itemize}
We will see later (in Section \ref{6.4}) that if the CM form $g_\xi$ associated to $\xi$ satisfies $\mathbf{(irred)}$ and $\mathbf{(dist)}$ then $\mathcal{L}_{\mathbf{f},\xi,\mathcal{K}}\in\mathbb{I}[[\Gamma_\mathcal{K}]]$.\\

\noindent In this paper, under certain conditions on $\mathbf{f},\xi,\mathcal{K}$, we prove one inclusion (or divisibility) of the Iwasawa-Greenberg main conjecture for $\mathcal{L}_{\mathbf{f},\xi,\mathcal{K}}$. Our first theorem is a three-variable result for Hida families.
\begin{theorem}\label{Theorem 1}
Let $\mathbf{f}$ be a Hida family of ordinary eigenforms that are new outside $p$ of square-free tame level $N$, and suppose $\mathbf{f}$ has a weight two specialization $f$ that has trivial nebentypus and is the ordinary stabilization of a new form of level $N$. Let $\bar{\rho}$ be the $\mathrm{mod}\ p$ residual $G_\mathbb{Q}$-representation associated with the Hida family $\mathbf{f}$. Let $\xi$ be a Hecke character of $\mathcal{K}^\times \backslash\mathbb{A}_\mathcal{K}^\times$ with infinity type $(\frac{\kappa}{2},-\frac{\kappa}{2})$ for some $\kappa\geq 6$.
If
\begin{itemize}
\item[(a)] $p\geq 5$;
\item[(b)] $\xi|_{\mathbb{A}_\mathbb{Q}^\times}=\omega\circ\mathrm{Nm}$ and $\kappa\equiv 0(\mathrm{mod}\ 2(p-1))$;
\item[(c)] $\bar{\rho}_\mathbf{f}|_{G_\mathcal{K}}$ is irreducible;
\item[(d)] there exists $q|N$ that does not split in $\mathcal{K}$ and such that $\bar{\rho}_\mathbf{f}$ is ramified at $q$;
\item[(e)] the CM eigenform $g_\xi$ associated to the character $\xi$ satisfies $(\mathbf{dist})$ and $(\mathbf{irred})$;
\item[(f)] For each non-split prime $v$ of $\mathbb{Q}$ we have that
    $$\epsilon(\pi_{f,v},\xi_v,\frac{1}{2})=1.$$
    (As in \cite{Hsi12} $\epsilon(\pi_{f,v},\xi_v,\frac{1}{2})$ is the local root number for the base change of $\pi_{f,v}$ to $\mathcal{K}_v$ twisted by $\xi_v$. It differs from the local root number for the Rankin-Selberg product of $\pi_{f,v}$ and $g_{\xi,v}$ by a factor $\chi_{\mathcal{K}/\mathbb{Q},v}(-1)$.)
\item[(g)] Suppose the conductor of $\xi$ is only divisible by primes split in $\mathcal{K}/\mathbb{Q}$.
\end{itemize}
Then $\mathcal{L}^{Hida}_{\mathbf{f},\xi,\mathcal{K}}\in\hat{\mathbb{I}}^{ur}[[\Gamma_\mathcal{K}]]$ and $(\mathcal{L}^{Hida}_{\mathbf{f},\xi,\mathcal{K}})\supseteq \mathrm{char}_{\hat{\mathbb{I}}^{ur}[[\Gamma_\mathcal{K}]]}(X_{\mathbf{f},\mathcal{K},\xi})$ as ideals of $\hat{\mathbb{I}}^{ur}[[\Gamma_\mathcal{K}]]$. Here $\mathrm{char}$ means the characteristic ideal.
\end{theorem}
We also have a two variable theorem for a single form.
\begin{theorem}\label{Theorem 2}
Let $N$, $f=f_0$, $\kappa$ and $\xi$ be as before. If
\begin{itemize}
\item[(a)] $p\geq 5$;
\item[(b)] the $p$-adic avatar of $\xi|\cdot|^{\kappa/2}(\omega^{-1}\circ \mathrm{Nm})$ factors through $\Gamma_\mathcal{K}$ and $\kappa\equiv 0(\mathrm{mod}\ 2(p-1))$;
\item[(c)] $\bar{\rho}_f|_{G_\mathcal{K}}$ is irreducible;
\item[(d)] there exists $q||N$ that does not split in $\mathcal{K}$.
\end{itemize}
Then
$$(\mathcal{L}_{f,\xi,\mathcal{K}})\supseteq \mathrm{char}_{\hat{\mathcal{O}}^{ur}_L[[\Gamma_\mathcal{K}]]\otimes_{\mathcal{O}_L}L}(X_{f,\mathcal{K},\xi})$$
is true as fractional ideals of $\hat{\mathcal{O}}^{ur}_L[[\Gamma_\mathcal{K}]]\otimes_{\mathcal{O}_L}L$.
\end{theorem}
Unlike the previous theorem, in this theorem we allow both global root numbers $+1$ and $-1$ cases. In particular Theorem \ref{Theorem 2} is not deduced as a consequence of Theorem \ref{Theorem 1}. Both theorems are deduced in the proof at the end of this paper. Theorem \ref{Theorem 2} is proved as the specialization of a ``weaker version'' (since the assumption is weaker than that of Theorem \ref{Theorem 1}) of the $3$-variable main conjecture, where we inverted all non-zero elements of $\mathbb{I}$. This is where the $\otimes L$ comes in. See the end of the paper for details.

Hida's $p$-adic $L$-functions $\mathcal{L}^{Hida}_{\mathbf{f},\xi,\mathcal{K}}$ are more canonical than the $\mathcal{L}_{\mathbf{f},\xi,\mathcal{K}}$ in that there is a constant in $\bar{\mathbb{Q}}_p^\times$ showing up in our interpolation formula (see Proposition \ref{$p$-adic $L$-function}) that depends on some choices. Under the assumptions of Theorem \ref{Theorem 1} we show that Hida's $p$-adic $L$-function is integral: it belongs to $\hat{\mathbb{I}}^{ur}[[\Gamma_\mathcal{K}]]$. Note that in the setting of Theorem \ref{Theorem 2} we do not know if $\mathcal{L}_{f_0,\mathcal{K},\xi}$ is actually in $\hat{\mathcal{O}}^{ur}_L[[\Gamma_\mathcal{K}]]$.

The assumptions on $\bar{\rho}_{f}|_{G_\mathcal{K}}$ and the local $\epsilon$-factors in Theorem \ref{Theorem 1} are needed to appeal to results of M. Hsieh \cite{Hsi11}, \cite{Hsi12} in proving the non-vanishing modulo $p$ of some special $L$-values or vanishing of the anti-cyclotomic $\mu$-invariant. The square-freeness of $N$ is put at the moment for simplicity (mainly to avoid local triple product integrals for supercuspidal representations and we may come back to remove it in the future).

The assumption (g) in Theorem \ref{Theorem 1} is due to lack of reference for \cite[Conjecture in Introduction]{HT93}. Details are explained in Definition \ref{Hida $p$-adic $L$-functions}.

Hypothesis (b) of Theorem \ref{Theorem 2} means that $\mathcal{L}_{f,\xi,\mathcal{K}}$ can be evaluated at the trivial character of $\Gamma_\mathcal{K}$, though it is not a point at which it interpolates classical $L$-values. As a result, Theorem \ref{Theorem 2} has interesting applications for the usual Bloch-Kato Selmer group of $f$.

The result of this paper is the foundation for several important breakthroughs on arithmetic of elliptic curves and modular forms.
Skinner \cite{Skinner} has recently been able to use Theorem \ref{Theorem 2} to prove a converse of the Gross-Zagier-Kolyvagin theorem: if the Mordell-Weil rank of an elliptic curve over $\mathbb{Q}$ is exactly one and the Shafarevich-Tate group is finite, then its $L$-function vanishes to exactly order one at the central critical point. The author has been able to prove an anti-cyclotomic main conjecture of Perrin-Riou when the root number is $-1$ \cite{WAN1} (by comparing the Selmer group in the theorem with the one studied by Perrin-Riou, using the Poitou-Tate long exact sequence and applying F. Castella's generalization \cite{Castella} of a formula of Bertolini-Darmon-Prasanna relating the different $p$-adic $L$-functions).

There is also joint work of the author with Skinner and Jetchev that uses Theorem \ref{Theorem 2} to deduce the $p$-adic part of the precise BSD formula in the rank one case \cite{JSW}.
We remark that in the above mentioned applications one can not appeal to the main conjecture proved in \cite{SU} since the global sign of the $L$-functions has to be $+1$ in \emph{loc.cit.}.

The methods of this paper can be adapted (with some additional arguments) to the case when $f$ is non-ordinary as well. This forms the foundation of the author's recent proof of the Iwasawa main conjecture for supersingular elliptic curves formulated by Kobayashi (see \cite{WAN2}).\\

Our proofs of Theorem \ref{Theorem 1} and Theorem \ref{Theorem 2} use Eisenstein congruences on the unitary group $\mathrm{U}(3,1)$, which first appeared in Hsieh's paper \cite{Hsieh13}. Recent works with a similar flavor include Skinner-Urban's \cite{SU} using the group $\mathrm{U}(2,2)$, and the work of Hsieh \cite{Hsieh CM} for CM characters using the group $\mathrm{U}(2,1)$. The difference between our results and Skinner-Urban's is that they studied the $p$-adic $L$-function of Rankin-Selberg product of a general modular form and a CM form such that the weight of the CM form is lower, while in our case the weight of the CM form is higher. This is the very reason we work with unitary groups of different signature.

We also mention there are works establishing the other divisibility of the main conjecture using Euler systems (\cite{WAN1}, \cite{LLZ}) under some more restrictions. Together with Theorem \ref{Theorem 1} and Theorem \ref{Theorem 2} these give the full equality of the main conjecture in the case when all hypotheses are satisfied.\\

\noindent For clarity purposes, we briefly discuss our proof of the theorems. The proof follows the main outline of Skinner-Urban's proof in \cite{SU} (which in turn followed the main outline of Wiles' proof of the Iwasawa main conjecture for totally real fields). However, carrying this out requires new arguments. The main steps are: (1) constructing a $p$-adic family of Eisenstein series whose constant terms are essentially the $p$-adic Rankin-Selberg $L$-function $\mathcal{L}^\Sigma_{\mathbf{f},\mathcal{K},\boldsymbol{\xi}}$; (2) proving that the Eisenstein series is co-prime to $p$-adic $L$-function (that is, modulo any divisor of the $p$-adic $L$-function it is still non-zero), which shows that its congruences with cuspforms is `measured' by the $p$-adic $L$-function; (3) the Galois argument.

The main differences between our proof and that of Skinner-Urban are in steps (1) and (2). First, we need to work with the unitary group $\mathrm{U}(3,1)$ instead of $\mathrm{U}(2,2)$ which is used in \cite{SU}. The reason is that by our assumption that the CM form has higher weight than $f$, the $L$-values interpolated by the $p$-adic $L$-function $\mathcal{L}_{\mathbf{f},\xi,\mathcal{K}}$ show up in the constant terms of holomorphic Eisenstein series on the group $\mathrm{U}(3,1)$ that are induced from the Klingen parabolic subgroup with Levi $\mathrm{U}(2)\times\mathcal{K}^\times$. The cuspidal representation on $\mathrm{U}(2)$ is determined by the automorphic representation $\pi_f$ and a Hecke character of $\mathbb{A}_\mathcal{K}^\times$ whose restriction to $\mathbb{A}_\mathbb{Q}^\times$ is the central character of $\pi_f$. As a result, the construction of the $p$-adic families of the Eisenstein series via the pullback formula requires finding the right Siegel section at $p$ (which turns out to be different from the one used in \cite{SU}). To have the right pullback and to make the Fourier-Jacobi coefficient computation not too hard, such choice of section is quite subtle. The idea for our choice is similar to that in \cite{EHLS} and is inspired by the formula for differential operators on $p$-adic $q$-expansions on the group $\mathrm{GU}(3,3)$. (The Siegel-Eisenstein series measure used here to construct the $p$-adic $L$-function is the special case of \emph{loc.cit.}. (See Section 4.3, part II.) We also refer to the paper of Eischen \cite[Sections 3, 4]{Ellen} for a nice exposition for details about those differential operators. These differential operators are not logically needed for the construction in this paper though.)

Step (2) is the core of the whole argument. In \cite{SU}, the Klingen-Eisenstein series on $\mathrm{U}(2,2)$ (which are also special cases of the series constructed in \cite{WAN}) has a Fourier expansion $E_{Kling}=\sum_Ta_Tq^T$, with $T$ running over $2\times 2$ Hermitian matrices. By the pullback formula, we have $a_T=\langle \mathrm{FJ}_TE_{sieg},\varphi_\pi\rangle_{\mathrm{U}(1,1)}$, where $E_{sieg}$ is a Siegel-Eisenstein series on $\mathrm{U}(3,3)$, $\mathrm{FJ}_TE_{sieg}$ is its $T$-Fourier-Jacobi coefficient (regarded as a form on $\mathrm{U}(1,1)$), and $\varphi_\pi$ is a form in the $\mathrm{U}(1,1)$ automorphic representation $\pi$ considered in \cite{SU} (again determined by $\pi_f$ and a Hecke character of $\mathbb{A}_\mathcal{K}^\times$). The Siegel Eisenstein measure is constructed in Proposition 12.3. Computation tells us $\mathrm{FJ}_TE_{sieg}$ is essentially a product of an Eisenstein series and a theta function, and thus this pairing, and hence $a_T$ is essentially a Rankin-Selberg product.

In our case, forms on $\mathrm{U}(3,1)$ only have Fourier-Jacobi expansions (instead of Fourier expansions):
$$F\mapsto \sum_{n\in\mathbb{Q}}a_n(F)q^n:=\mathrm{FJ}(F)$$
with $a_n(F)\in H^0(\mathcal{Z}^\circ_{[g]},\mathcal{L}(n))$, where $\mathcal{Z}^\circ_{[g]}$ is a two-dimensional abelian variety which is the abelian part of the universal semiabelian scheme over a point in the boundary of a toroidal compactification of the Shimura variety for $\mathrm{GU}(3,1)$. The sheaf $\mathcal{L}(n)$ is a line bundle on $\mathcal{Z}^\circ_{[g]}$. We can view each $a_n(F)$ as an automorphic form on the group $\mathrm{U}(2)\cdot N$, where $\mathrm{U}(2)$ is the definite unitary group appearing as a factor of the Levi of the Klingen parabolic subgroup of $\mathrm{U}(3,1)$, and $N$ is the unipotent radical of the parabolic subgroup, which is a Heisenberg group. It consists of matrices of the form $\begin{pmatrix}1&\times&\times&\times\\&1&&\times\\&&1&\times\\&&&1\end{pmatrix}$. To study $a_n(F)$ we use a functional $l_{\theta^\star}$ on $H^0(\mathcal{Z}^\circ_{[g]},\mathcal{L}(n))$. This is just the pairing over $N$ (modulo its center) with an explicit theta function $\theta^\star$ on $\mathrm{U}(2)\cdot N$ (defined in Lemma \ref{Lemma 6.32}). In the following we do the computation at an arithmetic point $\mathsf{z}$.

We divide our argument into five steps. (1) We first compute the $n$-th Fourier-Jacobi coefficient of a Siegel-Eisenstein series $E_{sieg,\mathsf{z}}$ (in Section \ref{GC}), considered as a form on the Jacobi group $N'\cdot \mathrm{U}(2,2)\subseteq \mathrm{U}(3,3)$ with $N'$ a unipotent subgroup of $\mathrm{U}(3,3)$. It consists of matrices of the form
$$\begin{pmatrix}1&\times&\times&\times&\times&\times\\&1&&\times&&\\&&1&\times&&\\&&&1&&\\&&&
\times&1&\\&&&\times&&1\end{pmatrix}.$$
This turns out to be a finite sum of products of the form $E_{2,\mathsf{z}}\cdot \Theta_\mathsf{z}$ (see Proposition \ref{Proposition 6.31}, with $E_{2,\mathsf{z}}$ a Siegel-Eisenstein series on $\mathrm{U}(2,2)$ and $\Theta_\mathsf{z}$ a theta function on the Jacobi group (the $E_{sieg,2}$ and $\Theta_{\Phi_\mathcal{D}}$ in Proposition \ref{Proposition 6.31}).

(2) Next we restrict this $n$-th Fourier-Jacobi coefficient to the group $$(N\cdot \mathrm{U}(2))\times\mathrm{U}(2)\subset \mathrm{U}(3,1)\times \mathrm{U}(2)\cap N'\cdot \mathrm{U}(2,2).$$ Another computation shows that $\Theta_{\mathsf{z}}$ essentially restricts to a form $\theta_{4}\times\theta_{2,\mathsf{z}}$ on $(N\cdot \mathrm{U}(2))\times \mathrm{U}(2)$. (The $\theta_{2,\mathsf{z}}$ is varying with the arithmetic point $\mathsf{z}$ while the $\theta_4$ is fixed, which justifies dropping the subscript $\mathsf{z}$.) The actual situation is slightly more complicated: it is actually a finite sum of such products. Applying a functional $l_{\theta^\star}$ (which is pairing with a \emph{fixed} theta function $\theta^\star$ to the $\theta_4$-component of each summand above), we show that $\langle\theta_4,\theta^\star\rangle_N$ is a constant function on $\mathrm{U}(2)$ by Lemma \ref{lemma 4.8} (which we manage to make non-zero), and then end up with a theta function $\theta_\mathsf{z}$ on (the lower) $\mathrm{U}(2)$ (a finite linear combination of $\theta_{2,\mathsf{z}}$'s of each summand). See Lemma \ref{Lemma 6.32} for a precise formula. So using the pullback formula, we get for $E_{Kling,\mathsf{z}}$ the Klingen Eisenstein series defined in (\ref{DefineKlingen}) (denoted $E_{Kling,\mathcal{D}}$ there),
$$l_{\theta^\star}(a_n(E_{Kling,\mathsf{z}}))=\langle E_{2,\mathsf{z}}|_{\mathrm{U}(2)\times\mathrm{U}(2)},f_\mathsf{z}\cdot\theta_{\mathsf{z}}
\rangle_{1\times\mathrm{U}(2)}$$
regarded as a form on the first $\mathrm{U}(2)$, which is the $\mathrm{U}(2)$ in the Levi of the Klingen parabolic subgroup. Note that by Lemma \ref{lemma8.19} when $\mathsf{z}$ is varying in a $p$-adic family the $l_{\theta^\star}$ takes values in the Iwasawa algebra (the parameter space).

(3) To study its $p$-adic property, we pair it with an auxiliary form $h_\mathsf{z}$ on $\mathrm{U}(2)$ (Definition \ref{444}):
$$\langle\langle E_{2,\mathsf{z}}|_{\mathrm{U}(2)\times\mathrm{U}(2)}, f_\mathsf{z}\cdot\theta_{\mathsf{z}}\rangle_{1\times \mathrm{U}(2)},h\rangle_{\mathrm{U}(2)}=(*)\cdot \langle h_\mathsf{z},f_\mathsf{z}\cdot\theta_{\mathsf{z}}\rangle.$$
To obtain this formula, we use the doubling method formula for $\mathrm{U}(2)\times\mathrm{U}(2)\hookrightarrow \mathrm{U}(2,2)$ applied to $h_\mathsf{z}$. The $(*)$ is some $p$-adic $L$-function factor for $h_\mathsf{z}$ coming from this (see Proposition \ref{Proposition 8.21} for details).

(4) We prove that such expression is interpolated by an element $\mathbf{B}_1$ in the Iwasawa algebra (in (\ref{116})). The pairing on the right hand side is just a triple product integral $\int_{[\mathrm{U}(2)]}h_{\mathsf{z}}(g)\theta_{\mathsf{z}}(g)f_{\mathsf{z}}(g)dg$. The fact that the $\theta_{\mathsf{z}}$ can be taken to be an eigenform follows from considering the central character (see the proof of Proposition \ref{Proposition 8.21}).

(5) We use Ichino's formula to evaluate this:
\begin{align*}
&&&(\int_{[U(2)]}h_{\mathsf{z}}(g)\theta_{\mathsf{z}}(g)f_{\mathsf{z}}dg)(\int_{[U(2)]}
\tilde{h}_{\mathsf{z}}(g)\tilde{\theta}_{3,\mathsf{z}}(g)\tilde{f}_{\mathsf{z}}(g)dg)&\\
&=\langle h_{\mathsf{z}},\tilde{h}_{\mathsf{z}}\rangle\langle\theta_{\mathsf{z}},\tilde{\theta}_{3,\mathsf{z}}\rangle\langle f_{\mathsf{z}},\tilde{f}_{\mathsf{z}}\rangle\cdot&&
\frac{L^\Sigma(\frac{1}{2},\pi_{f_\mathsf{z}}\times\chi_{1,\mathsf{z}})L^\Sigma(\frac{1}{2},\pi_{f_\mathsf{z}}
\times\chi_{2,\mathsf{z}})}
{L^\Sigma(2,\pi_{f_\mathsf{z}},\mathrm{ad})
L^\Sigma(2,\pi_{\theta_{\mathsf{z}}},\mathrm{ad})
L^\Sigma(2,\pi_{h_\mathsf{z}},\mathrm{ad})}
\prod_{v\in\Sigma}\frac{I_v(h_\mathsf{z}\otimes \theta_{\mathsf{z}}\otimes f_\mathsf{z},\tilde{h}_\mathsf{z}\otimes\tilde{\theta}_{3,\mathsf{z}}\otimes\tilde{f}_\mathsf{z})}{\langle h_{\mathsf{z},v},\tilde{h}_{\mathsf{z},v}\rangle\langle\theta_{\mathsf{z},v},\tilde{\theta}_{3,\mathsf{z},v}
\rangle\langle f_{\mathsf{z},v},\tilde{f}_{\mathsf{z},v}\rangle}.&
\end{align*}
\noindent Here $\tilde{h}_{\mathsf{z}}$, $\tilde{\theta}_{3,\mathsf{z}}$ and $\tilde{f}_{\mathsf{z}}$ means some forms or vectors in the contragredient representation of the automorphic representation for $h_{\mathsf{z}}, \theta_{\mathsf{z}}$ and $f_{\mathsf{z}}$. The factor $I_v$ is a local integral defined by Ichino, and $\chi_{1,\mathsf{z}}$ and $\chi_{2,\mathsf{z}}$ are two CM Hecke characters showing up in the computation. We interpolate everything in $p$-adic families and compare it to the product of several $p$-adic $L$-functions of modular forms or Hecke characters (see the proof of Proposition \ref{p-adic} for details). Furthermore:
\begin{itemize}
\item We can choose $h_{\mathsf{z}}$'s and $\theta_{\mathsf{z}}$'s so that these $p$-adic $L$-functions are units in $\hat{\mathbb{I}}^{ur}[[\Gamma_\mathcal{K}]]^\times$ times a fixed number in $\bar{\mathbb{Q}}_p$.
\item The ratio of the triple product and the product of these $p$-adic $L$-functions is a product of local factors (we show that the triple product is a $p$-adic analytic function, so the product of these local factors is a $p$-adic meromorphic function). We make the local choices such that for inert or ramified primes these local factors involve only the Hida-family variable of $\mathbf{f}$ (which has nothing to do with $\Gamma_\mathcal{K}^+$ or $\Gamma_\mathcal{K}^-$). For split primes, we compute these local factors explicitly.
\end{itemize}
The constructions above finally provide a non-zero element of $\mathbb{I}$, which is sufficient for our use. We are thus able to prove in Proposition \ref{p-adic} that height one divisors of $\mathbf{B}_1$ are those of $\mathbb{I}$.

After this, we can use the same argument as in \cite{SU} to deduce our main theorem: by a geometric argument we construct a cuspidal family on $\mathrm{U}(3,1)$ congruent modulo the $p$-adic $L$-function $\mathcal{L}_{\mathbf{f},\xi,\mathcal{K}}$ to the Eisenstein family constructed as above. Passing to the Galois side, we get a family of Galois representations coming from cuspidal forms that is congruent to the family coming from our Klingen-Eisenstein series, but which is ``more irreducible'' than the Eisenstein Galois representations. Then an argument (the ``lattice construction'') of E.Urban gives the required elements in the dual Selmer group.
\begin{remark}
In fact at the point where $\mathcal{L}_{\mathbf{f},\xi,\mathcal{K}}$ takes its central critical value, the Klingen Eisenstein family does not interpolate a classical Eisenstein series (i.e. not an interpolation point). Therefore even in the case when the global root number for $\mathcal{L}_{\mathbf{f},\xi,\mathcal{K}}$ is $-1$, so that the constant terms of the Eisenstein family vanish identically along the central critical subfamily, the $p$-adic Eisenstein series itself can still be non-zero in that subfamily.
\end{remark}
\begin{remark}
We emphasize here that $\theta^\star$ is fixed throughout the whole $p$-adic family (instead of varying). Note also that the space of theta functions with given Archimedean kernel function and level group at finite places is finite-dimensional. The space $H^0(\mathcal{Z}^\circ_{[g]},\mathcal{L}(n))\otimes_{\mathbb{Z}_p}\mathbb{Q}_p$ is generated by a finite number of such theta functions. We will show in the text that by pairing the $\hat{\mathbb{I}}^{\mathrm{ur}}[[\Gamma_\mathcal{K}]]$-adic Fourier-Jacobi coefficient with one rational theta function (not necessarily $p$-integral!), we get an element in $\hat{\mathbb{I}}^{\mathrm{ur}}[[\Gamma_\mathcal{K}]]\otimes_{\mathbb{Z}_p}\mathbb{Q}_p$. We then show that by choosing the datum properly, this element is the product of a unit in $\hat{\mathbb{I}}^{\mathrm{ur}}[[\Gamma_\mathcal{K}]]$ and a non-zero element in $\bar{\mathbb{Q}}_p$, and proving it is prime to the $p$-adic $L$-function we study. Such strategy is notably different from the one adopted in \cite{Hsi12}, where Hsieh argued $p$-integrally and proved with a stronger result that the Fourier-Jacobi coefficient is already a unit. This is the very reason why we do not need to study the theory of $p$-integral theta functions.
\end{remark}
\begin{remark}
In \cite{ZhB07} the special $L$-value showing up in the Fourier-Jacobi expansion is the near central point of the Rankin-Selberg $L$-function, while in our case it is the central value of the triple product $L$-function. Moreover, the Fourier-Jacobi coefficient considered in \cite{ZhB07} is non-zero only when $f$ is a CM form (see Theorem 4.12 in \emph{loc.cit.}). This is due to the fact that we are pairing the Fourier-Jacobi coefficient with the product of a theta function and an auxiliary form $h$ on $\mathrm{U}(2)$, while Zhang paired it with the theta function only (i.e. taking the $h$ in our case to be the constant function). Our strategy has the advantage that these central $L$-values are accessible to various results of non-vanishing modulo $p$ by Hsieh.
\end{remark}

\noindent The rest of this paper is organized as follows. In section 2, we recall some background information and formulate the main conjecture. In section 3, we discuss automorphic forms and $p$-adic automorphic forms on various unitary groups. In section 4 we recall the notion of theta function which plays an important role in studying Fourier-Jacobi expansions as outlined above. In sections 5 and 6, we make the local and global calculations for Siegel and Klingen-Eisenstein series using the pullback formula of Shimura. In section 7, we interpolate our previous calculations $p$-adically and construct the families. In section 8, we prove the co-primeness of (the Fourier-Jacobi coefficients of) the Klingen-Eisenstein series and the $p$-adic $L$-function. Finally, we deduce the main theorem in section 9.\\

\noindent\emph{Acknowledgement} The result of this paper was worked out while the author was a graduate student of Christopher Skinner. I would like to thank him for leading me to the problem, sharing many ideas and constant encouragement. I am also grateful to Ming-Lun Hsieh, from whose previous work \cite{Hsieh13}, \cite{Hsieh CM} the author has benefited a lot, and for answering many questions. I would also like to thank Haruzo Hida, Ye Tian, Eric Urban, and Bei Zhang for their insightful communications. The author is partially supported by the Chinese Academy of Science grant Y729025EE1, NSFC grant 11688101, 11621061 and an NSFC grant associated to the "Recruitment Program of Global Experts
\section{Background}
We first introduce our notation. We will usually take a finite extension $L/\mathbb{Q}_p$ and write $\mathcal{O}_L$ for its integer ring and $\varpi_L$ for a uniformizer. Let $G_\mathbb{Q}$ and $G_\mathcal{K}$ be the absolute Galois groups of $\mathbb{Q}$ and $\mathcal{K}$. Let $\Gamma_\mathcal{K}^{\pm}$ \index{$\Gamma_\mathcal{K}^{\pm}$} be the subgroups of $\Gamma_\mathcal{K}$ such that the complex conjugation $c$ acts by $\pm 1$. We take topological generators $\gamma^{\pm}$ so that $\mathrm{rec}^{-1}(\gamma^+)=((1+p)^{\frac{1}{2}},(1+p)^{\frac{1}{2}})$ and $\mathrm{rec}^{-1}(\gamma^-)=((1+p)^{\frac{1}{2}},(1+p)^{-\frac{1}{2}})$ where $\mathrm{rec}: \mathbb{A}_\mathcal{K}^\times\rightarrow G_\mathcal{K}^{\mathrm{ab}}$ is the reciprocity map normalized by the geometric Frobenius. Let $\Psi_\mathcal{K}$ be the composition $$G_\mathcal{K}\twoheadrightarrow \Gamma_\mathcal{K}\hookrightarrow \mathbb{Z}_p[[\Gamma_\mathcal{K}]]^\times.$$
Define $\Lambda_\mathcal{K}:=\mathcal{O}_L[[\Gamma_\mathcal{K}]]$. Recall we defined a branch character $\xi$ in the introduction. We will write $\sigma_\xi$ for the Galois character corresponding to $\xi$ via class field theory. We also let $\mathbb{Q}_\infty$ be the cyclotomic $\mathbb{Z}_p$ extension of $\mathbb{Q}$ and let $\Gamma_\mathbb{Q}=\mathrm{Gal}(\mathbb{Q}_\infty/\mathbb{Q})$. Define $\Psi_\mathbb{Q}$ to be the composition $G_\mathbb{Q}\twoheadrightarrow \Gamma_\mathbb{Q}\hookrightarrow \mathbb{Z}_p[[\Gamma_\mathbb{Q}]]^\times$. We also define $\varepsilon_\mathcal{K}$ and $\varepsilon_\mathbb{Q}$ \index{$\varepsilon_\mathcal{K}$, $\varepsilon_\mathbb{Q}$} to be the compositions $\mathcal{K}^\times\backslash \mathbb{A}_\mathcal{K}^\times\stackrel{rec}{\rightarrow} G_\mathcal{K}^{ab}\rightarrow \mathbb{Z}_p[[\Gamma_\mathcal{K}]]^\times$ and $\mathbb{Q}^\times \backslash \mathbb{A}_\mathbb{Q}^\times\stackrel{rec}{\rightarrow} G_\mathbb{Q}^{ab}\rightarrow \mathbb{Z}_p[[\Gamma_\mathbb{Q}]]^\times$ where the second arrows are the $\Psi_\mathcal{K}$ and $\Psi_\mathbb{Q}$ \index{$\Psi_\mathcal{K}$, $\Psi_\mathbb{Q}$} defined above. Let $\omega$ and $\epsilon$ \index{$\omega$, $\epsilon$} be the Techim{\"u}ller character and the cyclotomic character. We also write $\chi_{\mathcal{K}/\mathbb{Q}}$ \index{$\chi_{\mathcal{K}/\mathbb{Q}}$} to be the quadratic character associated to $\mathcal{K}/\mathbb{Q}$.

Write $c$ for the complex conjugation. For a Hecke character $\chi$ we write $\chi^c(x):=\chi(c(x))$. For a Galois character $\chi$ we define $\chi^c$ to be the composition of $\chi$ with conjugation by $c$ (regarding $c$ as an element in the Galois group).

\subsection{$p$-adic Families for $\mathrm{GL}_2/\mathbb{Q}$}

\noindent Let $M$ be a positive integer prime to $p$ and $\chi$ a character of $(\mathbb{Z}/pM\mathbb{Z})^\times$. Let $\Lambda_\mathbb{Q}:=\mathbb{Z}_p[[W]]$ (we call $\mathrm{Spec}\Lambda_\mathbb{Q}(\bar{\mathbb{Q}}_p)$ the weight space). Let $\mathbb{I}$ be a normal domain finite over $\Lambda_\mathbb{Q}$. A point $\phi\in\mathrm{Spec}(\mathbb{I})$ is called arithmetic if the image of $\phi$ in $\mathrm{Spec}\Lambda(\bar{\mathbb{Q}}_p)$ is the continuous $\mathbb{Z}_p$-homomorphism sending $(1+W)\mapsto \zeta(1+p)^{k-2}$ for some $k\geq 2$ and $\zeta$ a $p$-power root of unity. We usually write $k_\phi$ for this $k$, called the weight of $\phi$. We also define $\chi_\phi$ to be the character of $\mathbb{Z}_p^\times\simeq (\mathbb{Z}/p\mathbb{Z})^\times\times(1+p\mathbb{Z}_p)$ that is trivial on the first factor and given by $(1+p)\mapsto \zeta$ on the second factor.

\begin{definition}
An $\mathbb{I}$-adic family of modular forms of tame level $M$ and character $\chi$ is a formal $q$-expansion $\mathbf{f}=\sum_{n=0}^\infty a_nq^n, a_n\in\mathbb{I}$, such that for a Zariski dense set of arithmetic points $\phi\in\mathrm{Spec}(\mathbb{I})$ the specialization $f_\phi=\sum_{n=0}^\infty \phi(a_n) q^n$ of $\mathbf{f}$ at $\phi$ is the $q$-expansion of a modular form of weight $k_\phi$, character $\chi\chi_\phi\omega^{2-k_\phi}$ (where $\omega$ is the Techim{\"u}ller character), and level $Mp^{t_\phi}$ for some $t_\phi\geq 0$.
\end{definition}

The $U_p$ operator is defined in both the spaces of modular forms and families. It is given by:
$$U_p(\sum_{n=0}^\infty a_nq^n)=\sum_{n=0}^\infty a_{pn}q^n.$$
Note that $(U_p\cdot\mathbf{f})_\phi=U_p\cdot \mathbf{f}_\phi$. Hida's ordinary idempotent $e_p$ is defined by $e_p:=\lim_{n\rightarrow \infty}U_p^{n!}$. A form $f$ or family $\mathbf{f}$ is called ordinary if $e_pf=f$ or $e_p\mathbf{f}=\mathbf{f}$. (See for instance \cite[pp. 550]{Hida5}.) A well known fact is that every ordinary eigenform fits into an ordinary family of eigenforms $\mathbf{f}$ (\cite[Theorem II]{Hida88} for example).\\

According to the results of Deligne, Langlands, Shimura et al., there is a Galois representation $\rho_f: G_\mathbb{Q}\rightarrow \mathrm{GL}_2(\bar{\mathbb{Q}}_p)$ for $f$. If the residual representation $\bar{\rho}_f$ is irreducible then one can construct a Galois representation $\rho_\mathbf{f}: G_\mathbb{Q}\rightarrow \mathrm{GL}_2(\mathbb{I})$ such that it specializes to the Galois representation $\rho_{\mathbf{f}_\phi}$ of $\mathbf{f}_\phi$ at each arithmetic specialization $\phi\in\mathrm{Spec}(\mathbb{I})$. We write $T_\mathbf{f}$ for the representation space of $\rho_\mathbf{f}$.

\subsection{The Main Conjecture}\label{Formulation}
Before formulating the main conjecture, we first define characteristic ideals and Fitting ideals. We let $A$ be a Noetherian ring. We write $\mathrm{Fitt}_A(X)$ for the Fitting ideal in $A$ of a finitely generated $A$-module $X$. This is the ideal generated by the determinant of the $r\times r$ minors of the matrix giving the first arrow in a given presentation of $X$:
$$A^s\rightarrow A^r\rightarrow X\rightarrow 0.$$
If $X$ is not a torsion $A$-module, then $\mathrm{Fitt}_A(X)=0$. This definition does not depend on the choice of the presentation.\\

\noindent Fitting ideals behave well with respect to base change. For $I\subset A$ an ideal
$$\mathrm{Fitt}_{A/I}(X/IX)=\mathrm{Fitt}_A(X)\ \mathrm{mod}\ I.$$
\index{$\mathrm{Fitt}$} Now suppose $A$ is a Krull domain (a domain which is Noetherian and normal). Then the characteristic ideal is defined by: \index{$\mathrm{char}$}
$$\mathrm{char}_A(X):=\{x\in A:\mathrm{ord}_Q(x)\geq \mathrm{length}_Q(X) \mbox{ for any height one prime $Q$ of $A$}\}.$$
If $X$ is not a torsion $A$-module, then we define $\mathrm{char}_A(X)=0$.\\

\noindent We consider the Galois representation:
$$T_{\mathbf{f},\mathcal{K},\xi}:=T_\mathbf{f}\hat{\otimes}_{\mathbb{Z}_p}\Lambda_\mathcal{K}$$
with the $G_\mathcal{K}$ action given by $\rho_{\mathbf{f}}\sigma_{{\xi}^{-c}}\epsilon^{\frac{4-\kappa}{2}}\hat{\otimes} \Lambda_\mathcal{K}(\Psi_\mathcal{K}^{-c})$. (Here $c$ means composing with the complex conjugation in the idele group, and $-$ means taking inverse of the character.)
We define the Selmer group (recall $\kappa$ is assumed to be even)
$$\mathrm{Sel}_{\mathbf{f},\mathcal{K},\xi}:=\mathrm{ker}\{H^1(\mathcal{K},T_{\mathbf{f},\mathcal{K},\xi}
\otimes_{\mathbb{I}[[\Gamma_\mathcal{K}]]}\mathbb{I}[[\Gamma_\mathcal{K}]]^*)\rightarrow H^1(I_{\bar{v}_0},T_{\mathbf{f},\mathcal{K},\xi}\otimes\mathbb{I}[[\Gamma_\mathcal{K}]]^*)\times \prod_{v\nmid p}H^1(I_v,T_{\mathbf{f},\mathcal{K},\xi}\otimes\mathbb{I}[[\Gamma_\mathcal{K}]]^*)\}$$
where $*$ means the Pontryagin dual $\mathrm{Hom}_{\mathbb{Z}_p}(-, \mathbb{Q}_p/\mathbb{Z}_p)$. We also define the $\Sigma$-primitive Selmer groups:
$$\mathrm{Sel}_{\mathbf{f},\mathcal{K},\xi}^\Sigma:=\mathrm{ker}\{H^1(\mathcal{K},T_{\mathbf{f},\mathcal{K},
\xi}\otimes\mathbb{I}[[\Gamma_\mathcal{K}]]^*)\rightarrow H^1(I_{\bar{v}_0},T_{\mathbf{f},\mathcal{K},\xi}\otimes\mathbb{I}[[\Gamma_\mathcal{K}]]^*)\times \prod_{v\not\in \Sigma}H^1(I_v,T_{\mathbf{f},\mathcal{K},\xi}\otimes\mathbb{I}[[\Gamma_\mathcal{K}]]^*)\}.$$
We let \index{$X_{\mathbf{f},\mathcal{K},\xi}$, $X_{\mathbf{f},\mathcal{K},\xi}^\Sigma$}
$$X_{\mathbf{f},\mathcal{K},\xi}:=(\mathrm{Sel}_{\mathbf{f},\mathcal{K},\xi})^*.$$
$$X_{\mathbf{f},\mathcal{K},\xi}^\Sigma:=(\mathrm{Sel}_{\mathbf{f},\mathcal{K},\xi}^\Sigma)^*.$$
These are finitely generated $\mathbb{I}[[\Gamma_\mathcal{K}]]$-modules (see e.g. \cite[Lemma 3.3]{SU}).
We take the extension of scalars of them to $\hat{\mathbb{I}}^{ur}[[\Gamma_\mathcal{K}]]$ and still denote them by using the same notations. In section \ref{section 7} we construct $p$-adic $L$-functions $\mathcal{L}^{Hida}_{\mathbf{f},\mathcal{K},\xi}$ and $\mathcal{L}_{\mathbf{f},\xi,\mathcal{K}}^{\Sigma,Hida}$ which are elements in $\hat{\mathbb{I}}^{ur}[[\Gamma_\mathcal{K}]]$ or its fraction field. Their interpolation formulas are given in equation (\ref{equation (9)}) (see also Remark \ref{Hida $p$-adic $L$-functions}).
The three-variable Iwasawa main conjecture is
\begin{conjecture}
$X_{\mathbf{f},\mathcal{K},\xi}$ and $X_{\mathbf{f},\mathcal{K},\xi}^\Sigma$ are torsion $\hat{\mathbb{I}}^{ur}[[\Gamma_\mathcal{K}]]$-modules and
$$\mathrm{char}_{\hat{\mathbb{I}}^{ur}[[\Gamma_\mathcal{K}]]}X_{\mathbf{f},\mathcal{K},\xi}=(\mathcal{L}_{\mathbf{f},\xi,\mathcal{K}}^{Hida}),$$
$$\mathrm{char}_{\hat{\mathbb{I}}^{ur}[[\Gamma_\mathcal{K}]]}X_{\mathbf{f},\mathcal{K},\xi}^\Sigma=(\mathcal{L}_{\mathbf{f},\xi,\mathcal{K}}^{\Sigma,Hida}).$$
\end{conjecture}
We can also replace $\mathbf{f}$ with a single form $f_0$ and have
the two-variable main conjectures.
\begin{conjecture}
$X_{f_0,\mathcal{K},\xi}$ and $X_{f_0,\mathcal{K},\xi}^\Sigma$ are torsion $\hat{\mathcal{O}}_L^{ur}[[\Gamma_\mathcal{K}]]$-modules and
$$\mathrm{char}_{\hat{\mathcal{O}}^{ur}_L[[\Gamma_\mathcal{K}]]}X_{f_0,\mathcal{K},\xi}=(\mathcal{L}_{f_0,\mathcal{K},\xi}^{Hida}),$$
$$\mathrm{char}_{\hat{\mathcal{O}}^{ur}_L[[\Gamma_\mathcal{K}]]}X_{f_0,\mathcal{K},\xi}^\Sigma=(\mathcal{L}_{f_0,\mathcal{K},\xi}^{\Sigma,Hida}).$$
\end{conjecture}

\subsection{Control of Selmer Groups}
In this sub-section we prove a control theorem of Selmer groups which will be used to prove Theorem \ref{Theorem 2}. Let $\phi_0\in\mathrm{Spec}\mathbb{I}[[\Gamma_\mathcal{K}]](\bar{\mathbb{Q}}_p)$ be the point mapping $\gamma^\pm$ to $1$ and such that $\phi_0|_{\mathbb{I}}$ corresponds to the form $f_0$. Let $\wp=\mathrm{ker}\phi_0|_\mathbb{I}$ be the point of weight two and trivial character.  Then we prove the following proposition.
\begin{proposition}\label{control}
Suppose $\bar{\rho}_f|_{G_\mathcal{K}}$ is absolutely irreducible. There is an exact sequence of $\mathcal{O}_L[[\Gamma_\mathcal{K}]]$-modules
$$M\rightarrow X_{\mathbf{f},\mathcal{K},\xi}^{\Sigma}/\wp X_{\mathbf{f},\mathcal{K},\xi}^{\Sigma}\rightarrow X_{f_0,\mathcal{K},\xi}^\Sigma\rightarrow 0$$
where $M\otimes_{\mathcal{O}_L}L$ has support of codimension at least $2$ in $\mathrm{Spec}\mathcal{O}_L[[\Gamma_\mathcal{K}]]\otimes L$.
\end{proposition}
\begin{proof}
We write $\mathbb{I}_\mathcal{K}$ for $\mathbb{I}[[\Gamma_\mathcal{K}]]$ for simplicity. Write $\mathbb{T}=T_{\mathbf{f},\mathcal{K},\xi}$ as a $\mathbb{I}_\mathcal{K}$-module. Let $T$ be the $\Lambda_\mathcal{K}$-module $T_{f_0,\mathcal{K},\xi}$. Recall that $p=v_0\bar{v}_0$. We have an exact sequence
$$0\rightarrow T\otimes_{\Lambda_\mathcal{K}}\Lambda_\mathcal{K}^*\rightarrow \mathbb{T}\otimes_{\mathbb{I}_\mathcal{K}}\mathbb{I}_\mathcal{K}^*\rightarrow \mathbb{T}\otimes_{\mathbb{I}_\mathcal{K}}(\wp\mathbb{I}_\mathcal{K})^*\rightarrow 0.$$
Write $G_{\mathcal{K}_\Sigma}$ for the Galois group over $\mathcal{K}$ of the maximal algebraic extension of $\mathcal{K}$ unramified outside $\Sigma$. From this we deduce
$$H^1(G_{\mathcal{K}_\Sigma}, T\otimes_\Lambda\Lambda_\mathcal{K}^*)\stackrel{\sim}{\rightarrow}H^1(G_{\mathcal{K}_\Sigma}, \mathbb{T}\otimes_{\mathbb{I}_\mathcal{K}}\mathbb{I}_\mathcal{K}^*)[\wp]$$
as in \cite[Proposition 3.7]{SU}.
We also have an exact sequence:
\begin{align*}
&H^0(I_{\bar{v}_0}, \mathbb{T}\otimes_{\mathbb{I}[[\Gamma_\mathcal{K}]]}(\mathbb{I}
[[\Gamma_\mathcal{K}]]^*))
\stackrel{s_1}{\rightarrow} H^0(I_{\bar{v}_0}, \mathbb{T}\otimes_{\mathbb{I}[[\Gamma_\mathcal{K}]]}(\wp
[[\Gamma_\mathcal{K}]]^*))&\\
&\rightarrow H^1(I_{\bar{v}_0}, T\otimes_{\mathcal{O}_L[[\Gamma_\mathcal{K}]]}(\mathcal{O}_L
[[\Gamma_\mathcal{K}]]^*))
\rightarrow H^1(I_{\bar{v}_0}, \mathbb{T}\otimes_{\mathbb{I}[[\Gamma_\mathcal{K}]]}(\mathbb{I}
[[\Gamma_\mathcal{K}]]^*)).&
\end{align*}
From these we deduce an exact sequence of $\Lambda_\mathcal{K}$-modules

$$M:=((\mathrm{coker}s_1)^{G_{\bar{v}_0}})^*/\wp(\mathrm{coker}s_1)^{G_{\bar{v}_0}}\hookrightarrow X_{\mathbf{f},\mathcal{K},\xi}^\Sigma/\wp X_{\mathbf{f},\mathcal{K},\xi}^\Sigma\rightarrow X^\Sigma_{f_0,\mathcal{K},\xi}\rightarrow 0.$$

Let $\mathcal{K}_{\infty, \bar{v}_0}$ (resp. $\mathcal{K}_{\infty, v_0})$ be the $\mathbb{Z}_p$-extension of $\mathcal{K}$ unramified outside $\bar{v}_0$ (resp. $v_0$) and let $\Gamma_{\bar{v}_0}=\mathrm{Gal}(\mathcal{K}_{\infty,\bar{v}_0})$ (resp. $\Gamma_{v_0}=\mathrm{Gal}(\mathcal{K}_{v_0}/\mathcal{K})$). Let $\gamma_{\bar{v}_0}\in \Gamma_{\bar{v}_0}$ and $\gamma_{v_0}\in\Gamma_{v_0}$ be topological generators. It is well known that (e.g., see \cite[Section 3.3.5]{SU} that we have $$0\rightarrow T^+\rightarrow T\rightarrow T/T^+\rightarrow 0$$ as $G_{\mathbb{Q}_p}$-modules. By the description of the Galois action, there is a $\gamma\in I_{\bar{v}_0}$ such that $\gamma-1$ acts invertibly on $T^+\otimes_{\mathbb{I}[[\Gamma_\mathcal{K}]]}(\mathbb{I}
[[\Gamma_\mathcal{K}]]^*))$. We take a basis $(v_1,v_2)$ such that $v_1$ generates $T^+$ and the action of $\gamma$ on $T$ is diagonal under this basis. Then it is not hard to see (by looking at the $I_{\bar{v}_0}$-action) that if $$v\in H^0(I_{\bar{v}_0}, \mathbb{T}\otimes_{\mathbb{I}_\mathcal{K}}\mathbb{I}_\mathcal{K}^*)$$
we have $v\in (\mathbb{I}[[\Gamma_{v_0}]])^* v_2$ and if $v\in H^0(I_{\bar{v}_0}, \mathbb{T}\otimes_{\mathbb{I}_\mathcal{K}}(\wp\mathbb{I}_\mathcal{K})^*)$ then $v\in (\wp\mathbb{I}_\mathcal{K})^*v_2$. From the above discussion, we know that $((\mathrm{coker}s_1)^{G_{\bar{v}_0}})^*/\mathrm{ker}\phi_0'(\mathrm{coker}s_1)^{G_{\bar{v}_0}}$ is supported in $$\mathrm{Spec}(\mathcal{O}_L[[\Gamma_v]]\otimes L).$$

Moreover by looking at the action of $\mathrm{Frob}_{\bar{v}_0}$ we see it is killed by the function $a_p^{-1}R-1$ where $a_p$ is the invertible function in $\mathbb{I}$ which gives the $U_p$-eigenvalue of $\mathbf{f}$ and $R$ is the image in $\Gamma_v$ of $\mathrm{Frob}_{\bar{v}_0}$ under class field theory. But $a_p(\phi_0)\not=1$ and $R(\phi_0)=1$ so $a_p^{-1}R-1$ is non-zero at $\phi_0$. So the support of $M\otimes_{\mathcal{O}_L}L$ has support of dimension at most zero and this proves the proposition.

\end{proof}
\section{Unitary Groups}
In this section, we introduce our notation for unitary groups and develop the Hida theory on them. We mainly follow \cite[Sections 2,3,4]{Hsieh CM} in our presentation, which in turn, summarizes portions of Shimura's books \cite{Shi97} and \cite{Shi00}. We define $S_n(R)$ \index{$S_n(R)$} to be the set of $n\times n$ Hermitian matrices with entries in $\mathcal{O}_\mathcal{K}\otimes_\mathbb{Z}R$. We define a map $e_\mathbb{A}=\prod_ve_v: \mathbb{A}_\mathbb{Q}\rightarrow \mathbb{C}^\times$ where for each place $v$ of $\mathbb{Q}$, $e_v$ is the standard additive character as in \cite[8.1.2]{SU} (which again follows Shimura's convention). We refer to \cite[Section 2.8]{Hsieh CM} for the discussion of the CM period $\Omega_\infty\in\mathbb{C}^\times$ \index{$\Omega_\infty$} and the $p$-adic period $\Omega_p\in\hat{\mathbb{Z}}^{ur,\times}_p$ \index{$\Omega_p$}.
\subsection{Groups}
Let $\delta'\in\mathcal{K}$ be a totally imaginary element such that $-i\delta'$ is positive. Let $d=\mathrm{Nm}(\delta')$ which we assume to be a $p$-adic unit. Let $\mathrm{U}(2)=\mathrm{U}(2,0)$ (resp. $\mathrm{GU}(2)=\mathrm{GU}(2,0)$) be the unitary group (resp. unitary similitude group) associated to the skew-Hermitian matrix $\zeta=\begin{pmatrix}\mathfrak{s}\delta'&\\&\delta'\end{pmatrix}$ \index{$\zeta$} for some $\mathfrak{s}\in\mathbb{Z}_+$ \index{$\mathfrak{s}$} prime to $p$. More precisely $\mathrm{GU}(2)$ is the group scheme over $\mathbb{Z}$ defined by: for any $\mathbb{Z}$ algebra $A$,
$$\mathrm{GU}(2)(A)=\{g\in\mathrm{GL}_2(A\otimes_\mathbb{Z}\mathcal{O}_\mathcal{K})|{}^t\!\bar{g}\zeta g=\lambda(g)\zeta,\ \lambda(g)\in A^\times.\}$$
The map $\lambda: \mathrm{GU}(2)\rightarrow \mathbb{G}_m$, $g\mapsto\lambda(g)$ is called the similitude character and $\mathrm{U}(2)\subseteq \mathrm{GU}(2)$ is the kernel of $\lambda$.
Let $G=\mathrm{GU}(3,1)$ (resp. $\mathrm{U}(3,1)$) \index{$G, \mathrm{GU}(3,1)$} be the similarly defined unitary similitude group (resp. unitary group) over $\mathbb{Z}$ associated to the skew-Hermitian matrix  $\begin{pmatrix}&&1\\&\zeta&\\-1&&\end{pmatrix}$. We denote this Hermitian space as $V$. Let $P\subseteq G$ \index{$P$} be the parabolic subgroup of $\mathrm{GU}(3,1)$ consisting of those matrices in $G$ of the form $\begin{pmatrix}\times&\times&\times&\times\\&\times&\times&\times\\&\times&\times&\times\\&&&\times\end{pmatrix}$. Let $N_P$ \index{$N_P$} be the unipotent radical of $P$. Then if $X_\mathcal{K}$ is the $1$-dimensional space over $\mathcal{K}$,
$$M_P:=\mathrm{GL}(X_\mathcal{K})\times \mathrm{GU}(2)\hookrightarrow \mathrm{GU}(V),\ (a,g_1)\mapsto \mathrm{diag}(a,g_1,\mu(g_1)\bar{a}^{-1})$$
is the Levi subgroup of $P$. Let $G_P:=\mathrm{GU}(2)\subseteq M_P$ \index{$G_P$} be the set of elements $\mathrm{diag}(1,g_1,\mu(g_1))$ as above. Let $\delta_P$ be the modulus character for $P$. We usually use a more convenient character $\delta$ such that $\delta^3=\delta_P$.\\

\noindent Since $p$ splits as $v_0\bar{v}_0$ in $\mathcal{K}$, $\mathrm{GL}_4(\mathcal{O}_\mathcal{K}\otimes\mathbb{Z}_p)\stackrel{\sim}{\rightarrow}\mathrm{GL}_4
(\mathcal{O}_{\mathcal{K}_{v_0}})\times\mathrm{GL}_4(\mathcal{O}_{\mathcal{K}_{\bar{v}_0}})$. Here $$\mathrm{U}(3,1)(\mathbb{Z}_p)\stackrel{\sim}{\rightarrow}\mathrm{GL}_4(\mathcal{O}_{\mathcal{K}_{v_0}})
=\mathrm{GL}_4(\mathbb{Z}_p)$$ is the projection onto the first factor. Let $B$ and $N$ be the upper triangular Borel subgroup of $G(\mathbb{Q}_p)$ and its unipotent radical, respectively. Let $$K_p=\mathrm{GU}(3,1)(\mathbb{Z}_p)\simeq \mathrm{GL}_4(\mathbb{Z}_p),$$ and for any $n\geq 1$ let $K_0^n$ \index{$K_0^n$} be the subgroup of $K$ consisting of matrices upper-triangular modulo $p^n$. Let $K_1^n\subset K_0^n$ \index{$K_1^n$} be the subgroup of matrices whose diagonal
elements are $1$ modulo $p^n$.\\

\noindent The group $\mathrm{GU}(2)$ is closely related to a division algebra. Put \index{$D$}
$$D=\{g\in M_2(\mathcal{K})|g\zeta{}^t\!\bar{g}=\det(g)\zeta\}.$$
Then $D$ is a definite quaternion algebra over $\mathbb{Q}$ with local invariant $\mathrm{inv}_v(D)=(-\mathfrak{s},-D_{\mathcal{K}/\mathbb{Q}})_v$ (the Hilbert symbol).
The relation between $\mathrm{GU}(2)$ and $D$ is explained by
$$\mathrm{GU}(2)=D^\times\times_{\mathbb{G}_m}\mathrm{Res}_{\mathcal{K}/\mathbb{Q}}\mathbb{G}_m.$$
For each finite place $v$ we write $D_v^1$ for the set of elements $g_v\in D_v^\times$ such that $|\mathrm{Nm}(g_v)|_v=1$, where $\mathrm{Nm}$ is the reduced norm. In application we will choose the $D$ to be the quaternion algebra ramified exactly at $\infty$ and the $q$ in the main theorems. \\

\noindent Let $\Sigma$ \index{$\Sigma$} be a finite set of primes containing all the primes at which $\mathcal{K}/\mathbb{Q}$ or $\xi$ is ramified, the primes dividing the level of $f_0$ (as in the introduction), the primes dividing $\mathfrak{s}$, the primes such that $\mathrm{U}(2)(\mathbb{Q}_v)$ is compact and the prime $2$. Let $\Sigma^1$ and $\Sigma^2$, respectively be the set of non-split primes in $\Sigma$ such that $\mathrm{U}(2)(\mathbb{Q}_v)$ is non-compact, and compact. We will sometimes write $[D^\times]$ for $D^\times(\mathbb{Q})\backslash D^\times(\mathbb{A}_\mathbb{Q})$. We similarly write $[\mathrm{U}(2)]$, $[\mathrm{GU}(2,0)]$, etcetera. For two automorphic forms $f_1, f_2$ on $\mathrm{U}(2)$ we write $\langle f_1,f_2\rangle=\int_{[\mathrm{U}(2)]}f_1(g)f_2(g)dg$. Here the Haar measure is normalized so that at finite places $\mathrm{U}(2)(\mathbb{Z}_\ell)$ has measure $1$, and at $\infty$ the compact set $\mathrm{U}(1)(\mathbb{R})\backslash\mathrm{U}(2)(\mathbb{R})$ has measure $1$.\\

\noindent We define $G_n=\mathrm{GU}(n,n)$ \index{$G_n, \mathrm{GU}(n,n)$} for the unitary similitude group for the skew-Hermitian matrix $\begin{pmatrix}&1_n\\-1_n&\end{pmatrix}$ and $\mathrm{U}(n,n)$ for the corresponding unitary group.

\subsection{Hermitian Spaces and Automorphic Forms}
Let $(r,s)=(3,3)$ or $(2,2)$ or $(3,1)$ or $(2,0)$. Then the unbounded Hermitian symmetric domain for $\mathrm{GU}(r,s)$ is
$$X^+=X_{r,s}=\{\tau=\begin{pmatrix}x\\y \end{pmatrix}|x\in M_s(\mathbb{C}),y\in M_{(r-s)\times s}(\mathbb{C}),i(x^*-x)>iy^*\zeta^{-1}y\}.$$
We use $x_0$ to denote the Hermitian symmetric domain for $\mathrm{GU}(2)$, which is just a point.
We have the following embedding of Hermitian symmetric domains:
$$\iota: X_{3,1}\times X_{2,0}\hookrightarrow X_{3,3}$$
$$(\tau,x_0)\hookrightarrow Z_\tau,$$
where $Z_\tau=\begin{pmatrix}x&0\\y&\frac{\zeta}{2}\end{pmatrix}$ for $\tau=\begin{pmatrix}x\\y\end{pmatrix}$.\\

\noindent Let $G_{r,s}=\mathrm{GU}(r,s)$ and $H_{r,s}=\mathrm{GL}_r\times\mathrm{GL}_s$. Let $G_{r,s}(\mathbb{R})^+$ be the subgroup of elements of $G_{r,s}(\mathbb{R})$ whose similitude factors are positive. If $s\not=0$ we define a co-cycle: $$J: G_{r,s}(\mathbb{R})^+\times X^+\rightarrow H_{r,s}(\mathbb{C})$$ \index{$H$} by
$J(\alpha,\tau)=(\kappa(\alpha,\tau),\mu(\alpha,\tau))$, where for $\tau=\begin{pmatrix}x \\y \end{pmatrix}$ and
$\alpha=\begin{pmatrix}a&b&c\\g&e&f\\h&l&d\end{pmatrix}$ (block matrix with respect to the partition $(s+(r-s)+s)$),
$$\kappa(\alpha,\tau)=\begin{pmatrix}\bar{h}{}^t\!x+\bar{d}&\bar{h}{}^t\!y+l\bar{\zeta}\\-\bar{\zeta}^{-1}(\bar{g}{}^t\!x+\bar{f})&-\bar{\zeta}^{-1}
\bar {g}{}^t\!y+\bar{\zeta}^{-1}\bar{e}\bar{\zeta}\end{pmatrix},\ \mu(\alpha,\tau)=hx+ly+d$$
in the $\mathrm{GU}(3,1)$ case and
$$\kappa(\alpha,\tau)=\bar{h}{}^t\!x+\bar{d},\ \mu(\alpha,\tau)=hx+d$$
in the $\mathrm{GU}(3,3)$ case.
\noindent Let $i\in X^+$ be the point $\begin{pmatrix}i1_s\\0\end{pmatrix}$. Let $K_\infty^+$ be the compact subgroup of $\mathrm{U}(r,s)(\mathbb{R})$ stabilizing $i$ and let $K_\infty$  be the group generated by $K_\infty^+$ and $\mathrm{diag}(1_{r+s},-1_s)$. Then $$K_\infty^+\rightarrow H(\mathbb{C}),\ k_\infty\mapsto J(k_\infty,i)$$ defines an algebraic representation of $K_\infty^+$. Later on in Section \ref{sectionar} we will also consider a different choice $\mathbf{i}$ on the Symmetric domain for $(r,s)=(3,3)$ or $(2,2)$.
\begin{definition}
A weight $\underline{k}$ \index{$\underline{k}$} is defined to be an $(r+s)$-tuple
$$\underline{k}=(c_{r+s},...,c_{s+1};c_{1},...,c_{s})\in\mathbb{Z}^{r+s}$$ with $c_{1}\geq...\geq c_{r+s}$, $c_s\geq c_{s+1}+r+s$.
\end{definition}
Our convention for identifying a weight with a tuple of integers is different from others in the literature. For example, our $c_{s+i}$ ($1\leq i\leq r$) and $c_j$ ($1\leq j\leq s$) corresponds to $-a_{r+1-i}$ and $b_{s+1,j}$ in \cite[Section 3.1]{Hsieh CM}.

We refer to \emph{loc.cit.} for the definition of the algebraic representation $L_{\underline{k}}(\mathbb{C})$ of $H$ with the action denoted by $\rho_{\underline{k}}$ (note the different index for weight) and define a model $L^{\underline{k}}(\mathbb{C})$ of the representation $H(\mathbb{C})$ with the highest weight $\underline{k}$ as follows. The underlying space of $L^{\underline{k}}(\mathbb{C})$ is $L_{\underline{k}}(\mathbb{C})$ and the group action is defined by
$$\rho^{\underline{k}}(h)=\rho_{\underline{k}}({}^t\!h^{-1}),h\in H(\mathbb{C}).$$
We also note that if each $\underline{k}=(0,...,0;\kappa,...,\kappa)$ then $L^{\underline{k}}(\mathbb{C})$ is one-dimensional.

For a weight $\underline{k}$, define $\|\underline{k}\|$ by:
$$\|\underline{k}\|:=-c_{s+1}-...-c_{s+r}+c_{1}+...+c_{s}$$
and
$|\underline{k}|$ by:
$$|\underline{k}|=(c_{1}+...+c_{s}).\sigma-(c_{s+1}+...+c_{s+r}).\sigma c\in\mathbb{Z}^I.$$
Here $I$ is the set of embeddings $\mathcal{K}\hookrightarrow \mathbb{C}$ and $\sigma$ is the Archimedean place of $\mathcal{K}$ determined by our fixed embedding $\mathcal{K}\hookrightarrow \mathbb{C}$.
Let $\chi$ be a Hecke character of $\mathcal{K}$ with infinity type $|\underline{k}|$, i.e. the Archimedean part of $\chi$ is given by:
$$\chi_\infty(z)=(z^{c_{1}+...+c_{s}}\cdot \bar{z}^{-(c_{s+1}+...+c_{s+r})}).$$
\begin{definition}
Let $U$ be an open compact subgroup of $G(\mathbb{A}_{f})$.  We denote by $M_{\underline{k}}(U,\mathbb{C})$ the space of holomorphic $L^{\underline{k}}(\mathbb{C})$-valued functions $f$ on $X^+\times G(\mathbb{A}_{f})$ such that for $\tau\in X^+$, $\alpha\in G(\mathbb{Q})^+$ and $u\in U$ we have
$$
f(\alpha\tau,\alpha gu)=\mu(\alpha)^{-\|\underline{k}\|}\rho^{\underline{k}}(J(\alpha,\tau))f(\tau,g).$$
\end{definition}

Now we consider automorphic forms on unitary groups in the adelic language. The space of automorphic forms of weight $\underline{k}$ and level $U$ with central character $\chi$ consists of smooth and slowly increasing functions $F: G(\mathbb{A})\rightarrow L_{\underline{k}}(\mathbb{C})$ such that for every $(\alpha,k_\infty,u,z)\in G(\mathbb{Q})\times K_\infty^+\times U\times Z(\mathbb{A})$,
$$F(z\alpha gk_\infty u)=\rho^{\underline{k}}(J(k_\infty,i)^{-1})F(g)\chi^{-1}(z).$$
We can associate a $L_{\underline{k}}(\mathbb{C})$-valued function on $X^+\times G(\mathbb{A}_f)/U$ given by
\begin{equation}\label{CLASSICAL}f(\tau,g):=\chi_f(\mu(g))\rho^{\underline{k}}(J(g_\infty,i))F((f_\infty,g))
\end{equation}
where $g_\infty\in G(\mathbb{R})$ such that $g_\infty(i)=\tau$. If this function is holomorphic, then we say that the automorphic form $F$ is holomorphic.
\subsection{Galois Representations Associated to Cuspidal Representations}\label{2.24}
In this section, we follow \cite[Theorem 7.1, Lemma 7.2]{SU} to discuss the Galois representations associated to cuspidal automorphic representations on $\mathrm{GU}(r,s)(\mathbb{A}_\mathbb{Q})$. Let $\pi$ be an irreducible automorphic representation of $\mathrm{GU}(r,s)(\mathbb{A}_\mathbb{Q})$ generated by a holomorphic cuspidal eigenform with weight $\underline{k}=(c_{r+s},...,c_{s+1};c_{1},...,c_{s})$ and central character $\chi_\pi$. Let $\Sigma(\pi)$ be a finite set of primes of $\mathbb{Q}$ containing all the primes at which $\pi$ is unramified and all the primes dividing $p$. Then for some $L$ finite over $\mathbb{Q}_p$ there is a Galois representation (see \cite{Shin}, \cite{Morel} and \cite{Sk10}):
$$R_p(\pi):G_\mathcal{K}\rightarrow \mathrm{GL}_n(L)$$
($n=r+s$) such that:
\begin{itemize}
\item[(a)]$R_p(\pi)^c\simeq R_p(\pi)^\vee\otimes\rho_{p,\chi_\pi^{1+c}}\epsilon^{1-n}$ where $\rho_{p,\chi_\pi^{1+c}}$ denotes the $p$-adic Galois character associated to $\chi_\pi^{1+c}$ by class field theory and $\epsilon$ is the cyclotomic character.
\item[(b)]$R_p(\pi)$ is unramified at all finite places not above primes in $\Sigma(\pi)$, and for such a place $w\nmid p$:
$$\det (1-R_p(\pi)(\mathrm{Frob}_w)q_w^{-s})=L(BC(\pi)_w\otimes \chi_{\pi,w}^c,s+\frac{1-n}{2})^{-1}$$
Here, the $\mathrm{Frob}_w$ is the geometric Frobenius and $BC$ means the base change from $\mathrm{U}(r,s)$ to $\mathrm{GL}_{r+s}$.
Suppose $\pi_v$ is nearly ordinary  with respect to $\underline{k}$ (see Subsection \ref{Subsection 3.8}) and unramified at all primes $v$ dividing $p$. Recall $v_0|p$ corresponds to $\iota:\mathbb{C}\simeq \mathbb{C}_p$. If we write $\kappa_{i}=s-i+c_{i}$ for $1\leq i\leq s$ and $\kappa_{i}=c_{i}+s+r+s-i$ for $s+1\leq i \leq r+s$, then
$$R_p(\pi)|G_{\mathcal{K},v_0}\simeq\begin{pmatrix}\begin{matrix}\xi_{r+s,v}\epsilon^{-\kappa_{r+s}}&*
\\&\xi_{r+s-1,v}\epsilon^{\kappa_{r+s-1}}\end{matrix}&
\begin{matrix}*&*\\*&*\end{matrix}\\0&\begin{matrix}...&*\\&\xi_{1,v}\epsilon^{-\kappa_{1}}\end{matrix}
\end{pmatrix}$$
where $\xi_{i,v}$ are unramified characters. Using fact (a) above, we also know that $R_p(\pi)_{\bar{v_0}}$ is equivalent to an upper triangular representation as well (with the Hodge-Tate weight being $(-(\kappa_1+1-r-s-|\underline{k}|),...,-(\kappa_{r+s}+1-r-s-|\underline{k}|)$) (in our geometric convention $\epsilon^{-1}$ has Hodge-Tate weight one).
\end{itemize}

\subsection{Shimura Varieties}
Now we consider the group $\mathrm{GU}(3,1)$. For any open compact subgroup $K=K_pK^p$ of $\mathrm{GU}(3,1)(\mathbb{A}_f)$ whose $p$-component is $K_p=\mathrm{GU}(3,1)(\mathbb{Z}_p)$, we refer to \cite[Section 2.1]{Hsieh CM} for the definition and arithmetic model of the associated Shimura variety, which we denote as $S_G(K)_{/\mathcal{O}_{\mathcal{K},(v_0)}}$. The scheme $S_G(K)$ \index{$S_G(K)$} represents the following functor: for any $\mathcal{O}_{\mathcal{K},(v_0)}$-algebra $R$, $\underline{A}(R)=\{(A,\bar{\lambda},\iota,\bar{\eta}^p)\}$ where $A$ is an abelian scheme over $R$ of relative dimension four with CM by $\mathcal{O}_\mathcal{K}$ given by $\iota$, $\bar{\lambda}$ is an orbit of prime-to-$p$ polarizations and $\bar{\eta}^p$ is an orbit of prime-to-$p$ level structures. There is also a theory of compactifications of $S_G(K)$ developed by Lan in \cite{LAN2}. We denote by $\bar{S}_G(K)$ a toroidal compactification and $S^*_G(K)$ \index{$S_G(K), S^*_G(K), \bar{S}_G(K)$} the minimal compactification. We refer to \cite[Section 2.7]{Hsieh CM} for details. The boundary components of $S^*_G(K)$ are in one-to-one correspondence with the set of cusp labels defined below. For $K=K_pK^p$ as above we define the set of cusp labels to be:
$$C(K):=(\mathrm{GL}(X_\mathcal{K})\times G_P(\mathbb{A}_f))N_P(\mathbb{A}_f)\backslash G(\mathbb{A}_f)/K.$$
This is a finite set. We denote by $[g]$ the class represented by $g\in G(\mathbb{A}_f)$. For each such $g$ whose $p$-component is $1$ we define $K_P^g=G_P(\mathbb{A}_f)\cap gKg^{-1}$ and denote $S_{[g]}:=S_{G_P}(K_P^g)$ \index{$S_{[g]}$} the corresponding Shimura variety for the group $G_P$ with level group $K_P^g$. By strong approximation we can choose a set $\underline{C}(K)$ of representatives of ${C}(K)$ consisting of elements $g=pk^0$ for $p\in P(\mathbb{A}_f^{(\Sigma)})$ and $k^0\in K^0$ for $K^0$ the maximal compact subgroup of $G(\mathbb{A}_f)$ defined in \cite[Section 1.10]{Hsieh CM}.
\subsection{Igusa Varieties and $p$-adic Automorphic Forms}\label{2.4}
Now we recall briefly the notion of Igusa varieties in \cite[Section 2.3]{Hsieh CM}. We remark that these materials are special cases in Hida's book \cite[Chapter 8]{Hida04}. Let $V$ be the Hermitian space for the unitary group $\mathrm{GU}(3,1)$ and let $M$ be the standard lattice of $V$ as in \cite[Section 1.8]{Hsieh CM}. Let $M_p=M\otimes_{\mathbb{Z}}\mathbb{Z}_p$. Let $\mathrm{Pol}_p=\{N^{-1}, N^0\}$ be a polarization of $M_p$. Recall that this means that $N^{-1}$ and $N^0$ are maximal isotropic $\mathcal{O}_\mathcal{K}\otimes\mathbb{Z}_p$-submodules in $M_p$ such that they are dual to each other with respect to the Hermitian metric on $V$, and
$$\mathrm{rank}_{\mathbb{Z}_p} N_{v_0}^{-1}=\mathrm{rank}_{\mathbb{Z}_p}N_{\bar{v}_0}^0=3,\  \mathrm{rank}_{\mathbb{Z}_p}N_{\bar{v}_0}^{-1}=\mathrm{rank}_{\mathbb{Z}_p}N_{v_0}^0=1.$$

The Igusa variety $I_G(K^n)$ \index{$I_G(K^n)$} of level $p^n$ is the scheme over $\mathcal{O}_{\mathcal{K},(v_0)}$ representing the quintuple $\underline{A}(R)=\{(A,\bar{\lambda},\iota,\bar{\eta}^p,j)\}$ where the $A$, $\bar{\lambda}$, $\iota$, $\bar{\eta}^p$ are as in the definition for Shimura varieties of $\mathrm{GU}(3,1)$ as above, and an injection of group schemes
$$j:\mu_{p^n}\otimes_\mathbb{Z}N^0\hookrightarrow A[p^n]$$
over $R$ which is compatible with the $\mathcal{O}_\mathcal{K}$-action on both sides.
Note that the existence of $j$ implies that $A$ must be ordinary along the special fiber. There is also a theory of Igusa varieties over $\bar{S}_G(K)$. Let $\underline{\omega}$ be the automorphic vector bundle on $S_G(K)$ as defined in \cite[Subsection 2.7.3]{Hsieh CM}. As in \emph{loc.cit.} let $\bar{H}_{p-1}\in H^0(S_G(K)_{/\bar{\mathbb{F}}},\mathrm{det}(\underline{\omega})^{p-1})$ be the Hasse invariant. Over the minimal compactification, some power (say the $t$-th) of the Hasse invariant can be lifted to $\mathcal{O}_{v_0}$, by the ampleness of $\det\underline{\omega}$. We denote such a lift by $E$. By the Koecher principle we can regard $E$ as in $H^0(\bar{S}_G(K),\mathrm{det}(\underline{\omega}^{t(p-1)}))$. Let $\mathcal{O}_m:=\mathcal{O}_{\mathcal{K},v_0}/p^m\mathcal{O}_{\mathcal{K},v_0}$. Set $T_{0,m}:=\bar{S}_G(K)[1/E]_{/\mathcal{O}_m}$. For any positive integer $n$ define $T_{n,m}:=I_G(K^n)_{/\mathcal{O}_m}$ and $T_{\infty,m}=\varprojlim_n T_{n,m}$. Then $T_{\infty,m}$ is a Galois cover over $T_{0,m}$ with Galois group $\mathbf{H}\simeq \mathrm{GL}_3(\mathbb{Z}_p)\times \mathrm{GL}_1(\mathbb{Z}_p)$. Let $\mathbf{N}\subset \mathbf{H}$ be the upper triangular unipotent radical. Define:
$$V_{n,m}=H^0(T_{n,m},\mathcal{O}_{T_{n,m}}).$$
Let $V_{\infty,m}=\varinjlim_n V_{n,m}$ and $V_{\infty,\infty}=\varprojlim_m V_{\infty,m}$ be the space of $p$-adic automorphic forms on $\mathrm{GU}(3,1)$ with tame level $K$. We also define $W_{n,m}=V_{n,m}^\mathbf{N}$, $W_{\infty,m}=V_{\infty,m}^\mathbf{N}$ and $\mathcal{W}=\varinjlim_n\varinjlim_m W_{n,m}$. We define $V_{n,m}^0$, etcetera, to be the cuspidal part of the corresponding spaces.\\

\noindent We can make similar definitions for the definite unitary similitude groups $G_P$ as well and define $V_{n,m,P}$,$V_{\infty,m,P}$, $V_{\infty,\infty,P}$, $V_{n,m,P}^{\mathbf{N}}$, $\mathcal{W}_P$, and so forth.\\

\noindent Let $K_0^n$ and $K_1^n$ be the subgroup of $\mathbf{H}$ consisting of matrices which are in $B_3\times \mathrm{GL}_1$ or $N_3\times \mathrm{GL}_1$ modulo $p^n$, for $B_3$ and $N_3$ being the upper triangular Borel subgroup of $\mathrm{GL}_3$ and its unipotent radical, respectively. (These notations are already used for level groups of automorphic forms. The reason for using the same notation here is that automorphic forms with level group $K_\bullet^n$ are $p$-adic automorphic forms of level group $K_\bullet^n$.) We sometimes denote $I_G(K_1^n)=I_G(K^n)/K_1^n$ and $I_G(K_0^n)=I_G(K^n)/K_0^n$.\\

\noindent We can define the Igusa varieties for $G_P$ as well. For $\bullet=0,1$ we let $K_{P,\bullet}^{g,n}:= gK^n_\bullet g^{-1}\cap G_P(\mathbb{A}_f)$ and let $I_{[g]}(K_\bullet^n):=I_{G_P}(K_{P,\bullet}^{g,n})$ be the corresponding Igusa variety over $S_{[g]}$. We denote $A_{[g]}^n$ the coordinate ring of $I_{[g]}(K_1^n)$. Let $A_{[g]}^{\infty}=\varinjlim_n A_{[g]}^n$ and let $\hat{A}_{[g]}^\infty$ be the $p$-adic completion of $A^\infty_{[g]}$. This is the space of $p$-adic automorphic forms for the group $\mathrm{GU}(2,0)$ of level group $gKg^{-1}\cap G_P(\mathbb{A}_f)$.\\

\noindent\underline{For Unitary Groups}\\
\noindent Assume the tame level group $K$ is neat. Let $c$ be an element in $\mathbb{Q}_+\backslash \mathbb{A}_{\mathbb{Q},f}^\times /\mu(K)$. We refer to \cite[Section 2.5]{Hsieh CM} for the notion of $c$-Igusa schemes $I_{\mathrm{U}(2)}^0(K,c)$ for the unitary groups $\mathrm{U}(2,0)$ (not the similitude group). It parameterizes quintuples $(A,\lambda,\iota,\bar{\eta}^{(p)},j)_{/S}$ similar to the Igusa schemes for unitary similitude groups but requires $\lambda$ to be a prime to $p$ $c$-polarization of $A$ such that $(A,\bar{\lambda},\iota,\bar{\eta}^{(p)},j)$ is a quintuple as in the definition of Shimura varieties for $\mathrm{GU}(2)$. For $g_c$ with $\mu(g)\in\mathbb{A}_\mathbb{Q}^\times$ in the class of $c$. Let ${}^c\!K=g_cKg_c^{-1}\cap U(2)(\mathbb{A}_{\mathbb{Q},f})$. Then the space $I_{\mathrm{U}(2)}^0(K,c)$ is isomorphic to the space of forms on $I_{\mathrm{U}(2)}^0({}^c\!K,1)$ (see \emph{loc.cit.}).\\

\noindent\underline{Embedding of Igusa Schemes}\\
\noindent In order to use the pullback formula geometrically we need a map from the Igusa scheme of $\mathrm{U}(3,1)\times\mathrm{U}(0,2)$ to that of $\mathrm{U}(3,3)$ (or from the Igusa scheme of $\mathrm{U}(2,0)\times\mathrm{U}(0,2)$ to that of $\mathrm{U}(2,2)$) given by:
\begin{equation}\label{Em}i([(A_1,\lambda_1,\iota_1,\eta_1^pK_1,j_1)],[(A_2,\lambda_2,\iota_2,\eta_2^pK_2,j_2)])=[(A_1\times A_2,\lambda_1\times\lambda_2,\iota_1\times\iota_2,(\eta_1^p\times \eta_2^p)K_3,j_1\times j_2)].
\end{equation}

\subsection{Fourier-Jacobi Expansions}\label{F-J-E}
\noindent\underline{Analytic Fourier-Jacobi Coefficients}:\\
\noindent Let $\beta\in\mathbb{Q}_+$. Over $\mathbb{C}$ we have the $\beta$-analytic Fourier-Jacobi coefficient for a holomorphic automorphic form $f$ on $G=\mathrm{GU}(3,1)$ given by:
$$\mathrm{FJ}_\beta(f,g)=\int_{\mathbb{Q}\backslash \mathbb{A}} f(\begin{pmatrix}1&&n\\&1_2&\\&&1\end{pmatrix}g)e_\mathbb{A}(-\beta n)dn.$$
The Haar measure is normalized so that the set $(\mathbb{Q}\backslash \mathbb{A})$ has measure $1$.

\noindent\underline{$p$-adic Cusps}\\
\noindent As in \cite{Hsieh CM} each pair $(g,j)\in C(K)\times\mathbf{H}$ can be regarded as a $p$-adic cusp, i.e. cusps of the Igusa tower. In the following we give the algebraic Fourier-Jacobi expansion at $p$-adic cusps.\\

\noindent\underline{Algebraic Theory for Fourier-Jacobi Expansions}\\
\noindent We follow \cite[pp. 16-17]{Hsieh CM} to give some background about the algebraic theory for Fourier-Jacobi expansion on the group $G=\mathrm{GU}(r,1)$. These are special cases developed by Lan (\cite{LAN1}, \cite{LAN2}). Recall $[g]$ is a cusp label corresponding to class $g\in G(\mathbb{A}_f)$. One defines $\mathcal{Z}_{[g]}$ \index{$\mathcal{Z}_{[g]}$} a group scheme over $S_{[g]}$ using the universal abelian variety as in \emph{loc.cit.} and denote $\mathcal{Z}_{[g]}^\circ$ \index{$\mathcal{Z}_{[g]}^\circ$} the connected component over $S_{[g]}$. There is a line bundle $\mathcal{L}(\beta)$ \index{$\mathcal{L}(\beta)$} on $\mathcal{Z}_{[g]}$ determined by $\beta$ \cite[Subsection 2.7.4]{Hsieh CM}.\\

\noindent Now let $f\in H^0(I_G(K_1^n)_{/R},\omega_\kappa)$ be a scalar weight $\kappa\geq 6$ (i.e. of weight $(0,0,0;\kappa)$) modular form over an $\mathcal{O}$ algebra $R$, then by (\cite[Subsection 3.6.2]{Hsieh CM}) there is a Fourier-Jacobi expansion of $f$ at the $p$-adic cusp $(g,h)$ for $h\in\mathbf{H}$:
$$\mathrm{FJ}_{[g]}^h(f)=\sum_{\beta\in \mathscr{S}_{[g]}}a_{[g]}^h(\beta,f)q^\beta$$
where
$$a_{[g]}^h(\beta,f)\in (\hat{A}_{[g]}^\infty\otimes_\mathcal{O} R)\otimes_{A_{[g]}} H^0(\mathcal{Z}_{[g]}^\circ,\mathcal{L}(\beta))$$
and $\mathscr{S}_{[g]}$ is a sub-lattice of $\mathbb{Q}$ determined by the level subgroup.
This is given by evaluating $f$ at the Mumford family $(\mathfrak{M},h^{-1}j_\mathfrak{M},\omega_\mathfrak{M})$ where $j_\mathfrak{M}$ is a fixed level structure (see \cite[Subsection 2.7.4]{Hsieh CM}). (Note that we do not have the subscript $N_H^1$ there since it is a scalar weight $\kappa$.)\\

\noindent \underline{Siegel Operators}\\

We have a Siegel operator $\Phi$ at the $p$-adic cusp $(g,h)$ defined by:

$$\Phi_{[g]}^h:H^0(I_G(K_1^n)_{/R},\omega_\kappa)\rightarrow A_{[g]}^n\otimes_{\mathcal{O}}R$$

$$f\mapsto \Phi_{[g]}^h(f):=a^h_{[g]}(0,f).$$
The Siegel operator at $[g]$ can be defined analytically as follows:\\
For any $g\in G(\mathbb{A}_f)$ we define:
\begin{equation}\label{SiegelOperators}\Phi_{P,g}(f)=\int_{N_P(\mathbb{Q})\backslash N_P(\mathbb{A}_\mathbb{Q})}f(ng)dn.
\end{equation}
We fix the Haar measure on $N_P(\mathbb{Q})\backslash N_P(\mathbb{A}_\mathbb{Q})$ as in \cite[Section 8.2]{SU}.
The relation between the algebraic and analytic Siegel operator is given in \cite[(3.12)]{Hsieh CM}.

\subsection{Weight Space for $\mathrm{GU}(3,1)$}
Let $H=\mathrm{GL}_3\times \mathrm{GL}_1$ and $T\subseteq H$ be the diagonal torus. Then $\mathbf{H}\simeq H(\mathbb{Z}_p)$. We let $\Lambda_2=\Lambda$ \index{$\Lambda_2$} be the completed group algebra $\mathbb{Z}_p[[T(1+p\mathbb{Z}_p)]]$. This is (non-canonically) isomorphic to a formal power series ring with four variables. There is an action of ${T}(\mathbb{Z}_p)$ on the Igusa scheme given by its action on the embedding $j:\mu_{p^n}\otimes_\mathbb{Z} N^0\hookrightarrow  A[p^n]$. (See \cite[Definition 3.4]{Hsieh CM}, which in turn, follows Hida's convention in \cite[Section 8.2]{Hida04}.) This gives the spaces of $p$-adic modular forms for $\mathrm{GU}(3,1)$ a structure of $\Lambda$-algebra. A $\bar{\mathbb{Q}}_p$-point $\phi$ of $\mathrm{Spec}\Lambda$ is call arithmetic if it is determined by a character $[\underline{k}]\cdot[\zeta]$ of $T(1+p\mathbb{Z}_p)\simeq (1+p\mathbb{Z}_p)^4$ where $\underline{k}$ is a weight and $\zeta=(\zeta_1,\zeta_2,\zeta_3;\zeta_4)$ for $\zeta_i\in \mu_{p^\infty}$. Here $[\underline{k}]$ is the character by regarding $\underline{k}$ as a character of $T(1+\mathbb{Z}_p)$ by ${[\underline{k}]}(t_1,t_2,t_3,t_4)=(t_1^{-c_4}t_2^{-c_3}t_3^{-c_2}t_4^{-c_1})$ and $[\zeta]$ is the finite order character given by mapping $(1+p)$ to $\zeta_i$ at the corresponding entry of $T(\mathbb{Z}_p)$. We often write this point $\underline{k}_\zeta$. We also define $\omega^{[\underline{k}]}$ as a character of the torsion part of $T(\mathbb{Z}_p)$ (canonically isomorphic to $(\mathbb{F}_p^\times)^4$) given by $\omega^{[\underline{k}]}(t_1,t_2,t_3,t_4)=\omega(t_1^{-c_4}t_2^{-c_3}t_3^{-c_2}t_4^{-c_1})$.\\

\noindent We can define the weight ring $\Lambda_P$ for the definite unitary group $G_P$ as well.

\subsection{Nearly Ordinary Forms}\label{Subsection 3.8}
Here for convenience we again follow Hsieh's treatment of Hida theory, but point out that the results are actually due to Hida \cite{Hida02}. We refer to \cite[3.8.3, 4.3]{Hsieh CM} for the definition of the $\mathrm{U}_p$ operator and Hida's idempotent $e$ acting on the space $V_{\infty,\infty}^\mathbf{N}$ of $p$-adic automorphic forms on $\mathrm{GU}(3,1)$ and the nearly ordinary subspace of the space of $p$-adic modular forms. The space of nearly ordinary automorphic forms (cusp forms) is denoted as $\mathcal{W}_{ord}$ ($\mathcal{W}_{ord}^0$). For $q=0$ or $\emptyset$ we let $\textbf{V}_{ord}^q$ be the Pontryagin dual of $\mathcal{W}_{ord}^q$. Then we have the following theorem (\cite[Theorem 4.21]{Hsieh CM})

\begin{theorem}
Let $q=0$ or $\emptyset$. Then:\\
(1) $\mathbf{V}_{ord}^q$ is a free $\Lambda$ module of finite rank;\\
(2) For any $\underline{k}$ very regular we have natural isomorphisms:
$$\mathcal{M}_{ord}^q(K,\Lambda)\otimes\Lambda/P_{\underline{k}}\xrightarrow{\sim} eM_{\underline{k}}^q(K,\mathcal{O}_p)$$
where $\mathcal{M}_{ord}^q(K, \Lambda)$ is defined in Definition \ref{Lambda Adic}. Here we identify $eM_{\underline{k}}^q(K,\mathcal{O}_{v_0})$ with its image in the space of $p$-adic automorphic forms of weight $\underline{k}$ under $\beta_{\underline{k}}$ for the map $\beta_{\underline{k}}$ defined in \cite[Equation (3.3)]{Hsieh CM}.
\end{theorem}

\begin{remark}
If $\mathcal{K}$ is a general CM field, then the statement of the corresponding result is more complicated; see \cite[Section 4.5]{Hsieh CM}.
\end{remark}

\subsection{$\Lambda$-adic Forms}
\begin{definition}\label{Lambda Adic}
For any finite $\Lambda$ algebra $R$, and $q=0$ or $\emptyset$ we define the space of $R$-adic ordinary forms to be:
$$\mathcal{M}^q_{ord}(K,R):=\mathrm{Hom}_{\Lambda}(\mathbf{V}_{ord}^q,R).$$
Similarly, if $R$ is a $\Lambda_P$-algebra, then we define:
$$\mathcal{M}_{ord,[g],P}(K_{P,[g]}, R):=\mathrm{Hom}_{\Lambda_P}(\mathbf{V}_{ord,P,[g]},R).$$
Here the subscript $[g]$ means that the prime to $p$ level group is $K_P^{g}$ as defined previously.
\end{definition}

\noindent For any $f\in \mathcal{M}_{ord}(K,R)$ we have an $R$-adic Fourier-Jacobi expansion:
\begin{equation}\label{FJEX}
\mathrm{FJ}_{[g]}^h(f)=\sum_{\beta\in \mathscr{S}_{[g]}}a_{[g]}^h(\beta,f)q^\beta
\end{equation}
obtained from the Fourier-Jacobi expansion on $\mathcal{W}_{ord}^q$,
where $a_{[g]}^h(\beta,f)\in R\hat{\otimes}\hat{A}^\infty_{[g]}\otimes_{A_{[g]}}H^0(\mathcal{Z}_{[g]}^\circ,\mathcal{L}(\beta))$ (see \cite[Subsection 4.6.1]{Hsieh CM}).
We also have an $R$-adic Siegel operator which we denote as $\hat{\Phi}_{[g]}^h$. Let $w_3'=\begin{pmatrix}&&1&\\1&&&\\&1&&\\&&&1\end{pmatrix}\in \mathrm{GL}_4(\mathbb{Z}_p)\simeq \mathrm{U}(3,1)(\mathbb{Z}_p)$. (Notice that we used the place $v_0$ to identify $\mathrm{GL}_4(\mathbb{Z}_p)$ with $\mathrm{U}(3,1)(\mathbb{Z}_p)$ here. We use $w_3'$ instead of $w_3$ as in \cite[pp. 35]{Hsieh CM} to distinguish it from $w_3\in \mathrm{U}(3,3)$.) Now we have the following important theorem
\begin{theorem}\label{exact}
\cite[Theorem 4.26]{Hsieh CM} Let $R$ be as before. We have the following short exact sequence
$$0\rightarrow \mathcal{M}_{ord}^0(K,R)\rightarrow \mathcal{M}_{ord}(K,R)\xrightarrow{\hat{\Phi}^{w_3'}=\oplus \hat{\Phi}_{[g]}^{w_3'}}\oplus_{g\in C(K)}\mathcal{M}_{ord}(K_P^g,R)\rightarrow 0.$$
\end{theorem}
\noindent We need one more theorem which gives another definition of nearly ordinary $p$-adic modular forms using Fourier-Jacobi expansions.
\begin{definition}\label{definition 3.7}
Let $R$ be a finite torsion free $\Lambda$-algebra. Let $X(K)$ be the set $\{(g,h)\}$ where $g$ runs over a set of representatives of cusp labels $C(K)$ and $h$ runs over $\mathbf{T}$ which is the diagonal torus of $\mathbf{H}$. Let $\mathcal{N}_{ord}(K,R)$ be the set of formal Fourier-Jacobi expansions:
$$F=\{\sum_{\beta\in\mathscr{S}_{[g]}}a(\beta, F)q^\beta, a(\beta,F)\in (R\hat{\otimes}\hat{A}_{[g]}^\infty)^\Lambda \otimes H^0(\mathcal{Z}_{[g]}^\circ,\mathcal{L}(\beta))\}_{g\in X(K)}$$
(here $\hat\otimes$ means completed tensor product, and the superscript $\Lambda$ in $(R\hat{\otimes}\hat{A}_{[g]}^\infty)^\Lambda$ means that the $\Lambda$-action as a nebentypus character is compatible with the $\Lambda$-algebra structure of $R$), such that for a Zariski dense set $\mathcal{X}_F$ of points $\phi\in\mathrm{Spec}(R)$ such that the induced point in $\mathrm{Spec}(\Lambda)$ is some arithmetic weight $\underline{k}_\zeta$, the specialization $F_\phi$ of $F$ is the Fourier-Jacobi expansion of a nearly ordinary modular form with prime to $p$ level group $K$, weight $\underline{k}$ and nebentype at $p$ given by $[\underline{k}][\underline{\zeta}]\omega^{-[\underline{k}]}$.
\end{definition}
\noindent Then we have the following theorem (\cite[Theorem 4.25]{Hsieh CM})
\begin{theorem}\label{4.25}
$$\mathcal{M}_{ord}(K,R)=\mathcal{N}_{ord}(K,R).$$
\end{theorem}
\begin{remark}
The proof uses the $p$-adic $q$-expansion principle for $\mathrm{GU}(r,1)$, which is proved by Hida \cite[Theorem 0.1]{Hida09a} (also recalled in \cite[Subsection 3.6.4]{Hsieh CM}). The $q$-expansion principle follows from the Irreducibility of the Igusa scheme. As mentioned by Hida in \emph{loc.cit.} the Igusa scheme is not quite irreducible: in fact the component group is isomorphic to the quotient of $\mathrm{GL}_3(\mathbb{Z}_p)\times\mathrm{GL}_1(\mathbb{Z}_p)$ over his $M_1$, which is the subgroup consisting of matrices $(g_1,g_2)$ with $\det g_1=\det g_2$. (Hida proved the monodromy group, i.e. the image of $\pi_1(\bar{S}_K(G)_{/\mathbb{F}_p},s)$ in  $\mathrm{GL}_3(\mathbb{Z}_p)\times\mathrm{GL}_1(\mathbb{Z}_p)$ is exactly this $M_1$.) By our definition of $X(K)$ above, it clearly contains a representative of this quotient. So we still have the $q$-expansion principle.
\end{remark}

\section{Background for Theta Functions}\label{section 3}
Now we recall briefly the basic notions of theta functions and theta liftings, following closely to \cite{ZhB07} with some modifications.
\subsection{Heisenberg Group}
Let $W$ be a finite-dimensional vector space over $\mathbb{Q}_v$ with a non-degenerate alternating form $\langle,\rangle$. We define:
$$H(W):=\{(w,t)|w\in W,t\in \mathbb{Q}_v\}$$
with multiplication law:
$(w_1,t_1)(w_2,t_2)=(w_1+w_2,t_1+t_2+\frac{1}{2}\langle w_1,w_2\rangle)$.
\subsection{Schr\"odinger Representation}
Fix an additive character $\psi$ of $\mathbb{Q}_v$ and a complete polarization as $W=X\oplus Y$ of $W$ where $X$ and $Y$ are maximal totally isotropic subspaces of $W$. Let $S(X)$ be a space of Bruhat-Schwartz functions on $X$, and define a representation $\rho_\psi$ of $H(W)$ on $S(X)$ by:
$$\rho_\psi(x)f(z)=f(x+z),x\in X$$
$$\rho_\psi(y)f(z)=\psi(\langle z,y\rangle)f(z),y\in Y$$
$$\rho_\psi(t)f(z)=\psi(t)f(z),t\in \mathbb{Q}_v$$
This is called the Schr\"odinger representation. By the theorem of Stone and von Neumann, $\rho_\psi$ is the unique irreducible smooth representation on which
$\mathbb{Q}_v=\{(0,t)|t\in\mathbb{Q}_v\}$ acts via the character $\psi$.
\subsection{Metaplectic Groups and Weil Representations}\label{3.3}
Let $Sp(W)$ be the symplectic group preserving the alternating form $\langle,\rangle$ on $W$. Then $Sp(W)$ acts on $H(W)$ by $(w,t)g=(gw,t)$ (we use column vectors for $w\in W$ and the left action of $Sp(W)$as \cite{ZhB07}). By the uniqueness of $\rho_\psi$, there is an operator $\omega_\psi(g)$ on $S(X)$, determined up to scalar, such that
$$\rho_\psi(gw,t)\omega_\psi(g)=\omega_\psi(g)\rho_\psi(w,t)$$
for any $(w,t)\in H(W)$.
Define the metaplectic group $\tilde{Sp}_\psi(W)=\{(g,\omega_\psi(g)) \mbox{ as above }\}$ which we often abbreviate as $\tilde{Sp}$ for short. Thus $\tilde{Sp}(W)$ has an action $\omega_\psi$ on $S(X)$ called the Weil representation.\\

\noindent Now suppose $\psi=\prod_v\psi_v$ is a global additive character of $\mathbb{Q}\backslash\mathbb{A}_\mathbb{Q}$ and $W$ is a finite-dimensional vector space over $\mathbb{Q}$ equipped with an alternating pairing $\langle , \rangle$. We can put the above construction together for all $v$'s to get a representation of $\tilde{Sp}(W)(\mathbb{A})$ on $S(X(\mathbb{A}))$. This can be viewed as a projective representation of ${Sp}(W)$ (a representation with image in the infinite dimensional projective linear group). We now give formulas for this representation. Let $\{e_1,...,e_n;f_1,...,f_n\}$ be a basis of $W=X\oplus Y$ such that $\langle e_i,f_j\rangle=\delta_{ij}$. With respect to this basis, we write elements in $X$ in row vectors, and the projective representation of $\tilde{Sp}(W)(\mathbb{A}_\mathbb{Q})$ on $\mathrm{Proj}S(X(\mathbb{A}))$ is given by the formulas
\begin{itemize}
\item $\omega_\psi(\begin{pmatrix}A&\\&{}^t\!A^{-1}\end{pmatrix})\phi(x)=|\det A|^{\frac{1}{2}}\phi(xA)$;
\item $\omega_\psi(\begin{pmatrix}1&B\\&1\end{pmatrix})\phi(x)=\psi(\frac{xB{}^t\!x}{2})\phi(x)$;
\item $\omega_\psi(\begin{pmatrix}&1\\-1&\end{pmatrix})\phi(x)=\gamma\hat{\phi}(x)$ where $\hat{\phi}$ means the Fourier transform of $\phi$ with respect to the additive character $\psi$. The $\gamma$ is an $8$-th root of unity which is called the Weil constant.
\end{itemize}

\subsection{Dual Reductive Pairs}
A dual reductive pair is a pair of subgroups $(G,G')$ in the symplectic group $Sp(W)$ satisfying:
\begin{itemize}
\item[(1)] $G$ is the centralizer of $G'$ in $Sp(W)$ and vice versa;
\item[(2)] the action of $G$ and $G'$ are completely reducible on $W$.
\end{itemize}
We are mainly interested in the following dual reductive pairs of unitary groups. Let $\mathcal{K}$ be a quadratic imaginary extension of $\mathbb{Q}$, $(V_1,(,)_1)$ be a skew Hermitian space over $\mathcal{K}$ and $(V_2,(,)_2)$ a Hermitian space over $\mathcal{K}$. Then the unitary groups $\mathrm{U}(V_1)$ and $\mathrm{U}(V_2)$ form a dual reductive pair in $\mathrm{Sp}(W)$, where $W=V_1\otimes V_2$ is given the alternating form $\frac{1}{2}\mathrm{tr}_{\mathcal{K}/\mathbb{Q}}((,)_1\otimes\overline{(,)_2})$ over $\mathbb{Q}$. The embedding of the dual reductive pair $(\mathrm{U}(V_1),\mathrm{U}(V_2))$ into $\mathrm{Sp}(W)$ is
$$e:\mathrm{U}(V_1)\times \mathrm{U}(V_2)\rightarrow \mathrm{Sp}(W)$$
$$e(g_1,g_2)\cdot(v_1\otimes v_2)=v_1g_1\otimes g_2^{-1}v_2.$$
\subsection{Splittings}
Suppose $\mathrm{dim}_\mathcal{K}V_1=n$ and $\mathrm{dim}_\mathcal{K}V_2=m$. If $\chi_1$ and $\chi_2$ are Hecke characters of $\mathcal{K}^\times$ such that $\chi_1|_{\mathbb{A}_\mathbb{Q}^\times}=\chi_{\mathcal{K}/\mathbb{Q}}^n$ and $\chi_2|_{\mathbb{A}_\mathbb{Q}^\times}=\chi_{\mathcal{K}/\mathbb{Q}}^m$, then there is a splitting
(see \cite[Section 1]{HKS})
$$s:\mathrm{U}(V_1)\times\mathrm{U}(V_2)\rightarrow \tilde{Sp}(W)$$
determined by $\chi_2$ and $\chi_1$. This enables us to define the Weil representations of $\mathrm{U}(V_1)\times \mathrm{U}(V_2)$ on $S(X(\mathbb{A}))$ which we denote as $\omega_{\chi_1,\chi_2}=\omega_{\chi_1}\otimes\omega_{\chi_2}$.

\subsection{Theta Functions}
Now let us define theta functions.
\begin{definition}\label{Define Theta}
Let $\phi\in S(X(\mathbb{A}_\mathbb{Q}))$. Define the theta kernel function
$$\theta(\phi)=\sum_{l\in X(\mathbb{Q})}\phi(l).$$
\end{definition}
Let $J=H(W)\ltimes Sp(W)$ ($\tilde{J}=H(W)\ltimes\tilde{Sp}(W)$) be the Jacobi group with $Sp(W)$ acting on $H(W)$ by $(w,t)\cdot g=(wg,t)$ ($\tilde{Sp}(W)$ acts on $H(W)$ by $(w,t)\cdot\tilde{g}=(wg,t)$, where $g$ is the image of $\tilde{g}$ in $Sp(W)$). We define a theta kernel on $\tilde{J}(\mathbb{A}_\mathbb{Q})$ as below.
\begin{definition}
Let $\tilde{g}\in \tilde{Sp}(W)$ and $(w,t)\in H(W)$, define
$$\theta_\phi((w,t)\tilde{g})=\sum_{l\in X(\mathbb{Q})}\rho_\psi(w,t)\tilde{g}.\phi(l).$$
\end{definition}
Using the Weil representation of the dual reductive pair above (with the choices of the splitting characters) we define the theta kernel for the theta correspondence as follows:
\begin{definition}
$$\theta_\phi(g_1,g_2)=\theta(\omega(g_1,g_2))\phi).$$
\end{definition}
\subsection{Intertwining Maps}\label{3.6}
Here, we study the intertwining maps between theta series corresponding to different polarizations $(X,Y)$ of $W$. Suppose $r\in Sp(W)$, then $(rX,rY)$ gives another polarization of $W$, and all polarizations are obtained this way. If $\phi\in S(X)$ then we define an intertwining map (local or global) $\delta_\psi:S(X)\rightarrow S(rX)$ by
\begin{equation}\label{generalform}
\delta_\psi\phi(xr)=\omega_\psi(r^{-1})\phi(x)
\end{equation}
for $x\in X$. We can see that $\delta_\psi$ is an isometry intertwining the actions of $\tilde J$.\\

Let $W^-$ be the skew Hermitian space which is isomorphic to $W$ as $\mathbb{Q}_v$-vector spaces but equipped with the alternating pairing $-\langle,\rangle$. For a polarization $(X,Y)$ of $W$ we present the intertwining formula for the two polarizations $(X\oplus X^-)\oplus (Y\oplus Y^-)$ and $\{w\oplus w, w\in W\}\oplus \{w\oplus-w, w\in W\}$ of $W\oplus W^-$. We write the formula for the map $\delta_\psi: S(X(\mathbb{Q}_v)\oplus X^-(\mathbb{Q}_v))\rightarrow S(W(\mathbb{Q}_v))$ \index{$\delta_\psi$} and its inverse:
\begin{align}\label{intertwining operator}
&\delta_\psi(\phi)(x_1,y)=\int\psi(2\langle x_2,y\rangle)\phi(x_1+x_2,x_2-x_1)dx_2&\nonumber\\
&\delta_\psi^{-1}(\phi)(x_1,x_2)=\int\psi(\langle-x_1-x_2,y\rangle)\phi(x_1-x_2,\frac{y}{2})dy.&
\end{align}
Another easy property is that if the two polarizations $(X,Y)$ and $(rX,rY)$ are globally defined, then the theta kernels $\Theta_\phi$ and $\Theta_{\delta_\psi(\phi)}$ are defined and
$$\Theta_\phi(u,ng)=\Theta_{\delta_\psi(\phi)}(u,(rn)g).$$

\subsection{Special Cases}\label{3.7}
Here we give two special cases which are used later. Case One is used in Section 8.3 to construct families of theta functions on $\mathrm{U}(2)$.  Case Two is used in the computation of Fourier-Jacobi coefficients for the Siegel-Eisenstein series on $\mathrm{U}(3,3)$ as a finite sum of products of Siegel-Eisenstein series and theta functions.
\subsubsection{Case One}
\noindent We write $V$ for the two-dimensional Hermitian space over $\mathcal{K}$ for $\zeta/\delta$ with respect to the basis $(v_1,v_2)$, $V^-$ for the Hermitian space for $-\zeta/\delta$ with respect to the basis $(v_1^-,v_2^-)$, and $V_1$ for the one-dimensional skew Hermitian space with the metric $\delta$ with respect to the basis $v$. Let $W=V\otimes V_1$ and $W^-=V^-\otimes V_1$. We define several polarizations for the Hermitian space $\mathbb{W}:=W\oplus W^-$ (the alternating pairing being the direct sum of those for $W$ and $W^-$).
\begin{definition}\label{definition4.4}
$$X:=\mathbb{Q}v_1\otimes v\oplus \mathbb{Q}v_2\otimes v$$
$$X^-:=\mathbb{Q}v_1^-\otimes v\oplus \mathbb{Q}v_2^-\otimes v$$
$$Y:=\mathbb{Q}\delta v_1\otimes v\oplus \mathbb{Q}\delta v_2\otimes v$$
$$Y^-:=\mathbb{Q}\delta v_1^-\otimes v\oplus \mathbb{Q}\delta v_2^{-}\otimes v.$$
\end{definition}
Fix the additive character $\psi=\prod\psi_v$. Thus $W=X\oplus Y$ and $W^-=X^-\oplus Y^-$ are globally defined polarizations. For a split prime $v$ we write $v=w\bar{w}$ for its decomposition in $\mathcal{K}$. We will often use an auxiliary polarization $W_v=X_v'\oplus Y_v'$ \index{$X_v', Y_v'$} of $W_v=W\otimes_{\mathcal{K}}\mathcal{K}_v$ with respect to $\mathcal{K}_v\simeq \mathcal{K}_w\times \mathcal{K}_{\bar{w}}\simeq \mathbb{Q}_v^2$ and $W_v=X^{'-}_v\oplus Y^{'-}_v$ that is defined by $X'_v=\mathcal{K}_wv_1\otimes v\oplus \mathcal{K}_w v_2\otimes v, Y'_v=\mathcal{K}_{\bar{w}}v_1\otimes v\oplus \mathcal{K}_{\bar{w}}v_2\otimes v$ and similar for $X^{'-}_v, Y^{'-}_v$. This polarization is better suited for computing the Weil representation. For split primes $v$ let $\delta_\psi'': S(X_v')\rightarrow S(X_v)$ and $\delta_\psi^{-,''}: S(X_v^{-,'})\rightarrow S(X_v^{-})$ \index{$\delta_\psi''$, $\delta_\psi^{-,''}$} be the intertwining operators between the Schwartz functions defined above.\\

\noindent Let $\mathbb{W}^d=\{w\oplus w, w\in W\}$ ($d$ stands for diagonal). We denote the intertwining maps:
$$\delta_\psi: S(X_v\oplus X^-_v)\rightarrow S(\mathbb{W}^d_v)$$
\index{$\delta_\psi$} and if $v$ splits,
$$\delta_\psi': S(X_v'\oplus X_v^{'-})\rightarrow S(\mathbb{W}^d_v).$$
\index{$\delta_\psi'$}
Recall the formulas given in \ref{intertwining operator}.
\begin{remark}
In application in Section \ref{section 5} we compute the intertwining operator
$$\delta_\psi: S(X_v\oplus X^-_v)\rightarrow S(\mathbb{W}^d_v)$$
(for $\mathbb{W}=(V\oplus V^-)\otimes V_1$) in this special case and the Weil representations restricting to semi-direct products $H(\mathbb{W})\ltimes (\mathrm{U}(V\oplus V^-)\times \mathrm{U}(V_1))$ (recall $\mathrm{U}(V\oplus V^-)\times \mathrm{U}(V_1)\hookrightarrow Sp(\mathbb{W})$). We provide the matrix forms of these semi-direct products that will be used in Section \ref{section 5}. Let $\mathrm{U}_1$ and $\mathrm{U}_2$ be unitary groups associated to the matrices $\begin{pmatrix}&&1&\\&\zeta &&\\-1&&&\\&&&-\zeta \end{pmatrix}$ and $\begin{pmatrix}&1_3\\-1_3&\end{pmatrix}$ respectively, and let $\mathrm{U}_1'$ and $\mathrm{U}_2'$ be the unitary groups associated to $\begin{pmatrix}\zeta&\\&-\zeta\end{pmatrix}$ and $\begin{pmatrix}&1_2\\-1_2&\end{pmatrix}$, considered as subgroups of $\mathrm{U}_1$ and $\mathrm{U}_2$ respectively in the obvious way. Let $N_1$ be the subgroup of $\mathrm{U}_1$ consisting of matrices of the form $\begin{pmatrix}1&x_1&*&x_2\\&1_2&\zeta x_1^*&\\&&1&\\&&-\zeta x_2^*&1_2\end{pmatrix}$, and $N_2\subset \mathrm{U}_2$ the subgroup consisting of matrices of the form $\begin{pmatrix}1&x&t+\frac{1}{2}(xy^*-yx^*)&y\\&1_2&y^*&0\\&&1&\\&&-x^*&I_2\end{pmatrix}$. The corresponding semi-direct products mentioned above are $J_1=N_1\mathrm{U}_1'$ and $J_2=N_2\mathrm{U}_2'$. These are used later in computing Fourier-Jacobi coefficients (see Definition \ref{defJacobi} for $J_2$).
\end{remark}

\subsubsection{Case Two}
Now we discuss another special situation which will be used in the Fourier-Jacobi coefficient computations for the Siegel-Eisenstein series on $\mathrm{GU}(3,3)$.\\

\noindent\underline{The local set-up.}\\
\noindent Let $v$ be a place of $\mathbb{Q}$. Let $h\in S_1(\mathbb{Q}_v), h\not=0$. Let $\mathrm{U}_h$ be the unitary group of this matrix and let $V_v$ be the corresponding one-dimensional Hermitian space. Let $$V_{2,v}=\mathcal{K}^{2}_v\oplus \mathcal{K}^{2}_v=\mathbf{X}_v\oplus \mathbf{Y}_v$$ be the Hermitian space associated to $\mathrm{U}_2=\mathrm{U}(2,2)$ with the alternating pairing denoted as $\langle,\rangle_2$. Let $\mathbf{W}=V_v\otimes_{\mathcal{K}_v} V_{2,v}$. Then $$(-,-):=\mathrm{Tr}_{\mathcal{K}_v/\mathbb{Q}_v}(\langle-,-\rangle_h\otimes_{\mathcal{K}_v}\langle-,-\rangle_{2})$$ is a $\mathbb{Q}_v$ linear pairing on $\mathbf{W}$ that makes $\mathbf{W}$ into an eight-dimensional symplectic space over $\mathbb{Q}_v$. The canonical embedding of $\mathrm{U}_h\times \mathrm{U}_{2}$ into $Sp(\mathbf{W})$ realizes the pair $(\mathrm{U}_h,\mathrm{U}_{2})$ as a dual pair in $Sp(\mathbf{W})$. Let $\lambda_v$ be a character of $\mathcal{K}_v^\times$ such that $\lambda_v|_{\mathbb{Q}_v^\times}=\chi_{\mathcal{K}/\mathbb{Q},v}$. As noted earlier, there is a splitting $\mathrm{U}_h(\mathbb{Q}_v)\times \mathrm{U}_{2}(\mathbb{Q}_v)\hookrightarrow \tilde{Sp}(\mathbf{W},\mathbb{Q}_v)$ of the metaplectic cover $\tilde{Sp}(\mathbf{W},\mathbb{Q}_v)\rightarrow Sp(\mathbf{W},\mathbb{Q}_v)$ determined by the character $\lambda_v$. This gives the Weil representation $\omega_{\lambda_v,1}$, which we denote here as $\omega_{h,\lambda_v}(u,g)$ of $\mathrm{U}_h(\mathbb{Q}_v)\times \mathrm{U}_{2}(\mathbb{Q}_v)$ where $u\in \mathrm{U}_h(\mathbb{Q}_v)$ and $g\in \mathrm{U}_{2}(\mathbb{Q}_v)$, via the Weil representation of $\tilde{Sp}(\mathbf{W},\mathbb{Q}_v)$ on the space of Schwartz functions $\mathcal{S}(V_v\otimes_{\mathcal{K}_v} \mathbf{X}_v)$ (we use the polarization $\mathbf{W}=V_v\otimes_{\mathcal{K}_v}\mathbf{X}_v\oplus V_v\otimes_{\mathcal{K}_v}\mathbf{Y}_v$). Moreover, we write $\omega_{h,\lambda_v}(g)$ to mean $\omega_{h,\lambda_v}(1,g)$. For $X\in M_{1\times2}(\mathcal{K}_v)$, we define $\langle X,X\rangle_h:={}^t\!\bar{X}hX$ \index{$\langle X,X\rangle_h$} (note that this is a $2\times 2$ matrix). We record here some useful formulas for $\omega_{h,\lambda_v}$ which are generalizations of the formulas in \cite[Section 10.1]{SU}.

\begin{itemize}
\item  $\omega_{h,\lambda_v}(u,g)\Phi(X)=\omega_{h,v,\lambda_v}(1,g)\Phi(u^{-1}X)$
\item $\omega_{h,\lambda_v}(\mathrm{diag}(A,{}^t\!\bar {A}^{-1}))\Phi(X)=\lambda(\det A)|\det A|^{\frac{1}{2}}_\mathcal{K}\Phi(XA),$
\item $\omega_{h,\lambda_v}(r(S))\Phi(X)=\Phi(X)e_v(\mathrm{tr}\langle X,X\rangle_h S),$
\item $\omega_{h,\lambda_v}(\eta)\Phi(X)=|\det h|_v\int\Phi(Y)e_v(\mathrm{Tr}_{\mathcal{K}_v/F_v}(\mathrm{tr}\langle Y,X\rangle_h))dY.$
\end{itemize}

\noindent \underline{The global set-up}:\\
\noindent Let $h\in S_1(\mathbb{Q}),h>0$. We can define global versions of $\mathrm{U}_h,\mathrm{GU}_h,\mathbf{W}$, and $(-,-)$, similar as the local case above. Fixing an idele class character $\lambda=\otimes\lambda_v$ of $\mathbb{A}_{\mathcal{K}}^\times/\mathcal{K}^\times$ such that $\lambda|_{\mathbb{Q}^\times}=\chi_{\mathcal{K}/\mathbb{Q}}^1$, the associated local splitting described above then determines a global splitting $\mathrm{U}_h(\mathbb{A}_\mathbb{Q})\times \mathrm{U}_1 (\mathbb{A}_\mathbb{Q})\hookrightarrow \tilde{Sp}(\mathbf{W},\mathbb{A}_\mathbb{Q})$ and hence an action $\omega_h:=\otimes\omega_{h,\lambda_v}$ of $\mathrm{U}_h(\mathbb{A}_\mathbb{Q})\times \mathrm{U}_1(\mathbb{A}_\mathbb{Q})$ on the Schwartz space $S(V_{\mathbb{A}_\mathcal{K}}\otimes \mathbf{X})$. In application, we require the infinity type of $\lambda$ to be $(-\frac{1}{2},\frac{1}{2})$. \\

\noindent For any Schwartz function $\Phi\in S(V_{\mathbb{A}_\mathcal{K}}\otimes \mathbf{X})$ we define the theta function
$$\Theta_h(\Phi, -): J(\mathbb{A})\rightarrow \mathbb{C}.$$

associated to it by:
$$\Theta_h(\Phi,u,g)=\sum_{x\in V\otimes X}\omega_{h,\lambda}(g)\Phi(x), g\in J(\mathbb{A}).$$

\subsection{Theta Functions with Complex Multiplication}\label{ThetaCM}

\noindent We consider the situation of theta correspondences for $\mathrm{U}(\zeta)=\mathrm{U}(V)$ and $\mathrm{U}(V_1)$.
Let $V$ be a two-dimensional Hermitian vector space over $\mathcal{K}$. Let $L$ be an $\mathcal{O}_\mathcal{K}$ lattice such that it gives an abelian variety $\mathcal{A}_L=\mathbb{C}^2/L$.

Let $H$ be a Riemann form on $V$ and $\epsilon: L\rightarrow U$ be a map where $U$ is the unit circle of $\mathbb{C}$ (in application the $\epsilon$ is given by the formula after \cite[(38)]{ZhB07}, there is a line bundle $\mathfrak{L}_{H,\epsilon}$ on $\mathcal{A}_\mathcal{L}$ associated to $H$ and $\epsilon$ as follows: define an analytic line bundle $\mathfrak{L}_{H,\epsilon}\simeq \mathbb{C}\times\mathbb{C}^2/L$ with the action of $L$ given by
$$l\cdot (w,x)=(w+l,\epsilon(l)e(\frac{1}{2i}H(l,w+\frac{l}{2}))x), l\in L, (w,x)\in\mathbb{C}^2\times\mathbb{C}$$
where $e(x)=e^{2\pi ix}$. The space of global sections of this line bundle is canonically identified with the space $T(H,\epsilon,L)$ of theta functions consisting of holomorphic functions $f$ on $V$ such that:
$$f(w+l)=f(w)\epsilon(l)e(\frac{1}{2i}H(l+w+\frac{l}{2})),w\in V,l\in L.$$
There are arithmetic models for the above abelian variety and line bundle. Shimura defined subspaces $T^{ar}(H,\epsilon,L)\subset T(H,\epsilon,L)$ of arithmetic theta functions by requiring that the values at all CM points are in $\bar{\mathbb{Q}}$ which under the canonical identification, are identified with rational sections of the line bundle (see \cite[Section 2.6]{LAN1}, and also \cite[Appendix B]{ZhB07} for a discussion in terms of theta kernel functions. Note that we only need to consider arithmetic sections, but not integral sections.)\\

\noindent\underline{Adelic Theta Functions}\\
\noindent Now we consider theta functions for $\mathrm{U}(3,1)$. Let the Hermitian form on $V$ be defined by:
$$\langle v_1,v_2\rangle=v_1\zeta v_2^*-v_2\zeta v_1^*.$$
Let $U_f$ be some compact open subgroup of $U(\zeta)(\mathbb{A}_f)$.
\begin{definition}\label{adelictheta}
We define the space $T_{\mathbb{A}}(m,L,U_f)$ of adelic theta functions as the space of function:
$$\Theta:N(\mathbb{Q})\mathrm{U}(\zeta)(\mathbb{Q})\backslash N(\mathbb{A})\mathrm{U}(\zeta)(\mathbb{A})/\mathrm{U}(\zeta)_\infty \mathrm{U}_fN(L)_f\rightarrow \mathbb{C},$$
where $N=N_P\subset U(3,1)$ and $U(\zeta)\hookrightarrow U(3,1)$ as before;
$$N(L)_f=\{(w,t)|x\in \hat{L},t+\frac{w\zeta w^*}{2}\in\mu(L)\hat{\mathcal{O}}_\mathcal{K}\},$$
where $\mu(L)$ is the ideal generated by $w\zeta w^*$ for $w\in L$ and $\Theta$ satisfies
$$\Theta((0,t)r)=e(m t)\Theta(r), r\in N(\mathbb{A})U(\zeta)(\mathbb{A}).$$
\end{definition}

Since $\mathrm{U}(\zeta)$ is anisotropic, the set $\mathrm{U}(\zeta)\backslash U(\zeta)(\mathbb{A})/\mathrm{U}(\zeta)_\infty \mathrm{U}_f$ consists of finitely many points $\{x_1,...,x_s\}\subset \mathrm{U}(\zeta)(\mathbb{A}_f)$. We assume that for each $x_i$ the $p$-component is within $\mathrm{GL}_2(\mathbb{Z}_p)$ under the first projection $\mathrm{U}(\mathbb{Q}_p)\simeq \mathrm{GL}_2(\mathbb{Q}_p)$. In this paper, we consider the case satisfying the following\\

\noindent\underline{Assumption}:\\
\noindent The lattice $L$ is such that $m(\delta^{-1}_{\mathcal{K}/\mathbb{Q}})l\zeta{}^t\! \bar{l}$ is always integral for any $i$ and $l\in x_iL$.\\

Then for any
$$\Theta\in T_\mathbb{A}(m,L,U_f),$$
write
$\Theta_i(n)=\Theta(nx_i)$ for $n\in N(\mathbb{A})$ as functions on $N(\mathbb{A})$. Then for each $i$ we define the function:
$$\theta_i(w_\infty)=e(-m\frac{w\zeta w^*}{2})\Theta_i((w_\infty,0)).$$ If this function is holomorphic then it is a classical theta function in $T(H,\epsilon, x_i L)$ where $H$ and $\epsilon$ are defined as follows. The
$$H(v,v):=-2miv\zeta{}^t\!\bar{v}.$$
As in \cite[Subsection 7.2.3]{Hsieh CM}, we choose a finite idele $u=(u_v)\in\mathbb{A}_{\mathcal{K},f})$ such that $\mathcal{O}_{\mathcal{K}_v}=\mathbb{Z}_v\oplus\mathbb{Z}_v u_v$ for each finite place $v$. For each $l\in x_i L$ let $x_l$ and $y_l$ be the unique elements in $\mathbb{A}_{\mathbb{Q},f}$ such that
$$\frac{1}{2}l\zeta{}^t\!\bar{l}=x_l+y_lu.$$
Let $\psi$ be the standard additive character $\mathbb{Q}\backslash \mathbb{A}_\mathbb{Q}$ such that at the Archimedean place it is given by $\psi(x_\infty)=e^{2\pi ix_\infty}$. Let $\psi_m(x)=\psi(m\mathrm{tr}_{\mathcal{K}/\mathbb{Q}}(x)$ for $x\in\mathbb{A}_\mathcal{K}$. Define $\epsilon(l)=\psi_m(-x_l)$. Then under our assumptions the $\epsilon$ is well defined and takes a value of $\pm1$.

If it is the case that for all $i$ the $\theta_i$ is holomorphic, then we say it is a holomorphic adelic theta function and write the corresponding space as $T^{\mathrm{Hol}}_{\mathbb{A}}(m,L,U_f)$.\\

We make the following identification (recall we defined $\mathcal{Z}^\circ_{[g]}$ in Subsection \ref{F-J-E})
$$H^0(\mathcal{Z}^\circ_{[g]},\mathcal{L}(\beta))\otimes\mathbb{C}\simeq T^{\mathrm{Hol}}_{\mathbb{A}}(\beta,L,U_f)$$
for certain level group $U_f=U_p\times\prod_{v\nmid p}U_v$.

\noindent\underline{A functional}\\
Recall that we constructed a theta function $\theta_\phi$ on the Jacobi group $J(V)$ from the Schwartz function $\phi$. As mentioned in the introduction, we only need to develop a rational theory on theta functions instead of $p$-integral theory. Upon choosing $v_1\in V_1$ such that $\langle v_1,v_1\rangle=1$ we have an isomorphism $V\simeq W=V\otimes V_1$. We also consider $W^-=V^-\otimes V_1$. It is the space $W$ but with the metric being the negative of $W$. Let $H^-=H(W^-)$ be the corresponding Heisenberg group. We have an isomorphism of $H$ and $H^-$ (as Heisenberg groups) given by:
$$(w,t)\rightarrow (w,-t).$$
We construct a theta function $\theta^\star=\theta_{\phi_1}$ on $H^-\rtimes \mathrm{U}(V^-)$ for the Schwartz function $\phi_1$. We have chosen a set $\{x_1,...,x_s\}$ above.
We write
$$\langle\theta_{\phi_1}, \theta_{\phi}\rangle(x_i):=\int_{[N]}\theta_{\phi_1}(nx_i)\theta_{\phi}(nx_i)dn.$$
and
$$(\phi_1,\phi)_{x_i}=\int_{X(\mathbb{A}_\mathbb{Q})}\omega_{\lambda^{-1}}(x_i)(\phi_1)(x)\omega_\lambda(x_i)(\phi)(x)dx.$$
Then it is easy to check for each $x_i$,
\begin{equation}\label{(20)}\langle \theta_{\phi_1}, \theta_{\phi}\rangle(x_i)=(\phi_1,\phi)_{x_i}.
\end{equation}
We first construct a functional $l_{\theta^\star}'$ \index{$l_{\theta^\star}'$} on the space $H^0(\mathcal{Z}_{[g]}^\circ,\mathcal{L}(\beta))$, which we identify with some $T(\beta,\epsilon,L)$ for appropriate $\epsilon$ and $L$ as before (we save the notation $l_{\theta^\star}$ for later use) with values in $A_{[g]}$.

\begin{lemma}
The general elements in $H^0(\mathcal{Z}_{[g]}^\circ,\mathcal{L}(\beta))$ is a linear combination of $\frac{\theta_\phi}{\Omega_\infty}$'s with coefficients in $A_{[g]}$, where $\phi=\prod_v\phi_v$ are kernel functions such that $\phi_\infty$ is defined as in Definition \ref{Archimedean Kernel}, and $\phi_v$ takes $\bar{\mathbb{Q}}$-values for $v<\infty$.
\end{lemma}
\begin{proof}
This can be seen by interpreting the theta functions defined before \cite[Theorem B.2]{ZhB07} in terms of Weil representations presented here. Note that the CM period $\Omega_\infty$ is missing in \cite[Theorem B.2]{ZhB07}. The algebraicity follows from \cite[Theorem 2.5]{Shi76}. In fact in \cite{Shi76} the period is $h(z_0)$ for some weight $\frac{1}{2}$ form $h$ on $\mathrm{Sp}_4$, as our Hermitian space is two-dimensional, and $z_0$ is a CM point with $h(z_0)\not=0$. This $h(z_0)$ is just $\Omega_\infty$ up to multiplying by a non-zero algebraic number.
\end{proof}

We define:
\begin{equation}\label{CFunctional}l_{\theta^\star}'(\frac{\theta_\phi}{\Omega_\infty})(x_i):=\int_{N(\mathbb{Q})\backslash N(\mathbb{A}_\mathbb{Q})}\theta_{\omega_{\lambda^{-1}}(x_i)(\phi_1)}(n)\theta_{\omega_{\lambda}(x_i)(\phi)}(n)
dn=\int_{X(\mathbb{A}_\mathbb{Q})}\omega_{\lambda^{-1}}(x_i)(\phi_1)(x)\omega_\lambda(x_i)(\phi)(x)dx.
\end{equation}
The last equality is easily seen, and we denote the last term as $(\omega_\lambda^{-1}(x_i)(\phi_1),\omega_\lambda(x_i)(\phi))$. Note that in $N(\mathbb{Q})\backslash N(\mathbb{A}_\mathbb{Q})$ we identified $H$ with $H^-$ using the above isomorphism. (In the literature, people usually consider $\int_{N(\mathbb{Q})\backslash N(\mathbb{A}_\mathbb{Q})}\theta_\phi(n)\bar{\theta}_{\phi_1}(n)dn$ for $\theta_\phi$ and $\theta_{\phi_1}$ on the same space of theta functions. We use a different convention for the sake of simplicity.) So by taking appropriate $\phi_1$ the $l_{\theta^\star}'$ is a rational functional. We extend the definition of $l_{\theta^\star}'$ to whole $H^0(\mathcal{Z}{[g]}^\circ, \mathcal{L}(\beta))$ linearly. Thus, it is well defined.
\begin{lemma}\label{lemma 4.8}
The $l_{\theta^\star}'$ takes values in the space of constant functions on any theta function $\theta_\phi$ as above.
\end{lemma}
\begin{proof}
We note that for any $\phi$,
\begin{equation}\label{(13)}(\omega_{\lambda^{-1}}(x_i)(\phi_1),\omega_\lambda(x_i)(\phi))=(\phi_1,\phi).
\end{equation}
This is a standard fact and can be seen by simply unfolding the definition and integration. The lemma follows from the above equation.
\end{proof}
\begin{remark}
Later we will use this functional on Fourier-Jacobi coefficients for $\mathrm{U}(3,1)$. We can view it as a function on $G_PN_P(\mathbb{A})$ by $\mathrm{FJ}_{[g]}(p,f)=\mathrm{FJ}_P(pg,f)$ for $p\in G_PN_P(\mathbb{A})$ and thus an adelic theta function. Lan \cite{LAN1} has proved the following compatibility of the analytically and algebraically defined Fourier-Jacobi expansions using the usual identification of the global sections of $\mathcal{L}(\beta)$ with (classical or adelic) theta functions, keeping the rational structures
$$\mathrm{FJ}_{[hg]}(-,f)=\mathrm{FJ}_{[g]}^h(f)(-).$$
Note that the period factor appearing in \cite[Section 3.6.5]{Hsieh CM} is $1$ since we are in the scalar weight $\kappa$.
\end{remark}

\section{Klingen-Eisenstein Series}\label{Klingen Series}
From now on throughout this paper we define $z_\kappa=\frac{\kappa-3}{2}$ and $z'_\kappa=\frac{\kappa-2}{2}$.\index{$z_\kappa, z'_\kappa$}
\subsection{Archimedean Picture}
Let $(\pi_\infty, V_\infty)$ be a finite-dimensional representation of $D_\infty^\times$. Let $\psi_\infty$ and $\tau_\infty$ be characters of $\mathbb{C}^\times$ such that $\psi_\infty|_{\mathbb{R}^\times}$ is the central character of $\pi_\infty$. We assume here that $\tau_\infty(z)=z^{-\frac{\kappa}{2}}\bar{z}^{\frac{\kappa}{2}}$ and $\psi_\infty$ is trivial. Then there is a unique representation $\pi_\psi$ of $\mathrm{GU}(2)(\mathbb{R})$ determined by $\pi_\infty$ and $\psi_\infty$ such that the central character is $\psi_\infty$. These determine a representation $\pi_{\psi}\times \tau_\infty$ of $M_P(\mathbb{R})\simeq \mathrm{GU}(2)(\mathbb{R})\times \mathbb{C}^\times$. Here for $g\in\mathrm{GU}(2)(\mathbb{R})$ and $x\in\mathbb{C}^\times$, we identify with it an element
$$m(g,x)=\begin{pmatrix}\mu(g)x^{-1}&&\\&g&\\&&x\end{pmatrix}\in M_P(\mathbb{R}).$$

We extend this to a representation $\rho_\infty$ of $P(\mathbb{R})$ by requiring that $N_P(\mathbb{R})$ acts trivially. Let $I(V_\infty)=\mathrm{Ind}_{P(\mathbb{R})}^{G(\mathbb{R})}\rho_\infty$ (smooth induction) and $I(\rho_\infty)\subset I(V_\infty)$ be the subspace of $K_\infty$ -finite vectors. (Elements of $I(V_\infty)$ can be realized as functions on $K_\infty$.) For any $F\in I(V)$ and $z\in\mathbb{C}^\times$ we define a function $F_z$ on $G(\mathbb{R})$ by
$$F_z(g):=\delta(m)^{\frac{3}{2}+z}\rho(m)f(k), g=mnk\in P(\mathbb{R})K_\infty.$$
There is an action $\sigma(\rho,z)$ on $I(V_\infty)$ by
$$(\sigma(\rho,z)(g))(k)=F_z(kg).$$

We let $\rho^\vee_\infty$ and $I(\rho^\vee_\infty)$ be the corresponding objects by replacing $\pi_\infty,\psi_\infty,\tau_\infty$ with $\pi_\infty\otimes (\tau_\infty\circ \mathrm{Nm}), \psi_\infty\tau_\infty\tau^c_\infty,\bar{\tau}_\infty^c$. Let $w=\begin{pmatrix}&&&1\\&1&&\\&&1&\\-1&&&\end{pmatrix}$. Then there is an intertwining operator $A(\rho_\infty,z,-): I(\rho_\infty)\rightarrow I(\rho^\vee_\infty)$ by:
$$A(\rho_\infty,z,F)(k):=\int_{N_P(\mathbb{R})}F_z(wnk)dn.$$
In this paper, we use the case when $\pi_\infty$ is the trivial representation.
By the Frobenius reciprocity law there is a unique (up to scalar) vector $\tilde v\in I(\rho)$ such that $k.\tilde v=\det \mu(k,i)^{-\kappa}\tilde{v}$ for any $k\in K_\infty^+$. We fix $v$ and scale $\tilde v$ such that $\tilde v(1)=v$. In $\pi^\vee$ it has the action of $K_\infty^+$ given by multiplying by $\det \mu(k,i)^{-\kappa}$. There is a unique vector $\tilde v^\vee\in I(\rho^\vee)$ such that the action of $K_\infty^+$ is given by $\det \mu(k,i)^{-\kappa}$ and $\tilde v^\vee(w)=v$. Then by uniqueness there is a constant $c(\rho,z)$ such that $A(\rho,z,\tilde v)=c(\rho,z)\tilde v^\vee$. \index{$A(\rho,z,-)$}
\begin{definition}\label{3.1.1}
We define $F_\kappa\in I(\rho)$ to be the $\tilde{v}$ as above.
\end{definition}
We record the following lemma proved in \cite[Section 5.4.2]{WAN}.
\begin{lemma}
Let $\kappa\geq 6$ and $z_\kappa=\frac{\kappa-3}{2}$. Then $c(\rho,z_\kappa)=0$.
\end{lemma}

\subsection{$\ell$-adic Picture}
Let $(\pi_\ell, V_\ell)$ be an irreducible admissible representation of $D^\times(\mathbb{Q}_\ell)$ and $\pi_\ell$ is unitary and tempered if $D$ is split at $\ell$. Let $\psi$ and $\tau$ be characters of $\mathcal{K}_\ell^\times$ such that $\psi|_{\mathbb{Q}_\ell^\times}$ is the central character of $\pi_\ell$. Then similar to the Archimedean case, there is a unique irreducible admissible representation $\pi_\psi$ of $\mathrm{GU}(2)(\mathbb{Q}_\ell)$ determined by $\pi_\ell$ and $\psi_\ell$.
As before we have a representation $\pi_\psi\times \tau$ of $M_P(\mathbb{Q}_\ell)$ and extend it to a representation $\rho_\ell$ of $P(\mathbb{Q}_\ell)$ by requiring that $N_P(\mathbb{Q}_\ell)$ acts trivially. Let $I(\rho_\ell)=\mathrm{Ind}_{P(\mathbb{Q}_\ell)}^{G(\mathbb{Q}_\ell)}\rho_\ell$ be the admissible induction.

Define $F_z$ for $F\in I(\rho_\ell)$ and $\rho_\ell^\vee, I(\rho_\ell^\vee), A(\rho_\ell,z,F)$ etcetera as before. For $v\not\in\Sigma$ we have $D^\times(\mathbb{Q}_\ell)\simeq \mathrm{GL}_2(\mathbb{Q}_\ell)$. Moreover, we can choose isomorphism as a conjugation by elements in $\mathrm{GL}_2(\mathcal{O}_{\mathcal{K},\ell})$ (note that both groups are subgroups of $\mathrm{GL}_2(\mathcal{K}_\ell)$).
\begin{definition}When $\pi_\ell$, $\psi_\ell$, $\tau_\ell$ are unramified and $\varphi_\ell\in V_\ell$ are spherical vectors, then there is a unique vector $F_{\varphi_\ell}^0\in I(\rho_\ell)$ which is invariant under $G(\mathbb{Z}_\ell)$ and $F_{\varphi_\ell}^0(1)=\varphi_\ell$.
\end{definition}
\subsection{Global Picture}
\begin{definition}\label{EDatum}
We define an Eisenstein data $\mathcal{D}=(\Sigma,\pi,\psi,\tau,\varphi)$ \index{$\mathcal{D}$} to consist of the following:
\begin{itemize}
\item A finite set of primes $\Sigma$ containing all bad primes.
\item An irreducible unitary cuspidal automorphic representation $(\pi=\otimes_v\pi_v, V)$ of $D^\times(\mathbb{A}_\mathbb{Q})$ and a vector $\varphi=\prod_v\varphi_v\in\pi$, which is ordinary at $p$.
\item The $\psi=\prod \psi_v$ and $\tau=\prod \tau_v$, CM characters of $\mathcal{K}^\times\backslash \mathbb{A}^\times_\mathcal{K}$ of infinite types $(0,0)$ and $(-\frac{\kappa}{2},\frac{\kappa}{2})$, respectively, such that $\psi|_{\mathbb{A}^\times_\mathbb{Q}}$ is the central character of $\pi$. We define $\xi:=\psi/\tau$. \index{$\psi, \tau$}
\end{itemize}
\end{definition}
\begin{remark}
In application the $\varphi$ we use later on is not of the form $\prod_v\varphi_v$, but is a finite sum of such functions. However all theory extends to this situation easily by linear combination.
\end{remark}
We define $I(\rho)$ to be the restricted tensor product of $\otimes_v I(\rho_v)$ with respect to the unramified vectors $F_{\varphi_\ell}^0$ for some $\varphi=\otimes_v \varphi_v\in\pi$. We can define $F_z$, $I(\rho^\vee)$ and $A(\rho,z,F)$ similar to the local case. The $F_z$ takes values in $V$ which can be realized as automorphic forms on $D^\times(\mathbb{A}_\mathbb{Q})$. We also write $F_z$ for the scalar-valued functions $F_z(g):=F_z(g)(1)$ and define the Klingen-Eisenstein series:
$$E(F,z,g):=\sum_{\gamma\in P(\mathbb{Q})\backslash G(\mathbb{Q})}F_z(\gamma g).$$
This is absolutely convergent if $\mathrm{Re} z>>0$ and has meromorphic continuation to all $z\in \mathbb{C}$.

\subsection{Good Klingen Eisenstein Sections}\label{GKES}
We specify good choices for Klingen Eisenstein sections at local places. We write $w=\begin{pmatrix}&&&1\\&1&&\\&&1&\\-1&&&\end{pmatrix}$.
\begin{itemize}
\item For the Archimedean place we define $F_{\mathcal{D},\infty}:=F_\kappa$.
\item For finite places $v$ outside $\Sigma$, we define $F_{\mathcal{D},v}:=F^0_{\varphi_v}$.
\item For finite places $v$ inside $\Sigma$,
let $y$ be an element in $\mathcal{O}_v$ divisible by some high power of the uniformizer $\varpi_v$ at $v$ (to be made precise in the next chapter).
Let $\mathfrak{Y}$ be the set of matrices $A\in \mathrm{U}(2)(\mathbb{Q}_v)$ such that $M=A-1$ satisfies:
$$M(1+y\bar{y}N)=\zeta y\bar{y}$$
for some $N\in M_2(\mathcal{O}_v)$.
Let $\varphi$ be some vector invariant under the action of $\mathfrak{Y}$.
Let $K_v^{(2)}$ be the subgroup of $G(\mathbb{Q}_v)$ of the form $\begin{pmatrix}1&f&c\\&1_2&g\\&&1\end{pmatrix}$ where
$$g=-\zeta{}^t\!\bar{f}, c-\frac{1}{2}f\zeta{}^t\!\bar{f}\in\mathbb{Z}_\ell, f\in(y\bar y), g\in (\zeta y\bar y), c\in\mathcal{O}_v.$$
We define $F_{y,v}$ to be supported in $PwK_v^{(2)}$ and is invariant under the right action of $K_v^{(2)}$, and such that
$$F_{y,v}(w)=\tau(y\bar y )|(y\bar y)^2|_v^{-z-\frac{3}{2}}\mathrm{Vol}(\mathfrak{Y})\cdot\varphi.$$
This $F_{y,v}$ is the Klingen Eisenstein section $F_{\mathcal{D},v}$ that we choose.
\item For the $p$-adic places we use the following
\begin{definition}\label{Definition 6.27}
Define $F^{0,\bullet}_{\mathcal{D},p}\in I_p(\rho)$ to be the Klingen section described as follows It is supported in $P(\mathbb{Q}_p)w_3'B_t(\mathbb{Z}_p)$ where $B_t(\mathbb{Z}_p)$ consists of matrices in $\mathrm{GL}_4(\mathbb{Z}_p)$ which are upper triangular modulo $p^t$, $w_3'=\begin{pmatrix}&&1&\\1&&&\\&1&&\\&&&1\end{pmatrix}\in \mathrm{GL}_4(\mathbb{Z}_p)$, and is right invariant under $N_t(\mathbb{Z}_p)$ for $N_t(\mathbb{Z}_p)\subseteq B_t(\mathbb{Z}_p)$ consisting of matrices which are $1$ along the diagonal modulo $p^t$ (see \cite[Section 4]{WAN}, note the differences in the indices discussed there (Subsections 4.D.2 and 4.D.4 in \emph{loc.cit.}). The $w_3'$ here is the $w_{\mathrm{Borel}}$ there.). Moreover we require that the value of $F^{0,\bullet}_{\mathcal{D},p}$ on $w_3'$ is given by
$\varphi$. We define the $p$-part level group $K_t\subset\mathrm{GL}_4(\mathbb{Z}_p)$ to consist of matrices congruent to some elements in $B(\mathbb{Z}_p)$ modulo $p^t$.
\end{definition}
\end{itemize}

Now we briefly recall Hida's $\mathrm{U}_p$ operator for $\mathrm{U}(3,1)$. We refer to \cite[Section 3.8]{Hsieh CM} for geometric backgrounds for $\mathrm{U}_p$ operators, and that it is compatible with our representation theoretic description below. Let $\lambda_1,\cdots,\lambda_4$ be $4$ characters of $\mathbb{Q}^\times_p$, $\pi=\mathrm{Ind}_B^{GL_n}(\lambda_1,\cdots,\lambda_4)$.
\begin{definition}
Let $\underline{k}=(c_4;c_1,...,c_3)$ be a weight. We say $(\lambda_1,\cdots, \lambda_n)$ is nearly ordinary with respect to $\underline{k}$ if the set:

\begin{eqnarray*}
&\{\mathrm{val}_p\lambda_1(p),...,\mathrm{val}_p\lambda_4(p)\}\\
=&\{c_1-\frac{3}{2},c_2+\frac{3}{2},c_3+\frac{1}{2},c_4-\frac{1}{2}\}
\end{eqnarray*}
\end{definition}
\noindent We denote the set as $\{\kappa_1,...,\kappa_{r+s}\}$ so that $\kappa_1>...>\kappa_{r+s}$.

We define the $p$-component of the Eisenstein data is \emph{generic} if it satisfies Definition \ref{definegeneric}.
Let $\mathcal{A}_p:=\mathbb{Z}_p[t_1,t_2,...,t_n,t_n^{-1}]
$ be the Atkin-Lehner ring of $G(\mathbb{Q}_p)$, where $t_i$ is defined by $t_i=N(\mathbb{Z}_p)\alpha_i
N(\mathbb{Z}_p)$, $\alpha_i=\begin{pmatrix}1_{n-i}&\\&p 1_i\end{pmatrix}$. Then $t_i$ acts on $\pi^{N(\mathbb{Z}_p)}$ by $$v|t_i=\sum_{x\in N|\alpha_i^{-1}N\alpha_i}x_i\alpha_i^{-1}v.$$
We also define a normalized action with respect to the weight $\underline{k}$ following (\cite{Hida04})
$$v\|t_i:=\delta(\alpha_i)^{-1/2}p^{\kappa_1+\cdots+\kappa_i}v|t_i.$$
(The $\delta$ is the modulus character).
\begin{definition}
A vector $v\in \pi$ is called nearly ordinary if it is an eigenvector for all $||t_i$'s with eigenvalues that are $p$-adic units.
\end{definition}

Now we define
\begin{equation}\label{KlinEi}
E_{Kling,\mathcal{D}}(g):=E(\prod_{v\nmid p} F_{\mathcal{D},v}\times F^{0,\bullet}_{\mathcal{D},p}, z_\kappa,g).
\end{equation}
The following proposition is easily proved as in \cite[Proposition 9.8]{SU}.
\begin{proposition}
The classical automrophic form corresponding to $E_{Kling,\mathcal{D}}$ defined by (\ref{CLASSICAL}) is a holomorphic automorphic form on $\mathrm{U}(3,1)$ with weight $\underline{k}=(\kappa;0,0,0)$.
\end{proposition}
Suppose $\pi_p$ is nearly ordinary, in the sense that it is of the form $\pi(\chi_{1,p},\chi_{2,p})$ with $\mathrm{val}_p(\chi_{1,p}(p))=-\frac{1}{2}$, $\mathrm{val}_p(\chi_{2,p}(0))=\frac{1}{2}$. Then it is easy to check that the representation $I(\rho_p)$ is nearly ordinary with respect to the weight $\underline{k}$. Suppose moreover that the $p$-component of the Eisenstein datum is generic. Then we have
\begin{proposition}
The $F^{0,\bullet}_p$ is an eigenvector for all the actions $||t_i$ with eigenvalues $$\lambda_1\cdots\lambda_i(p^{-1})p^{\kappa_1+\cdots+\kappa_i},$$ which are clearly $p$-adic units.
\end{proposition}
\begin{proof}
The proof is a little convoluted and given in \cite[Subsection 4.4.1]{WAN}. It uses the intertwining operator which maps $F^{0,\bullet}_p$ to the section supported on the big cell (denoted $f^\ell$ there, see Lemma 4.16 of \emph{loc.cit.}), whose eigenvalues are easy to compute. It is proved in \cite[Lemma 4.17]{WAN} that the $f^\ell$ is indeed an ordinary operator. In the generic case, the intertwining operator gives an isomorphism between the corresponding principal series representations of $\mathrm{GL}_4(\mathbb{Q}_p)$. Although our definition of being ``generic'' is different from \emph{loc.cit.}, however the argument of Lemma 4.17 there still works. Then as in the proof of \cite[Lemma 4.19]{WAN}, one checks the $F^{0,\bullet}_p$ and the $f^\ell$ have the same action by the level group $K^t$, and that such vector is unique up to scalar in the corresponding principal series representation, identifying the $F^{0,\bullet}_p$ and the $f^\ell$ under the intertwining operator. These altogether implies that $F^{0,\bullet}_p$ is indeed a nearly ordinary vector.

In our $\mathrm{U}(3,1)$ situation the argument is also given by Hsieh in \cite[Section 6.2]{Hsieh13}.
\end{proof}
\begin{remark}
It is worth pointing out that the definition is quite different from the $\mathrm{U}(2,2)$ case in \cite{SU}. The nearly ordinary section is supported on the set containing the Weyl element $w'_3$ instead of the identity. This description also coincides with property that in \cite[Lemma 6.6]{Hsieh13} that the only Weyl element in which the ordinary section is non-zero is $w'_3$.
\end{remark}
\begin{definition}
Throughout this paper, we fix the tame level subgroup $K^{(3,1)}_\mathcal{D}$ of $\mathrm{U}(3,1)(\mathbb{A}^{p\infty})$, under which our $E_{Kling,\mathcal{D}}$ is invariant. We can do so by simply taking it to be the set of matrices congruent to $1$ modulo the $(y\bar{y})^2$ at each finite place not dividing $p$ for the $y$ above. We also define the $p$-component of the level group as the $N_t(\mathbb{Z}_p)$ above, where $t$ is the one in the definition for ``generic''.
\end{definition}

\subsection{Constant Terms}\label{constantterms}
\begin{definition}\index{$E_P$, $E_R$}
For any parabolic subgroup $R$ of $\mathrm{GU}(3,1)$ and an automorphic form $\varphi$ we define $\varphi_R$ to be the constant term of $\varphi$ along $R$ given by the following
$$\varphi_R(g)=\int_{N_R(\mathbb{Q})\backslash N_R(\mathbb{A}_\mathbb{Q})}\varphi(ng)dn$$
where $N_R$ is the unipotent radical of $R$.
\end{definition}
The following lemma is well-known (see \cite[Section II.1.7]{MoWa95}).
\begin{lemma}\label{CON}
Let $R$ be a standard $\mathbb{Q}$-parabolic subgroup of $\mathrm{GU}(3,1)$. Suppose $\mathrm{Re}(z)>\frac{3}{2}$.\\
(i) If $R\not= P$ then $E(f,z,g)_R=0$;\\
(ii) $E(f,z,-)_P=f_z+A(\rho,f,z)_{-z}$.
\end{lemma}
\subsection{Hecke Operators}
Let $K'=K'_{\Sigma\backslash\{p\}} K^\Sigma\subset G(\mathbb{A}_f^p)$ be an open compact subgroup with $K^\Sigma=G(\hat{\mathbb{Z}}^\Sigma)$ and
such that $K:=K'K_p^0$ is neat.
For each $v$ outside $\Sigma$ we have $\mathrm{GU}(3,1)(\mathbb{Q}_v)\simeq \mathrm{GU}(2,2)(\mathbb{Q}_v)$ with the isomorphism given by conjugation by some elements in $\mathrm{GL}_4(\mathcal{O}_{\mathcal{K},v})$. So we only need to study the unramified Hecke operators for $\mathrm{GU}(2,2)$ with respect to $\mathrm{GU}(2,2)(\mathbb{Z}_v)$. We follow closely to \cite[Sections 9.5, 9.6]{SU}.\\

\noindent \underline{Unramified Inert Case}\\
\noindent Let $v$ be a prime of $\mathbb{Q}$ inert in $\mathcal{K}$. Recall that as in \cite[Section 9.5.2]{SU} that $Z_{v,0}$ is the Hecke operator associated to the matrix $z_0:=\mathrm{diag}(\varpi_v,\varpi_v,\varpi_v,\varpi_v)$ by the double coset $Kz_0K$ where $K$ is the maximal compact subgroup of $G(\mathbb{Z}_v)$. Let $t_0:=\mathrm{diag}(\varpi_v,\varpi_v,1,1)$,
$t_1:=\mathrm{diag}(1,\varpi_v,1,\varpi_v^{-1})$ and $t_2:=\mathrm{diag}(\varpi_v,1,\varpi_v^{-1},1)$. As in \cite[9.5.2]{SU} we define
$$\mathcal{R}_v:=\mathbb{Z}[X_v,q^{1/2},q^{-1/2}]$$ where $X_v$ is $T(\mathbb{Q}_v)/T(\mathbb{Z}_v)$ and write $[t]$ for the image of $t$ in $X_v$. Let $\mathcal{H}_K$ be the abstract Hecke ring with respect to the level group $K$. There is a Satake map:
$\mathcal{S}_K:\mathcal{H}_K\rightarrow \mathcal{R}_v$ given by $$\mathcal{S}_K(KgK)=\sum\delta_B^{1/2}(t_i)[t_i]$$ if $KgK=\sqcup t_in_iK$ for $t_i\in T(\mathbb{Q}_v),n_i\in N_B(\mathbb{Q}_v)$ and extend linearly.
We define the Hecke operators $T_i$ for $i=1,2,3,4$ by requiring that
$$1+\sum_{i=1}^4\mathcal{S}(T_i)X^i=\prod_{i=1}^2(1-q_v^{\frac{3}{2}}[t_i]X)(1-q_v^{\frac{3}{2}}[t_i]^{-1}X)$$
be an equality of polynomials of the variable $X$.
We also define:
$$Q_v(X):=1+\sum_{i=1}^4T_i(Z_0X)^i.$$
\noindent\underline{Unramified Split Case}\\
\noindent Suppose $v$ is a prime of $\mathbb{Q}$ split in $\mathcal{K}$. In this case we define $z_0^{(1)}$ and $z_0^{(2)}$ to be $(\mathrm{diag}(\varpi_v,\varpi_v,\varpi_v,\varpi_v),1)$ and $(1,\mathrm{diag}(\varpi_v,\varpi_v,\varpi_v,\varpi_v))$ and define the Hecke operators $Z_0^{(1)}$ and $Z_0^{(2)}$ as above but replacing $z_0$ by
$z_0^{(1)}$ and $z_0^{(2)}$. Let $t_1^{(1)}:=\mathrm{diag}(1,(\varpi_v,1),1,(1,\varpi_v^{-1}))$, $t_2^{(1)}:=\mathrm{diag}((\varpi_v,1),1,(1,\varpi_v^{-1}),1)$. Define $t_i^{(2)}:=\bar{t}_i^{(1)}$ and $t_i=t_i^{(1)}t_i^{(2)}$ for $i=1,2$. We define $R_v$ and $\mathcal{S}_K$ in the same way as the inert case. Then we define Hecke operators $T_i^{(j)}$ for $i=1,2,3,4$ and $j=1,2$ by requiring that the following
$$1+\sum_{i=1}^4\mathcal{S}_K(T_i^{(j)})X^i=\prod_{i=1}^2(1-q_v^{\frac{3}{2}}[t_i^{(j)}]X)(1-q_v^{\frac{3}{2}}[t_i^{(j')}]^{-1}X)$$
be equalities of polynomials of the variable $X$. Here $j'=3-j$ and $[t_i^{(j)}]$'s are defined similarly to the inert case.
Now suppose $v=w\bar{w}$ for a place $w$ of $\mathcal{K}$ and a place $v$ of $\mathbb{Q}$. Define $i_w=1$ and $i_{\bar{w}}=2$. Then we define:
$$Q_w(X):=1+\sum_{i=1}^4T_i^{(i_w)}(Z_0^{(3-i_w)}X)^i,$$
$$Q_{\bar{w}}(X):=1+\sum_{i=1}^4T_i^{(i_{\bar{w}})}(Z_0^{(3-i_{\bar{w}})}X)^i.$$
\subsection{Galois Representations}
For the holomorphic Klingen-Eisenstein series, we can also associate a reducible Galois representation with the same recipe as in subsection \ref{2.24}. Write $\tau'$ for the restriction of $\tau$ to $\mathbb{A}_\mathbb{Q}$ and let $\sigma_{\tau'}$ be the corresponding Galois character of $G_\mathbb{Q}$ via class field theory. The resulting Galois representation associated to the Klingen-Eisenstein series we defined above can be seen as follows:
$$\sigma_{\tau'}\sigma_{\psi^c}\epsilon^{-\kappa}\oplus \sigma_{\psi^c}\epsilon^{-3}\oplus \rho_{\pi_f}.\sigma_{\tau^c}\epsilon^{-\frac{\kappa+3}{2}}.$$
Note that $\kappa+3$ is an odd number; however, $\pi_f$ is a unitary representation whose $L$-function is the usual $L$-function for $f$ shifted by $\frac{1}{2}$. So it makes sense to write in the above way. This can be obtained in the same manner as \cite[Sections 9.5, 9.6]{SU}, by studying the Hecke operators defined above. Indeed the Galois representation is determined by its local Euler factors at unramified places, which is already worked out in \emph{loc.cit.}. (See in particular, \cite[Proposition 9.14]{SU}.)

\section{Siegel-Eisenstein Series and Pullback}\label{section 5}
\subsection{Generalities}
\noindent\underline{Local Picture}:\\
Our discussion in this section follows \cite[Sections 11.1-11.3]{SU} closely. Let $Q=Q_n$ \index{$Q_n$} be the Siegel parabolic subgroup of $G_n$ consisting of matrices $\begin{pmatrix}A_q&B_q\\0&D_q\end{pmatrix}$. It consists of matrices whose lower-left $n\times n$ block is zero. For a place $v$ of $\mathbb{Q}$ and a character $\chi$ of $\mathcal{K}_v^\times$ we let $I_n(\chi_v)$ \index{$I_n(\chi_v)$} be the space of smooth
$K_{n,v}$-finite functions (here $K_{n,v}$ means the maximal compact subgroup $G_n(\mathbb{Z}_v)$)\index{$K_{n,v}$} $f: K_{n,v}\rightarrow \mathbb{C}$ such that $f(qk)=\chi_v(\det D_q)f(k)$ for all $q\in Q_n(\mathbb{Q}_v)\cap K_{n,v}$ (we write $q$ as block matrix $q=\begin{pmatrix}A_q&B_q\\0&D_q\end{pmatrix}$). For $z\in \mathbb{C}$ and $f\in I(\chi)$ we also define a function $f(z,-):G_n(\mathbb{Q}_v)\rightarrow \mathbb{C}$ by $f(z,qk):=\chi(\det D_q))|\det A_q D_q^{-1}|_v^{z+n/2}f(k),q\in Q_n(\mathbb{Q}_v)$ and $k\in K_{n,v}.$\\

\noindent For $f\in I_n(\chi_v),z\in \mathbb{C},$ and $k\in K_{n,v}$, the intertwining integral is defined by:
$$M(z,f)(k):=\bar\chi^n_v(\mu_n(k))\int_{N_{Q_n}(F_v)}f(z,w_nrk)dr$$
where $N_{Q_n}$ is the unipotent radical of $Q_n$, and $w_n:=\begin{pmatrix}&1_n\\-1_n&\end{pmatrix}$.
For $z$ in compact subsets of $\{\mathrm{Re}(z)>n/2\}$ this integral converges absolutely and uniformly, with the convergence being uniform in $k$. In this case it is easy to see that $M(z,f)\in I_n(\bar{\chi}^c_v).$
\noindent Let $\mathcal{U}\subseteq \mathbb{C}$ be an open set. By a meromorphic section of $I_n(\chi_v)$ on $\mathcal{U}$ we mean a function $\varphi:\mathcal{U}\mapsto I_n(\chi_v)$ taking values in a finite-dimensional subspace $V\subset I_n(\chi_v)$ and such that
$\varphi:\mathcal{U}\rightarrow V$ is meromorphic. A standard fact from the theory of Eisenstein series says that this has a continuation to a meromorphic section on all of $\mathbb{C}$.\\

\noindent\underline{Global Picture:}\\
For a Hecke character $\chi=\otimes\chi_v$ of $\mathbb{A}_\mathcal{K}^\times$ we define a space $I_n(\chi)$ to be the restricted tensor product defined using the spherical vectors $f_v^{sph}\in I_n(\chi_v)$ (invariant under $K_{n,v}$) such that $f_v^{sph}(K_{n,v})=1$, at the finite places $v$ where $\chi_v$ is unramified.\\

\noindent For $f\in I_n(\chi)$ we consider the Eisenstein series
$$E(f;z,g):=\sum_{\gamma\in Q_n(\mathbb{Q})\setminus G_n(\mathbb{Q})} f(z,\gamma g).$$
This series converges absolutely and uniformly for $(z,g)$ in compact subsets of $$\{\mathrm{Re}(z)>n/2\}\times G_n(\mathbb{A}_\mathbb{Q}).$$ The defined automorphic form is called Siegel-Eisenstein series. \\

\noindent The Eisenstein series $E(f;z,g)$ has a meromorphic continuation in $z$ to all of $\mathbb{C}$ in the following sense. If $\varphi:\mathcal{U}\rightarrow I_n(\chi)$ is a meromorphic section, then we put $E(\varphi;z,g)=E(\varphi(z);z,g).$ This is defined at least on the region of absolute convergence and it is well known that it can be meromorphically continued to all $z\in \mathbb{C}$.\\

\noindent Now for $f\in I_n(\chi),z\in \mathbb{C}$, and $k\in \prod_{v\nmid \infty}K_{n,v}\prod_{v|\infty}K_\infty$, there is a similar intertwining integral $M(z,f)(k)$ as above but with the integral being over $N_{Q_n}(\mathbb{A}_\mathbb{Q})$. This again converges absolutely and uniformly for $z$ in compact subsets of $\{\mathrm{Re}(z)>n/2\}\times K_n$. Thus $z\mapsto M(z,f)$ defines a holomorphic section $\{\mathrm{Re}(z)>n/2\}\rightarrow I_n(\bar{\chi}^c)$. This intertwining operator has a continuation to a meromorphic section on $\mathbb{C}$. For $Re(z)>n/2$, we have
$$M(z,f)=\otimes_v M(z,f_v),f=\otimes f_v.$$

\noindent The functional equation for Siegel-Eisenstein series is:
$$E(f,z,g)=\chi^n(\mu(g))E(M(z,f);-z,g)$$
in the sense that both sides can be meromorphically continued to all $z\in\mathbb{C}$ and the equality is understood as an equality of meromorphic functions of $z\in\mathbb{C}$.

\subsection{Embeddings}
We define some embeddings of a subgroup of $\mathrm{GU}(3,1)\times \mathrm{GU}(0,2)$ into some larger groups. This is used in the doubling method. First we define $\mathrm{GU}(3,3)'$ to be the unitary similitude group associated to:
$$\begin{pmatrix}&&1&\\&\zeta&&\\-1&&&\\&&&-\zeta\end{pmatrix}$$
and $\mathrm{GU}(2,2)'$ to be the unitary group associated to
$$\begin{pmatrix}\zeta&\\&-\zeta\end{pmatrix}.$$
We define embeddings
$$\alpha: \{g_1\times g_2\in \mathrm{GU}(3,1)\times \mathrm{GU}(0,2), \mu(g_1)=\mu(g_2)\}\rightarrow \mathrm{GU}(3,3)'$$
and
$$\alpha':\{g_1\times g_2\in \mathrm{GU}(2,0)\times \mathrm{GU}(0,2), \mu(g_1)=\mu(g_2)\}\rightarrow \mathrm{GU}(2,2)'$$
by $\alpha(g_1,g_2)=\begin{pmatrix}g_1&\\&g_2\end{pmatrix}$ and $\alpha'(g_1,g_2)=\begin{pmatrix}g_1&\\&g_2\end{pmatrix}$.
We also define isomorphisms:
$$\beta: \mathrm{GU}(3,3)'\xrightarrow{\sim} \mathrm{GU}(3,3)$$
and
$$\beta':\mathrm{GU}(2,2)'\xrightarrow{\sim} \mathrm{GU}(2,2)$$
by
$$g\mapsto S^{-1}g S$$
or
$$g\mapsto S'^{-1}gS'$$
where $$S=\begin{pmatrix}1&&&\\&1&&-\frac{\zeta}{2}\\&&1&\\&-1&&-\frac{\zeta}{2}\end{pmatrix}$$
and
$$S'=\begin{pmatrix}1&-\frac{\zeta}{2}\\-1&-\frac{\zeta}{2}\end{pmatrix}$$
\index{$S, S'$} We write $\gamma$ and $\gamma'$ for the embeddings $\beta\circ\alpha$ and $\beta'\circ\alpha'$, respectively.

We define an element \index{$\Upsilon$} $\Upsilon\in \mathrm{U}(3,3)(\mathbb{Q}_p)$ such that $\Upsilon_{v_0}=S^{-1}_{v_0}$ and $\Upsilon_{v_0}'=S^{-1,'}_{v_0}$, where $S$ is defined at the end of Section 6.2. We know that under the complex uniformization, taking the change of polarization into consideration the map \ref{Em} is given by
\begin{equation}\label{Emb}
i([\tau,g],[x_0,h])=[Z_\tau,(g,h)\Upsilon]
\end{equation}
(see \cite[Section 2.6]{Hsieh CM}.)

\subsection{Pullback Formula}
We recall the pullback formula of Shimura. Let $\chi$ be a unitary idele class character of $\mathbb{A}_\mathcal{K}^\times$. Given a cuspform $\varphi$ on $\mathrm{GU}(2)$ we consider
$$F_\varphi(f;z,g):=\int_{\mathrm{U}(2)(\mathbb{A}_\mathbb{Q})} f(z,S^{-1}\alpha(g,g_1h)S)\bar\chi(\det g_1g)\varphi(g_1h)dg_1,$$
$$f\in I_{3}(\chi),g\in \mathrm{GU}(3,1)(\mathbb{A}_\mathbb{Q}),h\in \mathrm{GU}(2)(\mathbb{A}_\mathbb{Q}),\mu(g)=\mu(h)$$
or
$$F_\varphi'(f';z,g)=\int_{\mathrm{U}(2)(\mathbb{A}_\mathbb{Q})} f'(z,S^{'-1}\alpha(g,g_1h)S')\bar\chi(\det g_1g)\varphi(g_1h)dg_1$$
$$f'\in I_{2}(\chi),g\in \mathrm{GU}(2)(\mathbb{A}_\mathbb{Q}),h\in \mathrm{GU}(2)(\mathbb{A}_\mathbb{Q}),\mu(g)=\mu(h)$$
This is independent of $h$.
We see that the above integrals can be factorized as local integrals, which we denote as $F_{\varphi_v}(f_v;z,g_v)$ and $F'_{\varphi_v}(f'_v;z,g_v)$, respectively.
The pullback formulas are the identities in the following proposition.
\begin{proposition}
Let $\chi$ be a unitary idele class character of $\mathbb{A}_\mathcal{K}^\times$.\\
(i) If $f'\in I_{2}(\chi),$ then $ F_\varphi'(f';z,g)$ converges absolutely and uniformly for $(z,g)$ in compact subsets of $\{\mathrm{Re}(z)>1\}\times \mathrm{GU}(2,0)(\mathbb{A}_\mathbb{Q})$, and for any $h\in \mathrm{GU}(2)(\mathbb{A}_\mathbb{Q})$ such that $\mu(h)=\mu(g)$
\begin{equation*}
\int_{\mathrm{U}(2)(\mathbb{Q})\setminus \mathrm{U}(2)(\mathbb{A}_\mathbb{Q})} E(f';z,S'^{-1}\alpha(g,g_1h)S')\bar{\chi}(\det g_1h)\varphi(g_1h)dg_1=F_\varphi'(f';z,g).
\end{equation*}
(ii) If $f\in I_{3}(\chi)$, then $F_\varphi(f;z,g)$ converges absolutely and uniformly for $(z,g)$ in compact subsets of $\{Re(z)>3/2\}\times \mathrm{GU}(3,1)(\mathbb{A}_\mathbb{Q})$ such that $\mu(h)=\mu(g)$
\begin{equation*}
\begin{split}
\int_{\mathrm{U}(2)(\mathbb{Q})\setminus \mathrm{U}(2)(\mathbb{A}_\mathbb{Q})} E(f;z,S^{-1}\alpha(g,g_1h)S)\bar{\chi}&(\det g_1h)\varphi(g_1h)dg_1\\
&=\sum_{\gamma\in P(\mathbb{Q})\setminus \mathrm{GU}(3,1)(\mathbb{Q})} F_\varphi(f;z,\gamma g),
\end{split}
\end{equation*}
with the series converging absolutely and uniformly for $(z,g)$ in compact subsets of $$\{\mathrm{Re}(z)>3/2\}\times \mathrm{GU}(3,1)(\mathbb{A}_\mathbb{Q}).$$
\end{proposition}
This is a special case of \cite[Proposition 3.5]{WAN}, which summarizes results proved in \cite{Shi97}.
\subsection{Fourier-Jacobi Expansion}
From now on we fix a splitting character $\lambda$ of $\mathcal{K}^\times\backslash\mathbb{A}^\times_\mathcal{K}$ of infinity type $(-\frac{1}{2},\frac{1}{2})$ which is unramified at $p$ and unramified outside $\Sigma$ and such that $\lambda|_{\mathbb{A}_\mathbb{Q}^\times}=\chi_{\mathcal{K}/\mathbb{Q}}$. Let $\tau$ be a Hecke character of $\mathcal{K}^\times\backslash \mathbb{A}_\mathcal{K}^\times$ of infinity type $(-\frac{\kappa}{2},\frac{\kappa}{2})$.
\begin{definition}
For $\beta\in S_n(\mathbb{Q})$ and $\varphi$ a holomorphic automorphic form on $\mathrm{GU}(n,n)$ we define the $\beta$-th Fourier coefficient
$$\varphi_\beta(g)=\int_{S_n(\mathbb{Q})\backslash S_n(\mathbb{A})}\varphi(\begin{pmatrix}1_n&S\\&1_n\end{pmatrix}g)e_\mathbb{A}(-\mathrm{Tr}\beta S)dS.$$
For a prime $v$ and $f_v\in I_n(\tau)$ we also define the local Fourier coefficient at $g_v\in\mathrm{GU}(n,n)(\mathbb{Q}_v)$ as
$$f_{v,\beta}(z,g_v)=\int_{S_n(\mathbb{Q}_v)}f_v(z,\omega_n(\begin{pmatrix}1_n&S_v\\&1_n\end{pmatrix}g_v)e_v(-\mathrm{Tr}\beta
S_v)dS_v.$$
For $\varphi$ a holomorphic automorphic form on $\mathrm{GU}(3,3)$ and $\beta\in\mathbb{Q}^+$ we define
$$\mathrm{FJ}_\beta(\varphi)(g)=\int_{\mathbb{Q}\backslash \mathbb{A}}\varphi(\begin{pmatrix}1_3&\begin{matrix}S&0\\0&0\end{matrix}\\&1_3\end{pmatrix}g)
e_\mathbb{A}(-\mathrm{Tr}\beta S)dS.$$
For $E(f;z,g)$ with $f\in I_3(\tau)$ we define
$$\mathrm{\mathrm{FJ}}_\beta(f;z,g)=\mathrm{FJ}_\beta(E(f;z,-))(g).$$
\end{definition}
The following formula is proved in \cite[Subsection 3.3.1]{WAN}.
\begin{proposition}
Suppose $f\in I_3(\tau)$ and $\beta\in \mathbb{Q}_+$. If $E(f;z,g)$ is the Siegel-Eisenstein series on $\mathrm{GU}(3,3)$ defined by $f$ for some $Re(z)$ sufficiently large, then the $\beta$-th Fourier-Jacobi coefficient $E_\beta(f;z,g)$ satisfies:
\begin{equation}\label{LGFJ}
E_\beta(f;z,g)=\sum_{\gamma\in Q_{2}(\mathbb{Q})\backslash \mathrm{GU}_{2}(\mathbb{Q})}\sum_{x\in \mathcal{K}^2}\int_{S_1(\mathbb{A})}f(w_3\begin{pmatrix}1_3&\begin{matrix}S&x\\ {}^t\!\bar x&0\end{matrix}\\&1_3\end{pmatrix}\mathbf{j}(1,\gamma)g)e_{\mathbb{A}}(-\beta S)dS
\end{equation}
where
$$\mathbf{j}:\mathrm{U}(1,1)\times\mathrm{U}(2,2)\hookrightarrow \mathrm{U}(3,3)$$
is given by
$$\mathbf{j}(g_1,g_2)=\begin{pmatrix}A&&B&\\&D'&&C'\\C&&D&\\&B'&&A'\end{pmatrix}$$
if $g_1=\begin{pmatrix}A&B\\C&D\end{pmatrix}\in \mathrm{U}(1,1),g_2=\begin{pmatrix}A'&B'\\C'&D'\end{pmatrix}\in\mathrm{U}(2,2)$.\index{$\mathbf{j}$}
\end{proposition}

\begin{definition}
If $g_v\in \mathrm{U}(2,2)(\mathbb{Q}_v)$, $x\in\mathcal{K}_v^2$ and $a\in \mathcal{K}^\times_v$, we define:
$$\mathrm{FJ}_\beta(f_v;z,x,g_v,a)=\int_{S_1(\mathbb{Q}_v)}f_{v,z}(w_3\begin{pmatrix}1_3&\begin{matrix}S&x\\ {}^t\!\bar {x}&0\end{matrix}\\&1_3\end{pmatrix}\mathbf{j}(\mathrm{diag}(a,{}^t\!\bar{a}^{-1}),g_v))e_{\mathbb{Q}_v}(-\mathrm{Tr}\beta S)dS.$$
\end{definition}

We have:
$$\mathrm{FJ}_\beta(f_v;z,x,g,a)=\tau_v(\det a)|\det a\bar{a}|_\mathbb{A}^{-(z+\frac{1}{2})}\mathrm{FJ}_{{}^t\!\bar{a}\beta x}(f_v;z,a^{-1}x,g,1).$$
\begin{definition}\label{defJacobi}
For $x,y\in\mathcal{K}^2_v$ and $t\in\mathbb{Q}_v$, we write $n(x,y,t)$ for $\begin{pmatrix}1&y&t+\frac{yx^*-xy*}{2}&x\\&1_2&x^*&0_2\\&&1&\\&&-y^*&1_2\end{pmatrix}$. So that it becomes a Heisenberg group if we give the pairing $\langle(x_1,y_1), (x_2,y_2)\rangle=y_1x_2^*+x_2y_1^*-y_2x_1^*-x_1y_2^*$. (see \cite[Section 4]{ZhB07}).
\end{definition}
\begin{lemma}\label{5.4}
We write
$$\mathrm{FJ}_\beta(f_v;z,n(x,y,t)\alpha(1,u))=\int_{S_1(\mathbb{Q}_v)}f(w_3\begin{pmatrix}1_3&\begin{matrix}S&0\\ 0&0\end{matrix}\\&1_3\end{pmatrix}(x,y,t)\mathbf{j}(1,u))e_{\mathbb{A}}(-\mathrm{Tr}\beta S)dS$$
for $u\in U(2,2)(\mathbb{Q}_v)$ and $n(x,y,t)$ as above.
Suppose for some place $v$ we have $$\mathrm{FJ}_\beta(f_v; z,n(x,0,0)\alpha(1,u))=f(u,z)\omega_{\beta,\lambda}(u)\phi(x)$$
for any $u$, $n(x,0,0)$ as above, and some Schwartz function $\phi\in S(W^d)$ and some $f\in I((\tau/\lambda)_v,z)$. Then we have:
$$\mathrm{FJ}_\beta(f_v; z,(x,y,t)\alpha(1,u))=f(u,z)\omega_{\beta,\lambda}(n(x,y,t)\cdot u)\phi(0)$$ for any $n(x,y,t)$.
\end{lemma}
Note that comparing with \cite{IKE} we have switched the roles played by $x$ and $y$.
\begin{proof}
Since
$$\begin{pmatrix}1_3&\begin{matrix}S&x\\ {}^t\!\bar x&\end{matrix}\\ &
1_3\end{pmatrix}
\begin{pmatrix}\begin{matrix}1_1&\\&\bar A^{-1}\end{matrix}&\begin{matrix}&\\&\end{matrix}\\ \begin{matrix}&\\&B\bar A^{-1}\end{matrix}&
\begin{matrix}1&\\&A\end{matrix}\end{pmatrix}=
\begin{pmatrix}\begin{matrix}1&xB\bar A^{-1}\\&\bar A^{-1}\end{matrix}&\begin{matrix}&\\&\end{matrix}\\ \begin{matrix}&\\&B\bar A^{-1}\end{matrix}&
\begin{matrix}1&\\-B\bar{x}&A\end{matrix}\end{pmatrix}
\begin{pmatrix}1_3&\begin{matrix}S-xB{}^t\!\bar x&xA\\ \bar A {}^t\!\bar x&\end{matrix}\\ \begin{matrix}&\\&\end{matrix}&
1_3\end{pmatrix},$$
it follows that
\begin{eqnarray*}
&\mathrm{FJ}_\beta(f_v;z,x,\begin{pmatrix}A&B\bar A^{-1}\\&\bar A^{-1}\end{pmatrix}g,a)=&\\
&\tau_v^c(\det A)^{-1}|\det A\bar A|_v^{z+3/2}e_v(-\mathrm{Tr}({}^t\!\bar a\beta aB))\mathrm{FJ}_\beta(f;z,xA,g,a).
\end{eqnarray*}
Now the lemma is a consequence of
$$\begin{pmatrix}1&-y&&\\&1_2&&\\&&1&\\&&y^*&1_2\end{pmatrix}=\begin{pmatrix}1&y&t+\frac{yx^*-xy^*}{2}
&x\\&1_2&x^*&0_2\\&&1&\\&&-y^*&1_2\end{pmatrix}
\begin{pmatrix}1&&t-\frac{yx^*+xy^*}{2}&x\\&1_2&x^*&0_2\\&&1&\\&&&1_2\end{pmatrix}$$
where we write $y^*$ for ${}^t\!\bar{y}$.
\end{proof}

\subsection{Archimedean Cases}\label{sectionar}
We let \index{$\mathbf{i}$} $\mathbf{i}:=\begin{pmatrix}i&\\&\frac{\zeta}{2}\end{pmatrix}$ or $\frac{\zeta}{2}$ depending on the size $(3\times 3)$ or $(2\times 2)$. Let $J_n(g,Z):=\det (C_gi+D_g)$ for $g=\begin{pmatrix}A_g&B_g\\C_g&D_g\end{pmatrix}$ be the automorphic factor for $\mathrm{U}(n,n)$.
The Siegel section we choose is $f_{sieg,\mathcal{D},\infty}=f_{sieg,\infty}:=f_\kappa(g,z):=J_3(g,\mathbf{i})^{-\kappa}|J_3(g,\mathbf{i})|^{\kappa-2z-3}$ and $f'_{sieg,\mathcal{D},\infty}=f'_{sieg,\infty}:=f_\kappa'(g,z)=J_2(g,\mathbf{i})^{-\kappa}|J_2(g,\mathbf{i})|^{\kappa-2z-2}$.
Recall for $\varphi\in\pi_\infty$ we define the pullback sections:
$$F_\kappa(z,g):=\int_{\mathrm{U}(2)(\mathbb{R})}f_\kappa(z,S^{-1}\alpha(g,g_1)S)\bar\tau(\det g_1)\pi(g_1)\varphi dg_1$$
and
$$F'_\kappa(z,g):=\int_{\mathrm{U}(2)(\mathbb{R})}f'_\kappa(z,S'^{-1}\alpha(g,g_1)S')\bar\tau(\det g_1)\pi(g_1)\varphi dg_1$$

\noindent If we define an auxiliary $f_{\kappa,n}^\circ(z,g)=J_n(g,i1_n)^{-\kappa}|J_n(g,i1_n)|^{\kappa-2z-n}$ for $n=2,3$, then $f_\kappa(g,z)=f_{\kappa,3}^\circ(gg_0)$ and $f_\kappa'(g,z)=f_{\kappa,2}^\circ(gg_0)$ for $$g_0=\mathrm{diag}(1,\frac{\mathfrak{s}^{\frac{1}{2}}d^{\frac{1}{4}}}{\sqrt{2}},\frac{d^{\frac{1}{4}}}{\sqrt{2}},1,(\frac{\mathfrak{s}^{\frac{1}{2}}d^{\frac{1}{4}}}{\sqrt{2}})^{-1},
(\frac{d^{\frac{1}{4}}}{\sqrt{2}})^{-1})$$
or
$$g_0=\mathrm{diag}(\frac{\mathfrak{s}^{\frac{1}{2}}d^{\frac{1}{4}}}{\sqrt{2}},\frac{d^{\frac{1}{4}}}{\sqrt{2}},(\frac{\mathfrak{s}^{\frac{1}{2}}d^{\frac{1}{4}}}{\sqrt{2}})^{-1},
(\frac{d^{\frac{1}{4}}}{\sqrt{2}})^{-1})$$
depending on the sizes.

\begin{lemma}\label{Apullback}
The integrals are absolutely convergent for $\mathrm{Re}(z)$ sufficiently large and for such $z$, we have:\\
(i)$$F_{\mathcal{D},Kling,\infty}(z,g):=F_\kappa(z,g)=F_{\kappa,z}(g);$$
(ii)$$F'_{\mathcal{D},\infty}:=F_\kappa'(z,g)=\pi(g)\varphi;$$
where $F_{\kappa,z}$ is defined in Definition \ref{3.1.1} using $\varphi$ as the $v$ there.
\end{lemma}

\noindent\underline{Fourier Coefficients}\\
The following lemma is  \cite[Lemma 11.4]{SU}.
\begin{lemma}\label{Fourier}
Suppose $\beta\in S_n(\mathbb{R})$. Then the function $z\rightarrow f_{\kappa,\beta}(z,g)$ has a meromorphic continuation to all of $\mathbb{C}$.
Furthermore, if $\kappa\geq n$ then $f_{\kappa,n,\beta}(z,g)$ is holomorphic at $z_\kappa:=(\kappa-n)/2$. For $y\in \mathrm{GL}_n(\mathbb{C}),f_{\kappa,n,\beta}^\circ(z_\kappa,\mathrm{diag}(y,{}^t\!\bar{y}^{-1}))=0$ if $\det\beta\leq 0$ and if $\det\beta>0$ then
$$f_{\kappa,n,\beta}^\circ(z_\kappa,\mathrm{diag}(y,{}^t\!\bar{y}^{-1}))=\frac{(-2)^{-n}(2\pi i)^{n\kappa}(2/\pi)^{n(n-1)/2}}{\prod_{j=0}^{n-1}(\kappa-j-1)!}e(i\mathrm{Tr}(\beta y{}^t\!\bar{y}))\det(\beta)^{\kappa-n}\det\bar{y}^\kappa.$$
\end{lemma}
The local Fourier coefficient for $f_\kappa$ can be easily deduced from that for $f_\kappa^\circ$.\\

\noindent\underline{Fourier-Jacobi Coefficients}\\
The following lemma can be found in \cite[Lemma 4.4]{WAN}.
\begin{lemma}
Let $z_\kappa=\frac{\kappa-3}{2}$, $\beta\in \mathbb{R}_+$. Then:\\

(i) $\mathrm{FJ}_{\beta}(z_\kappa,f_{\kappa,3}^\circ, x,\eta,1)=f_{\kappa,1,\beta}^\circ(z_\kappa+1,1)e(i\mathrm{Tr}({}^t\!\bar x\beta x))$. Recall that $\eta=\begin{pmatrix} &1_2 \\-1_2  &\end{pmatrix}$;\\

(ii) If $g\in \mathrm{U}(2,2)(\mathbb{R})$, then
$$\mathrm{FJ}_{\beta,\kappa}(z_\kappa,f_{\kappa,3}^\circ,x,g,1)=e(i\mathrm{Tr}\beta)c_1(\beta,\kappa)
f_{\kappa-1,2}^\circ(z_\kappa,g')\omega_{\beta,\lambda_\infty}(g')\Phi_{\beta,\infty}(x).$$
where $g'=\begin{pmatrix}1_2&\\&-1_2\end{pmatrix}g\begin{pmatrix}1_2&\\&-1_2\end{pmatrix}$, $c_1(\beta,\kappa)=\frac{(-2)^{-1}(2\pi i)^{\kappa}}{(\kappa-1)!}\det\beta^{\kappa-1}$ and $\Phi_{\beta,\infty}=e^{-2\pi\mathrm{Tr}(\langle x,x\rangle_\beta)}$. Recall that the $\langle x,x\rangle_\beta$ is a $2$ by $2$ matrix.
\end{lemma}

\begin{lemma}
We have
$$\mathrm{FJ}_\beta(f_\kappa,x,g,1)=e(i\mathrm{Tr}\beta)c_1(\beta,\kappa)J(g,\mathbf{i})^{-\kappa}
\omega_{\beta,\lambda_\infty}(g'g_0)\Phi_{\beta,\infty}(x)$$ for all $g\in \mathrm{U}(2,2)(\mathbb{R}),x\in\mathbb{C}^2$.
\end{lemma}
\begin{proof}
Note that $$f_\kappa(g,z_\kappa)=J(g,\mathbf{i})^{-\kappa}=J(gg_0,i)^{-\kappa}J(g_0,i)^\kappa=(\frac{\sqrt{d}}{2})^{-\kappa}J(gg_0,i)^{-\kappa}.$$
\end{proof}
\begin{lemma}
Let $x_1=(x_{11},x_{12}),x_2=(x_{21},x_{22})$ where the $x_{ij}\in\mathbb{R}$. Then
$$\delta_\psi^{-1}(\omega_{1,\lambda}(\eta g_0)\Phi_{1,\infty})(x_1,x_2)=\frac{\mathfrak{s}^{\frac{1}{2}}
d^{\frac{1}{2}}}{4}e^{-2\pi\sqrt{d}(\mathfrak{s}x_{11}^2+x_{12}^2)}e^{-2\pi\sqrt{d}(\mathfrak{s}x_{21}^2+x_{22}^2)}$$
\end{lemma}
\begin{proof}
Straightforward from the expression for $\Phi_{1,\infty}$ and $\delta_\psi$.
\end{proof}
We summarize the key definitions associated to the Archimedean datum below.
\begin{definition}\label{Archimedean Kernel}
Recall we defined
$f_{sieg,\mathcal{D},\infty}=f_{sieg,\infty}:=f_\kappa(g,z):=J_3(g,\mathbf{i})^{-\kappa}|J_3(g,\mathbf{i})|^{\kappa-2z-3}$ and $f'_{sieg,\mathcal{D},\infty}=f'_{sieg,\infty}:=f_\kappa'(g,z)=J_2(g,\mathbf{i})^{-\kappa}|J_2(g,\mathbf{i})|^{\kappa-2z-2}$.
$$\Phi_{\mathcal{D},\infty}=\omega_1(g_0)\Phi_{1,\infty}, \Phi_{\mathcal{D},\infty}''=\omega_1(\eta g_0)\Phi_{\mathcal{D},\infty},$$
$$f_{2,\mathcal{D},\infty}(g)=f'_{\kappa-1}(gg_0),f_{2,\mathcal{D},\infty}''(g)=f_{\kappa-1}'(g\eta g_0),$$
$$\phi_{1,\infty}(x_1,x_2)=\phi_{2,\infty}(x_1,x_2)=\frac{\mathfrak{s}^{\frac{1}{4}}
d^{\frac{1}{4}}}{2}e^{-2\pi\sqrt{d}(\mathfrak{s}x_{1}^2+x_{2}^2)}, x_1,x_2\in\mathbb{R}.$$
\end{definition}
Finally we record a lemma.
\begin{lemma}\label{+}
Let $Z\in X_{2,2}$ and $\Phi_{\beta,Z}(x)=e(\mathrm{tr}(\langle x,x\rangle_\beta Z))$. For any $g\in\mathrm{U}(2,2)(\mathbb{R})$,
$$\omega_{\beta,\lambda_\infty}(g)\Phi_{\beta,Z}=\det J(g,Z)^{-1}\Phi_{\beta,g(Z)}.$$
\end{lemma}
\begin{proof}
Similar as \cite[Lemma 10.1]{SU}.
\end{proof}

\begin{example}
We work out an example for the theta function constructed via the Weil representation whose Archimedean Schwartz function is given by Definition \ref{Archimedean Kernel}. We check that it is nothing but the adelic theta function defined before. We take the $w_2$ in \cite[Appendix B]{ZhB07} to be the identity. The $z$ there is thus equal to the $w_1$ there. We first note that $N\ni \begin{pmatrix}1 & z & t+z\zeta z^*/2\\ & 1_2 & \zeta z^*\\ & & 1\end{pmatrix}$ acting on $i$ gives $\begin{pmatrix}i+t+z\zeta z^*/2 \\ \zeta z^*\end{pmatrix}$. Thus the complex structure on $N/Z(N)_\infty$ is given by the complex conjugation of $z$. (Note that the $z$ is not the $z$ in \cite{ZhB07} -- instead it plays the role of $\bar{u}$ there.

Now write $z=(z_1,z_2)$, and write $x=(x_1,x_2)\in\mathbb{Q}^2$. A straightforward computation using the formulas for Weil representation of $H(W)$ implies that the \emph{classical} theta function is a sum
$$\sum_x \prod_{v<\infty}\phi_v(x)e^{-(x_1^2+2x_1\bar{z}_1+\frac{1}{2}\bar{z}^2_1)2\pi\sqrt{d}}\cdot e^{-(x_2+2x_2\bar{z}_2+\frac{1}{2}\bar{z}^2_2)2\pi\mathfrak{s}\sqrt{d}}$$
for $x$ running over some lattice of $\mathbb{Q}^2$. This is clearly holomorphic with respect to the complex structure. In fact comparing with the notations in \cite{ZhB07}, taking the $u$ there to be $2(\bar{z}_1, \bar{z}_2)\zeta$ and $z$ there to be $2\zeta$, this theta function is nothing but the one considered in \emph{loc.cit.}.

In applications later on, we also take Schwartz functions in $\phi_f\in\mathcal{S}(\mathbb{A}_f)$ and consider the associated theta function $\Theta_\phi$ for $\phi=\phi_{1,\infty}\times\phi_f$. Suppose the $\Theta_\phi$ is right invariant under the open compact level group $K\subseteq NU(\zeta)(\mathbb{A}_f)$. Then we define $L$ to be a certain lattice contained in
$$(N(\mathbb{Q})\cap K)/(Z(N)(\mathbb{Q})\cap K)$$
satisfying the assumption after Definition \ref{adelictheta}, and $U_f=K\cap \mathrm{U}(\zeta)(\mathbb{A}^\infty)$. The associated classical theta functions $\theta_\phi$ is indeed in $T^{\mathrm{Hol}}_{\mathbb{A}}(1,L,U_f)$.
\end{example}

\subsection{Unramified Cases}
Let $v$ be a prime outside $\Sigma$ (in particular $v\nmid p$). Then the Siegel sections $f_{sieg,\mathcal{D},v}=f_{v,sieg}=f_v^{sph}$ and $f_{sieg,\mathcal{D},v}'=f_{v,sieg}'=f_v^{sph,'}$ is defined to be the unique section that is invariant under $\mathrm{GU}(n,n)(\mathbb{Z}_v)$ ($n=3,2$) and is $1$ at identity.

\begin{lemma}\label{Upullback}
Suppose $\pi,\psi$ and $\tau$ are unramified and $\varphi\in\pi$ is a new vector. If $Re(z)>3/2$ then the pullback integral converges and
$$F_{Kling,\mathcal{D},v}:=F_\varphi(f_v^{sph};z,g)=
\frac{L(\tilde\pi,\xi,z+1)}{\prod_{i=0}^{1}L(2z+3-i,\bar\tau'\chi_\mathcal{K}^i)}
F_{\rho,z}(g)$$
where $F_\rho$ is the spherical section defined using $\varphi\in\pi$.
Also:
$$F'_{\mathcal{D},v}:=F_\varphi'(f_v^{sph,'},z,g)=\frac{L(\tilde{\pi},\xi,z+\frac{1}{2})}{\prod_{i=0}^1L(2z+2-i,\bar{\tau}'\chi_\mathcal{K}^i)}\pi(g)\varphi.$$
\end{lemma}
This is a special case of \cite[Proposition 3.3]{LR}.

\noindent\underline{Fourier Coefficients}
\begin{definition}
Let $\Phi_0$ be the characteristic function of $\mathcal{O}_\mathcal{K}^2$.
\end{definition}
\begin{lemma}
Let $\beta\in S_n(\mathbb{Q}_v)$ and let $r:=\mathrm{rank}(\beta)$. Then for $y\in \mathrm{GL}_n(\mathcal{K}_v)$,
\begin{eqnarray*}
f_{v,\beta}^{sph}(z,diag(y,{}^t\!\bar y^{-1}))=&\tau({\det} y)|{\det} y\bar y|_v^{-z+n/2}D_v^{-n(n-1)/4}\\
&\times \frac{\prod_{i=r}^{n-1} L(2z+i-n+1,\bar\tau'\chi_\mathcal{K}^i)}{\prod_{i=0}^{n-1}L(2z+n-i,\bar\tau'\chi_\mathcal{K}^i)}h_{v,{}^t\!\bar y\beta y}(\bar\tau'(q_v)q_v^{-2z-n}).
\end{eqnarray*}
where $h_{v,{}^t\!\bar y\beta y}\in \mathbb{Z}[X]$ is a monic polynomial depending on $v$ and ${}^t\!\bar y\beta y$ but not on $\tau$. If $\beta\in S_n(\mathbb{Z}_v)$ and $\det \beta\in \mathbb{Z}_v^\times$, then we say that $\beta$ is $v$-primitive and in this case $h_{v,\beta}=1$.
\end{lemma}
\begin{proof}
This is a computation of Shimura in \cite[Propositions 18.14 and 19.2]{Shi97}. See also \cite[Lemma 11.7]{SU}.
\end{proof}
\noindent\underline{Fourier-Jacobi Coefficients}
\begin{lemma}
Suppose $v\not\in\Sigma$ and not dividing $p$. Let $\beta\in S_1(\mathbb{Q}_v)$ such that $\beta\not=0$. Let $y\in \mathrm{GL}_{2}(\mathcal{K}_v)$ be such that ${}^t\!\bar y\beta y\in S_1(\mathbb{Z}_v)$, let $\lambda$ be an unramified character of $\mathcal{K}_v^\times$ such that $\lambda|_{\mathbb{Q}_v^\times}=1$.
If $\beta\in \mathrm{GL}_1(\mathcal{O}_{\mathcal{K},v})$, then for $u\in \mathrm{U}_\beta(\mathbb{Q}_v)$, we have
$$\mathrm{FJ}_\beta(f_3^{sph};z,x,g,u)=\tau(\det u)|\det u\bar u|_v^{-z+1/2}
\frac{f_{2}^{sph}(z,g')(\omega_{\beta,\lambda_v}(u,g')\Phi_{0})(x)}{L(2z+3,\bar\tau')}.$$
Here $g'=\begin{pmatrix}1_2&\\&-1_2\end{pmatrix}g\begin{pmatrix}1_2&\\&-1_2\end{pmatrix}$, and the $f_2^{\mathrm{sph}}$ on the right is in $I(\tau/\lambda)$.
\end{lemma}
This is a formal generalization of \cite[Lemma 11.8]{SU}.
\begin{definition}\label{Un Kernel}
Recall we have defined $f_{sieg,\mathcal{D},v}=f_{v,sieg}=f_v^{sph}$ and $f_{sieg,\mathcal{D},v}'=f_{sieg,v}'=f_v^{sph,'}$. We also define $\phi_{1,v}$ and $\phi_{2,v}$ to be the Schwartz function on $X_v$ which is the characteristic function of $\mathbb{Z}_v^2$. We define $f_{2,\mathcal{D},v}=f^{sph,'}_v$, $\Phi_{\mathcal{D},v}=\Phi_0$ and $\Phi_{\mathcal{D},v}''=\Phi_0$.
\end{definition}
\subsection{Ramified Cases}\label{5.7}
Let $f^\dag\in I_n(\tau)$ ($n=2$ or $3$) be the Siegel section supported on $Q(\mathbb{Q}_v)w_nQ(\mathbb{Z}_v)$, which takes value $1$ on $w_nN_Q(\mathbb{Z}_v)$. The Siegel section we choose is $I_3(\tau)\ni f_{sieg,v}=f_{sieg,\mathcal{D},v}=f^\dag(g\tilde{\gamma}_v)$ where $\tilde{\gamma}_v$ is $\begin{pmatrix}1&&&\\&1_2&&\frac{1}{y\bar{y}}1_2\\&&1&\\&&&1_2\end{pmatrix}$ where $y\in\mathcal{O}_v$ is some fixed element such that the valuation is sufficiently large (can be made precise in the text). We also define $$I_2(\tau)\ni f_{sieg,\mathcal{D},v}=f_{sieg,v}'=f^\dag(g\tilde{\gamma}_v')$$ where $\tilde{\gamma}_v'=\begin{pmatrix}1_2&\frac{1}{y\bar{y}}1_2\\&1_2\end{pmatrix}$. \\

\noindent\underline{Pullback Formulas}\\
Recall the notations in Section \ref{GKES}.

\begin{lemma}\label{Lpullback}
Let $\varphi$ be some vector invariant under the action of $\mathfrak{Y}$ defined before, then $$F_{Kling,\mathcal{D},v}:=F_{\varphi}(z,w)=\tau(y\bar y )|(y\bar y)^2|_v^{-z-\frac{3}{2}}\mathrm{Vol}(\mathfrak{Y})\cdot\varphi.$$ Also $F'_{\mathcal{D},v}:=F_\varphi'(f_{v,sieg}';z,g)=\tau(y\bar y )|(y\bar y)^2|_v^{-z-1}\mathrm{Vol}(\mathfrak{Y})\cdot\pi(g)\varphi$.
\end{lemma}
The proof is a special case of \cite[Lemma 4.9, 4.10]{WAN}.

\noindent\underline{Fourier Coefficients}
\begin{lemma}
\begin{itemize}
\item[(i)] Let $\beta\in S_{3}(\mathbb{Q}_\ell)$. Then $f_{v,\beta}(z,1)=0$ if $\beta\not\in S_{3}(\mathbb{Z}_\ell)^*$. If $\beta\in S_{3}(\mathbb{Z}_\ell)^*$ then
$$f_{v,\beta}(z,\mathrm{diag}(A,{}^t\!\bar{A}^{-1}))=D_\ell^{-\frac{3}{2}}\tau(\det A)|\det A\bar{A}|_\ell^{-z+\frac{3}{2}}e_\ell(\frac{\beta_{22}+\beta_{33}}{y\bar{y}})$$
where $D_\ell$ is the discriminant of $\mathcal{K}_\ell$.
\item[(ii)] If $\beta\in S_2(\mathbb{Q}_\ell)$, then $f_{v,\beta}(z,1)=0$ if $\beta\in S_2(\mathbb{Z}_\ell)^*$. If $\beta\in S_2(\mathbb{Z}_\ell)^*$ then
$$f_{v,\beta}'(z,\mathrm{diag}(A,{}^t\!\bar{A}^{-1}))=D_\ell^{-\frac{1}{2}}\tau(\det A)|\det A\bar{A}|_\ell^{-z+\frac{r}{2}}e_\ell(\frac{\beta_{11}+\beta_{22}}{y\bar{y}}).$$
\end{itemize}
\end{lemma}
The proof is a special case of \cite[Lemma 4.12]{WAN}.

\noindent\underline{Fourier-Jacobi Coefficients}
\begin{lemma}
If $\beta\not\in S_1(\mathbb{Z}_v)^*$ then $\mathrm{FJ}_\beta(f^\dag;z,x,g,1)=0$. If $\beta\in S_1(\mathbb{Z}_v)^*$ then
$$\mathrm{FJ}_\beta(f^\dag;z,x,g,1)=f^\dag(z,g'\eta)\omega_{\beta,\lambda_v}(h,g'\eta^{-1})\Phi_{0}(x).
\mathrm{Vol}(S_1(\mathbb{Z}_v)),$$
where $g'=\begin{pmatrix}1_{2}&\\&-1_{2}\end{pmatrix}g\begin{pmatrix}1_{2}&\\&-1_{2}\end{pmatrix}$.
\end{lemma}
The proof is a special case of \cite[Lemma 4.13]{WAN}.
\begin{definition}\label{Ra Kernel}
Let $A=\frac{1}{y\bar{y}}1_2$. Thus
$$\mathrm{FJ}_\beta(f_{sieg,v};z,x,g,h)=f^\dag(z,g'\begin{pmatrix}1&\\-A&1\end{pmatrix}\eta)(\omega_{\beta,\lambda_v}(h,g'
\begin{pmatrix}1&\\-A&1\end{pmatrix}\eta))\Phi_0(x)$$
for $h\in \mathrm{U}_\beta(\mathbb{Q}_\ell)$.
We define $\Phi_{\mathcal{D},v}'':=\omega_\beta(\begin{pmatrix}1&A\\&1\end{pmatrix})\Phi_0$ and $\Phi_{\mathcal{D},v}=\omega_{\beta,\lambda_v}(
\begin{pmatrix}1&\\-A&1\end{pmatrix}\eta)\Phi_0$. We also define $f_{2,\mathcal{D},v}=\rho(\begin{pmatrix}1&\\-A&1\end{pmatrix}\eta)f^\dag\in I_2(\frac{\tau}{\lambda})$.
\end{definition}
\noindent\underline{Split Case}\\
\noindent Suppose $v=w\bar{w}$ is a split prime. Recall we have the local polarization $X'_v\oplus Y'_v$. Now we write $x_1'=(x_{11}',x_{12}')$ and $x_2'=(x_{21}',x_{22}')$ with respect to $\mathcal{K}_v\simeq \mathcal{K}_w\times\mathcal{K}_{\bar{w}}$.
The following lemma follows from a straightforward computation and will be used later.
\begin{lemma}\label{lemma6.24}
Let $\chi_{\theta,v}$ be a character of $\mathbb{Z}_v^\times$ such that
$$\mathrm{cond}(\lambda_v)<\mathrm{cond}(\chi_{\theta,v})< \mathrm{ord}_v(y\bar{y}).$$
Then it is possible to choose a Schwartz function $\phi_1'$ such that the function
$$\phi_2'(x_2'):=\int_{X_1'}\delta_\psi'^{-1}(\Phi_v'')(x_1',x_2')\phi'_1(x_1')dx_1'$$ is given by
$$\phi_2'(x_2')=\left\{\begin{array}{ll}\lambda_v\chi_{\theta,v}(x_{22}') & x_{21}'\in\mathbb{Z}_v,x_{22}'\in \mathbb{Z}_v^\times\\0&\mbox{ otherwise. }\end{array}\right.$$
Moreover we can ensure that when we are moving our datum in $p$-adic families, this $\phi_1'$ is not going to change.
\end{lemma}
\begin{remark}
Note here we have flexibility on choosing the $\chi_{\theta,v}$ for split $v$. It is an analogue of the fact that in the doubling method, if we choose the Siegel section as in the beginning of this subsection, then the local pullback integral is nonzero if the test vector has appropriate conductor. This flexibility is important for our argument in Subsection \ref{section 7.2}.
\end{remark}
We define
\begin{equation}\label{spcase}
\phi_{1,v}=\delta_\psi''(\phi_1'), \phi_{2,v}=\delta_\psi^{-,''}(\phi_2').
\end{equation}

\noindent\underline{Non-split Case}
\begin{lemma}\label{5.18}
We consider the action of the compact abelian group $\mathrm{U}(1)(\mathbb{Q}_v)$ on $\delta_\psi^{-1}(\Phi_v'')$ by the Weil representation (using the splitting character $\lambda_v$) of
$$1\times \mathrm{U}(1)(\mathbb{Q}_v)\hookrightarrow 1\times \mathrm{U}(2)(\mathbb{Q}_v)\hookrightarrow \mathrm{U}(2)(\mathbb{Q}_v)\times \mathrm{U}(2)(\mathbb{Q}_v)\hookrightarrow \mathrm{U}(2,2)(\mathbb{Q}_v).$$
We can write $\delta_\psi^{-1}(\Phi_v'')$ as a sum of eigenfunctions of this action. Let $m=\mathrm{max}\{\mathrm{ord}_v(\mathrm{cond}\lambda_v),3)+1$. If $\mathrm{ord}_v(y\bar{y})>m$, then there is such a nonzero eigenfunction $\phi_{2,v}$ whose eigenvalue is a character $\lambda_v^{2}\chi_{\theta,v}$ for $\chi_{\theta,v}$ of conductor at least $\varpi_v^m$, such that there is a $\bar{\mathbb{Q}}_p$-valued Schwartz function $\phi_{1,v}$ and a set of constants $C_{v,i}\in\bar{\mathbb{Q}}_p$ and $u_{v,i}\in \mathrm{U}(1)(\mathbb{Q}_v)$'s such that the function
$$\phi_{2,v}(x_2)=\int_{X_1(\mathbb{Q}_v)}\sum_i\delta_\psi^{-1}(C_{v,i}\omega_{\beta,\lambda_v}(u_{v,i},1)\Phi_v'')
(x_1,x_2)\phi_{1,v}(x_1)dx_1$$
(here $1\in \mathrm{U}(2,2)(\mathbb{Q}_v)$).
\end{lemma}
\begin{proof}
Consider the embedding
$\mathrm{U}(1,1)\hookrightarrow \mathrm{U}(2,2)$ by
$$j: g\mapsto \begin{pmatrix}\frak{s}^{-1}&&&\\&1&&\\&&\mathfrak{s}^{-1}&\\&&&1\end{pmatrix}
\begin{pmatrix}a_g&&b_g&\\&a_g&&b_g\\d_g&&d_g&\\&c_g&&d_g\end{pmatrix}
\begin{pmatrix}\frak{s}&&&\\&1&&\\&&\frak{s}&\\&&&&1\end{pmatrix}$$
for $g=\begin{pmatrix}a_g&b_g\\c_g&d_g\end{pmatrix}$.
Define $f_{\Phi_v''}(g,\frac{1}{2})=(\omega_{\beta,\lambda_v}(j(g))\Phi_v'')(0)\in I_2(\frac{\tau}{\lambda_v})$. We define $i: \mathrm{U}(1)\times \mathrm{U}(1)\hookrightarrow \mathrm{U}(1,1)$ by
\begin{equation}\label{defni}
i(g_1,g_2)=\begin{pmatrix}\frac{1}{2}&-\frac{1}{2}\\-\delta^{-1}&-\delta^{-1}\end{pmatrix}
\begin{pmatrix}1&\\&g_1\end{pmatrix}\begin{pmatrix}1&-\frac{\delta}{2}\\-1&-\frac{\delta}{2}\end{pmatrix} .
\end{equation}
For $g_1\in \mathrm{U}(1)(\mathbb{Q}_v)$,
$$f_{\Phi_v''}(i(1,g_1),\frac{1}{2})=(\omega_{\beta,\lambda_v}(j\circ i(1,g_1))\Phi_v'')(0)=(\delta_\psi\omega_{\lambda_v}(1,g_1)\delta_\psi^{-1}\Phi_v'')(0).$$
Here in the first and last expression $1\in \mathrm{U}(2)(\mathbb{Q}_v)$ and $g_1$ is viewed as the element in the center of $\mathrm{U}(2)(\mathbb{Q}_v)$.
Thus we are reduced to proving the following lemma.
\end{proof}
\begin{lemma}\label{5.19}
Let $g_1=1+\varpi_v^m.a\in \mathrm{U}(1)(\mathbb{Q}_v)$ for $m$ as in the above lemma and $a\in\mathcal{O}_{\mathcal{K},v}$, if $n=y\bar{y}$ is such that $\mathrm{ord}_vn>m$ then $f_{\Phi_v''}(i(1,g_1);\frac{1}{2})\not=f_{\Phi_v''}(1;\frac{1}{2})$.
\end{lemma}
\begin{proof}
We have as follows:
\begin{align*}
&&&\begin{pmatrix}\frac{1}{2}&-\frac{1}{2}\\-\delta^{-1}&-\delta^{-1}\end{pmatrix}\begin{pmatrix}1&\\&g_1\end{pmatrix}\begin{pmatrix}1&-\frac{\delta}{2}\\-1&-\frac{\delta}{2}\end{pmatrix}&\\
&=&&\begin{pmatrix}\frac{1}{2}&-\frac{g_1}{2}\\-\delta^{-1}&-\delta^{-1}g_1\end{pmatrix}\begin{pmatrix}1&-\frac{\delta}{2}\\-1&-\frac{\delta}{2}\end{pmatrix}&\\
&=&&\begin{pmatrix}\frac{1}{2}+\frac{g_1}{2}&-\frac{\delta}{4}+\frac{\delta g_1}{4}\\-\delta^{-1}+\delta^{-1}g_1&\frac{1}{2}+\frac{g_1}{2}\end{pmatrix}&
\end{align*}
and
$$\begin{pmatrix}a&b\\c&d\end{pmatrix}\begin{pmatrix}1&\frac{1}{n}\\&1\end{pmatrix}=
\begin{pmatrix}a&\frac{a}{n}+b\\c&\frac{c}{n}+d\end{pmatrix}=\begin{pmatrix}\frac{n(ad-bc)}{c+nd}&\frac{a}{n}+b\\0&\frac{c}{n}+d\end{pmatrix}
\begin{pmatrix}1&\\ \frac{n}{1+\frac{nd}{c}}&1\end{pmatrix}.$$
Now the lemma follows readily.
\end{proof}
\begin{definition}
From now on we fix the choices and define the local characters $\chi_{\theta,v}$ as in Lemma \ref{lemma6.24} and Lemma \ref{5.18}.
\end{definition}
As before we summarize the definitions associated to our Eisenstein datum at $v$.
\begin{definition}
Recall we defined $f_{sieg,\mathcal{D},v}=f_{sieg,v}=f^\dag(g\tilde{\gamma}_v)$ where $\tilde{\gamma}_v$ is $\begin{pmatrix}1&&&\\&1_2&&\frac{1}{y\bar{y}}1_2\\&&1&\\&&&1_2\end{pmatrix}$, and defined $$f'_{sieg,\mathcal{D},v}=f_{sieg,v}'=f^\dag(g\tilde{\gamma}_v')$$ where $\tilde{\gamma}_v'=\begin{pmatrix}1_2&\frac{1}{y\bar{y}}1_2\\&1_2\end{pmatrix}$.
Let $$f_{2,\mathcal{D},v}=\rho(\begin{pmatrix}1&\\-A&1\end{pmatrix}\eta)f^\dag\in I_2(\frac{\tau}{\lambda}).$$
We define
$\Phi_{\mathcal{D},v}=\omega_{\beta,\lambda_v}(
\begin{pmatrix}1&\\-A&1\end{pmatrix}\eta)\Phi_0$. We also define the $\phi_{1,v}$ and $\phi_{2,v}$ as in (\ref{spcase}) or as in Lemma \ref{5.18}, depending on whether $v$ is split or not.
\end{definition}

\subsection{$p$-adic Cases}\label{Sectionn}
We recall some results in \cite[Section 4.D]{WAN} with some modifications. Recall that we have the triple $(\pi_p, \psi_p, \tau_p)$ and $\xi_p:=\psi_p/\tau_p$ in the $p$-component of the Klingen-Eisenstein datum $\mathcal{D}$, where $\chi_p$ is the central character of $\pi_p$ and $\psi_p|_{\mathbb{Q}_p^\times}=\chi_p$. Suppose $\pi_p$ is nearly ordinary in the sense that $\pi_p=\pi(\chi_{1,p},\chi_{2,p})$ such that $\mathrm{ord}_p(\chi_{1,p}(p))=-\frac{1}{2}$ and $\mathrm{ord}_p(\chi_{2,p}(p))=\frac{1}{2}$. We write $\tau_p=(\tau_1,\tau_2)$ and $\xi_p=(\xi_1,\xi_2)$.
\begin{definition}\label{definegeneric}
The triple $(\pi_p, \psi_p, \tau_p)$ is generic if there is a $t\geq 2$ such that $\xi_{1}, \xi_{2},\chi_p, \chi_p^{-1}\xi_{1}, \chi_p^{-1}\xi_{2}$ all have conductor $p^t$.
\end{definition}
Although the definition for generic points is different from \cite[Definition 4.21]{WAN}, the argument there goes through since the only place using this definition is Lemma 4.19 there, which can be proved completely in the same way under our definition for generic. We define: $\xi_1^\dag=\chi_{1,p}\bar{\xi}_2, \xi_2^\dag=\chi_{2,p}\bar{\xi}_2$. Let $K^{(3,3)}\subseteq \mathrm{GL}_6(\mathbb{Z}_p)$ be the subgroup consisting of matrices congruent to upper triangular matrices modulo $p^t$. We define
$f_t$ to be the Siegel section supported in $Q(\mathbb{Q}_v)K_t$, invariant under $K^{(3,3)}_t$ and takes value $1$ on the identity. We define
\begin{align}\label{siegelatp}
&f_{sieg,\mathcal{D},p}=f_{sieg,p}(g)&&:=\mathfrak{g}(\tau_p')^{-3}c_3^{-1}(\bar{\tau}_p',-z)p^{-3t}\mathfrak{g}(\xi_1^\dag)
\xi_1^\dag(-1)\mathfrak{g}(\xi_2^\dag)\xi_2^\dag(-1)&\\
&&&\times\sum_{a,b\in p^{-t}\mathbb{Z}_p^\times/\mathbb{Z}_p,m,n\in\mathbb{Z}_p/p^t\mathbb{Z}_p}\bar{\xi}_1^\dag(p^ta)\bar{\xi}_2^\dag(p^tb)f_t(g\Upsilon
\begin{pmatrix}1&&&&a+bmn&bm\\&1&&&bn&b\\&&1&&&\\
&&&1&&\\&&&&1&\\&&&&&1\end{pmatrix}).&
\end{align}
and
\begin{align*}
&f^\Box_{sieg,\mathcal{D},p}=f_{sieg,p}^{\Box}(g)&&:=\mathfrak{g}(\tau_p')^{-3}c_3^{-1}(\bar{\tau}_p',-z)p^{-3t}\mathfrak{g}
(\xi_1^\dag)\xi_1^\dag(-1)\mathfrak{g}(\xi_2^\dag)\xi_2^\dag(-1)&\\
&&&\times\sum_{a,b\in p^{-t}\mathbb{Z}_p^\times/\mathbb{Z}_p,m\in\mathbb{Z}_p/p^t\mathbb{Z}_p}\bar{\xi}_1^\dag(p^ta)\bar{\xi}_2^\dag(p^tb)f_t(g\Upsilon
\begin{pmatrix}1&&&&a&bm\\&1&&&&b\\&&1&&&\\
&&&1&&\\&&&&1&\\&&&&&1\end{pmatrix}).&
\end{align*}

We also define:

\begin{align}\label{siegelatp'}
&f'_{sieg,\mathcal{D},p}=f_{sieg,p}'(g)&&:=\mathfrak{g}(\tau_p')^{-3}c_2^{-1}(\bar{\tau}_p',-z)p^{-3t}\mathfrak{g}
(\xi_1^\dag)\xi_1^\dag(-1)\mathfrak{g}(\xi_2^\dag)\xi_2^\dag(-1)&\\
&&&\times\sum_{a,b\in p^{-t}\mathbb{Z}_p^\times/\mathbb{Z}_p,m,n\in\mathbb{Z}_p/p^t\mathbb{Z}_p}\bar{\xi}_1^\dag(p^ta)\bar{\xi}_2^\dag(p^tb)f_t
(g\Upsilon\begin{pmatrix}1&&a+bmn&bm\\&1&bn&b\\&&1&\\&&&1\end{pmatrix}).&
\end{align}
Here $\Upsilon\in \mathrm{U}(3,3)(\mathbb{Q}_p)$ is such that it is $\begin{pmatrix}1&&&\\&\frac{1}{2}.\mathrm{1_2}&&-\frac{1}{2}\mathrm{1_2}\\&&1&\\&-\zeta^{-1}&&-\zeta^{-1}\end{pmatrix}$ via the first projection
$\mathrm{U}(3,3)(\mathbb{Q}_p)\simeq \mathrm{GL}_6(\mathbb{Q}_p)$ and $c_n(\tau',z)=\tau'(p^{nt})p^{2ntz-tn(n+1)/2}$.

Among these sections the $f_{sieg,\mathcal{D},p}$ and $f'_{sieg,\mathcal{D},p}$ correspond to Siegel Eisenstein series that we interpolate. The relations between these sections are
$$f_{sieg,\mathcal{D},p}(g)=\sum_{n\in \mathbb{Z}_p/p^t\mathbb{Z}_p}f^\Box_{sieg,\mathcal{D},p}(g\gamma(1,\begin{pmatrix}1&\\n&1\end{pmatrix}_p).$$
The reason for introducing the $\Box$ sections is to help computing the Fourier-Jacobi expansion.

\noindent\underline{Pullback Formulas}\\
\noindent We refer to \cite[Subsection 4.D.1]{WAN} for the discussion of nearly ordinary vectors, which means the vector whose $U_p$-eigenvalues are $p$-adic units. Let $\varphi=\varphi^{ord}\in\pi_p$ be a nearly ordinary vector.
\begin{definition}
We define the $p$-adic Klingen section $F^0(g)$ to be the $F^{0,\bullet}_p$ defined in Definition \ref{Definition 6.27}, multiplied by
$$\mathfrak{g}(\tau_p')^{-1}\tau_p'(p^{-t})p^{\kappa-2}p^{(\kappa-3)t}\xi_{1,p}^2\chi_{1,p}^{-1}
\chi_{2,p}^{-1}(p^{-t})\mathfrak{g}(\xi_{1,p}\chi_{1,p}^{-1})
\mathfrak{g}(\xi_{1,p}\chi_{2,p}^{-1}).$$
\end{definition}

Then by the computations in \cite{WAN} we have the following (see the end of \cite[Section 4]{WAN}).
\begin{lemma}\label{Ppullback}
(1) $F_\varphi(f_{sieg,p},z_\kappa,g)=F^0(g):=F_{Kling,p}(g)$;\\
(2) $F_\varphi'(f_{sieg,p}',z_\kappa,g)=p^{(\kappa-3)t}\xi_{1,p}^2\chi_{1,p}^{-1}
\chi_{2,p}^{-1}(p^{-t})\mathfrak{g}(\xi_{1,p}\chi_{1,p}^{-1})
\mathfrak{g}(\xi_{1,p}\chi_{2,p}^{-1}).\pi(g)\varphi.$
\end{lemma}
\begin{proof}
For reader's convenience we briefly recall the proof of \cite{WAN} in case (1) of this lemma. In our $\mathrm{U}(3,1)$ case the proof is actually much simpler than the general case considered in \emph{loc.cit.}. The proof uses the trick of \cite[Proposition 11.13]{SU} of using the local and global functional equations to reduce the local pullback integral for $f_{sieg,p}$ to that of another Siegel section $f^\dag_{sieg,p}$ defined below. The computation of the pullback section of $f^\dag_{sieg,p}$ is easier than that of $f_{sieg,p}$. We define an auxiliary Siegel section $\tilde{f}^\dag_{sieg,p}$ to be supported in $Q_3(\mathbb{Q}_p)wQ_3(\mathbb{Z}_p)$, and such that it takes value $1$ on $wN_{Q_3}(\mathbb{Z}_p)$. We define
\begin{align*}
&f^\dag_{sieg,p}(g,z)&&:=\mathfrak{g}(\tau_p')^{-3}c_3^{-1}(\bar{\tau}_p',-z)p^{-3t}\mathfrak{g}(\xi_1^\dag)
\xi_1^\dag(-1)\mathfrak{g}(\xi_2^\dag)\xi_2^\dag(-1)&\\
&&&\times\sum_{a,b\in p^{-t}\mathbb{Z}_p^\times/\mathbb{Z}_p,m,n\in\mathbb{Z}_p/p^t\mathbb{Z}_p}\bar{\xi}_1^\dag(p^ta)\bar{\xi}_2^\dag(p^tb)\tilde{f}^\dag_{sieg,p}(g\Upsilon
\begin{pmatrix}1&&&&a+bmn&bm\\&1&&&bn&b\\&&1&&&\\
&&&1&&\\&&&&1&\\&&&&&1\end{pmatrix},z).&
\end{align*}
We can show by direct checking that
the $F_\varphi(f^\dag_{sieg,p},z,g)$ is invariant under $N_t(\mathbb{Z}_p)$ defined above (as a function of $g$). This can be seen by checking that the pullback of the Siegel section $f^\dag_{sieg,p}$ is already invariant under $N_t(\mathbb{Z}_p)$. Moreover the value of $F_\varphi(f^\dag_{sieg,p},z,g)$ at $ww'_3$ can also be computed directly (\cite[Lemma 4.38]{WAN}). Returning to the pullback section of $f_{sieg,p}$, we note that it is the image of $f^\dag_{sieg,p}$ under the intertwining operator $M(z,-)$. We apply \cite[Proposition 4.40]{WAN} (which is just a variant of \cite[Proposition 11.13]{SU}). By the uniqueness up to scalar of the vector with the same action of the level group $B_t(\mathbb{Z}_p)$ (see \cite[Lemma 4.19]{WAN}, which is just a variant of \cite[Proposition 9.5]{SU}. Note that the result in \cite{WAN} also works under our assumption of genericity), we know that
$F_\varphi(f_{sieg,p},z_\kappa,g)=F^0(g)$ up to scalar. Then applying \cite[Proposition 4.40]{WAN} again we can evaluate the $F_\varphi(f_{sieg,p},z_\kappa,g)$ at $w'_3$ and get
the lemma. Note that this simplified proof does not apply to the general case of \cite{WAN} including the $\mathrm{U}(2,2)$ case of \cite{SU}, since there by looking at the Siegel Eisenstein section we can only determine the action of a level group which is smaller than the level group corresponding to the ordinary vector.
\end{proof}
\noindent\underline{Fourier Coefficients}
\begin{definition}\label{39}
We define the function $\Phi_{\xi^\dag}$ as the function on the set of ($2\times 2$) $\mathbb{Q}_p$-matrices as follows. If $x=\begin{pmatrix}a&b\\c&d\end{pmatrix}$ is such that both its determinant and $a$ are in $\mathbb{Z}_p^\times$ then $\Phi_{\xi^\dag}(x)=\xi_1^\dag(a)\xi_2^\dag(\frac{\det x}{a})$. Otherwise $\Phi_{\xi^\dag}(x)=0$. The following lemma is proved in \cite[Lemma 4.46]{WAN}.
\end{definition}
\begin{lemma}
$$f_{sieg,p,\beta}(1)=\bar{\tau}'_p(\det\beta)|\det\beta|^{\kappa-3}_p\Phi_{\xi^\dag}(\begin{pmatrix}\beta_{21}&\beta_{22}\\ \beta_{31}&\beta_{32}\end{pmatrix})$$
for $\beta=\begin{pmatrix}\beta_{11}&\beta_{12}&\beta_{13}\\ \beta_{21}&\beta_{22}&\beta_{23}\\ \beta_{31}&\beta_{32}&\beta_{33}\end{pmatrix}$ with $\beta_{11}, \beta_{12}, \beta_{13}, \beta_{23}, \beta_{33}\in\mathbb{Z}_p$, and is $0$ otherwise.
\end{lemma}
\begin{proof}
This follows from straightforward computation using
\begin{itemize}
\item The Fourier coefficients of the section $f_t$ as computed in \cite[Lemma 11.12]{SU}.
\item The computation of the Fourier transform of the function $\Phi_{\xi^\dag}$ defined above as detailed in \cite[Lemma 4.28]{WAN} (note that the proof also works under our assumption of genericity). We compare it with the definition of $f_{sieg,p}$ using $f_t$.
\end{itemize}
\end{proof}

\noindent\underline{Fourier-Jacobi Coefficients}\\
\noindent For $\beta\in S_1(\mathbb{Q}_v)\cap \mathrm{GL}_1(\mathbb{Z}_v)$ we compute the Fourier-Jacobi coefficient for $f_t$ at $\beta$. We have the following (\cite[Lemma 4.54]{WAN})
\begin{lemma}
Let $x:=\begin{pmatrix}1&\\D&1\end{pmatrix}$ (this is a block matrix with respect to $(2+2)$).\\
(a) $\mathrm{FJ}_\beta(f_t;-z,v,x\eta^{-1},1)=0$ if $D\not\in p^t M_{2}(\mathbb{Z}_p)$;\\
(b) if $D\in p^t M_2(\mathbb{Z}_p)$ then $\mathrm{FJ}_\beta(f_t;-z,v,x\eta^{-1},1)=c(\beta,\tau,z)\Phi_0(v)$, where
$$c(\beta,\tau,z):=\bar\tau(-\det \beta)|\det\beta|_v^{2z+2}\mathfrak{g}(\tau')\mathfrak{g}(\tau_p')\tau_p'(p^t)p^{-2tz-3t}.$$
\end{lemma}

\noindent Note the formula:
$$\mathfrak{g}(\tau'_p)^3\tau'_p(p^{3t})p^{-6tz-6t}=\mathfrak{g}(\tau'_p)\tau'_p(p^t)p^{-2tz-3t}\mathfrak{g}(\tau'_p)^2\tau'_p(p^{2t})p^{-4tz-3t}.$$
\begin{definition}\label{Pa Kernel}
Let $K'_t\subseteq \mathrm{GL}_4(\mathbb{Z}_p)$ be the subgroup consisting of elements congruent to upper triangular matrices modulo $p^t$. We define the Siegel section $f'_t$ on $I_2(\frac{\tau}{\lambda})$ to be the section supported on $Q_2(\mathbb{Q}_p)K'_t$, invariant under the right action of $N_B(\mathbb{Z}_p)$ and takes value $1$ on the identity element. Let $A_{x}$ be the matrix $\begin{pmatrix}0&x\\0&0\end{pmatrix}$.
We define $f_{2,p}=\sum_b\mathfrak{g}(\xi_2^\dag)p^{-t}\bar{\xi}_2^\dag(-p^tb)\mathfrak{g}(\tau_p')^{-2}c_2^{-1}
(\bar{\tau}_p',-z)\rho(\begin{pmatrix}1&\\A_{-b}&1\end{pmatrix}\eta)f'_t$.
Let $\phi_{\xi_1^\dag}(x):=\xi^\dag_1(x^{-1})\mathds{1}_{\mathbb{Z}^\times_p}(x)$ where $\mathds{1}$ denotes the characteristic function. Define
$$\Phi_p(v_1, v_2, v_3, v_4)=\mathds{1}_{\mathbb{Z}_p\oplus\mathbb{Z}_p}(v_1, v_2)\hat{\phi}_{\xi^\dag_1}(v_3)\hat{\mathds{1}}_{p^t\mathbb{Z}_p}(v_4)$$
where the $\hat{\phi}$ means taking Fourier transform of $\phi$. The precise formula is given by
\begin{equation}\label{p-kernel}
\Phi_{\mathcal{D},p}(v_1,v_2,v_3,v_4)=\left\{\begin{array}{ll}\mathfrak{g}(\xi^\dag_1)p^{-2t}\bar{\xi}_1^\dag(a), & v_1,v_2\in\mathbb{Z}_p,v_4\in p^{-t}\mathbb{Z}_p,v_3\in \frac{a}{p^t}+\mathbb{Z}_p \mbox{ for some $a\in\mathbb{Z}_p^\times$ }\\
0, &\mbox{ otherwise.}\end{array}\right.
\end{equation}

These are defined just for computing Fourier-Jacobi coefficients and not used for the pullback formula.
\end{definition}

We prove the following lemma.
\begin{lemma}
\begin{eqnarray*}
&\mathfrak{g}(\tau_p')^{-3}c_3(\bar{\tau}_p',-z_\kappa)p^{-3t}\sum_{a,b,m}{\mathrm{FJ}}_1
(\rho\begin{pmatrix}1_3&\begin{matrix}&a&bm\\&&b\\&&\end{matrix}\\ &1_3\end{pmatrix}f_{t,p};z_\kappa,v,g)\bar{\xi}^\dag_1(-p^ta)\bar{\xi}^\dag_2(-p^tb)
\mathfrak{g}(\xi_1^\dag)\mathfrak{g}(\xi_2^\dag)\\
=&(\mathfrak{g}(\tau_p')^{-2}c_2^{-1}(\bar{\tau}_p',-z)p^{-t}\mathfrak{g}(\xi_2^\dag)\sum_b\bar{\xi}_2^\dag(p^tb)f'_t
(g\begin{pmatrix}1&\\A_{-b}&1\end{pmatrix}\eta))
(\omega_{\beta,\lambda_p}(g)\Phi_{\mathcal{D},p}(v_1,v_2,v_3,v_4)).
\end{eqnarray*}
Recall $A_x=\begin{pmatrix}0&x\\0&0\end{pmatrix}$. Also under the projection $\mathrm{U}(3,3)(\mathbb{Q}_p)\simeq \mathrm{GL}_6(\mathbb{Q}_p)$, the $v_1,v_2,v_3,v_4$ appear as $\begin{pmatrix}1&&&&v_3&v_4\\&1&&v_1&&\\&&1&v_2&&\\&&&1&&\\&&&&1&\\&&&&&1\end{pmatrix}$. The Weil representation is the one in subsection \ref{3.7} case two. We use $\rho$ to denote the right action of $\mathrm{GU}(3,3)(\mathbb{Q}_p)$ on the Siegel sections.
Note that for the Schwartz function $\Phi_{\mathcal{D},p}$, we used the identification
$\mathcal{K}^2_p\simeq \mathbb{Q}^4_p$ given by $(x_1,x_2)\rightarrow ((v_3,v_1), (v_4,v_2))$.
\end{lemma}
\begin{proof}
First we fix $b$ and consider the Fourier-Jacobi coefficient of:
\begin{equation}\label{(1)}
\sum_{a,m}\rho(\begin{pmatrix}1&&&&a&bm\\&1&&&&\\&&1&&&\\&&&1&&\\&&&&1&\\&&&&&1\end{pmatrix})f_t
\end{equation}
note that:
\begin{eqnarray*}
&w_3\begin{pmatrix}1_3&\begin{matrix}t&v_3&v_4\\v_1&D_1&D_2\\v_2&D_3&D_4\end{matrix}\\&1_3\end{pmatrix}\alpha(1,\eta^{-1})
\begin{pmatrix}1&&&&a&bm\\&1&&&&\\&&1&&&\\&&&1&&\\&&&&1&\\&&&&&1\end{pmatrix}\\
=&w_3\begin{pmatrix}1&-a&-bm&&&\\&1&&&&\\&&1&&&\\&&&1&&\\&&&&1&\\&&&&&1\end{pmatrix}\begin{pmatrix}1_3&\begin{matrix}t'&v_3'&v_4'\\v_1&D_1&D_2\\v_2&D_3&D_4\end{matrix}\\&1_3\end{pmatrix}
\alpha(1,\eta^{-1})
\end{eqnarray*}
where $v_3'=v_3+D_1.a+D_3.bm, v_4'=v_4+D_2.a+D_4.bm,t'=t+v_1.a+v_2.bm$. From this, a calculation using the above lemma (similar to \cite[pp. 203]{SU}) shows that the Fourier-Jacobi expansion of \eqref{(1)} at $g$ is:
$$(\mathfrak{g}(\tau_p')^{-2}c_2^{-1}(\bar{\tau}_p',-z)p^{-t}\mathfrak{g}(\xi_2^\dag)\rho(\eta)f'_t(g')
(\omega_{\beta,\lambda_p}(g')\Phi_{\mathcal{D},p})(v_1,v_2,v_3,v_4).$$
So the Fourier-Jacobi expansion of
$$\sum_{a,b,m}\rho(\begin{pmatrix}1&&&&a&bm\\&1&&&&b\\&&1&&&\\&&&1&&\\&&&&1&\\&&&&&1\end{pmatrix})f_t$$
is
$$\sum_b(\mathfrak{g}(\tau_p')^{-2}c_2^{-1}(\bar{\tau}_p',-z)p^{-t}\mathfrak{g}(\xi_2)
\rho(\begin{pmatrix}1&\\A_{-b}&1\end{pmatrix}\rho(\eta)f'_t(g')(\omega(g')
\omega_{\beta,\lambda_p}(g'\begin{pmatrix}1&\\A_{-b}&1\end{pmatrix})\Phi_{\mathcal{D},p})(v_1,v_2,v_3,v_4).$$
Note that $\omega_{\beta,\lambda_p}(\begin{pmatrix}1&\\A_{-b}&1\end{pmatrix})\Phi_{\mathcal{D},p}=\Phi_{\mathcal{D},p}$. We get the required Fourier-Jacobi coefficient.
\end{proof}

\noindent We record some formulas:
$$\rho(\eta)f_{2,p}=f_{2,p}''=\frac{\tau_p}{\lambda_p}(-1)\mathfrak{g}(\xi_2^\dag)\sum_b\bar{\xi}_2^\dag(b)
\mathfrak{g}(\tau_p')^{-2}c_2^{-1}(\bar{\tau}_p',-z)\rho(\begin{pmatrix}1&A_b\\&1\end{pmatrix})f'_t$$
\begin{equation}\label{(5)}
\Phi''_{\mathcal{D},p}:=(\omega_{\beta,\lambda_p}(\eta)\Phi_{\mathcal{D},p})(v_1,v_2,v_3,v_4)=\left\{\begin{array}{ll}\xi_1^\dag(-v_1),& v_3,v_4\in\mathbb{Z}_p,v_1\in\mathbb{Z}_p^\times,v_2\in p^t\mathbb{Z}_p\\0,&\mbox{ otherwise. }\end{array}\right.
\end{equation}
We define two Schwartz functions on $X^{-,'}_p$. Let $\phi_{1,p}'$ be the characteristic function of $\mathbb{Z}_p^2\subset \mathbb{Q}_p^2$. Let $\phi_{2,p}'(x)=p^{-2t}\mathfrak{g}(\xi^\dag_1)\bar{\xi}_1^\dag(-p^tx_1)$ if $x=(x_1,x_2)$ for $x_1\in p^{-t}\mathbb{Z}_p^\times,x_2\in p^{-t}\mathbb{Z}_p$, and is zero otherwise.
\begin{definition}\label{p-adic Kernel}
We summarize our definitions at the $p$-adic place. Recall we defined $f_{sieg,\mathcal{D},p}$ and $f'_{sieg,\mathcal{D},p}$ as in (\ref{siegelatp}) and (\ref{siegelatp'}).
The definition for $f_{2,\mathcal{D},p}$ is given in (\ref{f_2}) below.

We defined
$\Phi_{\mathcal{D},p}$ as in (\ref{p-kernel}) and defined $\Phi''_{\mathcal{D},p}$ as in (\ref{(5)}).
We also define $\phi_{1,p}=\delta_\psi''(\phi_{1,p}')$, $\phi_{2,p}=\delta_\psi^{-,''}(\phi_{2,p}')$. It is easily checked that
$$\delta_{\psi,p}^{',-1}(\omega_{\beta,\lambda_p}(\Upsilon)(\omega_{\beta,\lambda_p}(\eta)\Phi_{\mathcal{D},p}))
=\phi'_{1,p}\boxtimes\phi'_{2,p}.$$
(One compares these with later computations after Definition \ref{1234}.)
\end{definition}
For convenience of the reader we explain how this computation is done. We first take standard complex basis ${}^t\!(e_1, e_2; e^-_1, e^-_2)$  for the Hermitian space corresponding to $\mathrm{U}(\zeta)\times \mathrm{U}(-\zeta)$. To save space we write $e$ for ${}^t\!(e_1, e_2)$ and similarly for $e^-$. Then the embedding
$$\mathrm{U}(\zeta)\times\mathrm{U}(-\zeta)\hookrightarrow \mathrm{U}(2,2)$$
is reflected by the change of basis
$$\begin{pmatrix}e \\e^- \end{pmatrix}\mapsto {}^t S \begin{pmatrix}e \\e^- \end{pmatrix}=\begin{pmatrix}1_2&-1_2\\-\frac{\zeta}{2}&-\frac{\zeta}{2}\end{pmatrix}\begin{pmatrix}e \\e^- \end{pmatrix}.$$
We further take symplectic basis according to the polarization via $\mathcal{O}_p\simeq \mathbb{Z}_p\times\mathbb{Z}_p$. Then under this basis the $\Upsilon$ (as homomorphism of $8$-dimensional symplectic spaces) is given by the matrix
$$\begin{pmatrix}\frac{1}{2}1_2&&&-\frac{1}{2}1_2\\&-1_2&1_2&\\&1_2&1_2&\\ \frac{1}{2}1_2&&&\frac{1}{2}1_2\end{pmatrix}.$$
Using (\ref{generalform}), the matrix for the intertwining operator $\delta_{\psi,p}^{',-1}$ is
$$\begin{pmatrix}1_2&&&1_2\\-1_2&&&1_2\\&\frac{1}{2}1_2&\frac{1}{2}1_2&\\&\frac{1}{2}1_2&-\frac{1}{2}1_2&\end{pmatrix}
=\begin{pmatrix}1_2&&&\\&&&1_2\\&&1_2&\\&-1_2&&\end{pmatrix}\begin{pmatrix}1_2&&&1_2\\&-\frac{1}{2}&\frac{1}{2}&\\&\frac{1}{2}
&\frac{1}{2}&\\-1_2&&&1_2\end{pmatrix}.$$
We get the composed map is given by the matrix $\begin{pmatrix}1_2&&&\\&&&1_2\\&&1_2&\\&-1_2&&\end{pmatrix}$ and thus the formula.
\subsection{Pullback Formulas Again}\label{5.9}
In this section, we prove the local pullback formulas for $\mathrm{U}(2)\times \mathrm{U}(2)\hookrightarrow \mathrm{U}(2,2)$ which will be used to decompose the restriction to $\mathrm{U}(2)\times \mathrm{U}(2)$ of the Siegel-Eisenstein series associated to the character $\tau/\lambda$ on $\mathrm{U}(2,2)$ showing up in the Fourier-Jacobi expansion of $E_{sieg}$ on $\mathrm{U}(3,3)$. Fortunately, the local calculations are  the same as in the previous sections for $f_{sieg}'$ and $F_\varphi'$'s except for the case $v=p$.\\

\noindent\underline{$p$-adic case}:\\
We temporarily denote the $p$-component of an automorphic representation $\pi_h$ of some weight two cuspidal ordinary eigenform $h$ on $\mathrm{U}(2)(\mathbb{Q}_p)$ as $\pi(\chi_1,\chi_2)$ with $v_p(\chi_1(p))=-\frac{1}{2},v_p(\chi_2(p))=\frac{1}{2}$. We also temporarily write $\tau$ for $\tau/\lambda$ in this subsection. We let $\tau_p=(\tau_1,\tau_2)$ and require $\chi_1\tau_2^{-1}$ and $\chi_2\tau_2^{-1}\xi_2^\dag$ are unramified. We let $\varphi=\varphi^{ss}\in \pi_{h,p}$ for $\varphi^{ss}=\pi_p(\begin{pmatrix}&1\\p^t&\end{pmatrix})\varphi^{ord}$ for some nearly ordinary vector $\varphi^{ord}$.
Define \begin{equation}\label{f_2}
f_{2,\mathcal{D},p}(g):=\mathfrak{g}(\tau_p')^{-2}c_2(\bar{\tau}_p',-z_\kappa/2)^{-1}p^{-t}\mathfrak{g}(\xi_2^\dag)
\sum_{b\in\frac{p^{-t}\mathbb{Z}_p^\times}{\mathbb{Z}_p}}\bar{\xi}_2^\dag(bp^t)\rho
(\begin{pmatrix}1&A_b\\&1\end{pmatrix})f'_t(g\Upsilon').\end{equation}
It is hard to evaluate the integral directly. So we use the trick of using the functional equation as in \cite[Proposition 11.13]{SU} (which is also used in \cite{WAN}). We first evaluate the integral for the auxiliary $f^\dag_{2,p}$
$$F_\varphi(f^\dag_{2,p};z,g):=\int_{\mathrm{U}(2)(\mathbb{Q}_p)}f^{\dag}_{2,p}(z,S^{-1}\alpha(g_1,g)S)\bar{\tau}(\det g)\pi_h(g_1)\varphi dg_1$$
at $g=w=\begin{pmatrix}&&&1\\&1&&\\&&1&\\-1&&&\end{pmatrix}$ where $\varphi\in \pi_{h,p}$ and
$$f^{\dag}_{2,p}(g):=p^{-t}\mathfrak{g}(\xi_2^\dag)\sum_b\bar{\xi}_2^\dag(bp^t)\rho(\begin{pmatrix}1&A_b\\&1\end{pmatrix})f^\dag(g\Upsilon),\ \tilde{f}^{\dag}_{sieg,p}\in I_2(\bar{\tau}^c).$$
The $f_{2,p}$ and $f^\dag_{2,p}$ are related by the intertwining operator, as used in the proof of Lemma \ref{Ppullback}.
For $A\in \mathrm{U}(2)(\mathbb{Q}_p)\simeq \mathrm{GL}_2(\mathbb{Q}_p)$ note that
$$S^{-1}\mathrm{diag}(A,1)\begin{pmatrix}1&A_b\\&1\end{pmatrix}=\begin{pmatrix}-1&-A\\&-A\end{pmatrix}
\begin{pmatrix}1&\\-A_b-A^{-1}&1\end{pmatrix}w.$$
So in order for this to be in $\mathrm{supp}f^\dag$ we must have $A^{-1}+A_b\in M_2(\mathbb{Z}_p)$. So $A^{-1}$ can be written as
$\begin{pmatrix}1&\\n&1\end{pmatrix}\begin{pmatrix}&u\\v&\end{pmatrix}\begin{pmatrix}1&\\m&1\end{pmatrix}$ for $v\in\mathbb{Z}_p\backslash \{0\}$, $u\in -b+\mathbb{Z}_p$ and $m,n\in p^t\mathbb{Z}_p$. Thus
$$A=\begin{pmatrix}1&m\\&1\end{pmatrix}\begin{pmatrix}v^{-1}&\\&u^{-1}\end{pmatrix}\begin{pmatrix}&1\\1&\end{pmatrix}\begin{pmatrix}1&\\n&1\end{pmatrix}.$$
A direct computation gives the integral equals:
\begin{align*}
&&&\chi_2(-1)\chi_1\chi_2(p^t)p^{t(z+1)}\bar{\tau}^c((-p^t,-1))\sum_{i=1}^\infty(\chi_1(p^{-1})\bar{\tau}^c((p^{-1},1))p^{-z-\frac{1}{2}})^i\phi^{ord}&\\
&=&&p^{t(z+1)}\bar{\tau}^c((-p^t,-1))L_p(\pi,\tau,z+\frac{1}{2})\chi_1\chi_2(p^t)\chi_2(-1).&
\end{align*}

\noindent The $\pi$ in the $L$-factor means the base change of $\pi$ from $\mathrm{U}(2)$ to $\mathrm{GL}_2$. Note that it is not convergent at $z=-z_\kappa$ and is defined by analytic continuation at that point. \\

\noindent Now we apply the functional equation trick to evaluate the pullback integral for $f_{2,\mathcal{D},p}$, similar to the proof of Lemma \ref{Ppullback}. As in \cite[Proposition 11.28]{SU} the local constant showing up when applying the intertwining operator at $z=-z_\kappa$ is
$$\epsilon(\pi,\tau,-z_\kappa+\frac{1}{2})=\mathfrak{g}(\bar{\tau}_1\bar{\chi}_1)\tau_1\chi_1(p^t)\mathfrak{g}
(\bar{\tau}_1\bar{\chi}_2)\tau_1\chi_2(p^t)
\mathfrak{g}(\bar{\tau}_2\chi_2)\tau_2\chi_2^{-1}(p^t)p^{\frac{3}{2}\kappa-6}.$$
To sum up our original local integral for $f_{2,p}$ equals
$$L_p(\pi_h,\bar{\tau}^c,z_\kappa+\frac{1}{2})\mathfrak{g}(\bar{\tau}_1\bar{\chi}_1)\tau_1\chi_1(p^t)\mathfrak{g}(\bar{\tau}_1\bar{\chi}_2)
\tau_1\chi_2(p^t)p^{(2\kappa-5)t}\chi_1(p^t)p^{\frac{t}{2}}\varphi^{ord}.$$
\noindent Note that $\langle \tilde{\varphi}^{ord},\varphi^{ss}\rangle=\langle\tilde{\phi}^{ord},\varphi_{low}\rangle\cdot\chi_1(p^t)p^{\frac{t}{2}}$ where we define $\varphi_{low}=\pi_p(\begin{pmatrix}&1\\1&\end{pmatrix})\varphi^{ord}$. Thus if we replace $\varphi^{ss}$ by $\varphi_{low}$ in the definition for pullback integral then it equals
$$L_p(\pi_h,\bar{\tau}^c,z_\kappa+\frac{1}{2})\mathfrak{g}(\bar{\tau_1}\bar{\chi_1})\tau_1\chi_1(p^t)\mathfrak{g}(\bar{\tau_1}\bar{\chi_2})
\tau_1\chi_2(p^t)p^{(2\kappa-5)t}\varphi^{ord}.$$
Noting again that later when we are defining $E_{sieg,2}$ the $\tau$ here should be $\frac{\tau}{\lambda}$.

\subsection{Global Computations}\label{GC}
\subsubsection{Good Siegel Eisenstein Series}
\begin{definition}\label{definelevel}
As for Klingen Eisenstein series case, throughout this paper, we fix the tame level subgroup $K^{(3,3)}$ of $\mathrm{U}(3,3)(\mathbb{A}^{p\infty})$ , under which our $E_{sieg,\mathcal{D}}$ is invariant. We can do so by simply taking it to be the set of matrices congruent to $1$ modulo the $(y\bar{y})^2$ at each finite place not dividing $p$ for the $y$ defined above.  We define the $p$-component of the level group for $E_{sieg,\mathcal{D}}$ at $p$ to consist of elements congruent to $1$ modulo $p^{2t}$.

We also fix a tame level subgroup $K^{(2,0)}$ of $\mathrm{U}(2,0)(\mathbb{A}^{p\infty})$ consisting of matrices congruent to $1$ modulo the $(y\bar{y})^2$ above at each finite place away from $p$.
\end{definition}
We first define two normalization factors as in \cite[Subsection 5.3.1]{WAN}
\begin{eqnarray*}
B_\mathcal{D}:&=(\frac{(-2)^{-3}(2\pi i)^{3\kappa}(2/\pi)^{3}}{\prod_{j=0}^{2}(\kappa-j-1)!})^{-1}
\prod_{i=0}^{2}L^\Sigma(2z_\kappa+3-i,\bar\tau'\chi_\mathcal{K}^i),
\end{eqnarray*}
\begin{eqnarray*}
B_\mathcal{D}':&=(\frac{(-2)^{-2}
(2\pi i)^{2\kappa}(2/\pi)}{\prod_{j=0}^{1}(\kappa-j-1)!})^{-1}
\prod_{i=0}^{1}L^\Sigma(2z_\kappa+2-i,\bar\tau'\chi_\mathcal{K}^i).
\end{eqnarray*}
\index{$B_\mathcal{D}, B'_\mathcal{D}$}
Here recall $z_\kappa=\frac{\kappa-3}{2}$ and $z_\kappa'=\frac{\kappa-2}{2}$.
\begin{definition}
We define $E_{sieg,\mathcal{D}}(z,g)=E_{sieg}(z,f_{sieg},g)$ \index{$E_{sieg}, E_{sieg,\mathcal{D}}$} on $\mathrm{GU}(3,3)$ for $f_{sieg,p}=f_{sieg,\mathcal{D}}=B_\mathcal{D}\prod_v f_{sieg,\mathcal{D},v}$ and $E_{sieg,\mathcal{D}}'(z,g)=E_{sieg}'(z,f'_{sieg,\mathcal{D}},g)$ on $\mathrm{GU}(2,2)$ for $f'_{sieg}=B'_\mathcal{D}\prod_v f'_{sieg,v}$. (Note that compared to \cite{WAN}, the normalization factors at $p$ here are already included in our definitions of $p$-adic Siegel sections.)
We also define $E_{sieg,\mathcal{D}}^\Box(g)=E(z,f_{sieg,\mathcal{D}}^\Box,g)$ where $f_{sieg,\mathcal{D}}^\Box$ is the same as $f_{sieg,\mathcal{D}}$ at all primes not dividing $p$ and is $f_{sieg,\mathcal{D},p}^\Box$ at $p$.
\end{definition}
\subsubsection{Pullbacks}
For some $g_1\in \mathrm{U}(2)(\mathbb{A}_\mathbb{Q})$ (which we specify in Definition \ref{definition 8.17}) we define $E_{Kling,\mathcal{D}}$ \index{$E_{Kling,\mathcal{D}}$} by:
\begin{equation}\label{DefineKlingen}
E_{Kling,\mathcal{D}}(z,g)=\frac{1}{\Omega_\infty^{2\kappa}}\int_{[\mathrm{U}(2)]}
E_{sieg,\mathcal{D}}(z,\gamma(g,hg')\Upsilon)\bar{\tau}(\det g')\pi(g_1)\varphi_\mathcal{D}(g')dg'.
\end{equation}
Define
\begin{equation}
\varphi'_\mathcal{D}(z,g)=\frac{1}{\Omega_\infty^{2\kappa}}\int_{[\mathrm{U}(2)]}
E'_{sieg,\mathcal{D}}(z,\gamma(g,hg')\Upsilon')\bar{\tau}(\det g')\pi(g_1)\varphi_\mathcal{D}(g')dg'.
\end{equation}
Here we use the local components of a Klingen-Eisenstein data $\mathcal{D}$ in the construction.
The period factor showing up comes from the geometric pullback map (see \cite[Section 2.8, Subsection 5.6.5]{Hsieh CM}). In fact comparing the definition for the geometric pullback map and the pullback formula for automorphic forms, in order to get rational automorphic forms on $\mathrm{U}(2)\times\mathrm{U}(2)$ or $\mathrm{U}(3,1)\times\mathrm{U}(2)$ via pullbacks, such CM periods have to be divided out. The $\varphi_\mathcal{D}$ is defined as follows. First, recall that given a CM character $\psi$ and a form on $D^\times$ whose central character is $\psi|_{\mathbb{A}^\times_\mathbb{Q}}$ we can produce a form on $\mathrm{U}(2)$ whose central character is the restriction of $\psi$. So we often construct forms on $D^\times$ and get forms on $\mathrm{U}(2)$ this way. In section \ref{section 7.2} we construct a Dirichlet character $\vartheta$.
\begin{definition}
We define
$f_\vartheta$ as in the end of Section \ref{Section 8.4}, and an element $g_1$ in Definition \ref{definition 8.17}. Our $\varphi$ is defined as $\pi(g_1)f_\vartheta$.
\end{definition}
\begin{proposition}
The $E_{Kling,\mathcal{D}}(z_\kappa,-)$ defined above is the Klingen-Eisenstein series constructed using the Klingen section $F_{Kling,\mathcal{D}}=\prod_vF_{Kling,\mathcal{D},v}$ for $F_{Kling,\mathcal{D},v}$ the Klingen-Eisenstein sections defined in previous subsections. We also have that
$$\varphi'_\mathcal{D}(z'_\kappa,-)=\prod_v F'_{\varphi_v}(f'_{sieg,v},z'_\kappa,1).$$
Note that we have used $\pi(g_1)\varphi_\mathcal{D}$ in the place of $\varphi_\mathcal{D}$ in \ref{KlinEi}, and there is a normalization factor appearing in the $p$-adic Klingen section here.
\end{proposition}
\begin{proof}
It is a straightforward consequence of Shimura's pullback formula, and our local pullback computations in previous sections.
\end{proof}

We record the following easy lemma, which explains the motivation for the definition of $f_\Sigma$: to pick up a certain Iwahori-invariant vector from the unramified representation $\pi_v$ for $v\in \Sigma\backslash \{v, v|N\}$.
\begin{lemma}\label{lemma 6.39}
Consider the model for the unramified principal series representation $$\pi(\chi_{1,v},\chi_{2,v})=\{f:K_v\rightarrow \mathbb{C}, f(qk)=\chi_{1,v}(a)\chi_{2,v}(d)\delta_B(q)f(k),q=\begin{pmatrix}a&b\\&d\end{pmatrix}\in B(\mathbb{Z}_v)\}.$$ Let $f_{ur}$ be the constant function $1$ on $K_v$, $f_0$ be the function supported and takes value $1$ on $K_1$ for $K_1=\{\begin{pmatrix}a&b\\c&d\end{pmatrix},\varpi_v|c\}$. Then $$(\chi_{2,v}(\varpi_v)q_v^{-\frac{1}{2}}-\chi_1(\varpi_v)q_v^{\frac{1}{2}})f_0=\pi(\begin{pmatrix}&1\\ \varpi_v&\end{pmatrix})f_{ur}-\chi_{1,v}(\varpi_v)q_v^{\frac{1}{2}}f_{ur}.$$
\end{lemma}
\noindent\underline{Fourier-Jacobi Coefficients}
\begin{proposition}\label{Proposition 6.31}
The Fourier-Jacobi coefficient for $\beta=1$ is given by:
\begin{align*}
&\mathrm{FJ}_1(E_{sieg,\mathcal{D}})(z_\kappa,\mathrm{diag}(u,1_2,u,1_2)n'\mathbf{j}(1,g))&&=\sum_{n\in\mathbb{Z}_p/p^t\mathbb{Z}_p}
E_{sieg,2}(z_\kappa,f_{2,\mathcal{D}},
g'\gamma(1,\begin{pmatrix}1&\\n&1\end{pmatrix}_p))&\\&&&\times
\Theta_{\Phi_\mathcal{D}}(u,n'g'\gamma(1,\begin{pmatrix}1&\\n&1\end{pmatrix}_p))&
\end{align*}
for $n'\in N_2$ ($N_2$ defined in subsection \ref{3.6}), $g\in\mathrm{U}(2,2)(\mathbb{A}_\mathbb{Q})$, $g'=\begin{pmatrix}1_2&\\&-1_2\end{pmatrix}g\begin{pmatrix}1_2&\\& -1_2\end{pmatrix}$, and $u\in \mathrm{U}(1)(\mathbb{A}_\mathbb{Q})$. Here $f_{2,\mathcal{D}}=\prod_vf_{2,\mathcal{D},v}$ and $\Phi_\mathcal{D}=\prod_v\Phi_{\mathcal{D},v}$ \index{$\Phi_\mathcal{D}$} are given in Definitions \ref{Archimedean Kernel}, \ref{Un Kernel}, \ref{Ra Kernel} and \ref{Pa Kernel}. The $E_{sieg, 2}$ is considered as a function on $\mathrm{U}(2,2)(\mathbb{A}_\mathbb{Q})$.
\end{proposition}
\begin{proof}
This follows from our computations for local Fourier-Jacobi coefficients and Lemma \ref{5.4}.
\end{proof}
So far we have used the embedding $\mathbf{j}(1,g)=\begin{pmatrix}1&&&\\&D_g&&C_g\\&&1&\\&B_g&&A_g\end{pmatrix}$ for $g=\begin{pmatrix}A&B\\C&D\end{pmatrix}$, to keep accordance with the convention of \cite{SU} and \cite{WAN}. In our actual applications later on we will use another embedding $\alpha''$ as below. Now that we have
$$\mathrm{FJ}_1(E_{sieg,\mathcal{D}}^\Box,z_\kappa,g,x,u)=E_{sieg,2}(z_\kappa,f_{2,\mathcal{D}},g')
\Theta_{\Phi_\mathcal{D}}(u,g'),$$
we consider another embedding \index{$\mathbf{j}''$} $\mathbf{j}''(1,g)=\begin{pmatrix}1&&&\\&A_g&&B_g\\&&1&\\&C_g&&D_g\end{pmatrix}$. (The $g'$ is as defined in Proposition \ref{Proposition 6.31}).
Let \index{$\Phi''_\mathcal{D}$} $\Phi''_\mathcal{D}=\omega_\beta(\eta)\Phi_\mathcal{D}$ and $f_{2,\mathcal{D}}''=\rho(\eta)f_{2,\mathcal{D}}$ and we let $\mathrm{FJ}_\beta''$ be defined as $\mathrm{FJ}_\beta$ but replacing $\mathbf{j}$ by $\mathbf{j}''$. By observing that $E_{sieg,\mathcal{D},2}$ and $\Theta$ are automorphic forms and thus invariant under left multiplication by $\eta^{-1}$, we get:
\begin{equation}\label{(15)}
\mathrm{FJ}_\beta''(E_{sieg,\mathcal{D}}^\Box,z_\kappa,\mathrm{diag}(u,1_2,u,1_2)n'\mathbf{j}''(1,g))
=E_{sieg,\mathcal{D},2}(z_\kappa,f_{2,\mathcal{D}}'',g)\Theta_{\Phi''_\mathcal{D}}(u,n'g).
\end{equation}
\begin{definition}\label{definition 6.43}
Let $\prod_{v\nmid p}U_v$ be the intersection of $g^{-1}K^{(3,1)}g$ (the $K^{(3,1)}$ is defined in Definition \ref{definelevel}) with $\mathrm{U}(2,0)(\mathbb{A}^{p\infty})$, and $U_p$ consists of matrices in $\mathrm{U}(2,0)(\mathbb{Z}_p)$ which are upper triangular modulo the $p^t$ in Subsection \ref{Sectionn}. The $L$ is defined to be the intersection of $\mathcal{K}^2$ (the quotient of $N_P(\mathbb{Q})$ over the center of it) with the image of $(K^{(3,1)}\times\mathrm{U}(3,1)(\mathbb{Z}_p))\cap N_P(\mathbb{A}_f)$.
Define $\phi_1=\prod_v\phi_{1,v}$ and $\theta^\star_\mathcal{D}=\theta_{\phi_1}$ as an element in the $\mathbb{C}$-dual space of $$H^0(\mathcal{Z}^\circ_{[g]},\mathcal{L}(\beta))\otimes\mathbb{C}\simeq T^{\mathrm{Hol}}_{\mathbb{A}}(\beta,L,U_f)$$ as in the end of Section \ref{ThetaCM} for $m=1$, $U_f\subseteq \mathrm{U}(2)(\mathbb{A}_f)$ an open compact subgroup defined as $K^{(2,0)}\times\mathrm{GL}_2(\mathbb{Z}_p)$. The $L\subset \mathcal{K}^{\oplus2}$ being the ideal generated by $(y\bar{y})^2$. (These level groups are fixed throughout the family).
\end{definition}
We will usually write $E_{sieg,\mathcal{D},2}$ for $E_{sieg,\mathcal{D},2}(z_\kappa,f_{2,\mathcal{D}}'',g)$ for short.
\begin{lemma}\label{Lemma 6.32}
Suppose $\delta_{\psi}^{-1}(\Phi_\infty'')=\phi_{1,\infty}\boxtimes\phi_{2,\infty}$, and for each $v<\infty$
\begin{equation}
\phi_{2,v}(x)=\int_{X_v}\delta_{\psi}^{-1}(\Phi_v'')(x',x)\phi_{1,v}(x')dx'.
\end{equation}
Then
\begin{equation}
l_{\theta^\star_\mathcal{D}}'(\gamma^{-1}(\Theta_{\Phi''_\mathcal{D}}))(h)=\theta_{\phi_2}(h).
\end{equation}
Here we consider $\gamma^{-1}(\Theta_{\Phi''_\mathcal{D}})$ as a function on $(N\mathrm{U}(2))\times\mathrm{U}(2)\hookrightarrow \mathrm{U}(3,1)\times\mathrm{U}(2)$ and apply $l_{\theta^\star_\mathcal{D}}'$ to it on the $N\mathrm{U}(2)$ part.
\end{lemma}
\begin{proof}
We write $\delta_\psi^{-1}\Phi''$ for each finite place $v$ as a finite sum of expressions of the form $\phi'_{1,i,v}\boxtimes\phi_{2,i,v}$ and let $\phi'_{1,i}=\prod_{v\nmid\infty}\phi'_{1,v}\times\phi_{1,\infty}$ and $\phi_{2,i}=\prod_{v\nmid\infty}\phi_{2,i,v}\times\phi_{2,\infty}$. We have a finite sum
$$\Theta_{\Phi''_\mathcal{D}}(vh_1,h_2)=\sum_{i=1}^r\theta_{\phi'_{1,i}}(vh_1)\theta_{\phi_{2,i}}(h_2)$$
and that the function
$$T_i(h):=\int_{[V]}\theta_{\phi'_{1,i}}(vh)\theta^\star_\mathcal{D}(vh)dv$$
is a constant function by Lemma \ref{lemma 4.8}. Then the expression $\sum_{i=1}^rT_i\cdot \theta_{\phi_{2,i}}$ is clearly equal to $\theta_{\phi_2}$ by (\ref{(20)}) and (\ref{(13)}).
\end{proof}
\begin{corollary}\label{Corollary 6.29}
\begin{align*}
&&&l_{\theta^\star_\mathcal{D}}'(\mathrm{FJ}_1(\prod_v\sum_i \rho(\mathrm{diag}(u_{v,i},1_2,u_{v,i})))E_{Kling,\mathcal{D}})(g)&\\
&=&&\langle\frac{1}{\Omega_\infty^{2\kappa-2}}\sum_{n\in\mathbb{Z}_p/p^t\mathbb{Z}_p}E_{sieg,\mathcal{D},2}
(\alpha(g,-\begin{pmatrix}1&\\n&1\end{pmatrix}_p)\cdot\frac{1}{\Omega_\infty}\theta_{2,\mathcal{D}}(-\begin{pmatrix}1&\\n&1
\end{pmatrix}_p,\varphi_\mathcal{D}(-)\rangle&
\end{align*}
where $\theta^\star_\mathcal{D}$ and $\theta_{2,\mathcal{D}}$ are the theta functions on $U(-\zeta)(\mathbb{A}_\mathbb{Q})$ defined using the Schwartz functions $\phi_1=\prod_v\phi_{1,v}$ and $\phi_2=\prod_v\phi_{2,v}$ for $\phi_{1,v}$ and $\phi_{2,v}$'s defined as before (Definitions \ref{Archimedean Kernel}, \ref{p-adic Kernel}, also \ref{Define Theta}, and the corresponding definition for ramified places). Note that the $\phi_2$ is defined using $\phi_1$, which explains the dependence of the right hand side on $\phi_1$. The inner product is over the group $1\times\mathrm{U}(2)\hookrightarrow \mathrm{U}(3,3)$. Moreover, suppose $\varphi_p\in\pi_p$ is chosen such that $\varphi_p$ is the ordinary vector, then the above expression is
$$\frac{1}{\Omega^{2\kappa-1}_\infty}\langle E_{sieg,\mathcal{D},2}(\alpha(g,-))\cdot\theta_{2,\mathcal{D}}(-),\varphi_\mathcal{D}(-)\rangle.$$
\end{corollary}
\begin{proof}
It follows from the Proposition \ref{Proposition 6.31}, Lemma \ref{Lemma 6.32}, noting that the pairing $l'_{\theta^\star_\mathcal{D}}$ is essentially applied to the $\Theta_\Phi$ factor on the right hand side of Proposition \ref{Proposition 6.31} (recall also the meaning for intertwining maps defined in subsection \ref{3.6}). The last sentence follows from describing the pairing between $\pi_p$ and $\pi_p^\vee$.
\end{proof}

\section{$p$-adic Interpolation}\label{section 7}
\subsection{Congruence Module and the Canonical Period}\label{section7.1}
We now discuss the theory of congruences of modular forms on $\mathrm{GL}_2/\mathbb{Q}$, following \cite[Section 12.2]{SU}. Let $R$ be a finite discrete valuation ring extension of $\mathbb{Z}_p$ and $\varepsilon$ a finite order character of $\hat{\mathbb{Z}}^\times$ whose $p$-component has conductor dividing $p$. Let $M_\kappa^{ord}(Mp^r,\varepsilon; R)$ be the space of ordinary modular forms on $\mathrm{GL}_2/\mathbb{Q}$ with level $N=Mp^r$, character $\varepsilon$ and coefficient $R$. Let $S_\kappa^{ord}(Mp^r,\varepsilon; R)$ be the subspace of cusp forms. Let $\mathbb{T}_\kappa^{ord}(N,\varepsilon;R)$ ($\mathbb{T}_\kappa^{ord,0}(Mp^r,\varepsilon;R)$) be the $R$-sub-algebra of
$\mathrm{End}_R(M_\kappa^{ord}(Mp^r,\varepsilon;R))$ (respectively, $\mathrm{End}_R(S_\kappa^{ord}(Mp^r,\varepsilon;R))$) generated by the Hecke operators
$T_v$ (these are Hecke operators defined using the double cosets $\Gamma_1(N)_v\begin{pmatrix}\varpi_v&\\&1\end{pmatrix}\Gamma_1(N)_v$ for the $v$'s) for all $v$. Let any $f\in S_\kappa^{ord}(N,\varepsilon;R)$ be an ordinary eigenform. Then we have an
$1_f\in \mathbb{T}_\kappa^{ord,0}(N,\varepsilon;R)\otimes_R F_R=\mathbb{T}_\kappa'\times F_R$ the projection onto the second factor.\\

\noindent Let $\mathfrak{m}_f$ be the maximal ideal of the Hecke algebra corresponding to $f$. The $\mathbb{T}^{ord,0}(M,\varepsilon;R)\cap (0\otimes F_R)$ is free of rank one over $R$. We let $\ell_f$ be a generator and so
$\ell_f=\eta_f1_f$ for some $\eta_f\in R$. This $\eta_f$ is called the congruence number of $f$.\\

\noindent Now let $\mathbb{I}$ be as in the introduction. Suppose $\mathbf{f}\in M^{ord}(M,\varepsilon;\mathbb{I})$ is an ordinary $\mathbb{I}$-adic cuspidal eigenform. Then as above
$\mathbb{T}^{ord,0}(M,\varepsilon;\mathbb{I})\otimes F_\mathbb{I}\simeq \mathbb{T}'\times F_\mathbb{I}$, $F_\mathbb{I}$ being the fraction field of
$\mathbb{I}$ where projection onto the second factor gives the eigenvalues for the actions on $\mathbf{f}$. Again let $1_\mathbf{f}$ be the idempotent
corresponding to projection onto the second factor. Then for a $\mathbf{g}\in S^{ord}(M,\varepsilon;\mathbb{I})\otimes_\mathbb{I} F_\mathbb{I}$,
$1_\mathbf{f}\mathbf{g}=c\mathbf{f}$ for some $c\in F_\mathbb{I}$. In the case when the localized Hecke algebra $\mathbb{T}^{ord,0}(M,\varepsilon;R)_{\mathfrak{m}_\mathbf{f}}$ is a Gorenstein $\mathbb{I}$-algebra (which is indeed the case under assumptions $(\mathbf{dist})$ and $(\mathbf{irred})$),
$\mathbb{T}^{ord,0}(M,\varepsilon; \mathbb{I})\cap (0\otimes F_\mathbb{I})$ is a rank one $\mathbb{I}$-module. Under the Gorenstein property for $\mathbb{T}_{\mathbf{f}}$, we can define $\ell_\mathbf{f}$ and $\eta_\mathbf{f}$.\\

\begin{definition}
From now on, we will define $D$ to be the unique quaternion algebra ramified exactly at $\infty$ and the $q$ in our main theorems in the introduction. We choose the group $\mathrm{U}(2)$ with $D^\times$ being its associated quaternion algebra. It is clear that this is possible.
\end{definition}
We also make the following definition for $p$-adic families of forms on $D^\times$.
\begin{definition}\label{def D}
For any complete local $\mathcal{O}_L[[W]]$-algebra $R$ (need not be finite over $\mathcal{O}_L[[W]]$) we define the space of $R$-adic families on $G=D^\times$ with level group $K_D\subset D^\times(\mathbb{A}_f)$ which is $\mathrm{GL}_2(\mathbb{Z}_p)$ at $p$ to be the space of continuous functions
$$f: D^\times(\mathbb{Q})\backslash D^\times_f(\mathbb{A})/K_D^{(p)}\rightarrow R$$
such that
$$f(x\begin{pmatrix}a&b\\0&d\end{pmatrix})=f(x)\langle d\rangle_W\chi(d)$$
for $a,d\in\mathbb{Z}_p^\times$, $b\in\mathbb{Z}_p$, where $\chi$ is a fixed character of $\mathbb{Z}^\times_p$ trivial on $1+p\mathbb{Z}_p$, $\langle d\rangle_W\in \mathcal{O}_L[[W]]\simeq \mathcal{O}_L[[\Gamma_\mathbb{Q}]]$ is the image of $d$ under the local reciprocity law. Here we make an identification of $W$ with $\mathrm{rec}_p(1+p)$ which is a topological generator of $\Gamma_\mathbb{Q}$. We also equip $\mathrm{GL}_2(\mathbb{Z}_p)\subset D^\times(\mathbb{A}_f)$ with the topology as a $p$-adic Lie group. We write
$$M(S_G(K^{(p)}_D), R)$$
for the space of all such forms.

One can define Hecke operators $T_\ell$ at primes $\ell$ where $K$ is $\mathrm{GL}_2(\mathbb{Z}_\ell)$, and the $U_p$ operator defined by
$$U_p f(g)=\sum_{n\in\mathbb{Z}_p/p\mathbb{Z}_p}f(g\begin{pmatrix}1&n\\&1\end{pmatrix}_p\begin{pmatrix}p&\\&1\end{pmatrix}_p).$$

We make similar definitions for $R$-adic families on the definite unitary group $G=\mathrm{U}(2)$ of some prime to $p$ level group $K^{(p)}$, which we denote as $M(S_G(K^{(p)}), R)$. It is also possible to add one more variable allowing twisting by characters (so that the nebentypus will be a character of $\mathrm{diag}(a,d)$ instead of just $d$) as follows. Let $\Lambda_{\mathrm{U}(2)}=\Lambda_{2,0}=\mathbb{Z}_p[[T_1,T_2]]$. Let $\mathcal{M}_{\mathrm{ord}}(K^{(2,0)}, \Lambda_{2,0})$ be the space of $\Lambda_{2,0}$-adic ordinary modular forms on $\mathrm{U}(2,0)$, consisting of functions $$\mathbf{f}: \mathrm{U}(2,0)\backslash\mathrm{U}(2,0)(\mathbb{A}_\mathbb{Q})/K^{(2,0)}\mathrm{U}(2,0)_\infty\rightarrow\Lambda_{2,0}$$ such that
$$\mathbf{f}(g\begin{pmatrix}a&\\&d\end{pmatrix})=\mathbf{f}(g)\cdot\langle a\rangle_{T_1}\langle d\rangle_{T_2}, a,d\in\mathbb{Z}^\times_p.$$
For later use in Subsection \ref{IPIP}, we will also define a space $\breve{\mathcal{M}}_{\mathrm{ord}}(K^{(2,0)},\Lambda_{2,0})$ be the space of $\Lambda$-adic ordinary modular forms consisting of functions $\mathbf{f}:\mathrm{U}(2,0)\backslash\mathrm{U}(2,0)(\mathbb{A}_\mathbb{Q})/K^{(2,0)}\mathrm{U}(2,0)_\infty\rightarrow \Lambda_{2,0}$ such that
$$\mathbf{f}(g\begin{pmatrix}a&\\&d\end{pmatrix})=\mathbf{f}(g)\cdot\langle d\rangle^{-1}_{T_1}\langle a\rangle^{-1}_{T_2}.$$
\end{definition}

\noindent\underline{The Ordinary Family $\mathbf{f}$ on $D^\times$}\\
\noindent Let $\mathbf{f}$ be a Hida family of ordinary cuspidal eigenforms new outside $p$ as in Theorem \ref{Theorem 1}. Suppose $\mathbb{T}_{\mathfrak{m}_\mathbf{f}}$ is Gorenstein. Thus we have the integral projector $\ell_\mathbf{f}$. We construct from it a Hida family of ordinary forms on $D^\times(\mathbb{A}_\mathbb{Q})$, also denoted as $\mathbf{f}$. We refer to \cite[Sections 2 and 3]{Hida88} for the definition and theory of ordinary forms on quaternion algebras. By our assumption in Theorem \ref{Theorem 1} and our choice for $D$, we may choose $f_0$ a form on $D^\times(\mathbb{A}_\mathbb{Q})$ which is in the Jacquet-Langlands correspondence of $f$ in $\mathrm{GL}_2(\mathbb{A}_\mathbb{Q})$ with values in $\mathcal{O}_L$ and $f_0$ is not divisible by $\pi_L$ (the uniformizer of the maximal ideal of $\mathcal{O}_L$). Under the assumption of Theorem \ref{Theorem 1}, we do the following we first take an ordinary Hida family $\mathbf{g}$ such that $\ell_\mathbf{f}\mathbf{g}$ is nonzero. Note that $\mathbb{I}$ is a Dedekind domain. So we can divide $\ell_\mathbf{f}\mathbf{g}$ by an element in $\mathbb{I}$ such that the quotient (which we still denote as $\mathbf{f}$ to save notations) is integral (i.e. $\mathbb{I}$-valued) and there is no non-trivial common divisor for the values of $\mathbf{f}$.

\subsection{Eisenstein Datum}\label{EiDatum}
\begin{definition}\label{Eisens}
An Eisenstein datum $\mathbf{D}=(\Sigma,L,\mathbb{I},\mathbf{f}, \psi, \tau)$ \index{$\mathbf{D}$} consists of:
\begin{itemize}
\item A finite set of primes $\Sigma$ containing all bad primes.
\item A finite extension $L/\mathbb{Q}_p$;
\item a finite normal $\mathcal{O}_L[[W]]$-algebra $\mathbb{I}$;
\item an $\mathbb{I}$-adic Hida family $\mathbf{f}$ of cuspidal ordinary eigenforms on $D^\times$ new outside $p$, of square-free tame conductor $N$ such that some weight $2$ specialization $f_0$ has trivial character;
\item two $L$-valued Hecke characters $\psi$ and $\tau$ of $\mathcal{K}^\times\backslash \mathbb{A}_\mathcal{K}^\times$ whose $p$-part conductors divide $p$, and whose infinity types are $(0,0)$ and $(-\frac{\kappa}{2},\frac{\kappa}{2})$, respectively, such that $\psi|_{\mathbb{A}^\times_\mathbb{Q}}$ is equal to the central character of $\pi_{f_0}$. Let $\xi=\psi/\tau$. We define $p$-adic deformations $\boldsymbol{\psi},\boldsymbol{\tau}$ of them in (\ref{family character}).
\end{itemize}
\end{definition}
Note that for any arithmetic point, the specialization $\mathcal{D}_\phi$ of $\mathbf{D}$ gives an Eisenstein datum in the sense of Definition \ref{EDatum}. Now we need to modify our $\mathbb{I}$. By taking an irreducible component of the normalization of a series of quadratic extensions of $\mathbb{I}$ we may assume that for each $v\in\Sigma$ not dividing $N$, we can find two functions $\alpha_v,\beta_v\in\mathbb{I}$ interpolating the Satake parameters of $\pi_v$. This enables us to do the constructions in the global computations in Lemma \ref{lemma 6.39} in families. We still denote $\mathbb{I}$ for the new ring for simplicity. At the end of this paper, we will see how to deduce the main conjecture for the original $\mathbb{I}$ from that for the new $\mathbb{I}$.

Let $\hat{\mathbb{I}}^{ur}_\mathcal{K}:=\hat{\mathbb{I}}^{ur}[[\Gamma_\mathcal{K}]]$ (see Introduction for the notion of $\hat{\mathbb{I}}^{ur}$). We define $\alpha: \mathcal{O}_L[[\Gamma_\mathcal{K}]]\rightarrow \hat{\mathbb{I}}^{ur}[[\Gamma_\mathcal{K}^-]]$ and $\beta: \mathcal{O}_L[[\Gamma_\mathcal{K}]]\rightarrow \hat{\mathbb{I}}^{ur}[[\Gamma_\mathcal{K}]]$ by: $$\alpha(\gamma_+)=(1+W)^{\frac{1}{2}},\alpha(\gamma_-)=\gamma_-, \beta(\gamma_+)=\gamma_+, \beta(\gamma_-)=\gamma_-.$$

We let \begin{equation}\label{family character}
\boldsymbol{\psi}=\psi\cdot\alpha\circ\Psi_\mathcal{K}, \boldsymbol{\xi}=(\beta\circ\Psi_\mathcal{K})\cdot\xi,\boldsymbol{\tau}=\boldsymbol{\psi}/\boldsymbol{\xi}.
\end{equation}
\index{$\boldsymbol{\psi},\boldsymbol{\xi},\boldsymbol{\tau}$}

Define $\psi_\phi:=\phi\circ\boldsymbol{\psi}$ and $\xi_\phi=\phi\circ\boldsymbol{\xi}$. Let $\Lambda_\mathbf{D}=\hat{\mathbb{I}}^{ur}[[\Gamma_\mathcal{K}]][[\Gamma_\mathcal{K}^-]]$. We give $\Lambda_\mathbf{D}$ a $\Lambda_2$-algebra structure by first defining a homomorphism $\Gamma_2=(1+p\mathbb{Z}_p)^4\rightarrow \Gamma_\mathcal{K}\times\Gamma_\mathcal{K}$ given by:
$$(a,b,c,d)\rightarrow \mathrm{rec}_\mathcal{K}(db,a^{-1}c^{-1})\times \mathrm{rec}_\mathcal{K}(d^{-1},c)$$
and then compose with $\alpha\otimes \beta$.

We remark here that only the subring $\hat{\mathbb{I}}^{ur}[[\Gamma_\mathcal{K}]]$ of $\Lambda_\mathbf{D}$ really matters: the $\Gamma_\mathcal{K}^-$ variable corresponds to twisting everything by the same character and does not affect the $p$-adic $L$-functions and the Selmer groups.

\begin{definition}
A point $\phi\in\mathrm{Spec}\Lambda_\mathbf{D}$ is called arithmetic if $\phi(1+W),\phi(\gamma),\phi(\gamma^-)$ for $\gamma\in\Gamma_\mathcal{K}$ and $\gamma^-\in\Gamma_\mathcal{K}^-$ are all $p$-power roots of unity. We call it generic if the $p$-part of $(f_\phi,\psi_\phi,\tau_\phi)$ is generic in the sense defined in Definition \ref{definegeneric}. Note that whether $\phi$ is generic only depends on its image in $\mathrm{Spec}\Lambda_2$. It also only depends on the subring $\hat{\mathbb{I}}^{ur}[[\Gamma_\mathcal{K}]]$. So it makes sense to talk about generic points on these weight spaces as well. We let $\mathcal{X}_\mathbf{D}$ \index{$\mathcal{X}_\mathbf{D}$, $\mathcal{X}^{gen}_\mathbf{D}$} be the set of arithmetic points and $\mathcal{X}^{gen}_\mathbf{D}$ be the set of generic arithmetic points. Later when we are working with families, it is easily seen that these points are Zariski dense due to the fact that $p\geq 5$ (in fact this is the only place where we used the fact $p\geq 5$).
\end{definition}
Let us write down the weight map $j_0: \Lambda_{2,0}\rightarrow \Lambda_\mathbf{D}$ for the family (which we still denote as $\mathbf{f}\in\mathcal{M}_{\mathrm{ord}}(K^{(2,0)},\Lambda_{2,0})\otimes_{j_0}\Lambda_\mathbf{D}$) on $\mathrm{U}(2,0)$ constructed from the Hida family $\mathbf{f}$ and the character $\boldsymbol{\psi}$. In fact
$$j_0(1+T_1)=\boldsymbol{\psi}_2^{-1}(1+p),\ j_0(1+T_2)=\boldsymbol{\chi}_\mathbf{f}\boldsymbol{\psi}_2|_{\mathbb{Z}^\times_p}(1+p).$$
Here $\boldsymbol{\chi}_{\mathbf{f}}$ is the central character of $\mathbf{f}$, and we write $\boldsymbol{\psi}_2$ for the restriction of $\boldsymbol{\psi}$ to $\mathcal{K}^\times_{\bar{v}_0}\simeq\mathbb{Q}^\times_p$.

\subsection{Siegel-Eisenstein Measure}\label{inner}
\begin{proposition}\label{Siegel Measure}
There are $\Lambda_2$-adic formal Fourier expansions $\mathbf{E}_{\mathbf{D},sieg}$ and $\mathbf{E}_{\mathbf{D},sieg}'$ such that
$$\mathbf{E}_{\mathbf{D},sieg,\phi}=E_{sieg,\mathcal{D}_\phi}(\prod_v f_{sieg,v}, z_\kappa,-)$$
$$\mathbf{E}'_{\mathbf{D},sieg,\phi}=E'_{sieg,\mathcal{D}_\phi}(\prod_v f'_{sieg,v}, z'_\kappa,-)$$
in terms of formal Fourier expansions. Here the datum $\mathcal{D}_\phi=(f_\phi,\xi_\phi,\psi_\phi)$ is the specialization of $\mathbf{D}$ at $\phi$.
\end{proposition}
\begin{proof}
It is a special case of \cite[Lemma 5.7]{WAN} and follows from our computations of the local Fourier coefficients for the Siegel-Eisenstein series. Recall $\boldsymbol{\tau}=\boldsymbol{\psi}/\boldsymbol{\xi}$ and write $\tau_\phi$ for the localization of $\boldsymbol{\tau}$ at an arithmetic point $\phi$. On the other hand, we have families of characters $\boldsymbol{\xi}^\dag_1$ and $\boldsymbol{\xi}^\dag_2$ of $\mathbb{Z}^\times_p$ interpolating the restriction of characters $\xi_{\phi,1}^\dag$ and $\xi^\dag_{\phi,2}$ to $\mathbb{Z}^\times_p$. It follows from our computations in Section 6 for local Fourier coefficients and \cite[Lemma 11.2]{SU}, that we can form the $\beta$-th Fourier coefficient of $$E_{\mathbf{D},sieg}(f_{sieg}; z_\kappa, g)$$ at $\mathrm{diag}(y, {}^t\!\bar{y})$ for $y\in\mathrm{GL}_3(\mathbb{A}_{\mathcal{K},f})$ is given by
\begin{align*}
&&&\prod_{\ell\in\Sigma, \ell\nmid p}D_\ell^{-\frac{3}{2}}\boldsymbol{\tau}(\det y_\ell)|\det y_\ell\bar{y}_\ell|_\ell^{-\frac{\kappa}{2}}e_\ell(\frac{\beta_{22}+\beta_{33}}{y_\ell\bar{y}_\ell})\times\prod_{\ell\nmid p}\boldsymbol{\tau}'_\ell(\det\beta)&\\
&\times&&\prod_{\ell\not\in\Sigma}h_{\ell, {}^t\!\bar{y}_\ell\beta y_\ell}(\bar{\boldsymbol{\tau}}'(\ell)\ell^{-\kappa})\times(\det\beta|\det\beta|_p)^{\kappa-3}&\\
&\times&&\boldsymbol{\xi}^\dag_{1}(\beta_{21})\boldsymbol{\xi}^\dag_{2}(\frac{\beta_{21}\beta_{32}-\beta_{22}\beta_{31}
}{\beta_{21}})\cdot \mathrm{char}(\mathbb{Z}_p, \beta_{11})\mathrm{char}(\mathbb{Z}_p, \beta_{12})\mathrm{char}(\mathbb{Z}_p, \beta_{13})\mathrm{char}(\mathbb{Z}_p, \beta_{23})\mathrm{char}(\mathbb{Z}_p, \beta_{33})&
\end{align*}
if $\det\beta>0$ and $\beta_{21}$ and $\det\begin{pmatrix}\beta_{21}&\beta_{22}\\ \beta_{31}&\beta_{32}\end{pmatrix}$ are both in $\mathbb{Z}^\times_p$, and is equal to $0$ otherwise. Here, $\mathrm{char}(\mathbb{Z}_p, x)$ is the characteristic function for $\mathbb{Z}_p$ for the variable $x$. That the Fourier coefficient is zero if $\det\beta\leq 0$ follows from our computations in Lemma \ref{Fourier}. Note the definition of $\Phi_{\xi^\dag_\phi}$ in Definition \ref{39} for the part in the third row involving the $p$-adic place.
This whole expression is clearly interpolated by an element in the Iwasawa algebra. The case for
$$E'_{\mathbf{D},sieg}(f'_{sieg};z,g)$$ is similar.

Now we can obtain the Siegel-Eisenstein measure from the abstract Kummer congruence as detailed in \cite[Lemma 3.15]{Hsieh13}. From the $\mathrm{mod}\ p$ $q$-expansion principle and that all Fourier coefficients above are interpolated by elements in the Iwasawa algebra, we see that $\mathbf{E}_{\mathbf{D},sieg}$ and $\mathbf{E}_{\mathbf{D},sieg}'$ indeed give a measure with values in the space of $p$-adic automorphic forms on $\mathrm{GU}(3,3)$ (see \cite[Theorem 3.16]{Hsieh13} also).
\end{proof}
\noindent This formal Fourier expansion gives a measure on $\Gamma_\mathcal{K}\times\mathbb{Z}_p$ with values in the space of $p$-adic automorphic forms on $\mathrm{GU}(3,3)$, which we denote as $\mathcal{E}_{\mathcal{D},sieg}$ and $\mathcal{E}_{\mathcal{D},sieg}'$, respectively.\\

\subsection{Interpolating Petersson Inner Products}\label{IPIP}
In this subsection we make an additional construction of pairing Hida families on definite unitary groups following \cite{Hsi17}, which is crucial for our later construction.
\begin{definition}\label{define measure}
For a \emph{neat} tame level group $K\subset \mathrm{U}(2)(\mathbb{A}^{p\infty})$ we use the notation $\mathbf{B}_K\langle-,-\rangle$ to denote the $\Lambda_{\mathrm{U}(2)}$-pairing
$$\mathbf{B}_K: \mathcal{M}_{\mathrm{ord}}(K,\Lambda_{\mathrm{U}(2)})\times\breve{\mathcal{M}}_{\mathrm{ord}}(K,\Lambda_{\mathrm{U}(2)})\rightarrow \Lambda_{\mathrm{U}(2)}$$
such that for any $\mathbf{f}\in\mathcal{M}_{\mathrm{ord}}(K, \Lambda_{\mathrm{U}(2)})$, $\mathbf{g}\in\breve{\mathcal{M}}_{\mathrm{ord}}(K, \Lambda_{\mathrm{U}(2)})$ and $\phi\in\mathrm{Spec} \Lambda_{\mathrm{U}(2)}(\mathbb{C}_p)$ a weight two point, for any $n$ we define
\begin{align*}
&\mathbf{B}_{K,n}\langle\mathbf{g},\mathbf{f}\rangle&&:=\sum_{[x_i]\in\mathrm{U}(2)(\mathbb{Q})\backslash \mathrm{U}(2)/KU_0(p^n)}U^{-n}_p\mathbf{f}(x_i)\mathbf{g}(x_i\begin{pmatrix}&1\\p^n&\end{pmatrix})&\\
&&&(\mathrm{mod}(1+T_1)^{p^n}-1, (1+T_2)^{p^n}-1).&
\end{align*}
Then one checks
$$\mathbf{B}_{K,n+1}\equiv\mathbf{B}_{K,n}(\mathrm{mod}(1+T_1)^{p^n}-1,(1+T_2)^{p^n}-1).$$
We define
$$\mathbf{B}_K\langle\mathbf{g},\mathbf{f}\rangle=\lim_n\mathbf{B}_{K,n}\langle\mathbf{g},\mathbf{f}\rangle.$$
By definition we have
$$\phi(\mathbf{B}_{K}\langle\mathbf{g},\mathbf{f}\rangle)=\sum_{[x_i]\in\mathrm{U}(2)(\mathbb{Q})\backslash \mathrm{U}(2)/KU_0(p^n)}U^{-n}_p\mathbf{f}_\phi(x_i)\mathbf{g}_\phi(x_i\begin{pmatrix}&1\\p^n&\end{pmatrix})$$
and hence
\begin{align*}&\phi(\mathbf{B}_{K}\langle\mathbf{g},\mathbf{f}\rangle)&&=\mathrm{vol}(KU_0(p^n))^{-1}\int_{[\mathrm{U}(2)]}U^{-n}_p\mathbf{f}_\phi(h)\mathbf{g}_\phi
(h\begin{pmatrix}&1\\p^n&\end{pmatrix})dh&\\
&&&=\mathrm{vol}(KU_0(p^n))^{-1}\int_{[\mathrm{U}(2)]}\mathbf{f}_\phi(h\begin{pmatrix}&1\\1&\end{pmatrix}_p)\mathbf{g}_\phi
(h)dh&
\end{align*}
if $\phi$ corresponds to an ordinary form whose $p$-part conductor is $p^n$.
\index{$\mathbf{B}\langle,\rangle$}
In the following we will fix the tame level group $K^{(2,0)}$ as defined before, and will sometimes suppress the subscript $K$ in $B_K$.
\end{definition}
\subsection{$p$-adic $L$-functions}\label{6.4}
We have the following proposition for $p$-adic $L$-functions. Recall $\Sigma$ is a finite set of primes containing all bad primes.
\begin{proposition}\label{$p$-adic $L$-function}
Notations are as before. There is an element $\mathcal{L}^\Sigma_{\mathbf{f},\xi,\mathcal{K}}$ \index{$\mathcal{L}^\Sigma_{\mathbf{f},\xi,\mathcal{K}}$} in $\hat{\mathbb{I}}^{ur}[[\Gamma_\mathcal{K}]]$, and a $p$-integral element $C_{\mathbf{f},\xi,\mathcal{K}}^\Sigma\in\bar{\mathbb{Q}}_p^\times$ such that for any generic arithmetic point $\phi$ of conductor $p^t$, we have:
\begin{equation}\label{equation (8)}
\phi(\frac{\mathcal{L}^\Sigma_{\mathbf{f},\xi,\mathcal{K}}}{\Omega^{2\kappa}_p})=C_{\mathbf{f},\xi,\mathcal{K}}^\Sigma\frac{2\pi ip^{(\kappa-3)t}\xi_{1,p}^2\chi_{1,p}^{-1}
\chi_{2,p}^{-1}(p^{-t})\mathfrak{g}(\xi_{v_0}\chi_{1,p}^{-1})
\mathfrak{g}(\xi_{v_0}\chi_{2,p}^{-1})L^{\Sigma}(\mathcal{K},\pi_{f_\phi},\bar{\chi}_\phi\xi_\phi,
\frac{\kappa}{2}-\frac{1}{2})(\kappa-1)!(\kappa-2)!}{\Omega_\infty^{2\kappa}}.
\end{equation}
Here $\chi_{1,p},\chi_{2,p}$ is such that the unitary representation $\pi_{f_\phi}\simeq\pi(\chi_{1,p},\chi_{2,p})$ with $\mathrm{val}_p(\chi_{1,p}(p))=-\frac{1}{2}$, $\mathrm{val}(\chi_{2,p}(p))=\frac{1}{2}$. We remark that the fraction on the right hand side is an algebraic number, by the definition of the periods. Moreover, by making different choices for the N{\'e}ron differential of the CM elliptic curve, the $\Omega_\infty$ and $\Omega_p$ are changed by multiplying by the same non-zero algebraic number. 
Let the local $\epsilon$-factor at $p$ for a ramified character $\lambda$ of $\mathbb{Q}^\times_p$ with conductor $c$ be defined by 
\begin{equation}
\epsilon_p(\lambda, s)=\int_{c^{-1}\mathbb{Z}_p^\times}\lambda^{-1}(a)|a|^{-s}_pe_p(\mathrm{Tr}a)da.
\end{equation}
Note $1-\frac{\kappa-1}{2}=\frac{3-\kappa}{2}$.
The above interpolation formula can be written in terms of local $\epsilon$-factors by
\begin{equation}\label{equation (91)}
\phi(\frac{\mathcal{L}^\Sigma_{\mathbf{f},\xi,\mathcal{K}}}{\Omega^{2\kappa}_p})=C_{\mathbf{f},\xi,\mathcal{K}}^\Sigma\frac{2\pi i\epsilon_p(\frac{3}{2}-\frac{\kappa}{2},\xi^{-1}_{v_0}\chi_{1,p})\epsilon_p(\frac{3}{2}-\frac{\kappa}{2},\xi^{-1}_{v_0}\chi_{2,p}) L^{\Sigma}(\mathcal{K},\pi_{f_\phi},\bar{\chi}_\phi\xi_\phi,
\frac{\kappa}{2}-\frac{1}{2})(\kappa-1)!(\kappa-2)!}{\Omega_\infty^{2\kappa}}.
\end{equation}
\end{proposition}

\begin{proof}
Suppose we are under the assumption of Theorem \ref{Theorem 1}. Take $g_0$ to be a point on the Igusa scheme for $\mathrm{GU}(2)$ defined over $\hat{\mathcal{O}}_L^{ur}$ such that $\mathbf{f}(g_0)$ is nonzero in $\hat{\mathbb{I}}^{ur}$ and take a $h_0\in \mathrm{GU}(2)(\mathbb{A}_\mathbb{Q})$ such that $\mu(g_0)=\mu(h_0)$. It is noted in \cite[Section 2.8]{Hsieh CM} that
$$I_{\mathrm{GU}(2)}(K_1^n)(\hat{\mathcal{O}}_L^{ur})=\mathrm{GU}(2)(\mathbb{Q})^+\backslash\mathrm{GU}(2)
(\mathbb{A}_f)/K_1^n.$$
In the following we write $\hat{F}$ for the $p$-adic modular form associated to a form $F$.
We define $\mathcal{L}^\Sigma_{\mathbf{f},\xi,\mathcal{K}}$ such that
\begin{equation*}
\mathcal{L}^\Sigma_{\mathbf{f},\xi,\mathcal{K}}=\mathbf{B}\langle e^{\mathrm{U}(2)}_{ord}\hat{\mathbf{E}}'_{\mathbf{D},sieg}(\underline{A}_{g_0},-)\boldsymbol{\bar{\tau}}\circ\det(-),\pi(h_0)\mathbf{f}\rangle/\mathbf{f}(g_0)
\end{equation*}
where $\underline{A}$ is the quintuple associated to $g_0$, and we regard $\hat{\mathbf{E}}'_{\mathbf{D},sieg}(-,-)$ as a measure of forms on $I_{\mathrm{U}(2,0)}(K^{(2,0)})\times I_{\mathrm{U}(2)}(K^{(2,0)})$ under the embedding $i$ as in \ref{Em}. The $e^{\mathrm{U}(2)}_{ord}$ means applying the ordinary projector to the $\mathrm{U}(2)$-factor. Therefore for any generic arithmetic point $\phi$ of conductor $p^t$,
$$\mathrm{Vol}(K^\infty)^{-1}f_\phi(g_0)^{-1}\sum_{\underline{B}\in I_{\mathrm{U}(2)}(K^{(2,0)}K_0(p^t))}\hat{E}'_{\mathcal{D}_\phi,sieg}(\underline{A}_{g_0},\underline{B})
\bar{\tau}_\phi\circ\det(\underline{B})\times\pi(h_0)
\hat{f}_\phi(\underline{B}).$$
The $\mathbf{B}\langle,\rangle$ is in terms of Definition \ref{define measure}. Let the character $\boldsymbol{\tau}=\boldsymbol{\psi}/\boldsymbol{\xi}$. The function $\boldsymbol{\tau}(\det g_2)$ means the function taking value $\boldsymbol{\tau}(\det g_2)$ at the point $(g_1,g_2)$ in the above set. The integration is in the sense of subsection \ref{inner} with respect to the level group $h^{-1}_0(K_\mathcal{D}\cap(1\times \mathrm{GU}(2)(\mathbb{A}_f))h_0$ (in fact by pullback we get a measure of forms on the $h_0$-Igusa schemes, see the last part of subsection \ref{2.4}).
This $\mathcal{L}_{\mathbf{f},\xi,\mathcal{K}}^\Sigma$ satisfies the proposition (by Lemmas \ref{Apullback}, \ref{Upullback}, \ref{Lpullback}, \ref{Ppullback}).

This construction only implies the $p$-adic $L$-function is in $\hat{\mathbb{I}}^{ur}[[\Gamma_\mathcal{K}]]\otimes\mathrm{Frac}(\mathbb{I})$. To see the integrality, we take different choices for $g_0$ and note that by our choices for $\mathbf{f}$, its values have no non-trivial common divisors in $\mathbb{I}$.

If we are under assumption of Theorem \ref{Theorem 2} then we just pick up a $g_0$ such that $\mathbf{f}$ has non-zero specialization at $g_0$. Note that the period factors $\Omega_\infty^{2\kappa}$ and $\Omega_p^{2\kappa}$ come from the pullback as discussed in \cite[Section 2.8, Subsection 5.6.5]{Hsieh CM}.
\end{proof}

\begin{definition}\label{Hida $p$-adic $L$-functions}
Now we define Hida's $p$-adic $L$-function $\mathcal{L}_{\mathbf{f},\xi,\mathcal{K}}^{Hida}\in\Lambda_\mathbf{D}$ \index{$\mathcal{L}_{\mathbf{f},\xi,\mathcal{K}}^{Hida}$}. As in the main theorems we assume the $\xi$ has split conductor.

We consider the Hida family of ordinary $CM$ eigenforms $g_{\boldsymbol{\xi}}$ associated to $\xi$ (for simplicity we consider here ordinary forms by twisting the adelic nearly ordinary form corresponding to $\xi_\phi$'s by the unique Dirichlet character of $p$-power conductor, which makes it an ordinary form, i.e. being invariant under the right action of $\begin{pmatrix}\mathbb{Z}^\times_p&\\&1\end{pmatrix}$).
For our purpose we further twist the automorphic representation of $g_{\boldsymbol{\xi}}$ by a choice of a fixed finite order Dirichlet character $\chi^{\mathrm{tw}}_{g_{\boldsymbol{\xi}}}$, unramified outside the prime-to-$p$ places where $\xi$ is ramified, such that at each bad prime $\ell\not=p$, the resulting automorphic representation at $\ell$ is minimal in the sense of \cite[Section 7]{HT93} (namely it is principal series induced from two characters in which at least one is unramified). The reason is to ensure that we can compare the Petersson inner product of $g_\phi$'s with certain Katz $p$-adic $L$-functions without different Euler factors. (Alternatively we can also multiply the $\xi$ by $\chi^{\mathrm{tw}}_{g_{\boldsymbol{\xi}}}\circ\mathrm{Nm}$ and construct the primitive ordinary Hida family associated to this product character $\xi'=\xi\cdot \chi^{\mathrm{tw}}_{g_{\boldsymbol{\xi}}}\circ\mathrm{Nm}$). We denote the resulting primitive Hida family as $g_{\boldsymbol{\xi}'}$. Note that 
$$\xi\xi^{-c}=\xi'\xi^{\prime-c}.$$

We consider the $p$-adic $L$-function $\mathscr{D}$ constructed in \cite[Theorem I]{Hida91} choosing $\mathbf{f}$ there to be the $g_{\boldsymbol{\xi}'}$  and $\mathbf{g}^\rho$ there to be the ordinary eigenforms of our $\mathbf{f}$, twisted by the Dirichlet character $\chi^{\mathrm{tw}}_{g_{\boldsymbol{\xi}}}$ above. 

Instead of the CM period $\Omega_\infty$, the period factor in Hida's construction is the Petersson period (see \cite[Theorem I, Lemma 5.3 (vi)]{Hida91})
$$W'(g_\phi)^{-1}\langle g_\phi, g_\phi|_{\kappa+1}\begin{pmatrix}&-1\\N_{g_\phi}&\end{pmatrix}\rangle\eta_{g_\phi}(p^{-t_\phi}),$$
where $g_\phi$ is a specialization of $g_{\boldsymbol{\xi}}$ of weight $\kappa+1$, $N_{g_\phi}$ is the conductor of $g_\phi$, the $p^{t_\phi}$ is the $p$-part of its conductor, the $p$-component of the automorphic representation associated to the unitarization of $g_\phi$ (in the sense of \cite[Introduction]{Hida91}) is $\pi(\eta_{g_\phi},\eta'_{g_\phi})$ with $\mathrm{ord}_p\eta_{g_\phi}(p)<\mathrm{ord}_p\eta'_{g_\phi}(p)$.  Note that the $S(P)^{-1}$ and the Gauss sum in the denominator of the second row of the definition of $W(P,Q)$ in \cite[Theorem I]{Hida91} are included as part of the period (see Lemma 5.3 (vi) of \emph{loc.cit.}).   
The above expression equals (by \cite[Theorem 7.1, (8.8b)]{HT93} and \cite[Lemma 5.3 (vi)]{Hida91})
\begin{align*}
&&&\langle g_\phi, g_\phi\rangle\cdot \eta'_{g_\phi}\eta^{-1}_{g_\phi}(p^{t_\phi})\mathfrak{g}(\eta^{\prime -1}_{g_\phi}\eta_{g_\phi})\cdot p^{-t_\phi}&\\&=&&\mathfrak{L}(\mathrm{ad}, g_\phi,1)\eta'_{g_\phi}\eta_{g_\phi}^{-1}(p^{t_\phi})\mathfrak{g}(\eta^{\prime -1}_{g_\phi}\eta_{g_\phi})\kappa!2^{-2\kappa-1}\pi^{-\kappa-2}&
\end{align*}
where the $\mathfrak{L}(\mathrm{ad}, -)$ is in the sense of \cite[Section 7]{HT93}. By our assumption on that $\xi$ has split conductor, we see that the $\Delta(1)$ in 
\cite[(0.7a)]{HT93} equals $1$. So 
$$\mathfrak{L}(\mathrm{ad}, g_\phi,1)=L(1,\chi_{\mathcal{K}})L(\xi_\phi\xi^{-c}_\phi, 1).$$
On the other hand, there is a Katz $p$-adic $L$-function (see \cite[(8.2)]{HT93} for this Katz measure, we remove the factor $\mathrm{Im}(\delta)^{\kappa-1}$ there since it is a unit Iwasawa element and has no effect to us.)
$\mathcal{L}^{Katz}_{\mathcal{K},\xi}\in\hat{\mathcal{O}}^{ur}_L[[\Gamma_\mathcal{K}]]$  interpolating the values
$$\eta'_{g_\phi}\eta^{-1}_{g_\phi}(p^{t_\phi})\frac{\pi^{\kappa-1}\kappa!L(\xi_\phi\xi_\phi^{-c},1)\cdot\mathfrak{g}(\eta^{\prime -1}_{g_\phi}\eta_{g_\phi})}{(2i)^{2\kappa}\Omega_\infty^{2\kappa}}\Omega_p^{2\kappa}$$

Let $\mathrm{Cl}_\mathcal{K}$ be the class number of $\mathcal{K}$. By class number formula 
$$\frac{L(1,\chi_\mathcal{K})}{\pi}=\mathrm{Cl}_\mathcal{K}\cdot D^{-\frac{1}{2}}_\mathcal{K}.$$
We multiply Hida's $p$-adic $L$-function $\mathscr{D}$ by $2i^{2\kappa}
\mathrm{Cl}_\mathcal{K}D^{-\frac{1}{2}}_\mathcal{K}\cdot\mathcal{L}^{Katz}_{\mathcal{K},\xi}$, and further divide it by the first row in the definition of $W(P,Q)$ of \cite[Theorem I]{Hida91}. Note that this last factor is a unit Iwasawa element.
We denote this result as $\mathcal{L}^{Hida}_{\mathbf{f},\xi,\mathcal{K}}$. One checks readily the interpolation formula for it is
\begin{equation}\label{equation (9)}
\phi(\frac{\mathcal{L}_{\mathbf{f},\xi,\mathcal{K}}^{Hida}}{\Omega_p^{2\kappa}})=\frac{2\pi i\epsilon_p(\frac{3}{2}-\frac{\kappa}{2},\xi^{-1}_{v_0}\chi_{1,p})\epsilon_p(\frac{3}{2}-\frac{\kappa}{2},\xi^{-1}_{v_0}\chi_{2,p})L(\mathcal{K},\pi_{f_\phi},\bar{\chi}_\phi\xi_\phi,
\frac{\kappa}{2}-\frac{1}{2})(\kappa-1)!(\kappa-2)!}{\Omega_\infty^{2\kappa}}.
\end{equation}
(Recall we have divided out the product of prime to $p$ root numbers in \emph{loc.cit.} which are $p$-units and moves $p$-adic analytically.)
If $\xi$ is such that $g_\xi$ satisfies the (dist) and (irred) in the introduction, then the local Hecke algebra for $g_{\boldsymbol{\xi}}$ is Gorenstein. By the main conjecture proved in \cite{HT93}, \cite{HT94} and \cite{Hida06} (see \cite[Theorem and Page 468, (F)]{Hida06}, the $\mathrm{Cl}_\mathcal{K}\cdot\mathcal{L}^{Katz}_{\mathcal{K},\xi}$ generates the congruence module for $\mathbf{g}$ and our $\mathcal{L}^{Hida}_{\mathbf{f},\xi,\mathcal{K}}$ is integral (i.e. in $\Lambda_{\mathbf{D}}$).) We explain here the weight map parameterizing the Hida family $\mathbf{g}$ with coefficient ring $\mathcal{O}_L[[\Gamma^-_\mathcal{K}]]$. Let $\mathcal{K}^-$ be the anticyclotomic $\mathbb{Z}_p$ extension of $\mathcal{K}$ and $\mathcal{K}^{\mathrm{ur}}$ be the maximal subextension of $\mathcal{K}^-/\mathcal{K}$ unramified everywhere. Let $p^a$ be the index of $\mathcal{K}^{\mathrm{ur}}/\mathcal{K}$ and $\Gamma^{-,\prime}_\mathcal{K}$ be the corresponding subgroup of $\Gamma^-_\mathcal{K}$. Then local class field theory gives an isomorphism $$\sigma_{v_0}: (1+p\mathbb{Z}_p)^\times\simeq \Gamma^{-,\prime}_\mathcal{K}.$$ Write $\mathcal{O}_L[[W]]$ for the weight algebra of the Hida family $\mathbf{g}$, then the weight map is
$$\mathcal{O}_L[[W]]\simeq \mathcal{O}_L[[\Gamma^{-,\prime}_\mathcal{K}]]\hookrightarrow \mathcal{O}_L[[\Gamma^-_\mathcal{K}]],$$
where the first map is determined by
$$(1+W)\mapsto (1+p)^{-1}\sigma_{\bar{v}_0}(1+p),$$
and the last term is a finite free module over the second term.
\end{definition}

Given a finite set of primes $\Sigma$ we can define the $\Sigma$-primitive Hida $p$-adic $L$-function $\mathcal{L}^{\Sigma,Hida}_{\mathbf{f},\mathcal{K},\xi}$ by removing local Euler factors at $\Sigma$. Obviously, it is just multiplying $\mathcal{L}^{Hida}_{\mathbf{f},\mathcal{K},\xi}$ by a finite number of nonzero elements in $\mathbb{I}[[\Gamma_\mathcal{K}]]$. Note that Hida proved the interpolation formula for general arithmetic points. We may compare (\ref{equation (8)}) and (\ref{equation (9)}). If we write $\mathcal{L}_{f_0,\xi,\mathcal{K}}^\Sigma$ \index{$\mathcal{L}_{f_0,\xi,\mathcal{K}}^\Sigma$} for the specialization of $\mathcal{L}_{\mathbf{f},\xi,\mathcal{K}}^\Sigma$ to $\hat{\mathcal{O}}^{ur}_L[[\Gamma_\mathcal{K}]]$ at $f_0$, then we get the interpolation formula
\begin{equation}\label{equation (10)}
\phi(\frac{\mathcal{L}^\Sigma_{f_0,\xi,\mathcal{K}}}{\Omega_p^{2\kappa}})=C_{\mathbf{f},\xi,\mathcal{K}}^\Sigma\frac{2\pi i\epsilon_p(\frac{3}{2}-\frac{\kappa}{2},\xi^{-1}_{v_0}\chi_{1,p})\epsilon_p(\frac{3}{2}-\frac{\kappa}{2},\xi^{-1}_{v_0}\chi_{2,p})L^{\Sigma}(\mathcal{K},\pi_{f_0},
\xi_\phi,\frac{\kappa}{2}-\frac{1}{2})(\kappa-1)!(\kappa-2)!}{\Omega_\infty^{2\kappa}}.
\end{equation}
for $\xi_\phi$'s of conductor $(p^t,p^t)$ at $p$.\\

\noindent\underline{Anti-cyclotomic $\mu$-invariants}:\\
\noindent Now assume we are under assumption of Theorem \ref{Theorem 1} in the introduction. We define $\phi_0$ to be the $\bar{\mathbb{Q}}_p$-point in $\mathrm{Spec}\mathbb{I}[[\Gamma_\mathcal{K}]]$ sending $\gamma^{\pm}$ to $1$ and such that $\phi_0|_{\hat{\mathbb{I}}^{ur}}$ correspond to $f_0$. Our assumptions on $\xi$ and $\kappa$ ensure that our $p$-adic families pass through this point. (This is not an arithmetic point in Definition \ref{Eisens}, however it still interpolates the algebraic part of the special $L$-value by \cite{Hida91}.) Consider the one-dimensional subspace of $\mathrm{Spec}\hat{\mathbb{I}}^{ur}[[\Gamma_\mathcal{K}]]$ of anti-cyclotomic twists by characters of order and conductor powers of $p$ that passes through $\phi_0$. We look at the ratio between the specialization of Hida's $p$-adic $L$-function $\mathcal{L}^{Hida}_{\mathbf{f},\xi,\mathcal{K}}$ to this subspace and the anti-cyclotomic $p$-adic $L$-function considered by \cite{Hsi12} (note that the local sign assumptions there are satisfied). We explain the fudge factors: recall the $S(P)^{-1}$ and the Gauss sum in the denominator of the definition of $W(P,Q)$ of \cite[Theorem I]{Hida91} are already included in the period studied above. Also the first row of the definition of $W(P,Q)$ are also divided out in our definition for $\mathcal{L}^{Hida}_{\mathbf{f},\xi,\mathcal{K}}$. The remaining factors are: the $C(\pi,\lambda)$ and $\mathfrak{F}$ in \cite{Hsi12} which is a \emph{fixed} $p$-adic unit, and the powers of $\mathrm{Im}(\delta)$ and $2$ which are unit Iwasawa elements. So by result proved in \cite{Hsi12} the anti-cyclotomic $p$-adic $L$-function has $\mu$-invariant $0$. Thus it is easy to see that any height one prime $P$ of $\hat{\mathbb{I}}^{ur}[[\Gamma_\mathcal{K}]]$ containing $\mathcal{L}^{Hida}_{\mathbf{f},\mathcal{K},\xi}$ can not be the pullback of a height $1$ prime of $\hat{\mathbb{I}}^{ur}[[\Gamma_\mathcal{K}^+]]$. Therefore for any height $1$ prime containing $\mathcal{L}^{Hida}_{\mathbf{f},\mathcal{K},\xi}$,
$$\mathrm{ord}_P(\mathcal{L}^{Hida}_{\mathbf{f},\mathcal{K},\xi})=\mathrm{ord}_P(\mathcal{L}
_{\mathbf{f},\xi,\mathcal{K}})$$
and $\mathrm{ord}_P(\mathcal{L}_{\bar{\chi}_\mathbf{f}\xi})=0$. The $\mathcal{L}_{\mathbf{f},\xi,\mathcal{K}}$ is obtained by putting back the Euler factors at primes in $\Sigma$ on $\mathcal{L}^\Sigma_{\mathbf{f},\xi,\mathcal{K}}$. (There might be factors coming from Euler factors at non-split primes contributing to the anticyclotomic $\mu$-invariant of the $\Sigma$-primitive $p$-adic $L$-functions, however. We will explain how to treat those factors when proving our main theorem.)

\subsection{$p$-adic Eisenstein Series}
\begin{proposition}\label{Proposition 7.7}
There is a $\hat{\mathbb{I}}^{ur}[[\Gamma_\mathcal{K}]]$-adic formal Fourier-Jacobi expansion
$$\mathbf{E}_{\mathbf{D},Kling}\in \mathcal{M}_{\mathrm{ord}}(K,\Lambda_\mathbf{D})$$ \index{$\mathbf{E}_{\mathbf{D},Kling}$} such that for each generic arithmetic point $\phi\in\mathrm{Spec}\hat{\mathbb{I}}^{ur}[[\Gamma_\mathcal{K}]]$, the specialization $\mathbf{E}_{\mathbf{D},Kling,\phi}$ is the Fourier-Jacobi expansion of the nearly ordinary Klingen-Eisenstein series $E_{Kling,\mathbf{D}_\phi}$ we constructed in (\ref{DefineKlingen}) using the Eisenstein datum at $\phi$. Moreover, recall the fundamental exact sequence in Theorem \ref{exact} and the Siegel operator $\hat{\Phi}^h_{[g]}$ there, then the constant term $\hat{\Phi}^{w'_3}_{[g]}(\mathbf{E}_{Kling,\mathbf{D}})$'s are divisible by $\mathcal{L}^\Sigma_{\mathbf{f},\xi,\mathcal{K}}.\mathcal{L}^\Sigma_{\bar{\chi}\xi'}$, where $\mathcal{L}^\Sigma_{\bar{\chi}\xi'}$ is the element in $\mathbb{I}[[\Gamma_\mathcal{K}^+]]$ which is the Dirichlet $p$-adic $L$-function interpolating the algebraic part of the special values $L^\Sigma(\bar{\chi}_\phi\xi'_\phi,\kappa_\phi-2)$.
\end{proposition}
\begin{proof}
Our construction is more similar to \cite[Theorem 4.4]{Hsieh13} than to \cite[Theorem 12.11]{SU}. Recall Definition \ref{definition 3.7} the notion of $\hat{\mathbb{I}}^{ur}[[\Gamma_\mathcal{K}]]$-adic Fourier-Jacobi expansion. It is a special case of \cite[Theorem 1.1 (3)]{WAN}. In our cases, the local choices are slightly different but the arguments are the same, which we give below. For a $[g]$ we take a basis $(\theta_{1,\beta}',\cdots, \theta_{m_\beta,\beta}')$ of the $\mathcal{O}_L$-dual space of $H^0(\mathcal{Z}^\circ_{[g]}, \mathcal{L}(\beta))$ consisting of theta functions. Write $(\theta_{1,\beta}'',\cdots, \theta_{m_\beta,\beta}'')$ for the dual basis. Suppose $\theta$ is one of the $\theta_i'$s. For any $g\in \mathrm{GU}(2)(\mathbb{A}_\mathbb{Q})\subset \mathrm{GU}(3,1)(\mathbb{A}_\mathbb{Q})$ we take $h\in \mathrm{GU}(2)(\mathbb{A}_\mathbb{Q})$ such that $\mu(g)=\mu(h)$. Recall we have denoted the $\beta$-Fourier-Jacobi coefficient $a_{[g]}^h(\beta, F)$ for forms on $\mathrm{GU}(3,1)$. We write $a_{[g],\theta}^h(\beta, F)$ for the pairing of it with $\theta$. We define
\begin{align}\label{666}
&a^1_{[g],\theta}(\beta, \mathbf{E}_{\mathbf{D}, Kling})=\mathbf{B}_{K^{(2,0)}}\langle e^{\mathrm{U}(2)}_{ord}a^1_{[g],\theta}(\beta, \hat{\mathbf{E}}_{\mathbf{D},sieg}(-,-))\bar{\boldsymbol{\tau}}\circ\det(-),\pi(h)\mathbf{f}\rangle&
\end{align}
As before the $e_{ord}^{\mathrm{U}(2)}$ means applying the ordinary projector for the $\mathrm{U}(2)$ factor. We regard $\hat{\mathbf{E}}_{\mathbf{D},sieg}(-,-)$ as evaluated on $(\underline{A},\underline{B})\in I_{\mathrm{U}(3,1)}(K^{(3,1)})\times I_{\mathrm{U}(2)}(K^{(2,0)})$ under the embedding \ref{Em}. The $\bar{\boldsymbol{\tau}}(\det-)$ is regarded as a function for the $\mathrm{U}(2)$ factor. In view of the algebraic embedding of Igusa schemes, the pullback of the Siegel-Eisenstein measure gives a measure with values in the space of $p$-adic automorphic forms on the group $\{g,h\in\mathrm{GU}(3,1)\otimes\mathrm{GU}(2), \det g=\det h\}$. Fix the $g$, applying the $\beta$-th Fourier-Jacobi coefficient operator to the $\mathrm{GU}(3,1)$-factor and take the $\theta$-component we get a $\hat{\mathbb{I}}^{ur}[[\Gamma_\mathcal{K}]]$-valued family of forms on (the lower) $\mathrm{GU(2)}$ in the sense of Definition \ref{def D}, which is the integrand of (\ref{666}). Then we form the pairing $\langle,\rangle$ of (\ref{666}) in the sense of subsection \ref{IPIP} with respect to the level group $h^{-1}(K_\mathcal{D}\cap(1\times \mathrm{GU}(2)(\mathbb{A}_f))h$ (In fact, by pullback, we get a measure of forms on the $h$-Igusa schemes). We obtain the family of Fourier-Jacobi expansions. It is clear from the construction that this interpolates the Fourier-Jacobi expansions of the ordinary Klingen-Eisenstein series we constructed at arithmetic points (see (\ref{DefineKlingen})). We get the Fourier-Jacobi expansion at $g$ as in Definition \ref{definition 3.7}
$$\mathbf{E}_{\mathbf{D},Kling}(g)=\sum_\beta a^1_{[g]}(\beta, E_{\mathbf{D},Kling})q^\beta $$
where $$a^1_{[g]}(\beta, \mathbf{E}_{\mathbf{D},Kling})=\sum_ia^1_{[g],\theta'_{i,\beta}}(\beta,\mathbf{E}_{\mathbf{D},Kling})\otimes
\theta''_{i,\beta}$$ with
$\theta''_{i,\beta}\in H^0(\mathcal{Z}^\circ_{[g]}, \mathcal{L}(\beta))$. At a generic arithmetic point $\phi$ of conductor $p^t$ we have
$$a^1_{[g],\theta}(\beta, E_{\mathcal{D}_\phi,Kling})=\sum_{\underline{B}\in I_{\mathrm{U}(2)}(K^{(2,0)}K_0(p^t))}a^1_{[g],\theta}(\beta,\hat{E}_{\mathcal{D}_\phi,sieg}(-,\underline{B}))\cdot\bar{\tau}_\phi\circ\det(\underline{B})\times\pi(h)\varphi_\phi
(\underline{B})\rangle.$$

Next we explain the assertion on constant terms. The constant terms is simply interpolating the $\beta$-th Fourier-Jacobi coefficient for $\beta=0$ (i.e. the Siegel operator $\Phi_{P,g}(\mathbf{E}_{Kling,\mathbf{D},\phi})$). Let's consider the case when $g_v$'s are in the support of $F_{Kling,v}$ for $v\nmid p$, and $g_p=\omega'_3$. We claim that
$$\Phi^{w'_3}_{[g]}(\mathbf{E}_{Kling,\mathbf{D}})=C_\mathbf{D}
\mathcal{L}^\Sigma_{\mathbf{f},\xi,\mathcal{K}}\cdot\mathcal{L}^\Sigma_{\bar{\xi}\xi'}\mathbf{f}$$
for $C_\mathbf{D}$ being the product of the constants in the local pullback sections at primes outside $p$. It is a fixed non-zero number throughout the family.

To see the claim, specializing to an arithmetic point $\phi$, this is simply the constant term computation in Section \ref{constantterms} for $R$ being the Klingen parabolic subgroup $P$. This constant term is given by Lemma \ref{CON}. On the other hand, from the Archimedean computation in \cite[Corollary 5.11]{WAN} we see the contribution $A(\rho,f,z)_{-z}$ is actually $0$ in our case. So we only need to work out the pullback Klingen-Eisenstein
section. Now it is an easy consequence of our computations in Section 6 of local pullback integrals (Lemmas \ref{Apullback}, \ref{Upullback}, \ref{Lpullback}, \ref{Ppullback}), together with the normalization factors in Section \ref{GC}.
\end{proof}
It follows that the formal Fourier-Jacobi expansion $\mathbf{E}_{\mathbf{D},Kling}$ comes from a family in $\mathcal{M}^{ord}(K^p, \Lambda_\mathbf{D})$, which we still denote as $\mathbf{E}_{\mathbf{D}, Kling}$. (In fact, Theorem \ref{4.25} is still true after replacing $A=\mathbb{I}[[\Gamma_\mathcal{K}]]$ by $\hat{\mathbb{I}}^{ur}[[\Gamma_\mathcal{K}]]$.)

\section{$p$-adic Properties of Fourier-Jacobi Coefficients}\label{8}
\noindent\underline{Notation}: to avoid confusion in this section, we use $\mathsf{z}\in\mathrm{Spec}\mathbb{I}[[\Gamma_\mathcal{K}]](\bar{\mathbb{Q}}_p)$ instead of $\phi$ to denote arithmetic points on the weight space. The $\phi$ will usually denote Schwartz functions in theta correspondences. Such convention is only valid in this section.

The purpose of this section is to prove Proposition \ref{p-adic}.

\subsection{Preliminaries}\label{D/U}
\noindent\underline{Some Local Representation Theories}\\
(Non-compact case): Let $v$ be a non-split prime where $\mathrm{U}(2)(\mathbb{Q}_v)\simeq \mathrm{U}(1,1)(\mathbb{Q}_v)$. Then $D_v^\times\simeq \mathrm{GL}_2(\mathbb{Q}_v)$. For some irreducible admissible representation $\pi^{\mathrm{U}(2)_v}$ of $\mathrm{U}(2)_v$ we can find an irreducible representation $\pi^{\mathrm{GU}(2)_v}$ of $\mathrm{GU}(2)_v$ such that $\pi^{\mathrm{U}(2)_v}$ is a summand of $\pi^{\mathrm{GU}(2)_v}$ restricting to $\mathrm{U}(2)_v$ (note here the superscripts do not mean invariant subspaces). Thus we have:
$$\pi^{\mathrm{GU}(2)_v}|_{\mathrm{U}(2)_v}=\pi^{\mathrm{U}(2)_v}\oplus {}^\alpha\!\pi^{\mathrm{U}(2)_v}\ \mbox{or}\ \pi^{\mathrm{U}(2)_v}$$
for irreducible representations $\pi^{\mathrm{U}(2)(\mathbb{Q}_v)}$ of $\mathrm{U}(2)(\mathbb{Q}_v)$. Here $\alpha$ is some element such that $\mathrm{Nm}(\alpha)\not\in \mathrm{Nm}(\mathcal{K}_v/\mathbb{Q}_v)$.  The ${}^\alpha\!$ means the representation composed with the automorphism given by conjugation by $\alpha$.
Also the restriction of $\pi^{\mathrm{GU}(2)_v}$ to $D_v^\times$ is clearly irreducible.\\

\noindent (Compact Case) If $D_v^\times$ modulo center is compact, then we let $\alpha$ be some element such that $\mathrm{Nm}(\alpha)\not\in \mathrm{Nm}(\mathcal{K}_v/\mathbb{Q}_v)$. For $\pi^{\mathrm{U}(2)_v}$ we similarly have $\pi^{\mathrm{GU}(2)_v}$, $\pi^{D_v^\times}$. These can all be considered as finite-dimensional representations of finite groups. The $\pi^{\mathrm{GU}(2)_v}=\pi^{D_v^\times}$ as vector spaces and $\pi^{\mathrm{GU}(2)_v}|_{\mathrm{U}(2)_v}=\pi^{\mathrm{U}(2)_v}$ or $\pi^{\mathrm{U}(2)_v}\oplus {}^\alpha\!\pi^{\mathrm{U}(2)_v}$.\\

In both cases, we write $\iota_\alpha$ for the isomorphism between $\pi^{\mathrm{U}(2)_v}$ and $ {}^\alpha\!\pi^{\mathrm{U}(2)_v}$ given by right action by $\alpha$ (as vector spaces, the group actions may differ by a conjugation however.)

\noindent\underline{Forms on $D^\times$ and $\mathrm{U}(2)$}\\
\noindent We first define $\breve{D}^\times\subset D^\times(\mathbb{A}_\mathbb{Q})$ as the index $2$ subgroup consisting of elements whose reduced norms are in $\mathbb{Q}^\times \mathrm{Nm}(\mathbb{A}_\mathcal{K}^\times)$ and let $\breve{D}^\times(\mathbb{Q}_v)$ be the set of elements whose determinants are in $\mathrm{Nm}(\mathcal{K}^\times_v)$. Suppose $\varphi$ is a form on $\mathrm{U}(2)(\mathbb{Q})\backslash \mathrm{U}(2)(\mathbb{A}_\mathbb{Q})$, $\chi$ is a Hecke character of $\mathcal{K}^\times\backslash \mathbb{A}_\mathcal{K}^\times$. Suppose the central action of $\mathrm{U}(1)(\mathbb{Z}_p)$ on $\varphi$ is given by $\chi|_{\mathrm{U}(1)(\mathbb{Z}_p)}$, we can define a form $\varphi^D\boxtimes\chi$ on $D^\times(\mathbb{A}_\mathbb{Q})$ as follows. We first define $\varphi_\chi'$ on $\mathrm{U}(2)(\mathbb{A}_\mathbb{Q})$ as
$$\varphi'_\chi(g):=\int_{[U(1)]}\varphi(gt)\chi^{-1}(t)dt.$$
We now define a form on $D^\times(\mathbb{A}_\mathbb{Q})$. Recall that the image of reduced norm map from $D^\times(\mathbb{Q})$ consists of all positive elements $\mathbb{Q}^+$ in $\mathbb{Q}^\times$ (\cite[Page 206]{Weil}). Note that
$$\mathbb{A}^{\times,+}_\mathbb{Q}/\mathbb{Q}^+\mathrm{Nm}(\mathbb{A}^\times_\mathcal{K})\rightarrow \mathbb{A}^\times_\mathbb{Q}/\mathbb{Q}^\times\mathrm{Nm}(\mathbb{A}^\times_\mathcal{K})$$
is an isomorphism, where $\mathbb{A}^+_\mathbb{Q}$ is the set of ideles with positive Archimedean component, which is also the image of the reduced norm map from $D^\times(\mathbb{A}_\mathbb{Q})$.
Thus for $g \in \breve D^\times$, write $g=bag'$, $b \in D^\times(\mathbb{Q})$, $a\in \mathbb{A}_\mathcal{K}^\times, g'\in \mathrm{U}(2)(\mathbb{A}_\mathbb{Q})$, define
$$\varphi^D\boxtimes\chi(g)=\varphi_\chi(g)\chi(a).$$
Note that this is well defined since $\mathbb{Q}^\times\cap \mathrm{Nm}(\mathbb{A}_\mathcal{K}^\times/\mathbb{A}_\mathbb{Q}^\times)=\mathrm{Nm}(\mathcal{K}^\times/\mathbb{Q}^\times)$. For $g$ outside $\breve{D}^\times$ we define $\varphi^D\boxtimes\chi(g)=0$. When $\chi$ is clear from the context we simply drop the subscript $\chi$.\\

\begin{lemma}
Let $\pi_\xi$ be the irreducible automorphic representation of $\mathrm{GL}_2/\mathbb{Q}$ associated to a CM character $\xi$ of $\mathcal{K}^\times\backslash\mathbb{A}^\times_\mathcal{K}$. If $\varphi_\xi$ is in an irreducible automorphic representation of $\mathrm{U}(2)$ whose restriction to $\mathrm{SU}(2)$ is in the restriction of the automorphic representation of $D^\times(\mathbb{A}_\mathbb{Q})$ corresponding to $\pi_\xi$ under the Jacquet-Langlands correspondence, then $\varphi^D_\xi\boxtimes\xi$ itself is in $\pi_\xi$.
\end{lemma}
\begin{proof}
Clearly the $\varphi_\xi^D$ and $\pi_\xi$ have the same Hecke eigenvalues at split $v$'s. Note that the set of primes of $\mathcal{K}$ sitting over split primes of $\mathcal{K}/\mathbb{Q}$ has Dirichlet density one. Write $\varphi^D_\xi$ as a sum of forms in irreducible automorphic representations. Then for any such automorphic representation $\pi_i$, the corresponding Galois representation $\rho_{\pi_i}$ satisfies
$$\rho_{\pi_i}|_{G_\mathcal{K}}\simeq \xi\oplus\xi^c.$$
This implies each $\pi_i$ is isomorphic to $\pi_\xi$.
\end{proof}

\noindent We relate the integrals over $[\mathrm{U}(2)]$ to that over a subset of $[\mathbb{Q}^\times\backslash D^\times](\mathbb{A}_\mathbb{Q})$. It is elementary to check that there is a constant $C^D_{\mathrm{U}(2)}$ \index{$C^D_{\mathrm{U}(2)}$} depending only on the groups $D^\times $ and $\mathrm{U}(2)$ such that if $\chi=1$ then:
\begin{equation}\label{555}\int_{[\mathrm{U}(2)]}\varphi_{\mathrm{U}(2)}(g)dg=C_{\mathrm{U}(2)}^D
\int_{D^\times(\mathbb{Q})\mathbb{A}_\mathbb{Q}^\times\backslash \breve{D}^\times(\mathbb{A}_\mathbb{Q})}\varphi^D\boxtimes\chi(g)dg.
\end{equation}
Here, we normalize the Haar measure so that the measure $\mathrm{U}(1)\backslash [\mathrm{U}(2)]=1$ and the measure of $[D^\times]$ modulo center is also $1$.

\subsection{Constructing Auxiliary Families of Theta Functions}\label{theta}
\noindent\underline{Convention}\\
\noindent From now on, we usually do the computations at a generic arithmetic point $\mathsf{z}\in\mathrm{Spec}\Lambda_\mathbf{D}(\bar{\mathbb{Q}}_p)$ \index{$\mathsf{z}$}. We usually write bold symbols for $p$-adic families constructed (e.g. $\mathbf{h}$), and write their specializations using non-bold symbols (e.g. $h_\mathsf{z}$ for specializations of $\mathbf{h}$).

In this section we fix finite order CM characters $\eta$ and $\eta'$ of $\mathcal{K}^\times\backslash\mathbb{A}^\times_\mathcal{K}$. This notation is only used in this subsection so as not to confuse with our use of $\eta$ in previous sections.

Before continuing, we need to introduce some more families of characters. Let $\Gamma_{v_0}\simeq \mathbb{Z}_p$ be the quotient of $\Gamma_\mathcal{K}$ corresponding to the maximal subextension of $\mathcal{K}_\infty$ unramified outside $v_0$. Define $\Gamma_{\bar{v}_0}$ similarly. We need to further enlarge our parameter space. We write $u_{v_0}$ and $u_{\bar{v}_0}$ for topological generators of $\mathrm{Gal}(\mathcal{K}_\infty/\mathcal{K}_{v_0})$ and $\mathrm{Gal}(\mathcal{K}_\infty/\mathcal{K}_{\bar{v}_0})$ respectively. Let $\mathcal{K}^{ur}$ be the maximal (finite) everywhere unramified subextension of $\mathcal{K}_\infty/\mathcal{K}$. Then $u_{v_0}$ and $u_{\bar{v}_0}$ generates $\Gamma'_\mathcal{K}:=\mathrm{Gal}(\mathcal{K}_\infty/\mathcal{K}^{ur})$. We define an abstract group $\Gamma''_\mathcal{K}:=\frac{\mathbb{Z}_p}{p^a}u_{v_0}\oplus\frac{\mathbb{Z}_p}{p^a} u_{\bar{v}_0}$ with $p^a=\sharp\mathrm{Gal}(\mathcal{K}^{ur}/\mathcal{K})$. Then we have
$$\Gamma'_\mathcal{K}\subseteq \Gamma_\mathcal{K}\subseteq \Gamma''_\mathcal{K}$$
with each containment of index $p^a$. We consider the natural projection $\Gamma'_\mathcal{K}\rightarrow \mathbb{Z}_p u_{v_0}$, which extends canonically to a surjection $\Gamma_\mathcal{K}\rightarrow \frac{\mathbb{Z}_p}{p^a}u_{v_0}\subseteq \Gamma''_\mathcal{K}$. Recall we have defined a $\Lambda_\mathbf{D}$-adic character $\boldsymbol{\xi}$. Let $\Lambda''_\mathbf{D}:=\Lambda_\mathbf{D}\otimes_{\mathbb{Z}_p[[\Gamma_\mathcal{K}]]}
\mathbb{Z}_p[[\Gamma''_\mathcal{K}]]$ \index{$\Lambda''_\mathbf{D}, \Gamma'_\mathcal{K},\Gamma''_\mathcal{K}$}. This is an enlarged parameter space.
Then we can define a $\Lambda''_\mathbf{D}$-valued character $\boldsymbol{\xi}_2$ as the composition of $\boldsymbol{\xi}$ with the surjection $\Gamma_\mathcal{K}\rightarrow \frac{\mathbb{Z}_p}{p^a}u_{v_0}\subseteq \Gamma''_\mathcal{K}$ above. It is easy to see that at any arithmetic point $\mathsf{z}$ the specialization of $\boldsymbol{\xi}_2$ is unramified at $v_0$, and its restriction at $\mathcal{O}^\times_{\bar{v}_0}$ is the same as that for $\boldsymbol{\xi}_\mathsf{z}$. Using the same construction, we can define a $\Lambda''_\mathbf{D}$-valued character $\boldsymbol{\eta}$ be such that the specialization to $\phi_0$ ($\phi_0$ is defined in Section \ref{6.4}) is $\eta$ defined before and the specialization to each arithmetic point $\mathsf{z}$ satisfies $(\boldsymbol{\eta})_{\mathsf{z},p}|_{(\mathbb{Z}_p^\times, \mathbb{Z}^\times_p)}=(1, \bar{\xi}_{\mathsf{z},v_0}^\dag
|_{\mathbb{Z}_p^\times})$ (recall the definition for $\xi^\dag_1$ in Section \ref{Sectionn}, the triple there is the $p$-component of the specialization $(f_\mathsf{z},\psi_\mathsf{z},\tau_\mathsf{z})$ here).

Similarly starting with the character $\eta'$ before we can define another family of characters of $\mathcal{K}^\times\backslash \mathbb{A}^\times_\mathcal{K}$ with values in the enlarged $\mathbb{I}$ (taking tensor product of the original $\mathbb{I}$ with some degree $p^a$ extension of $\mathcal{O}_L[[W]]$ and take a reduced irreducible component and normalization) upon appropriate choice of identification of $\mathbb{Z}_p[[W]]$ with $\mathbb{Z}_p[[u_{{v}_0}]]$, such that at any arithmetic point $\mathsf{z}$ the specialization is unramified at ${v}_0$ and is equal to $\chi_{f,\mathsf{z}}\tau_{\mathsf{z},\bar{v}_0}\psi^{-1}_{\mathsf{z},\bar{v}_0}$ when restricting to $\mathcal{O}^\times_{{v}_0}\simeq\mathbb{Z}^\times_p$. Then we can define a $\Lambda''_\mathbf{D}$-adic character $\boldsymbol{\eta}''$ such that its specialization to $\phi_0$ is $\eta'$, and that $(\boldsymbol{\eta}'')_{\mathsf{z},\bar{v}_0}|_{\mathbb{Z}_p^\times}=\chi_{f,\mathsf{z}}
\tau_{\mathsf{z},\bar{v}_0}\psi^{-1}_{\mathsf{z},\bar{v}_0}|_{\mathbb{Z}_p^\times}$, $(\boldsymbol{\eta}'')_{\mathsf{z},v_0}|_{\mathbb{Z}_p^\times}=1$. Moreover there is a character $\boldsymbol{\chi}$ of $\mathcal{K}^\times\backslash \mathbb{A}^\times_\mathcal{K}$ which factors through $\Gamma_\mathcal{K}$ (again we use class field theory), such that $\boldsymbol{\chi}_{\mathsf{z},\bar{v}_0}|_{\mathbb{Z}_p^\times}=\chi^{-1}_{f,\mathsf{z}}\psi_{\mathsf{z},\bar{v}_0}|_{\mathbb{Z}_p^\times}$, $\boldsymbol{\chi}_{\mathsf{z},v_0}|_{\mathbb{Z}_p^\times}=\chi^{-1}_{f,\mathsf{z}}\psi_{\mathsf{z},\bar{v}_0}|_{\mathbb{Z}_p^\times}$.
Define $\boldsymbol{\eta}':=\boldsymbol{\eta}''\cdot\boldsymbol{\chi}$.

Note that we have enlarged our parameter space. At the end of Section 9 we will first prove the main theorems for this enlarged Iwasawa algebra and then go back to prove it for the original one.\\

\noindent\underline{Rallis Inner Product Formula}\\
\begin{diagram}
\mathrm{U}(1,1) (\omega_{\lambda^2})& \mathrm{U}(2) (\omega_\lambda)\times \mathrm{U}(2)(\omega_\lambda)\\
\uTo  & \uTo\\
\mathrm{U}(1)(\omega_{\lambda^2})\times \mathrm{U}(1)(\omega_{\lambda^2})&\mathrm{U}(2)(\omega_{\lambda^2})
\end{diagram}

\noindent Recall we fixed a splitting character $\lambda$ of $\mathcal{K}^\times\backslash\mathbb{A}^\times_\mathcal{K}$ of infinity type $(-\frac{1}{2},\frac{1}{2})$ such that its restriction to $\mathbb{A}^\times_\mathbb{Q}$ is $\chi_{\mathcal{K}/\mathbb{Q}}$. We use the background for dual reductive pair, splitting characters and theta correspondences in \cite[Sections 1,2,3]{HKS} freely. We consider the seesaw pair above. The $\mathrm{U}(2)$ above is for the Hermitian matrix $\begin{pmatrix}\mathfrak{s}&\\&1\end{pmatrix}$ and the $\mathrm{U}(1)$'s are for the skew-Hermitian matrices $\delta$ and $-\delta$. The embedding $\mathrm{U}(1)\times \mathrm{U}(1)\hookrightarrow \mathrm{U}(1,1)$ is given by the $i$ defined in the proof of Lemma \ref{5.18}. The splitting characters used are indicated in the brackets beside the groups. We want to consider the component of theta correspondence such that the first $\mathrm{U}(1)$ on the lower left corner acts by $\lambda^2\eta_{\mathsf{z}}$ and the second $\mathrm{U}(1)$ acts by $\eta_{\mathsf{z}}^{-1}$. We consider a theta function on $\mathrm{U}(2,2)$ by some Schwartz function $\phi$ such that $\phi=\delta_\psi(\phi_3\boxtimes\phi_2)$ for some $\phi_3$ and $\phi_2$ (recall the notion of intertwining operators in Subsection \ref{3.6}. We consider
\begin{equation}\label{Siegel Weil}\int_{[\mathrm{U}(2)]}\int_{[\mathrm{U}(1)]\times [\mathrm{U}(1)]}\theta_\phi(u_1,u_2,g)\lambda^{-2}\eta_{\mathsf{z}}^{-1}(u_1)\eta_{\mathsf{z}}(u_2)\bar\lambda(\det g)du_1du_2dg.
\end{equation}
The splitting characters are consequences of \cite[Lemma A.7]{HKS} and the discussion on ``reflection principle'' right after it. Here the $\bar{\lambda}(\det g)$ shows up due to the splitting $\omega_{\lambda^2}$ on $\mathrm{U}(2)$ since in the Siegel-Weil formula (\cite{Ichi2}), the splitting character on $\mathrm{U}(2)$ is the trivial character. Thus the restriction of the theta kernel to $\mathrm{U}(1,1)\times\mathrm{U}(2)$ is $\lambda(\det g)$ times the restriction to it of the theta kernel appearing in the Siegel-Weil formula of \cite{Ichi2}. So this is exactly the same formula considered in \cite[Section 6]{HLS}. 

On one hand, one can check that this is nothing but the inner product of the theta liftings $\theta_{\phi_3,\lambda}(\lambda\eta_\mathsf{z})$ and $\theta_{\phi_2,\lambda}(\lambda^{-1}\eta^{-1}_\mathsf{z})\cdot(\bar{\lambda}\circ\det)$ (by writing $\theta_{\phi,\lambda}$ we take the splitting character for $\mathrm{U}(1)$ to be trivial and for $\mathrm{U}(2)$ to be $\lambda$. We need to notice the different choices of splitting characters). On the other hand, if we change the order of integration using the Siegel-Weil formula for $\mathrm{U}(1,1)\times \mathrm{U}(2)$ as proved by Ichino (\cite{Ichi2}), this equals:
\begin{equation}\label{Siegel Weil2}
\int_{[\mathrm{U}(1)]\times[\mathrm{U}(1)]}E(f_{\delta_\psi(\phi)},\frac{1}{2},i(u_1,u_2))
\lambda^{-2}\eta_{\mathsf{z}}^{-1}(u_1)\eta_{\mathsf{z}}(u_2)du_1du_2
\end{equation}
Here $i$ is defined right before Lemma \ref{5.19} and $f_{\delta_\psi(\phi)}$ is the Siegel section defined by:
$$f_{\delta_\psi(\phi)}(g):=\omega_{\lambda^2}(j(g))\delta_\psi(\phi)(0), g\in\mathrm{U}(1,1)$$
where $j$ is defined in the proof of Lemma \ref{5.18}. Thus we reduced the Petersson inner product of theta liftings to the pullback formula of the Siegel-Eisenstein series on $\mathrm{U}(1,1)$.\\

\noindent\underline{Functorial Properties of Theta Liftings}\\
\noindent For any Hecke character $\chi$ of $\mathrm{U}(1)$ (in application $\chi(z_\infty)=z_\infty^{\pm1}$ for $z_\infty\in \mathrm{U}(\mathbb{R})$), we describe the $L$-packet of theta correspondence $\theta_\lambda(\chi)$ (possibly zero) of $\chi$ to $\mathrm{U}(2)$ where $\lambda$ is a Hecke character of $\mathbb{A}_\mathcal{K}^\times$ such that $\lambda|_{\mathbb{A}_\mathbb{Q}^\times}=\omega_{\mathcal{K}/\mathbb{Q}}$. We pick a Hecke character $\breve{\chi}$ such that $\breve{\chi}|_{\mathrm{U}(1)(\mathbb{A}_\mathbb{Q})}=\chi^{-1}$. In our application the automorphic representation of $\theta_\chi$ is supercuspidal at all primes where $D$ is compact modulo center. So we have the Jacquet-Langlands correspondence $\pi_{\breve{\chi}}$ on $D^\times$ of the automorphic representation of $\mathrm{GL}_2/\mathbb{Q}$ generated by the CM form $\theta_{\breve{\chi}}$ corresponding to $\breve{\chi}$. We form an automorphic representation in the way we introduced before:
$\pi_{\breve{\chi}}\boxtimes \breve{\chi}\lambda$ of $\mathrm{GU}(2)$. Then by looking at the local $L$-packets (see \cite[Section 7]{HKS}) $\theta_\lambda(\chi)$ is a subspace of the restriction of this representation to $\mathrm{U}(2)$. (This restriction is not necessarily irreducible.) The representations at split primes are irreducible. Therefore, we still have not specified the automorphic representation on $\mathrm{U}(2)$.)\\

\noindent\underline{Constructing Families of Theta Liftings}\\
\noindent Let $v$ be a prime inert or ramified in $\mathcal{K}$. Thanks to the recent work \cite{GT}, we know that the Howe duality conjecture is true for any characteristic. Recall as in Definition \ref{definition4.4}, consider the theta lifting from $\mathrm{U}(1)$ to $\mathrm{U}(2)$ at $v$ (the $\mathrm{U}(2)(\mathbb{Q}_v)$ might be $\mathrm{U}(1,1)(\mathbb{Q}_v)$ or compact). Write $S(X_v,\eta_{\mathsf{z},v}^{-1})$ for the summand of $S(X_v)$ such that $\mathrm{U}(1)$ acts by $\eta_{\mathsf{z},v}^{-1}$. Given a Schwartz function $\phi^v$ on $\otimes_{w\not=v} S(X_w)$ we consider the map $S(X_v,\eta^{-1}_{\mathsf{z},v})\rightarrow \pi_{\theta_{\eta_\mathsf{z}}}$ (the $\pi_{\theta_{\eta_\mathsf{z}}}$ is the automorphic representation of $\mathrm{U}(2)$ by Howe duality corresponding to $\eta^{-1}_{\mathsf{z}}$) by $\iota_v:\phi_v\rightarrow \Theta_{\phi^v\otimes \phi_v}(\eta_{\mathsf{z}}^{-1})$. By the Howe duality conjecture, we know that there is a maximal proper sub-representation $V_v$ of $S(X_v,\eta_{\mathsf{z},v}^{-1})$ such that $S(X_v,\eta_{\mathsf{z},v}^{-1})/V_v$ is irreducible and isomorphic to the local theta correspondence $\pi_{\theta,\mathsf{z},v}$ of $\eta_{\mathsf{z},v}^{-1}$ by the local and global compatibility of theta lifting (see \cite[Theorem 8.5]{Prasad1}). Suppose there is some $\phi_v$ so that $\iota_v(\phi_v)\not=0$ (a finite sum of pure tensors in $\pi_{\theta_{\eta_\mathsf{z}}}$). We consider the representation of $\mathrm{U}(2)(\mathbb{Q}_v)$ on $\pi_{\theta_{\eta_\mathsf{z}}}(\mathrm{U}(2)(\mathbb{Q}_v))(\iota_v(\phi_v))$. This is a sub-representation of a direct sum of finite number of $\pi_{\theta_{\eta_\mathsf{z}},v}$'s. The $\iota_v$ gives a homomorphism of representations of $\mathrm{U}(2)(\mathbb{Q}_v)$ from $S(X_v)$ to $\pi(\mathrm{U}(2)(\mathbb{Q}_v))(\iota(\phi_v))\hookrightarrow \oplus \pi_{\theta_{\eta_\mathsf{z}},v}$. Note that the automorphism group of the representation $\pi_{\theta_{\eta_\mathsf{z}},v}$ consists of scalar multiplications. Thus it is easy to see that the kernel of the above embedding is exactly $V_v$ and we have the following lemma:
\begin{lemma}\label{Lemma 8.3}
Fix the $\phi^v$ as above. Let $v_{\phi_v}$ be the image of $\phi_v$ in $\pi_{\theta_{\eta_\mathsf{z}},v}$ under the Howe duality isomorphism to $S(X_v,\eta_{\mathsf{z},v}^{-1})/V_v$. Then $\iota_v(\phi_v)\in\pi_{\theta_{\eta_\mathsf{z}}}$ can be written as a finite sum of pure tensors of the form $$v_{\phi_v}\otimes(\sum_i\prod_{w\not=v}\phi_{w,i})$$ for $\phi_{w,i}\in\pi_{\theta_{\eta_\mathsf{z}},w}$.
\end{lemma}

We define the weight map $j_1:\Lambda_{2,0}\rightarrow \Lambda''_\mathbf{D}$ is given by
$$(1+T_1)\mapsto 1,\ (1+T_2)\mapsto \boldsymbol{\tau}^{-1}_{\bar{v}_0}\boldsymbol{\psi}_{\bar{v}_0}|_{\mathbb{Z}^\times_p}(1+p)$$
where we write $\boldsymbol{\tau}_2$ for the restriction of $\boldsymbol{\tau}$ to $\mathcal{K}^\times_{\bar{v}_0}\simeq\mathbb{Q}^\times_p$, and similarly to $\boldsymbol{\psi}_2$.

\begin{proposition}\label{CMTheta}
Suppose $\eta$ is such that for each non-split prime $v\in\Sigma$, $\eta|_{\mathrm{U}(1)(\mathbb{Q}_v)}$ equals the $\chi_{\theta,v}$ defined in Section \ref{5.7}; for each split prime $v\nmid p$ in $\Sigma$ with $v=w\bar{w}$, we have $\eta_w$ is unramified and $\eta_{\bar{w}}$ is the $\chi_{\theta,v}$ defined in Section \ref{5.7}. We can construct a family $\boldsymbol{\theta}=\boldsymbol{\theta}^-_{\boldsymbol{\eta},\phi_{2}}\in\mathcal{M}_{\mathrm{ord}}
(K^{(2,0)},\Lambda_{2,0})\otimes_{j_1}\Lambda''_\mathbf{D}$ whose specialization to $\mathsf{z}\in\mathrm{Spec}\Lambda''_\mathbf{D}$ of conductor $p^t$ equals
$$\theta_\mathsf{z}(g):=\sum_{i=1}^{h_\mathcal{K}}\sum_{n\in\mathbb{Z}_p/p^t\mathbb{Z}_p}\eta_{\mathsf{z}}
(\breve{u}_i)^{-1}\frac{\Omega_p\theta_{2,\mathcal{D}_\mathsf{z}}(g\begin{pmatrix}1&0\\n&1\end{pmatrix}_p
\breve{u}_i)}{\Omega_\infty},$$
where the subscript $\phi_2$ stands for the Schwartz functions $\phi_{2,\mathsf{z}}$ we define in the proof below, the  $\theta_{2,\mathbf{D}_\mathsf{z}}$ is the one appearing in Corollary \ref{Corollary 6.29} with the $\mathcal{D}$ being the specialization $\mathbf{D}_\mathsf{z}$ of $\mathbf{D}$. The superscript $-$ is to indicate that the theta function is constructed through pullback under the map $1\times\mathrm{U}(2)\hookrightarrow \mathrm{U}(2,2)$. (Later on we will constructed theta functions via pullback under the map $\mathrm{U}(2)\times 1 \hookrightarrow \mathrm{U}(2,2)$, which we use a superscript $+$ to indicate).
\end{proposition}
\begin{proof}
For an eigenform $\theta$ such constructed we sometimes write $\pi_\theta$ for the automorphic representation of $\mathrm{U}(2)$ of $\theta$.
First, we give the choices for the Schwartz function $\phi$ for the construction using the embedding $1\times \mathrm{U}(2)\hookrightarrow \mathrm{U}(2,2)$.\\

\noindent\underline{Local Computations}\\
In the following we define some Schwartz function. The $\phi_{2,\mathsf{z},v}$'s depend on the arithmetic point $\mathsf{z}$ (in fact only varying at the $p$-adic place), while the $\phi_{3,v}$ are fixed throughout the family and we thus suppress the subscript $\mathsf{z}$.\\
\noindent\underline{Case 0}:\\
\noindent At finite places outside $\Sigma$ we choose the obvious spherical kernel functions.\\

\noindent \underline{Case 1}:\\
\noindent If $v=\infty$, we let $\phi_{v,\mathsf{z}}=\omega_\lambda(\begin{pmatrix}&1_2\\-1_2&\end{pmatrix} g_0)\Phi_\infty$. Recall that $g_0=\mathrm{diag}(\frac{\mathfrak{s}^{\frac{1}{2}}d^{\frac{1}{4}}}{\sqrt{2}},\frac{d^{\frac{1}{4}}}{\sqrt{2}},
(\frac{\mathfrak{s}^{\frac{1}{2}}d^{\frac{1}{4}}}{\sqrt{2}})^{-1},
(\frac{d^{\frac{1}{4}}}{\sqrt{2}})^{-1})$ and $\Phi_\infty=e^{-2\pi\mathrm{Tr}(\langle x,x\rangle_1)}$. Let $\phi_{3,v}=\frac{(\mathfrak{s}d)^{\frac{1}{4}}}{2}e^{-2\pi\sqrt{d}(\mathfrak{s}x_{11}^2+x_{12}^2)}$ and
$\phi_{2,v,\mathsf{z}}=\frac{(\mathfrak{s}d)^{\frac{1}{4}}}{2}e^{-2\pi\sqrt{d}(\mathfrak{s}x_{21}^2+x_{22}^2)}$. By our computation in Section 6.5 we have:
$\delta_\psi(\phi_{3,v}\boxtimes \phi_{2,v,\mathsf{z}})=\phi_{v,\mathsf{z}}$.\\

\noindent \underline{Case 2}: \\
\noindent If $v\in S$ is split and $v\nmid p$ we recall that we have two different polarizations $W=X_v\oplus Y_v=X_v'\oplus Y_v'$ where the first one is globally defined which we use to define theta function and the second is defined using $\mathcal{K}_v\equiv \mathbb{Q}_v\times\mathbb{Q}_v$ which is more convenient for computing the actions of level groups. We have defined intertwining operators $\delta_\psi''$ between $S(X_v)$ and $S(X_v')$ intertwining the corresponding Weil representations. Consider the theta correspondence of $\mathrm{U}(1)$ to $\mathrm{U}(2)$ on $S(X_v)$ and $S(X_v^-)$. We write $X_v\ni x_{3,v}'=(x_{3,v}'',x_{3,v}''')$ and $X_v^{'-}\ni x_{2,v}'=(x_{2,v}'',x_{2,v}''')$. We define
$$\phi_{3,v}'(x_{3,v}'',x_{3,v}''')=\left\{\begin{array}{ll}(\lambda\eta_{\mathsf{z}})^{-1}_v(x_{3,v}'') & x_{3,v}''\in\mathbb{Z}_v^\times,x_{3,v}'''\in\mathbb{Z}_v\\
0 & \mbox{ otherwise. }\end{array}\right.$$
where $\varpi_v^{t_v}$ is the conductor of $\eta_{\mathsf{z},v}$ and
$$\phi_{2,\mathsf{z},v}'(x_{2,v}'',x_{2,v}''')=\left\{\begin{array}{ll}(\lambda\eta_{\mathsf{z}})_v(x_{2,v}'') & x_{2,v}''\in\mathbb{Z}_v^\times,x_{2,v}'''\in\mathbb{Z}_v\\
0 & \mbox{ otherwise. }\end{array}\right.$$
We define $\phi_{v,\mathsf{z}}\in S(\mathbb{W}^d)$ by $\phi_{v,\mathsf{z}}=\delta_\psi'(\phi'_{3,\mathsf{z},v}\boxtimes \phi'_{2,\mathsf{z},v})$ and define $\phi_{2,\mathsf{z},v}=\delta_\psi^{-,''}(\phi_{2,\mathsf{z},v}')$, $\phi_{3,v}=\delta_\psi''(\phi_{3,v}')$. Then if $f_v\in I_v(\lambda^2)$ is the Siegel section corresponding to $\phi_{v,\mathsf{z}}$ in the Rallis inner product formula, we have
$$f(i(1,u_2))=\langle \phi_{3,v},\omega(u_2,1)\phi_{2,\mathsf{z},v}\rangle$$
by the formula for the intertwining operator.
This is zero unless $u_2\in \mathbb{Z}_v^\times$ ($\mathrm{U}(1)(\mathbb{Q}_v)\simeq \mathbb{Q}_v^\times$) and equals $\lambda^2\eta_{\mathsf{z}}(u_2)$ for those $u_2$'s. To sum up, for such $v$ the local integral in the Rallis inner product formula is a non-zero constant $(\frac{q_v-1}{q_v})^2$. We modify the definitions for $\phi_{2,\mathsf{z},v}$ and $\phi_{3,v}$ by multiplying them by an element $c'_v$ in $\mathbb{Z}^\times_p$ to make the resulting $\phi_{v,\mathsf{z}}$ integral. Note that our definition for the Schwartz functions $\phi_{2,\mathsf{z},v}$, $\phi_{3,v}$ and thus $\phi_{v,\mathsf{z}}$ are independent of the choice of $\mathsf{z}$. Let $c_v=(c'_v)^2(\frac{q_v-1}{q_v})^2$. Later when we are moving things $p$-adically, this constant is not going to change.\\

\noindent \underline{Case 3}:\\
\noindent For $v\in S$ ramified or inert such that $\mathrm{U}(2)(\mathbb{Q}_v)$ is not compact. In this case $\mathrm{U}(1)(\mathbb{Q}_v)$ is a compact abelian group. We let $\phi_{2,\mathsf{z},v}$ be the Schwartz function $\phi_{2,v}$ on $S(X_v^-)$ constructed in Section \ref{5.7}, with the Eisenstein datum $\mathcal{D}=\mathbf{D}_\mathsf{z}$. Let $\phi_{3,v}$ be a $p$-integral valued Schwartz function on $S(X_v)$ such that $(\phi_{3,v},\phi_{2,\mathsf{z},v})=\int_{X_v\simeq X_v^-}\phi_{3,v}(x)\phi_{2,\mathsf{z},v}(x)dx\not=0$ and that the action of $\mathrm{U}(1)(\mathbb{Q}_v)$ via the Weil representation is given by a certain character. (It is easy to see that this character is $\lambda_v^2\eta_{\mathsf{z},v}$, thus the action of the center of $\mathrm{U}(2)(\mathbb{Q}_v)$ via $\mathrm{U}(2)\times 1$ is given by $\eta_{\mathsf{z},v}^{-1}$.) We define $\phi_{v,\mathsf{z}}=\delta_\psi(\phi_{3, v}\boxtimes\phi_{2,\mathsf{z},v})\in S(\mathbb{W}^d_v)$. We note that at all arithmetic points $\mathsf{z}$, the character $\eta_{\mathsf{z},v}$ as a representation of $\mathrm{U}(1)(\mathbb{Q}_v)$ since the character is changed by an unramified character and the group $\mathrm{U}(1)(\mathbb{Q}_v)$ is compact. We further multiply the $\phi_{3,v}$ and $\phi_{2,\mathsf{z},v}$ by an element $c'_v$ of $\mathbb{Z}^\times_p$ to make them and the resulting $\phi_{\mathsf{z},v}$ integral.
Let $$c_v=(c'_v)^2\mathrm{Vol}(\mathrm{U}(1)(\mathbb{Q}_v))(\phi_{3,v},\phi_{2,\mathsf{z},v})=\int_{X_v\simeq X_v^-}\phi_{3,v}(x)\phi_{2,\mathsf{z},v}(x)dx$$

The $\phi_{3, v}$ and $\phi_{2,\mathsf{z},v}$, and therefore the $c_v$ and $\phi_{v,\mathsf{z}}$ are fixed throughout the family.\\

\noindent\underline{Case 4}: \\
\noindent For $v$ such that $\mathrm{U}(2)(\mathbb{Q}_v)$ is compact. Note that the local representation $\pi_{\theta_{\eta_\mathsf{z}},v}$ is finite-dimensional with some level group $K_v$. Again let $\phi_{w,\mathsf{z},v}$ be the Schwartz function $\phi_{2,v}$ defined in Section \ref{5.7} with the Eisenstein datum being $\mathcal{D}=\mathbf{D}_\mathsf{z}$. We write $v_1$ for the image of $\phi_{2,\mathsf{z},v}$ in $\pi_{\theta_\mathsf{z},v}$ under the Howe duality. We fix an $\mathrm{U}(2)(\mathbb{Q}_v)/K_v$-invariant measure of $\pi_{\theta_\mathsf{z},v}$ and extend $v_1$ to $\{v_1,\cdots,v_{d_v}\}$ an orthonormal basis of $\pi_{\theta_\mathsf{z},v}$.  Let $(\tilde{v}_1,\cdots)$ be the dual basis. Let $\phi_{3,v}$ be a $p$-integral valued Schwartz function on $S(X_v)$ pairing non-trivially with $\phi_{2,\mathsf{z},v}$, such that the action of $\mathrm{U}(1)(\mathbb{Q}_v)$ via the Weil representation is given by a certain character. (As before this character is $\lambda_v^2\eta_{\mathsf{z},v}$, thus the action of the center of $\mathrm{U}(2)(\mathbb{Q}_v)$ via $\mathrm{U}(2)\times 1$ is given by $\eta_{\mathsf{z},v}^{-1}$.) We require also that the image of $\phi_3$ in the representation of $\mathrm{U}(2)(\mathbb{Q}_v)$ (which is the dual of $\pi_{\theta_\mathsf{z},v}$) is $\tilde{v}_1$.
We define $\phi_{v,\mathsf{z}}=\delta_\psi(\phi_{3,v}\boxtimes\phi_{2,\mathsf{z},v})\in S(\mathbb{W}^d_v)$. We further multiply the $\phi_{3,v}$ and $\phi_{2,\mathsf{z},v}$ by an element $c'_v$ of $\mathbb{Z}^\times_p$ to make them and the resulting $\phi_{v,\mathsf{z}}$ integral. Define $c_v=(c'_v)^2$. \\

\noindent \underline{Case 5}: \\
We write $\eta_{\mathsf{z},p}$ for $\eta_{\mathsf{z}}|_{\mathbb{Q}_p^\times}$. For $v=p$, $W_p=X_p'\oplus Y_p'$, we write elements $$x_p'=(x_{p,1}',x_{p,2}')\in X_p', y_p'=(y_{p,1}',y_{p,2}')\in Y_p'.$$ We define $\phi_{p,\mathsf{z}}(x_p',y_p')=\eta_{\mathsf{z},p}(y_{p,1}')$ if $y_{p,1}'\in \mathbb{Z}_p^\times$ and $x_{p,1}',x_{p,2}',y_{p,2}'\in\mathbb{Z}_p$ and $\phi_{p,\mathsf{z}}(x_p',y_p')=0$ otherwise.
\begin{definition}\label{1234}
For later interpolation we define a $\Lambda''_\mathbf{D}$ valued function
$\boldsymbol{\phi}_p(x_p',y_p')=\boldsymbol{\eta}_{p}(y_{p,1}')$ if $y_{p,1}'\in \mathbb{Z}_p^\times$ and $x_{p,1}',x_{p,2}',y_{p,2}'\in\mathbb{Z}_p$ and $\boldsymbol{\phi}_p(x_p',y_p')=0$ otherwise.
\end{definition}
We also write $x_{3,p}'=(x_{3,p}'',x_{3,p}''')\in X_p',x_{2,p}'=(x_{2,p}'',x_{2,p}''')\in X_p^{'-}$ (note that we use $x_{2,p}'$ to distinguish from $x_{p,2}'$ above). A straightforward computation gives
$$\delta_{\psi,p}^{',-1}(\omega_\lambda(\Upsilon)(\phi_{p,\mathsf{z}}))(x_{3,p}',x_{2,p}')=\frac{\mathfrak{g}
(\eta_{\mathsf{z},p})}{p^t}\eta^{-1}_{\mathsf{z},p}(-x_{2,p}''p^t)$$ if $x_{2,p}''\in p^{-t}\mathbb{Z}_p^\times$ and $x_{3,p}'',x_{3,p}''',x_{2,p}'''\in\mathbb{Z}_p$, and equals $0$ otherwise. We write $\phi_{3,p}'$ to be the characteristic function of $\mathbb{Z}_p^2$ on $X_p'$ and define $\phi_{2,\mathsf{z},p}'(x_{2,p}')=\frac{\mathfrak{g}(\eta_{\mathsf{z},p})}
{p^t}\eta_{\mathsf{z},p}^{-1}(-x_{2,p}''p^t)$ if $x_{2,p}''\in p^{-t}\mathbb{Z}_p^\times$ and $x_{2,p}'''\in\mathbb{Z}_p$, and is $0$ otherwise. We define $\phi_{3,p}$ and $\phi_{2,\mathsf{z},p}$ as the images of $\phi'_{3,p}$ and $\phi'_{2,\mathsf{z},p}$ under $\delta_{\psi,p}^{-1}\circ \delta_{\psi,p}'$. We note that the $\phi_{2,\mathsf{z},p}$ here is not the $\phi_{2,p}$ constructed in Section 6. In fact $$\phi_{2,\mathsf{z},p}=\sum_{n\in\mathbb{Z}_p/p^t\mathbb{Z}_p}\omega_\lambda(\begin{pmatrix}1&0\\n&1\end{pmatrix}_p)
\phi_{2,p}.$$
This can be seen by comparing the $\Phi''_p$ in Section 6 with the $\phi_{p,\mathsf{z}}$ here. This is exactly where the $\sum_{n\in\mathbb{Z}_p/p^t\mathbb{Z}_p}\omega_\lambda(\begin{pmatrix}1&0\\n&1\end{pmatrix}_p)$ appears in the statement of the proposition. \\

\noindent \underline{Global Case}:
\begin{definition}\label{THETA}
Let $\phi_\mathsf{z}=\prod_v\phi_{v,\mathsf{z}}$ with $\phi_v$ defined in Case 0 through Case 5 before. We define a $\Lambda''_\mathbf{D}$-adic formal $q$-expansion $\boldsymbol{\Theta}$ on $\mathrm{U}(2,2)$ interpolating our theta kernel functions on $\mathrm{U}(2,2)$
$$\sum_{x\in\mathcal{K}^2}\prod_{v\nmid p\infty}\phi_{v}(x)\times\boldsymbol{\phi}_p(x)e^{2\pi i\mathrm{tr}({}^t\!\bar{x}Zx)}.$$
(The $\boldsymbol{\phi}_p$ is defined in Definition \ref{1234} and note the $\phi_v$'s are the $\phi_{v,\mathsf{z}}$'s, which are fixed throughout the family for $v\nmid p\infty$, justifying suppressing the subscript $\mathsf{z}$.)
\end{definition}
This is easily seen using Lemma \ref{+} and the computations in \cite[Section 10.3]{SU}.  (We note that the $\phi_{\infty,\mathsf{z}}$ here is $\Phi_{1,\infty}$ right translated by $g_0$. On the other hand, our distinguished point $\mathbf{i}$ is chosen as $\zeta\in X_{2,2}$. Note also that the $\phi_{2,\mathsf{z},v}$ and $\phi_{3,v}$ are independent of $\mathsf{z}$ for $v\nmid p$.) As in the pullback formula for Siegel Eisenstein family under $\mathrm{U}(2)\times \mathrm{U}(2)\hookrightarrow \mathrm{U}(2,2)$, the $\boldsymbol{\Theta}(-\Upsilon)$ also pulls back to $p$-adic analytic family of forms on $\mathrm{U}(2)\times \mathrm{U}(2)$. The following lemma is immediate from our computations from Case 0 to Case 5.
\begin{lemma}
The pullback of the specialization $\boldsymbol{\Theta}_\mathsf{z}$ to $\mathrm{U}(2)\times\mathrm{U}(2)$ is in fact $\theta_{\phi_{3}}\boxtimes \theta_{\phi_{2,\mathsf{z}}}$ for $\phi_{2,\mathsf{z}}=\prod_v\phi_{2,\mathsf{z},v}$ and $\phi_{3}=\prod_v\phi_{3,v}$. (We omit the subscript $\mathsf{z}$ for $\phi_3$ since it is fixed along the family by definition.)
\end{lemma}
It follows easily using the Rallis inner product formula (see Proposition \ref{Proposition8.9} and its proof, note the arithmetic specialization of the Katz $p$-adic $L$-function there is nonzero) and our choices for the Schwartz functions that for some $u_{aux}\in \mathrm{U}(2)(\mathbb{A}_\mathbb{Q})$, $\theta_{\phi_{3}}(u_{aux})\not=0$. Thus $\frac{\theta_{\phi_{3}}(u_{aux})\Omega_p}{\Omega_\infty}\cdot\frac{\theta_{\phi_{2,\mathsf{z}}}\Omega_p}{\Omega_\infty}$ is a $\Lambda''_\mathbf{D}$-adic family of forms on $1\times \mathrm{U}(2)\hookrightarrow \mathrm{U}(2,2)$. Now we take a representative $(\breve{u}_1,...\breve{u}_{h_\mathcal{K}})$ of $\mathrm{U}(1)(\mathbb{Q})\mathrm{U}_1(\mathbb{R})\backslash \mathrm{U}(1)(\mathbb{A}_\mathbb{Q})/\mathrm{U}(1)(\hat{\mathbb{Z}})$ considered as elements of the center of $\mathrm{U}(2)$.
\begin{definition}\label{define theta}
Write $c_\theta:=\frac{\theta_{\phi_{3}}(u_{aux})\Omega_p}{\Omega_\infty}$.  We denote $\boldsymbol{\theta}=\boldsymbol{\theta}^-_{\boldsymbol{\eta},\phi_2}\in\mathcal{M}_{\mathrm{ord}}(K^{(2,0)},\Lambda_{2,0})
\otimes_{j_1}\Lambda''_\mathbf{D}$ \index{$\boldsymbol{\theta}$} for the $\Lambda''_\mathbf{D}$-adic family constructed by
$$\boldsymbol{\theta}(g)=\sum_{i=1}^{h_\mathcal{K}}
\boldsymbol{\eta}(\breve{u}_j)^{-1}
\omega_{\lambda^{-1}}(\breve{u}_j)\frac{\Omega_p^2\boldsymbol{\Theta}(u_{aux},g)}{\Omega^2_\infty}
\bar{\lambda}(\det g).$$

Its specialization to $\mathsf{z}$ is
\begin{equation}\label{n}
\theta_\mathsf{z}(g):=c_\theta\cdot\frac{\Omega_p}{\Omega_\infty}\sum_{i=1}^{h_\mathcal{K}}
\eta_{\mathsf{z}}
(\breve{u}_j)^{-1}
\omega_{\lambda^{-1}}(\breve{u}_j)
(\theta_{\phi_{2,\mathsf{z}}}\cdot\bar{\lambda})(g).
\end{equation}
\index{$\theta_\mathsf{z}$}
\end{definition}
The property required by the proposition follows from comparing our choices of theta kernel function with the (local and global) computations for $\theta_2$ in Section \ref{section 5}.
\end{proof}

\noindent We can do the same thing to construct a $\Lambda''_\mathbf{D}$-adic family of forms on $\mathrm{U}(2)\times 1\hookrightarrow \mathrm{U}(2,2)$. This time we define $\phi$ such that for $v\not=p$ the local components are as before. If $v=p$ recall that $W_p=X_p'\oplus Y_p'$. If $x_p'=(x_{p,1}',x_{p,2}')$ and $y_p'=(y_{p,1}',y_{p,2}')$ we define $\phi_{\mathsf{z},p}(x_p',y_p')=\chi_{\theta,\mathsf{z},p}(x_{1,p}')$ for $x_{1,p}'\in\mathbb{Z}_p^\times$ and $x_{2,p}',y_{1,p}',y_{2,p}'\in\mathbb{Z}_p$ and $\phi_{\mathsf{z},p}(x_p',y_p')=0$ otherwise. Direct computation by plugging in the intertwining operator gives, if we write $x_{1,p}'=(x_{1,p}'',x_{1,p}''')$ and $x_{2,p}'=(x_{2,p}'',x_{2,p}''')$, then
$\delta_{\psi,p}^{',-1}(\omega_\lambda(\Upsilon)\phi_{\mathsf{z},p})(x_{1,p}',x_{2,p}')=\chi_{\theta,\mathsf{z},p}(x_{1,p}'')$ if $x_{1,p}''\in\mathbb{Z}_p^\times$ and $x_{1,p}'',x_{2,p}'',x_{2,p}'''\in\mathbb{Z}_p$, and equals $0$ otherwise. We also get new $\phi_{2,\mathsf{z}}$ and $\phi_{3,\mathsf{z}}$ in this case from the $\phi_v$'s defined here.

As before we move $\chi_{\theta,\mathsf{z}}$ $p$-adically. This time our $\theta_{\phi_{2}}$ is fixed and non-zero at some point $u_{aux}'\in \mathrm{U}(2)(\mathbb{A}_\mathbb{Q})$ and $\theta_{\phi_{3,\mathsf{z}}}$ is moving $p$-adic analytically. We write $c_{\tilde{\theta}_3}=\theta_{\phi_2}(u'_{aux})\cdot\frac{\Omega_p}{\Omega_\infty}$. Thus $\frac{\theta_{\phi_{2}}(u_{aux}')\Omega^2_p}{\Omega_\infty^2}\theta_{\phi_{3,\mathsf{z}}}$ is a $\Lambda_\mathcal{D}$-adic form on $\mathrm{U}(2)\times 1\subset \mathrm{U}(2,2)$.

As before we define $\boldsymbol{\Theta}$ as in Definition \ref{THETA}, and write
$$\tilde{\boldsymbol{\theta}}_3(g)=\boldsymbol{\theta}^+_{\boldsymbol{\eta},\phi_3}(g)=\frac{\Omega_p}{\Omega_\infty}\sum_{i=1}^{h_\mathcal{K}}
\boldsymbol{\eta}_\theta(\breve{u}_j)
\omega_{\lambda^{-1}}(\breve{u}_j)\frac{\boldsymbol{\Theta}(u'_{aux},g)\Omega^2_p}{\Omega^2_\infty}
\bar{\lambda}(\det g).$$
The superscript $+$ stands for the fact that the theta function is constructed via the pullback $\mathrm{U}(2)\times 1\hookrightarrow \mathrm{U}(2,2)$.
Its specialization $\tilde{\theta}_{3,\mathsf{z}}$ to $\mathsf{z}$ satisfies
$$\tilde{\theta}_{3,\mathsf{z}}(g)=c_{\tilde{\theta}_3}\frac{\Omega_p}{\Omega_\infty}\sum_{i=1}^{h_\mathcal{K}}\theta_{\phi_{3,\mathsf{z}}}(g\breve{u}_i)
\eta_{\mathsf{z}}(\breve{u}_i).$$

\begin{definition}
We define forms $\theta^D_\mathsf{z}$ \index{$\theta^D$} and $\tilde{\theta}^D_{3,\mathsf{z}}$ on $D^\times(\mathbb{A}_\mathbb{Q})$ as $\theta^D_\mathsf{z}\boxtimes\eta_\mathsf{z}$ and $\tilde{\theta}^D_{3,\mathsf{z}}\boxtimes\bar{\eta}_3$ respectively and characters $\eta_\mathsf{z}$ and $\bar{\eta}_\mathsf{z}$, as at the beginning of Subsection \ref{D/U}. Sometimes we drop the superscript $D$ when it is clear from the context. The key functorial property of it (and some other automorphic forms constructed) is summarized in Definition \ref{444}.
\end{definition}

\noindent As before we let $\mathcal{L}^{Katz}_{\lambda\eta}$ be the Katz $p$-adic $L$-function interpolating the values
$$\Omega_p^{2}\frac{L(\lambda^2\chi_{\theta,\mathsf{z}}\chi_{\theta,\mathsf{z}}^{-c},1)}
{(\lambda^2\eta_{\mathsf{z}}\eta_
{\mathsf{z}}^{-c})_{2,p}G(\lambda^2\eta_{\mathsf{z}}\eta_{\mathsf{z}}^{-c})\Omega_\infty^2}.$$
We now compute the Petersson inner product of $\theta_\mathsf{z}$ and $\tilde{\theta}_{3,\mathsf{z}}$ at a generic arithmetic point $\mathsf{z}$.
\begin{proposition}\label{Proposition8.9}
We have
$$\mathbf{B}_{K^{(0,2)}}(\boldsymbol{\theta}_{\boldsymbol{\eta},\phi_2}^-,\boldsymbol{\theta}_{\boldsymbol{\eta},\phi_3}^+)=c_\theta \cdot c_{\tilde{\theta}_3}\prod_v c_v\cdot\mathcal{L}^{Katz}_{\lambda\eta}$$
for the $c_v$ in Case 2, Case 3 and Case 4 as before. Note that these are constants fixed throughout the family.
\end{proposition}
\begin{proof}
We do the computation at a generic arithmetic point $\mathsf{z}$, and apply (\ref{Siegel Weil}) and (\ref{Siegel Weil}2). By the doubling method for $\mathrm{U}(1)\times\mathrm{U}(1)\hookrightarrow \mathrm{U}(1,1)$ we are reduced to local pullback formulas.

We first construct the local theta kernels $\phi_{v,\mathsf{z}}$'s. At places outside $p$, the choices are the same as Case 0 through Case 4 above. At the $p$-adic place, we construct the Schwartz function in $S(W_p)=S(X_p'\oplus Y_p')$ first and apply the intertwining operators $\delta_\psi^{-1}$. We write $x_p'=(x_{p,1}',x_{p,2}')\in X_p'$ and $y_p'=(y_{p,1}',y_{p,2}')\in Y_p'$. We define $\phi_{\mathsf{z},p}(x_p,y_p)=\eta_{\mathsf{z},p}(x_{p,1}',y_{p,1}')$ if $x_{p,1}',y_{p,1}'\in\mathbb{Z}_p^\times$ and $x_{p,2}',y_{p,2}'\in\mathbb{Z}_p$, and equals $0$ otherwise. Now we compute the local pullback integrals.
\begin{itemize}
\item For the Archimedean place, this pullback integral is $1$ as in \cite[Section 4.1]{WAN} ($n=1$, and $\kappa=2$ there).
\item For $v\nmid p$, by our choices the corresponding local integrals are non-zero constants which are fixed along the $p$-adic families (see the computations in Case 0 through Case 4 above). The product of such local integrals is $\prod_v c_v$.
\item For $v=p$, as before we have $\phi_{3,\mathsf{z},v}$ and $\phi_{2,\mathsf{z},v}$'s (this time both Schwartz functions depend on the arithmetic point $\mathsf{z}$!), and can compute that $\delta_\psi^{',-1}(\phi_{\mathsf{z},p})=\phi_{3,\mathsf{z},p}\times \phi_{2,\mathsf{z},p}\in S(X_{1,p}')\times S(X_{2,p'})$ where $\phi_{3,\mathsf{z},p}(x_{1,p}'',x_{1,p}''')=\eta_{\mathsf{z},p}(x_{1,p}'')$ if $x_{1,p}''\in\mathbb{Z}_p$, $x_{1,p}'''\in\mathbb{Z}_p$ and equals $0$ otherwise, $$\phi_{2,\mathsf{z},p}(x_{2,p}'',x_{2,p}''')=\frac{\mathfrak{g}(\eta_{\mathsf{z},p})}
{p^t}\eta_{\mathsf{z},p}^{-1}(x_{2,p}'')$$ if $x_{2,p}''\in p^{-t}\mathbb{Z}_p^\times$ and $x_{2,p}'''\in\mathbb{Z}_p$, and equals $0$ otherwise. So these are exactly the theta functions whose inner product we want to compute. Now we compute the Siegel section $f\in I_1(\lambda^2)$ on $\mathrm{U}(1,1)(\mathbb{Q}_p)$. From the form of the Schwartz functions, it is easy to see that the Siegel-Weil section $f_{\phi_{\mathsf{z},p}}=f_{1,p}f_{2,p}$ for $f_{2,p}\in I_1(\lambda)$ to be the spherical function which takes value $1$ on the identity and $f_{1,p}\in I_1(\lambda)$ is the Siegel-Weil section on $\mathrm{U}(1,1)$ of $\mathrm{U}(1,1)\times \mathrm{U}(1)$ for the Schwartz function $\phi_p\in S(\mathcal{K}_p)$ which, with respect to $\mathcal{K}_v\equiv \mathbb{Q}_v\times \mathbb{Q}_v$, is $\phi_p(x_1,x_2)=\eta_{\mathsf{z},p}(x_1\cdot x_2)$ for $x_1,x_2\in\mathbb{Z}_p^\times$, and $\phi_p(x_1,x_2)=0$ otherwise. However, this section is nothing but the Siegel section $f^\dag$ we constructed in \cite[section 4]{WAN} for the one-dimensional unitary group case. Thus, the local integral is easily computed to be:
$$\frac{\mathfrak{g}(\eta_{\mathsf{z},p})}{p^t}.\lambda_p^{-2}((p^{-t},p^t))p^t\lambda^2_p(p^{-t},1)
=\lambda^2_p((1,p^t))\mathfrak{g}(\eta_{\mathsf{z},p}).$$
\end{itemize}
Recall the $c_{\theta}$ and $c_{\tilde{\theta}_3}$ we used in the construction of $\boldsymbol{\theta}$ and $\tilde{\boldsymbol{\theta}}_3$. To sum up, we obtain the proposition from above computations.
\end{proof}

\subsubsection{Constructing $\mathbf{h}$} \label{444}\index{$\mathbf{h}$, $h_\mathsf{z}$}
We repeat the above process to construct another family $\mathbf{h}$ from the family $\boldsymbol{\eta}'$, which will be used in computing the Fourier-Jacobi expansion as well.
We put an assumption
\begin{itemize}
\item Suppose for each non-split bad prime $v$ such that $\mathrm{U}(2)(\mathbb{Q}_v)$ is compact, $\eta|_{\mathrm{U}(1)(\mathbb{Q}_v)}=\eta'|_{\mathrm{U}(1)(\mathbb{Q}_v)}$.
\end{itemize}
Again we compute at a generic arithmetic point $\mathsf{z}\in\mathrm{Spec}\Lambda''_\mathbf{D}(\bar{\mathbb{Q}}_p)$. First we construct forms $h'_\mathsf{z}$ and $\tilde{{h}}'_{3,\mathsf{z}}$ using theta lifting in the same way as we constructed $\theta_\mathsf{z}$ and $\tilde{\theta}_{\mathsf{z},3}$, except with $\eta''_{\mathsf{z}}$ in place of $\eta_{\mathsf{z}}$ and slightly different theta kernels described as follows. (Therefore, in our application, the forms $h'_\mathsf{z}$'s are still CM forms.)
More precisely, we make the Schwartz functions as follows.
\begin{itemize}
\item In case 0 and case 1 (unramified and Archimedean cases) our Schwartz functions $\phi_{v,\mathsf{z}}$, $\phi_{2,\mathsf{z},v}$, $\phi_{3,v}$ are chosen by the same formula;
\item In case 2 (split bad primes) the Schwartz functions $\phi'_{2,\mathsf{z},v}$ and $\phi'_{3,v}$ are chosen as before except replacing $\eta_{\mathsf{z},v}$ by $\eta'_{\mathsf{z},v}$. We define $\phi_{v,\mathsf{z}}=\delta_\psi'(\phi'_{3,v}\boxtimes \phi'_{2,\mathsf{z},v})$ and define $\phi_{2,\mathsf{z},v}=\delta_\psi^{-,''}(\phi_{2,\mathsf{z},v}')$, $\phi_{3,v}=\delta_\psi''(\phi_{3,v}')$. We further multiply the Schwartz functions by a (fixed) nonzero element in $\mathbb{Z}_p$ to make the $\phi_{3,v}$, $\phi_{2,\mathsf{z},v}$, and $\phi_{v,\mathsf{z}}$ are integral.
\item In case 3 (non-split bad primes with $\mathrm{U}(2)(\mathbb{Q}_v)$ not compact). In this case we recall the local theta correspondence from $\mathrm{U}(1)$ to $\mathrm{U}(1,1)$ is always non-vanishing from any character of $\mathrm{U}(1)(\mathbb{Q}_v)$ since it is in the ``stable range'' (see \cite[Propositions 4.3 and 4.5]{Kudla}).  We take some Schwartz function $\phi_{2,\mathsf{z},v}$ and $\phi_{3,v}$ with nonzero pairing so that they are eigenvectors under the action of $\mathrm{U}(1)(\mathbb{Q}_v)$, and the eigenvalue for $\phi_{3,v}$ is $\lambda_v^2\eta'_{\mathsf{z},v}$. (The existence is guaranteed by the previous ``stable range'' result.) These Schwartz functions are fixed throughout the family. Define $\phi_{v,\mathsf{z}}=\delta_\psi(\phi_{3,v}\boxtimes\phi_{2,\mathsf{z},v})\in S(\mathbb{W}_v^d)$ as before. We further multiply the Schwartz functions by a (fixed) nonzero element in $\mathbb{Z}_p$ to make the $\phi_{3,v}$, $\phi_{2,\mathsf{z},v}$, and $\phi_{v,\mathsf{z}}$ are integral.
\item In case 4 (non-split bad primes with $\mathrm{U}(2)(\mathbb{Q}_v)$ being compact) Recall we assumed $\eta|_{\mathbb{U}(1)(\mathbb{Q}_v)}=\eta'|_{\mathbb{U}(1)(\mathbb{Q}_v)}$. As discussed at the beginning of Section \ref{D/U}, it is easy to see that the $\pi_{h'_\mathsf{z},v}$ is either $\pi^\vee_{\theta_{\eta_\mathsf{z}},v}\otimes\eta_\mathsf{z}$ or ${}^\alpha\!\pi^\vee_{\theta_{\eta_\mathsf{z}},v}\otimes\eta_\mathsf{z}$. Under this identification we choose the local Schwartz function $\phi_{2,\mathsf{z},v}$ at $v$ so that the image in $\pi_{h'_\mathsf{z},v}$ is $\tilde{v}_1$ or $\iota_\alpha(\tilde{v}_1)$ depending on whether $\pi^\vee_{\theta_{\eta_\mathsf{z}},v}\otimes\eta_\mathsf{z}$ or ${}^\alpha\!\pi^\vee_{\theta_{\eta_\mathsf{z}},v}\otimes\eta_\mathsf{z}$ (notations as in Section \ref{D/U}), and the Schwartz function $\phi_{3,v}$ whose image is $v_1$ or $\iota_\alpha(v_1)$. Define $\phi_{v,\mathsf{z}}=\delta_\psi(\phi_{3,v}\boxtimes\phi_{2,\mathsf{z},v})\in S(\mathbb{W}_v^d)$ as before. We further multiply the Schwartz functions by a (fixed) nonzero element in $\mathbb{Z}_p$ to make the $\phi_{3,v}$, $\phi_{2,\mathsf{z},v}$, and $\phi_{v,\mathsf{z}}$ are integral.
\item In case 5 ($p$-adic places) we choose $\phi_{p,\mathsf{z}}$, $\phi_{2,\mathsf{z},v}$ and $\phi_{3,v}$ as in case 5 before Definition \ref{1234} except replacing $\eta_{\mathsf{z},p}$ by $\eta''_{\mathsf{z},p}$.
\end{itemize}
Let $\mathbf{H}'$ be the family of theta functions on $\mathrm{U}(2,2)$ defined as the $\boldsymbol{\Theta}$ in Definition \ref{THETA} as
$$\sum_{x\in\mathcal{K}^2}\prod_{v\nmid p\infty}\phi_{v}(x)\times\boldsymbol{\phi}_p(x)e^{2\pi i\mathrm{tr}({}^t\!\bar{x}Zx)}$$
but replacing the local Schwartz functions there by the ones defined here.
As in Definition \ref{define theta}, we denote $\mathbf{h}'$ for the $\Lambda''_\mathbf{D}$-adic family constructed by
$$\mathbf{h}'(g)=\boldsymbol{\theta}^-_{\boldsymbol{\eta}'',\phi_2}=\sum_{i=1}^{h_\mathcal{K}}
\boldsymbol{\eta}''(\breve{u}_j)^{-1}
\omega_{\lambda^{-1}}(\breve{u}_j)\frac{\mathbf{H}'(u_{aux},g)\Omega^2_p}{\Omega^2_\infty}
\bar{\lambda}(\det g).$$
Note that its central character is $\boldsymbol{\eta}''$. Write $c_{h'}=\theta_{\phi_3}(u'_{aux})\cdot\frac{\Omega_p}{\Omega_\infty}$. Clearly it is interpolating theta functions
\begin{equation}
h'_\mathsf{z}:=c_{h'}\sum_{i=1}^{h_\mathcal{K}}
\eta''_{\mathsf{z}}
(\breve{u}_j)^{-1}
\omega_{\lambda^{-1}}(\breve{u}_j)
(\theta_{\prod_v\phi_{2,\mathsf{z},v}}\cdot\bar{\lambda}).
\end{equation}

Again we define $h^D_\mathsf{z}$ \index{$h^D_\mathsf{z}$} on $D^\times(\mathbb{A}_\mathbb{Q})$ as ${h'}^D_\mathsf{z}\boxtimes\eta'_{\mathsf{z}}$ (note the character is not $\eta''_{\mathsf{z}}$) using the procedure at the beginning of Section \ref{D/U}. Clearly we can also form the corresponding family which we denote as $\mathbf{h}^D=\mathbf{h}^{'D}\boxtimes\boldsymbol{\eta}'$.

The automorphic representation for $h^D_\mathsf{z}$ ($\theta^D_\mathsf{z}$) is the Jacquet-Langlands correspondence of the CM form associated to $\lambda\eta'_{\mathsf{z}}$ ($\lambda\eta_{\mathsf{z}}$ respectively). We define $h_\mathsf{z}$ (or $\mathbf{h}$) on $\mathrm{GU}(2)(\mathbb{A}_\mathbb{Q})$ using $h^D_\mathsf{z}$ (or $\mathbf{h}^D$) and the character $\eta_{\mathsf{z}}^{-1}\psi_{\mathsf{z}}$ (not $\eta'_{\mathsf{z}}$! This is crucial for our $p$-adic analysis of Fourier-Jacobi coefficients) (or $\boldsymbol{\eta}^{-1}\boldsymbol{\psi}$) as
$$h_\mathsf{z}(ag)=\eta_{\mathsf{z}}^{-1}\psi_{\mathsf{z}}(a)h^D_\mathsf{z}(g),$$
$$\mathbf{h}(ag)=\boldsymbol{\eta}^{-1}\boldsymbol{\psi}(a)\mathbf{h}^D(g)$$
for $a\in\mathbb{A}^\times_\mathcal{K}$ and $g\in D^\times(\mathbb{A}_\mathbb{Q})$.
The $h_{\mathsf{z}}$'s are certainly interpolated by the family $\mathbf{h}$. We write $\pi_{h,\mathsf{z}}$ for the corresponding automorphic representation. Similarly let $\tilde{h}^D_{3,\mathsf{z}}$ be the form ${\tilde{h}}^{'D}_{3,\mathsf{z}}\boxtimes{\eta'}^{-1}_{h,\mathsf{z}}$, and let $\tilde{h}_{3,\mathsf{z}}$ be the form on $\mathrm{GU}(2)$ constructed from $\tilde{h}^D_{3,\mathsf{z}}$ using the character $\eta_{\mathsf{z}}\psi^{-1}_{\mathsf{z}}$. Let $\tilde{\mathbf{h}}_3$ be the corresponding family.

Now we give the map of the weight spaces for $\mathbf{h}$ and the Fourier-Jacobi coefficients of the $\mathbf{E}_{\mathbf{D},Kling}$. In fact the nebentypus of the $\mathbf{E}_{\mathbf{D},Kling,\mathsf{z}}$ is given by $\mathrm{diag}(\psi^{-1}_{2,\mathsf{z}},\chi_{f_{\mathsf{z}}}\psi^{-1}_{2,\mathsf{z}},\tau^{-1}_{2,\mathsf{z}},\tau_{1,\mathsf{z}})$. Therefore the nebentypus of $el_{\theta^\star}\mathrm{FJ}_\beta(\mathbf{E}_{\mathbf{D},Kling,\mathsf{z}})$ as a form on $\mathrm{U}(2)$ is given by $\chi_{f_{\mathsf{z}}}\psi^{-1}_{2,\mathsf{z}},\tau^{-1}_{2,\mathsf{z}}$. So it is interpolated by a family in $\mathcal{M}_{ord}(K^{(2,0)},\Lambda_{2,0})\otimes_{j_2}\Lambda''_\mathbf{D}$ with the weight map $j_2: \Lambda_{2,0}\rightarrow \Lambda''_\mathbf{D}$ given by
\begin{equation}\label{wtmap}
(1+T_1)\mapsto \boldsymbol{\chi}_{\mathbf{f},p}\boldsymbol{\psi}^{-1}_2|_{\mathbb{Z}^\times_p}(1+p), (1+T_2)\mapsto\boldsymbol{\tau}^{-1}_2(1+p).
\end{equation}
Here $\boldsymbol{\chi}_{\mathbf{f},p}$ is the $p$-part of the central character of $\mathbf{f}$, and we write $\boldsymbol{\psi}_2$ for the restriction of $\boldsymbol{\psi}$ to $\mathcal{K}^\times_{\bar{v}_0}\simeq\mathbb{Q}^\times_p$, and similarly for $\boldsymbol{\tau}_2$. It is also straightforward to check that $\mathbf{h}\in \breve{\mathcal{M}}_{ord}(K^{(2,0)},\Lambda_{2,0})\otimes_{j_2}\Lambda''_\mathbf{D}$ for the same $j_2$.\\

\noindent\underline{Convention}:\\
\noindent Oftentimes, when we constructed $\theta_\mathsf{z}\in\pi_{\theta_\mathsf{z}}$ with central character $\chi_{\theta,\mathsf{z}}$ then by $\tilde{\theta}_\mathsf{z}\in\pi_{\tilde{\theta}_\mathsf{z}}$ we mean $\theta_\mathsf{z}\cdot(\chi_{\theta,\mathsf{z}}^{-1}\circ\det)$. If we have constructed $\tilde{\theta}_{3,\mathsf{z}}$ with central character $\chi^{-1}_{\theta,\mathsf{z}}$, then by $\theta_{3,\mathsf{z}}$ we mean $\tilde{\theta}_{3,\mathsf{z}}\cdot(\chi_{\theta,\mathsf{z}}\circ\det)$. We use the same conventions for $h_\mathsf{z}$'s as well. \\

We have the following immediate corollary:
\begin{corollary}
The $\theta_\mathsf{z}$, $\tilde{\theta}_{3,\mathsf{z}}$, $\theta^D_\mathsf{z}$, $\tilde{\theta}_{3,\mathsf{z}}^D$, $h_\mathsf{z}$, $\tilde{h}_{3,\mathsf{z}}$, $h^D_\mathsf{z}$, $\tilde{h}_{3,\mathsf{z}}^D$ constructed before are pure tensors in the corresponding automorphic representations.
\end{corollary}
\begin{proof}
This follows immediately from Lemma \ref{Lemma 8.3} and the constructions above.
\end{proof}
\subsection{Choosing Some Characters}\label{section 7.2}
In this section we make choices for some Hecke characters for the $\eta$ and $\eta'$ in the previous section. These are important in our study for Fourier-Jacobi coefficients for Klingen Eisenstein series later on.

We first give a result of Pin-Chi Hung \cite[Theorem C]{Hsi13}. Let $\chi$ be a finite order Hecke character of $\mathcal{K}^\times\backslash\mathbb{A}_\mathcal{K}^\times$ of conductor $M\mathcal{O}_\mathcal{K}$ for some $M>0$. Let $f\in S_k(\Gamma_0(N))$ be an elliptic cusp form of even weight $k$, level $\Gamma_0(N)$ with $q$-expansion
$$f(q)=\sum_{n\geq 0}a_n(f)q^n.$$
We decompose $N=N^+N^-$, where $N^+$ is a product of primes split in $\mathcal{K}$ and $N^-$ is a product of primes ramified or inert in $\mathcal{K}$. Suppose $N^-$ is square-free and $N^-=N_f^-N_\chi^-$ where $N_f^-$ is a product of an odd number of primes co-prime to $M$ and $N_{\chi}^-$ is a divisor of $M$. Let $\ell$ be a rational prime split in $\mathcal{K}$. Let $\mathcal{K}_\ell^-$ be the unique abelian anticyclotomic $\mathbb{Z}_\ell$-extension of $\mathcal{K}$ and $\Gamma^-$ be the Galois group $\mathrm{Gal}(\mathcal{K}_\ell^-/\mathcal{K})$.
\begin{theorem}\label{7.1}
Suppose $\ell^2\nmid N$. Let $p$ be a rational prime such that
\begin{itemize}
\item $p\nmid \ell ND_\mathcal{K}$ and $p\geq k-2$,
\item for every non-split $q|M$, $q+1$ is not divisible by $p$,
\item for every $q|N_f^-$ ramified in $\mathcal{K}$, $a_q(f)=\chi(\mathfrak{q})(=\pm1)$, where $q=\mathfrak{q}^2$,
\item the residual Galois representation $\bar{\rho}_{f,\lambda}|_{\mathrm{Gal}(\bar{\mathbb{Q}}/\mathcal{K})}$ is absolutely irreducible.
\end{itemize}
Then there is a finite extension $L/\mathbb{Q}_p$ with integer ring $\mathcal{O}_L$ and uniformizer $\lambda$. We have for all but finitely many characters $\nu:\Gamma^-\rightarrow \mu_{\ell^{\infty}}$, we have
$$\frac{L(f/\mathcal{K},\chi\nu,\frac{k}{2})}{\Omega_{f, N^-}}\not\equiv 0 (\mathrm{mod}\lambda).$$
Here the $\Omega_{f,N^-}$ is a period factor defined in \emph{loc. cit}.
\end{theorem}

\noindent Now we choose the characters needed. From now on we fix once for all a split prime $\ell$ outside $\Sigma$ and write a new ``$\Sigma$'' for $\Sigma\cup \{\ell\}$. We choose $\chi_\theta$ \index{$\chi_\theta$} a Hecke character of $\mathcal{K}^\times\backslash\mathbb{A}_\mathcal{K}^\times$ as follows: $\chi_{\theta,\infty}$ is trivial. At $p$ we require that $\chi$ be unramified. For $v\in\Sigma$ non-split in $\mathcal{K}/\mathbb{Q}$, then we let $\chi_{\theta,v}|_{\mathrm{U}(1)}$ to be the character chosen in Section 6. For split $v\in\Sigma, v\nmid p,\ell$, $v=w\bar{w}$, suppose $\mathrm{cond}(\pi_v)=(\varpi_v^{t_{1,v}})$, we require that $\chi_{\theta,w}$ be unramified and $\mathrm{cond}(\chi_{\theta,\bar{w}})=(\varpi_v^{t_{2,v}})$ for $t_{2,v}\geq 2t_{1,v}+2$. We choose a character \index{$\chi_{\mathrm{aux}}$} $\chi_{\mathrm{aux}}$ of $\mathcal{K}^\times\backslash \mathbb{A}_\mathcal{K}^\times$ as follows: $\chi_{\mathrm{aux}}|_{\mathbb{A}_\mathbb{Q}^\times}=1$, it is trivial at $\infty$, and is only ramified at primes in $\Sigma$ not dividing $p$ such that $\mathrm{U}(2)(\mathbb{Q}_v)$ is not compact. For split such $v$ we require that
$$\mathrm{cond}(\chi_{\mathrm{aux}})=\left\{\begin{array}{ll}(\varpi_v^{t_{2,v}-t_{1,v}})& \mbox{ if }t_{1,v}\not=0 \\(\varpi_v^{t_{2,v}-1})& \mbox{ if } t_{1,v}=0\end{array}\right.$$
At non-split such primes we require that $$\mathrm{cond}(\chi_{aux,v})>\mathrm{cond}(\pi_v), \mathrm{cond}_v(\lambda^2\chi_\theta^{-c}\chi_\theta\chi_{\mathrm{aux}})>\mathrm{cond}(\pi_v)$$ and $\chi_{aux,v}|_{\mathbb{Q}_v^\times}$ has a smaller conductor than $\chi_{aux,v}$ (here $>$ means the conductor of the former is of higher power of the uniformizer than the latter). Also for each prime $q$ such that $\mathrm{U}(2)(\mathbb{Q}_q)$ is compact and $q$ is ramified as $w^2$ in $\mathcal{K}$, suppose $\pi_q\simeq \mathrm{Steinberg}\otimes \chi_{q,1}$ for some unramified quadratic character $\chi_{q,1}$ we require that $\chi_{q,1}(q)=\chi_{\mathrm{aux}}(\varpi_w)$ (these are used in the next paragraph to make sure that the special $L$-values are of the correct local signs when applying Theorem \ref{7.1}). To ensure the existence, we may need to enlarge the $\Sigma$ by including one prime $\ell'$ which is prime to $N$ and inert in $\mathcal{K}$. 
Let $\chi_h=\chi_\theta^{-c}\chi_{\mathrm{aux}}$ \index{$\chi_h$}.\\

\noindent We further require that $\frac{L(\pi_f,\lambda^2\chi_\theta\chi_h,\frac{1}{2})}{\pi^3\Omega_\infty^4}
\mathrm{Eul}_p(\pi_f,\lambda^2\chi_\theta\chi_h,\frac{1}{2})$, $\frac{\Gamma(\kappa-1)L(\chi_{\mathrm{aux}}\tau^{-c},\frac{\kappa-2}{2})}{\pi^{\kappa-1}\Omega_\infty^{\kappa}}
\mathrm{Eul}_p(\chi_{\mathrm{aux}}\tau^{-c},\frac{\kappa-2}{2})$ and $\frac{\Gamma(\kappa-2)L(\lambda^2\chi_\theta^{-c}\chi_\theta\chi_{\mathrm{aux}}\tau^{-1},\frac{\kappa-2}{2})}
{\pi^{\kappa-2}\Omega_\infty^{\kappa-2}}\mathrm{Eul}_p(\lambda^2\chi_\theta^{-c}\chi_\theta\chi_{\mathrm{aux}}\tau^{-1},\frac{\kappa-2}{2})$ are $p$-adic units where the $\mathrm{Eul}_p$ are the local Euler factors for the corresponding $p$-adic $L$-functions at $p$ when everything is unramified at $p$ (we refer to \cite[Equation (0.2)]{Hsi12}, \cite[Equation (4.16)]{Hsi14} for their precise definitions). The first uses \cite{Hsi12}. Our assumptions above on conductors imply that at all non-split primes the local root numbers in Theorem A of \emph{loc.cit.} are all $+1$ (this uses \cite[Proposition 3.8]{JL}, as also mentioned in \cite[Introduction]{Br}). Then we take a split prime $\ell\nmid Np$ and apply that theorem to see that there exists a twist by anticyclotomic character of $\ell$-power conductor which satisfies the requirements. The second and third uses \cite{Hsi11} (we are in the residually non self-dual case there) and again we can achieve the requirements by twisting by an appropriate anti-cyclotomic character of conductor powers of $\ell$. Further, we assume that $1-a_p(f)^{-1}\chi_{\theta,p,2}\chi_{h,p,1}(p)$, $1-a_p(f)\chi_{\theta, p,1}\chi_{h,p,2}(p)^{-1}$ and $1-\lambda_{p,2}^2\chi_{h,p,2}\chi_{\theta,p,2}\tau^{-1}_{p,2}(p)p^{-\frac{\kappa-2}{2}}$ are $p$-adic units. At each prime $v$ of $\mathcal{K}$ above a prime where $\mathrm{U}(2)(\mathbb{Q}_v)$ is compact, we require that $1-\chi_{\mathrm{aux}}\tau^{-c}(q_v)q_v^{-\frac{\kappa-2}{2}}$ be a $p$-adic unit. We also require that $\frac{L(\pi_f,\chi_\theta^c\chi_h,\frac{1}{2})}{\pi^2 \Omega_f^+\Omega_f^-}$ is non-zero (do not need to be non-zero modulo $p$!) using the Theorem \ref{7.1} recalled above (by choosing a different ``$p'$'' and prove non-vanishing the new $p'$. Note that the mod $p'$ residual representation $\bar{\rho}_f|_{G_\mathcal{K}}$ is absolutely irreducible for all but finitely many primes $p'$).
\begin{definition}
Now we take the $\eta$ and $\eta'$ in the previous section to be the $\chi_\theta$ and $\chi_h$ above (indicating they are used for forms $\theta_\mathsf{z}$ and $h_\mathsf{z}$'s. Clearly they satisfy all the requirements there.).
\end{definition}

\noindent Now we define a character $\vartheta$ of $\mathbb{Q}^\times\backslash \mathbb{A}_\mathbb{Q}^\times$ (the reason for doing so is just a cheap way to use the new form theory at split primes to pick different vectors inside an automorphic representation of $\mathrm{U}(2)$). We require that these be ramified
only at split primes in $\Sigma\backslash\{p\}$. At such $v=w\bar{w}$ and require that $\vartheta|_{\mathbb{Z}_v^\times}=\chi_h|_{\mathcal{O}_{\mathcal{K},w}^\times}$. These uniquely determine the character $\vartheta$.

\subsection{Triple Product Formula}\label{Section 8.4}

\noindent \underline{Background for Ichino's formula}\\
\noindent Let $\pi_1,\pi_2,\pi_3$ be three irreducible cuspidal automorphic representations for $\mathrm{GL}_2/\mathbb{Q}$ such that the product of their central characters is trivial and the archimedean components are holomorphic discrete series of weight two. Let $\pi_i^D$ be the Jacquet-Langlands correspondence of them to $D^\times$ (assume they do exist). Let $\phi_i\in\pi_i^D$ and $\tilde{\phi}\in\tilde{\pi}_i^D$. Write $\Pi=\prod_{i=1}^3\pi_i$, $\phi=\prod_i\phi_i\in\Pi$, $\tilde\phi=\prod_i\tilde\phi_i\in\tilde{\Pi}$, and $r$ the natural eight-dimensional representation of $\mathrm{GL}_2\times \mathrm{GL}_2\times \mathrm{GL}_2$. We write $$I(\phi\otimes\tilde\phi)=(\int_{[D]} \phi_1(g)\phi_2(g)\phi_3(g)dg)(\int_{[D]} \tilde\phi_1(g)\tilde\phi_2(g)\tilde\phi_3(g)dg).$$ Now look at the local picture. Suppose $\phi_i=\otimes_v\phi_{i,v}$ and $\tilde\phi_i=\otimes_v\tilde\phi_{i,v}$. We fix $\langle,\rangle$ a $D_v^\times$ invariant pairing between $\pi^D_i$ and $\tilde\pi^D_i$. Let $\zeta$ be the Riemann zeta-function. Define:
$$I_v(\phi_v\otimes \tilde{\phi}_v)=\zeta_v(2)^{-2}\frac{L_v(1,\Pi_v, Ad)}{L_v(1/2,\Pi_v,r)}.\int_{\mathbb{Q}_v^\times\backslash D^\times(\mathbb{Q}_v)}\prod_i\langle\pi_v^D\phi_v(x_v),\tilde\phi_v\rangle d^\times x_v.$$
Note that this depends on the choice of the pairing.\\

\noindent Let $\Sigma$ be a finite set of primes including all bad primes, then we have the following formula of Ichino \cite{Ichi1}:
$$\frac{I(\phi\otimes \tilde\phi)}{\prod_i\langle\phi_i,\tilde\phi_i\rangle}=\frac{C}{8}\zeta^2(2)\frac{L^\Sigma(\frac{1}{2},\Pi,r)}
{L^\Sigma(1,\Pi,Ad)}\prod_{v\in\Sigma}
\frac{I_v(\phi_v\otimes\tilde{\phi}_v)}{\langle\phi_v,\tilde{\phi}_v\rangle}$$
where $C$ is the Tamagawa number for $D^\times$.
This does not depend on the choice of the pairing.\\

\noindent In application, our $\langle\phi,\tilde{\phi}\rangle$ is usually $0$, thus we need a slight variant of the above formula. Suppose we have elements $g'_i=\prod_v g'_{i,v}$ such that $\langle\phi_i,\pi(g'_i)\tilde{\phi}_i\rangle\not=0$ for $i=1,2,3$, where $g'_{i,v}$ are elements in the group algebra $\bar{\mathbb{Q}}_p[D^\times(\mathbb{Q}_v)]$. Then:
$$\frac{I(\phi\otimes \tilde\phi)}{\prod_i\langle\phi_i,\pi(g'_i)\tilde\phi_i\rangle}=\frac{C}{8}\zeta^2(2)\frac{L^\Sigma(\frac{1}{2},\Pi,r)}
{L^\Sigma(1,\Pi,Ad)}\prod_{v\in\Sigma}
\frac{I_v(\phi_v\otimes\tilde{\phi}_v)}{\langle\phi_v,\pi(g'_v)\tilde{\phi}_v\rangle}$$
with $g_v=\prod_vg_{i,v}$.\\

\noindent \underline{Local Triple Product Computations}\\
We remark that in \cite[Sections 5, 6]{Hsi17} the local test vectors and triple product integrals are worked out in full generality. However here for the special cases needed in this paper we include the computations for convenience of the reader.\\
\noindent\underline{Split Case Principal Series}\\
\noindent Suppose $v$ is a split prime of $\mathcal{K}/\mathbb{Q}$ with $q_v$ being the cardinality of its residue field. We assume $\pi_{1,v}$ and $\pi_{2,v}$ are principal series representation and $\pi_{3,v}$ is either principal series representation or special representation with square-free conductor. For $K=\mathrm{GL}_2(\mathbb{Z}_v)$ the maximal compact subgroup of $\mathrm{GL}_2(\mathbb{Q}_v)$, we use the realizations of induced representations as functions on $K$:
$$\mathrm{Ind}_{B(\mathbb{Q}_v)}^{\mathrm{GL}_2(\mathbb{Q}_v)}(\chi_{1,v},\chi_{2,v})=\{v:K\rightarrow \mathbb{C},v(qk)=\chi(q)v(K),q\in B(\mathbb{Q}_v)\cap K\}$$
where $\chi(q)=\chi_{1,v}(a)\chi_{2,v}(d)\delta_B(q)$ for $q=\begin{pmatrix}a&b\\&d\end{pmatrix}$.
We realize the inner products as $$\langle v_1,v_2\rangle=\int_K v_1(k)v_2(k)dk$$ for $v_1\in \mathrm{Ind}_B^{\mathrm{GL}_2}(\chi_{1,v},\chi_{2,v})$, $v_2\in \mathrm{Ind}_B^{\mathrm{GL}_2}(\chi_{1,v}^{-1},\chi_{2,v}^{-1})$. For a positive integer $t$, let $K_t\subset K$ consist of matrices in $B(\mathbb{Z}_v)$ modulo $\varpi_v^t$. For $f\in\pi(\chi_1,\chi_2)$, $\tilde{f}\in\pi(\chi_1^{-1},\chi_2^{-1})$, we define the matrix coefficient $\Phi_{f,\tilde{f}}(g)=\langle\pi(g)f,\tilde{f}\rangle$. Let $\sigma_n=\begin{pmatrix}\varpi_v^n&\\&1\end{pmatrix}$.
\begin{lemma}\label{7.2}
Suppose $t\geq 1$, $\mathrm{cond}(\chi_1\chi_2^{-1})=(\varpi_v^t)$. Let $w=\begin{pmatrix}&1\\1&\end{pmatrix}$ in this lemma. If
$$f_{\chi}(k_v)=\left\{\begin{array}{ll}\chi_1(a)\chi_2(d), & k_v\in K_t\\
0, & \mbox{ otherwise}\end{array}\right.$$
and $\tilde{f}_{{\chi}^{-1}}$ is defined similar to $f_\chi$ but with $\chi$ replaced by $\chi^{-1}$.
Then $\Phi_{f_\chi,\tilde{f}_{{\chi}^{-1}}}(g)=0$ on $$\cup_nK_1w\sigma_n K_1\cup_nK_1\sigma_nw K_1.$$ On $\cup_nK_1\sigma_n K_1\cup_nK_1w\sigma_n\omega K_1$, it is supported in $$\cup_n\begin{pmatrix}1&\\&\varpi^n_v\end{pmatrix}\begin{pmatrix}1&\\ \varpi_v^{t-n}\mathbb{Z}_v&1\end{pmatrix}K_t\cup_n
\begin{pmatrix}\varpi^n_v&\\&1\end{pmatrix}\begin{pmatrix}1&\varpi_v^{-n}\mathbb{Z}_v\\&1\end{pmatrix}K_t.$$ The corresponding values at $\begin{pmatrix}1&\\&\varpi^n_v\end{pmatrix}$ and $\begin{pmatrix}\varpi^n_v&\\&1\end{pmatrix}$ are
$\mathrm{Vol}(K_t)\alpha_2^nq^{-\frac{n}{2}}_v$ and $\mathrm{Vol}(K_t)\alpha_1^nq^{-\frac{n}{2}}_v$ respectively, where $\alpha_i=\chi_i(\varpi_v)$ for $i=1,2$.
\end{lemma}
\begin{proof}
It is easy to check by considering the supports of $f_\chi$ and $\tilde f_\chi$ that $\Phi_{f_\chi,\tilde{f}_\chi}(g)=0$ on $\coprod_{n\geq 0}K_1\sigma_n wK_1$. ($K_1\sigma_n wK_1$ does not intersect $\mathrm{supp}f_\chi$.)\\

\noindent Now suppose $g\in K_1w\sigma_n K_1$ for $n\geq 1$, without loss of generality we assume $$g=\begin{pmatrix}1&\\c&1\end{pmatrix}w\sigma_n \begin{pmatrix}1&\\b&1\end{pmatrix}=g=\begin{pmatrix}1&\\&\varpi_v\end{pmatrix}\begin{pmatrix}1&\\\frac{c}{\varpi_v}&1\end{pmatrix}
\begin{pmatrix}1&b\\&1\end{pmatrix}w$$ for $\varpi_v|b$, $\varpi_v|c$. Plugging in the formula for matrix coefficients, write $K_t\ni g'=
\begin{pmatrix}a'&b'\\&d'\end{pmatrix}\begin{pmatrix}1&\\c'&1\end{pmatrix}$ for $a',d'\in\mathbb{Z}_v^\times$, $b'\in\mathbb{Z}_v,c'\in\varpi_v^t\mathbb{Z}_v$.
\begin{eqnarray*}
&g'g\\
=&\begin{pmatrix}a'&b'\\&d'\end{pmatrix}\begin{pmatrix}1&\\c+c'&1\end{pmatrix}w\sigma_n\begin{pmatrix}1&\\b&1\end{pmatrix}\\
=&\begin{pmatrix}a'&b'\\&d'\end{pmatrix}\begin{pmatrix}1&\\&\varpi_v^n\end{pmatrix}\begin{pmatrix}1&\\ \frac{c'+c}{\varpi_v^n}&1\end{pmatrix}
\begin{pmatrix}1&b\\&1\end{pmatrix}w\\
=&\begin{pmatrix}a'&b'\\&d'\end{pmatrix}\begin{pmatrix}1&\\&\varpi_v^n\end{pmatrix}\begin{pmatrix}1&\frac{1}{c''}\\&1\end{pmatrix}
\begin{pmatrix}-\frac{1}{c''}&\\&c''\end{pmatrix}\begin{pmatrix}1&\\\frac{bc''+1}{c''}&1\end{pmatrix}
\end{eqnarray*}
Here we write $c''=\frac{c'+c}{\varpi_v^n}$. We need to fix $g$ and do the integration for $g'$. The first observation is that we only need to
consider integration with respect to $c'$. Next we can integrate for those $c'$ such that $p^t|b+\frac{1}{c''}$. We divide the problem into four cases
according to whether $\varpi_v^t|c$ and whether $\varpi_v^t|b$. In any case, it is not difficult to check that the integration is $0$ since $\mathrm{cond}(\chi_1\chi^{-1}_2)=(\varpi_v^t)$. We leave the verification for $g\in K_1\sigma_nK_1$ and $K_1w\sigma wK_1$ to the reader.
\end{proof}
\begin{lemma}
Suppose $\mathrm{cond}(\chi_1\chi_2^{-1})=(\varpi_v^t)$, $t>0$ and $\vartheta$ is a character with conductor $(\varpi_v^s)$, $s>t$. Define
$$f=f_{\chi,\vartheta}(k_v)=\left\{\begin{array}{ll}\chi_1(a)\chi_2(d)\vartheta(\frac{c}{\varpi_v^t}), & a,d\in \mathbb{Z}_v^\times, b\in\mathbb{Z}_v, c\in\varpi_v^t\mathbb{Z}_v^\times\\
0, & \mbox{ otherwise}\end{array}\right.$$
for $k_v=\begin{pmatrix}a&b\\&d\end{pmatrix}\begin{pmatrix}1&\\c&1\end{pmatrix}$. We similarly define
$$\tilde{f}=\tilde{f}_{\chi^{-1},\vartheta^{-1}}.$$
Then on $\cup_nK_1\sigma_n K_1\cup_nK_1w\sigma_nw K_1$ , $\Phi_{f,\tilde f}(g)$ is supported in $K_1$. Moreover if $g\in K_{t+s}$ with $\Phi_{f,\tilde{f}}(g)\not=0$, then its upper right entry is divisible by $\varpi^{s-t}_v$, and $\Phi_{f,\tilde{f}}(g)=(q_v-1)q^{-t}_v\mathrm{Vol}(K_1)$.
\end{lemma}
\begin{proof}
Suppose $\Phi_{f_{\chi,\vartheta},\tilde{f}_{\chi^{-1},\vartheta^{-1}}}(g)\not=0$. In the following proof we write $\Phi$ for $\Phi_{f_{\chi,\vartheta},\tilde{f}_{\chi^{-1},\vartheta^{-1}}}$ for short. If $g\in K_1\begin{pmatrix}\varpi_v^n&\\&1\end{pmatrix}K_1$ for some $n\geq 0$, then as before we have $g\in \begin{pmatrix}\varpi_v^n&\\&1\end{pmatrix}\begin{pmatrix}1&\varpi_v^{-n}\mathbb{Z}_v\\&1\end{pmatrix}K_t$. For $g'\in K_t$, write $g'=\begin{pmatrix}a'&b'\\&d'\end{pmatrix}\begin{pmatrix}1&\\c'&1\end{pmatrix}$. Without loss of generality assume $g=\begin{pmatrix}1&b\\&1\end{pmatrix}\begin{pmatrix}\varpi_v^n&\\&1\end{pmatrix}\begin{pmatrix}1&\\c&1\end{pmatrix}$ for $b\in\mathbb{Z}_v$, $\varpi_v^t|c$. Then
$$g'g=\begin{pmatrix}a'&b'\\&d'\end{pmatrix}\begin{pmatrix}1&\\c'&1\end{pmatrix}\begin{pmatrix}1&b\\&1\end{pmatrix}
\begin{pmatrix}\varpi_v^n&\\&1\end{pmatrix}\begin{pmatrix}1&\\c&1\end{pmatrix}.$$ Plugging in the formula for matrix coefficients we need to integrate for $a',b',c',d'$. Again we only need to consider integral with respect to $c'\in \varpi_v^t\mathbb{Z}_v$. Thus, $$\begin{pmatrix}1&\\c'&1\end{pmatrix}\begin{pmatrix}1&b\\&1\end{pmatrix}\begin{pmatrix}\varpi_v^n&\\&1\end{pmatrix}\begin{pmatrix}1&\\c&1\end{pmatrix}
=\begin{pmatrix}1&\frac{b}{c'b+1}\\&1\end{pmatrix}\begin{pmatrix}\frac{1}{bc'+1}&\\&bc'+1\end{pmatrix}
\begin{pmatrix}\varpi_v^n&\\&1\end{pmatrix}\begin{pmatrix}1&\\\frac{c'\varpi_v^n}{c'b+1}&1\end{pmatrix}\begin{pmatrix}1&\\c&1\end{pmatrix}.$$ If $n\geq1$, then the integral is $0$.
If $n=0$ and $\Phi(g)\not=0$, then $g\in K_t$. Suppose $g=\begin{pmatrix}1&b\\&1\end{pmatrix}\begin{pmatrix}1&\\c&1\end{pmatrix}$ with $c$ divisible by $\varpi^{s+t}_v$ then we easily see that $b$ is divisible by $\varpi^{s-t}_v$ and $\Phi(g)$ is given as in the lemma. For $g\in K_1w\begin{pmatrix}\varpi^n_v&\\&1\end{pmatrix}wK_1$ again if $\Phi(g)\not=0$ then $g\in \begin{pmatrix}1&\\&\varpi_v^n\end{pmatrix}\begin{pmatrix}1&\\\varpi_v^{t-n}\mathbb{Z}_v&1\end{pmatrix}K_t$. Without
loss of generality write $g\in \begin{pmatrix}1&\\&\varpi_v^n\end{pmatrix}\begin{pmatrix}1&\\c&1\end{pmatrix}$ for $c\in\varpi_v^{t-n}\mathbb{Z}_v$, $K_t
\ni g'=\begin{pmatrix}a'&b'\\&d'\end{pmatrix}\begin{pmatrix}1&\\c'&1\end{pmatrix}$, $g'g=\begin{pmatrix}a'&b'\\&d'\end{pmatrix}\begin{pmatrix}1&\\&\varpi_v^n\end{pmatrix}
\begin{pmatrix}1&\\\frac{c'}{\varpi_v^n}+c&1\end{pmatrix}$. If $n\geq1$, then one can check that the integration is again $0$.
\end{proof}
Suppose $\chi^{-1}_1\chi_2$ is unramified and $\vartheta$ has conductor $(\varpi_v^s)$, $s>0$, then we define $f_{\chi,\vartheta}\in \pi$ by:
$$f_{\chi,\vartheta}(k_v)=\left\{\begin{array}{ll}\chi_1(a)\chi_2(d)\vartheta(\frac{c}{\varpi_v}), &g=\begin{pmatrix}a&b\\&d\end{pmatrix}
\begin{pmatrix}1&\\c&1\end{pmatrix},\begin{pmatrix}a&b\\&d\end{pmatrix}\in B(\mathbb{Z}_v),c\in\varpi_v\mathbb{Z}_v^\times\\
0& \mbox{ otherwise } \end{array}\right.$$
We define similarly $\tilde{f}_{\tilde{\chi},\tilde{\vartheta}}\in\tilde{\pi}$ by replacing $\chi,\vartheta$ by $\chi^{-1},\vartheta^{-1}$.
\begin{lemma}\label{7.4}
Write $f=f_{\chi,\vartheta}$, $\tilde{f}=\tilde{f}_{\tilde{\chi},\tilde{\vartheta}}\in\tilde{\pi}$ then on $\sqcup_nK_0\sigma_nK_0\sqcup_n K_0w\sigma_n wK_0$ it is supported in $K_1$. Moreover if $g\in K_{s+1}$ and $\Phi_{f,\tilde{f}}(g)\not=0$, then the upper right entry of it is divisible by $\varpi^{s+1}_v$, and
$$\Phi_{f,\tilde{f}}(g)= \frac{q_v-1}{q_v}\mathrm{Vol}(K_1)$$
\end{lemma}
\begin{proof}
Similar to the above lemma.
\end{proof}
\begin{lemma}\label{8.7}
Suppose $\chi_1$ and $\chi_2$ are both unramified. Let $f^{\mathrm{sph}}$ be the spherical vector which takes value $1$ at identity in the model above. Then
$$\sum_{a\in\frac{\varpi_v\mathbb{Z}_v}{\varpi_v^{1+s}\mathbb{Z}_v}}\vartheta_v(-\frac{a}{\varpi_v})\pi(\begin{pmatrix}1&\\a&1\end{pmatrix}
\begin{pmatrix}\varpi_v^{-s}&\\&1\end{pmatrix})(1-q_v\chi_1/\chi_2(\varpi_v))^{-1}(1-q_v^{-\frac{1}{2}}\chi_2^{-1}(\varpi_v)\pi_v
(\begin{pmatrix}\varpi_v&\\&1\end{pmatrix})f^{sph}$$
equals $\chi_1(\varpi_v^{-s})q_v^{-\frac{s}{2}}f_{\chi,\vartheta}$.
\end{lemma}
\begin{proof}
Straightforward computations.
\end{proof}
Next we evaluate the local triple product integral for certain sections. The following lemma follows from the lemmas above.
\begin{lemma}
Let $\chi_{f,1},\chi_{f,2},\chi_{\theta,1},\chi_{\theta,2},\chi_{h,1},\chi_{\theta,2},\vartheta$ be characters of $\mathbb{Q}_v^\times$ and $t_1<s<t_2$ be non-negative integers such that if $t_1\not=0$ then $t_1+s=t_2$, and if $t_1=0$ then $s+1=t_2$. Suppose $\mathrm{cond}(\chi_{f,1}\chi_{f,2}^{-1})=(\varpi_v^{t_1})$ and $\mathrm{cond}(\vartheta)=(\varpi_v)^s$ and $\mathrm{cond}(\chi_{\theta,1}\chi_{\theta,2}^{-1})=\mathrm{cond}(\chi_{h,1}\chi_{h,2}^{-1})=(\varpi_v^{t_2})$. Assume:
$\chi_{f,1}.\chi_{\theta,1}.\chi_{h,1}.\vartheta=1$ and $\chi_{f,2}.\chi_{\theta,2}.\chi_{h,2}.\vartheta^{-1}=1$. We also define $f_{\chi_f,\vartheta}\in\pi(\chi_{f,1},\chi_{f,2}),f_{\chi_\theta}\in\pi(\chi_{\theta,1},\chi_{\theta,2}),f_{\chi_h}\in\pi(\chi_{h,1},\chi_{h,2})$ as above. Similarly for $\tilde{f}_{\tilde{\chi}_f,\tilde{\vartheta}},\tilde{f}_{\tilde{\chi}_\theta},\tilde{f}_{\tilde{\chi}_h}$. Then Ichino's local triple product is
$$I_v(f_{\chi_f,\vartheta}\otimes f_{\chi_\theta}\otimes f_{\chi_h}, \tilde{f}_{\tilde{\chi}_f,\tilde{\vartheta}}\otimes\tilde{f}_{\tilde{\chi}_\theta}\otimes
\tilde{f}_{\tilde{\chi}_h})=\frac{(q_v-1)^2}{q^{2s+1}_v}\mathrm{Vol}(K_{1})^2\mathrm{Vol}(K_{t_2})^2.$$
(In this lemma, the $\chi_{h},\chi_{\theta}$ are defined  using $\chi_{h,1},\chi_{h,2},\chi_{\theta,1},\chi_{\theta,2}$ similarly as in Lemma \ref{7.2}.)
\end{lemma}
\noindent\underline{Special Representations}\\
We consider the induced representation $\pi(\chi_1,\chi_2)=\{f:K_v\rightarrow K,f(qk)=\chi_1(a)\chi_2(d)\delta_B(q),q=\begin{pmatrix}a&b\\&d\end{pmatrix}\in B(\mathbb{Z}_v)\}$ where $\chi_1=\chi_2|\cdot|$. The special representation $\sigma(\chi_1,\chi_2)\subset \pi(\chi_1,\chi_2)$ consists of functions $f$ such that $\int_Kf(k)dk=0$. We consider the case when $\pi_3$ is the special representation $\sigma(\chi_{v,1},\chi_{v,2})\subset \mathrm{Ind}_{B(\mathbb{Q}_v)}^{\mathrm{GL}_2(\mathbb{Q}_v)}(\chi_{v,1},\chi_{v,2})$ at $v$ with a square-free conductor. Here $\chi_{v,i}$ are unramified characters. Similar to the unramified principal series case, we use the model of induced representations. It is easy to see that the $f_{\chi,\vartheta}$ defined above is inside $\sigma(\chi_{1,v},\chi_{2,v})$. Note that in the model for $\pi(\chi_1^{-1},\chi_2^{-1})$, there is a one-dimensional subrepresentation and the quotient is $\sigma(\chi_{1,v}^{-1}),\chi_{2,v}^{-1})$. The inner product of $\sigma(\chi_{1,v},\chi_{2,v})$ and $\sigma(\chi_{1,v}^{-1},\chi_{2,v}^{-1})$ is still given by $\langle v_1,v_2\rangle=\int_Kv_1(k)v_2(k)dk$. The formula for the triple product integral is the same as the one in the case of principal series representations.\\Let $f_{\mathrm{new},v}\in\sigma(\chi_1,\chi_2)$ be the $f$ such that $f(k)=q_v$ for $k\in K_1$ and $f(k)=-1$ otherwise. Clearly it is the new vector in the special representation. Then we have the following lemma:
\begin{lemma}\label{8.9}
$$\sum_{a\in\frac{\varpi_v\mathbb{Z}_v}{\varpi_v^{1+s}\mathbb{Z}_v}}\vartheta_v(-\frac{a}{\varpi_v})\pi(\begin{pmatrix}1&\\a&1\end{pmatrix}
\begin{pmatrix}\varpi_v^{-s}&\\&1\end{pmatrix})f_{\mathrm{new},v}$$
is $\chi_1(\varpi_v^{-s})q_v^{1-\frac{s}{2}}\cdot f_{\chi,\vartheta}$ where $f_{\chi,\vartheta}$ is defined above.
\end{lemma}
\begin{proof}
Let $f_0$ and $f_1$ be the characteristic functions on $K_1$ and $K_1wK_1$. Then $f_{\mathrm{new},v}=q_vf_0-f_1$. A computation shows that
$$\sum_{a\in\frac{\varpi_v\mathbb{Z}_v}{\varpi_v^{1+s}\mathbb{Z}_v}}\vartheta_v(-\frac{a}{\varpi_v})\pi(\begin{pmatrix}1&\\a&1\end{pmatrix}
\begin{pmatrix}\varpi_v^{-s}&\\&1\end{pmatrix})f_0=\chi_1(\varpi_v^{-s})q_v^{-\frac{s}{2}}f_{\vartheta},$$
$$\sum_{a\in\frac{\varpi_v\mathbb{Z}_v}{\varpi_v^{1+s}\mathbb{Z}_v}}\vartheta_v(-\frac{a}{\varpi_v})\pi(\begin{pmatrix}1&\\a&1\end{pmatrix}
\begin{pmatrix}\varpi_v^{-s}&\\&1\end{pmatrix})f_1=0.$$
The lemma follows.
\end{proof}

\begin{remark}
The reason why the local integrals at split primes showing up in the triple product formula later on are the ones considered in this subsection is a consequence of the computations in Subsection \ref{5.7}.
\end{remark}
Now we consider non-split primes.\\
\noindent\underline{Non-split Case 1}\\
\noindent The $\mathrm{U}(2)(\mathbb{Q}_v)$ is compact. This case is easier since we are in the representation theory for finite groups. By construction the representation $\pi^D_{f,v}$ is one dimensional. Recall our construction for $\mathbf{h}$ in Section \ref{theta}. For such $v$, we let $g_{2,v}$ and $g_{4,v}$ be either the identity or $\alpha$ there depending on whether $\pi'_{h_\mathsf{z},v}$ is $\pi^\vee_{\theta_{\eta_\mathsf{z}},v}\otimes\eta_\mathsf{z}$ or ${}^\alpha\!\pi^\vee_{\theta_{\eta_\mathsf{z}},v}\otimes\eta_\mathsf{z}$. Let $g_{1,v}$ and $g_{3,v}$ are the identify elements. By our assumptions on $\chi_{\theta,v},\chi_{h,v}$ and $\pi_{f,v}$ and our chosen vectors $h,\theta,\tilde{\theta}_3,\tilde{h}_3$ and that $\pi_{f,v}^D$ is one-dimensional, it is easily checked from the construction (note that by the choices in Section \ref{section 7.2}, the local triple product root number at $v$ is $-1$)
$$\pi_{h^D,v}\simeq \pi^\vee_{\theta^D,v}\otimes \pi^{D,\vee}_{f,v}.$$ We conclude that
by the inner product formula of matrix coefficients of representations of finite groups
$$\frac{I_v(\pi_{h,v}(g_{2,v})h_v^D\otimes v_{\phi_v}^D\otimes \pi_{f,v}(g_{1,v})f_{\vartheta,v}^D,\tilde{\pi}_{h,v}({g}_{4,v})\tilde{h}_v^D\otimes \tilde{v}_{\phi_v}^D\otimes \tilde{\pi}_{f,v}({g}_{3,v})\tilde{f}_{\tilde{\vartheta},v}^D)}{\langle\pi_{h,v}(g_{2,v})h_v^D, \tilde{\pi}_{h,v}({g}_{4,v})\tilde{h}_v^D\rangle\langle v_{\phi_v}^D, \tilde{v}_{\phi_v}^D\rangle\langle \pi_{f,v}(g_{1,v})f_{\vartheta,v}^D, \tilde{\pi}_{f,v}({g}_{3,v})\tilde{f}_{\tilde{\vartheta},v}^D\rangle}=\frac{1}{d_{\pi_{h,v}}}.$$
where $d_{\pi_{h,v}}$ is the dimension of the representation $\pi_{h,v}$.
When we are moving our datum $p$-adic analytically, this integral is not going to change.

\noindent\underline{Non-split Case 2}\\
The $\mathrm{U}(2)(\mathbb{Q}_v)$ is not compact. Recall that we fix a generic arithmetic point. By \cite[Theorem 1.4]{Prasad} we know that there are $$g_{1,v},g_{2,v},g_{3,v},g_{4,v}\in D^\times(\mathbb{Q}_v)\subset \mathrm{GU}(2)(\mathbb{Q}_v)$$ such that
$$I_v(\pi_{h,v}(g_{2,v})h_v^D\otimes v_{\phi_v}^D\otimes \pi_{f,v}(g_{1,v})f_{\vartheta,v}^D,\tilde{\pi}_{h,v}({g}_{4,v})\tilde{h}_v^D\otimes \tilde{v}_{\phi_v}^D\otimes \tilde{\pi}_{f,v}({g}_{3,v})\tilde{f}_{\tilde{\vartheta},v}^D)\not=0.$$ (We write $v_{\phi_v}^D$ for the image of $v_{\phi_v}$ in the corresponding $D^\times(\mathbb{Q}_v)$ representation and similarly for $\tilde{v}_{\phi_v}^D$ (notations as in Lemma \ref{Lemma 8.3}. To apply Prasad's result, note that the local sign for this triple product is $+1$ by our choices of high conductors.)
\begin{definition}\label{definition 8.17}
We define
$$g_i=\prod_vg_{i,v}$$
for $i=1,2,3,4$. (We take $g_{i,v}=1$ if $v$ is split in $\mathcal{K}/\mathbb{Q}$.)
\end{definition}
\index{$g_1, g_2, g_3, g_4$}
The local triple product integrals are non-zero by our computations. By Ichino's formula and our requirement on special $L$-values (note that the product of the central characters for $f_{\mathsf{z}},\theta_\mathsf{z},h_\mathsf{z}$ is trivial by construction) the global trilinear form is also non-zero. So by our definitions for $h^D_\mathsf{z},\theta^D_\mathsf{z}$, etcetera in Section \ref{D/U}, we know that $\prod_vg_{2,v}$ has to be in $\breve{D}^\times$. Thus up to a nonzero constant fixed throughout the family, we have $$\int_{[D^\times]}(\pi(g_2)h^D_\mathsf{z})(g)\theta^D_\mathsf{z}(g)(\pi(g_1)f_{\mathsf{z},\vartheta})(g)dg=
\int_{[\mathrm{U}(2)]}(\pi(g_2)h_\mathsf{z})(g)\theta_\mathsf{z}(g)(\pi(g_1)f_{\mathsf{z},\vartheta})(g)dg.$$
We have similarly $\prod_vg_{4,v}\in \breve{D}^\times$ and up to a nonzero constant fixed throughout the family,
$$\int_{[D^\times]}(\pi(g_4)\tilde{h}^D_{3,\mathsf{z}})(g)(\pi_{f_\mathsf{z}}(g_3)\tilde{f}_{\mathsf{z},\tilde{\vartheta}})
(g)\tilde{\theta}^D_{3,\mathsf{z},low}(g)dg=\int_{[\mathrm{U}(2)]}(\pi(g_4)\tilde{h}_{3,\mathsf{z}})(g)(\pi_{f_\mathsf{z}}(g_3)\tilde{f}_{\mathsf{z},\tilde{\vartheta}})
(g)\tilde{\theta}_{3,\mathsf{z},low}(g)dg.$$

These $g_{i,v}$ are chosen only at the arithmetic point $\mathsf{z}$. We fix them when moving the Eisenstein datum in $p$-adic families. 

\begin{lemma}\label{Remark 8.13}
Let $v$ be a non-split prime. Then for different arithmetic points $\mathsf{z}$ with fixed $\mathsf{z}|_{\mathbb{I}}$, the $\pi_{h,\mathsf{z},v}^D$ and $\pi_{\theta,\mathsf{z},v}^D$ only differ by twisting by unramified characters which are inverse to each other. 
\end{lemma}
\begin{proof}
In fact, at non-split primes, the local Weil representations on unitary groups are unchanged throughout the family since the characters $\chi_{\theta,\mathsf{z},v}|_{\mathrm{U}(1)}$ are unchanged. (They differ by multiplying by unramified characters and $\mathrm{U}(1)(\mathbb{Q}_v)$ is compact.) So for different $\mathsf{z}$ and $\mathsf{z}'$, the difference between $\pi_{h,\mathsf{z},v}^D$ and $\pi_{h,\mathsf{z}',v}^D$ only comes from the characters extending the form on $\mathrm{U}(2)$ to $\mathrm{GU}(2)$, and similarly for $\pi_{\theta,\mathsf{z},v}^D$. Note moreover that from construction the product of the central characters of $\pi_{h,\mathsf{z},v}^D$, $\pi_{\theta,\mathsf{z},v}^D$ and $\pi_{f_\mathsf{z}}$ is trivial. The lemma thus follows.
\end{proof}

Thus if the theta kernel we used to define $h_\mathsf{z}$ and $\theta_\mathsf{z}$ are fixed, then the local triple product integrals
$$\frac{I_v(\pi_{h,\mathsf{z},v}(g_{2,v})h_{\mathsf{z},v}^D\otimes v_{\phi_v}^D\otimes \pi_{f_\mathsf{z},v}(g_{1,v})f_{\mathsf{z},\vartheta,v}^D,\tilde{\pi}_{h,\mathsf{z},v}({g}_{4,v})
\tilde{h}_{\mathsf{z},v}^D\otimes \tilde{v}_{\phi_v}^D\otimes \tilde{\pi}_{f_\mathsf{z},v}({g}_{3,v})\tilde{f}_{\mathsf{z},\tilde{\vartheta},v}^D)}{\langle h_{\mathsf{z},v}^D,\tilde{h}_{3,\mathsf{z},v}^D\rangle \langle v_{1}^D,\tilde{v}^D_1\rangle
\langle f_{\mathsf{z},\vartheta,v}^D,\tilde{f}_{\mathsf{z},\tilde{\vartheta},v}^D\rangle}$$ does not change (note that the $\langle f_{\mathsf{z},\vartheta,v}^D,\tilde{f}_{\mathsf{z},\tilde{\vartheta},v}^D\rangle $ has a non-zero inner product). This observation is crucial in proving Proposition \ref{p-adic} later on.

We now define $\mathbb{I}$-adic families from the family $\mathbf{f}$ defined in Section \ref{section7.1} using Lemma \ref{8.7} and \ref{8.9}. However, in order to do so we may have to replace $\mathbb{I}$ by a larger normal domain finite over $\mathbb{I}$, so that the $\chi_{1,v}(\varpi_v)$ and $\chi_{2,v}(\varpi_v)$'s at primes in $\Sigma$ where $\pi_f$ is unramified will be elements of this newly defined $\mathbb{I}$.
\begin{definition}\label{ftheta}
We define $\mathbf{f}_{\vartheta}$ as \index{$f_{\vartheta}$}
\begin{align*}
&\prod_v\sum_{a\in\frac{\varpi_v\mathbb{Z}_v}{\varpi_v^{1+s}\mathbb{Z}_v}}\vartheta_v(-\frac{a}{\varpi_v})\pi(\begin{pmatrix}1&\\a&1\end{pmatrix}
\begin{pmatrix}\varpi_v^{-s}&\\&1\end{pmatrix})(1-q_v\chi_1/\chi_2(\varpi_v))^{-1}(1-q_v^{-\frac{1}{2}}\chi_2^{-1}(\varpi_v)\pi_v
(\begin{pmatrix}\varpi_v&\\&1\end{pmatrix})&\\
&\times\prod_v\sum_{a\in\frac{\varpi_v\mathbb{Z}_v}{\varpi_v^{1+s}\mathbb{Z}_v}}\vartheta_v(-\frac{a}{\varpi_v})\pi(\begin{pmatrix}1&\\a&1\end{pmatrix}
\begin{pmatrix}\varpi_v^{-s}&\\&1\end{pmatrix})\mathbf{f}.&
\end{align*}
We define $\tilde{\mathbf{f}}=\mathbf{f}\otimes\boldsymbol{\chi}^{-1}_\mathbf{f}$ and similarly define $\tilde{\mathbf{f}}_{\tilde{\vartheta}}$.

We also denote its specialization at an arithmetic point $\mathsf{z}\in\mathrm{Spec}\Lambda''_\mathbf{D}(\bar{\mathbb{Q}}_p)$ by $f_{\mathsf{z},\vartheta}$.
\end{definition}

\subsection{Evaluating the Integral}\label{7.5}
Recall that we have defined a theta function $\theta^\star$ in Corollary \ref{Corollary 6.29} (we suppress the subscript $\mathcal{D}$ as it is fixed throughout the family). Now we construct: for any $F\in \mathcal{M}_{ord}(K^{(3,1)},\Lambda''_\mathbf{D})$ (weight map as in Section \ref{EiDatum}), \index{$l_{\theta^\star}$}
$$l_{\theta^\star}(F):=l_{\theta^\star}'a^1_{[1]}(1, \prod_{v\in\Sigma^1\cup \Sigma^2}(\sum_{i}C_{v,i}\rho(\begin{pmatrix}u_{v,i}&&\\&I_2&\\&&u_{v,i}\end{pmatrix})(F)))$$
whose value is in $\mathcal{M}(K^{(2,0)},\Lambda_{2,0})\otimes_{j_2}\Lambda''_\mathbf{D}$. Here $C_{v,i}$ are defined in Lemma \ref{5.18}. The $a^1_{[1]}$ is the family version of the Fourier-Jacobi expansion in the sense of Definition \ref{definition 3.7}, and $l'_{\theta^\star}$ is applied to the theta function part of the Fourier-Jacobi expansion as in (\ref{CFunctional}). See the proof of Proposition \ref{Proposition 7.7}. Note that we can easily make sure that $\theta^\star$ is a member in the basis $\theta'_i$'s there. Recall that for the specialization $F_\mathsf{z}$ of $F$,
$$l'_{\theta^\star} a^1_{[1]}(1,F_\mathsf{z})(h)):=\int_{[W]}a^1_{[1]} (1,F_\mathsf{z}(wh))\theta^\star(wh)dw$$
with the Heisenberg group $W\hookrightarrow \mathrm{U}(3,1)$
$$w\mapsto \begin{pmatrix}1&w&\frac{1}{2}\langle w,w\rangle\\&1&\zeta w^*\\&&1\end{pmatrix}.$$
It is clear that $l'_{\theta^\star} \mathrm{FJ}_\beta(F)(h)$ is an automorphic form on $\mathrm{U}(2)$.
We can also define $l_{\theta^\star}$ on a single form on $\mathrm{U}(3,1)$ instead of on families, using the same formula. It is clear that
$$\mathsf{z}(l_{\theta^\star}(F))=l_{\theta^\star}(F_{\mathsf{z}}).$$
\begin{lemma}\label{lemma8.19}
The $l_{\theta^\star}(F) \in\mathcal{M}(K^{(2,0)},\Lambda_{2,0})\otimes_{j_2}(\Lambda''_\mathbf{D}\otimes_{\mathbb{Z}_p}\mathbb{Q}_p)$.
\end{lemma}
\begin{proof}
Recall that in Definition \ref{definition 3.7} there is a $\Lambda_\mathcal{D}$-adic Fourier-Jacobi expansion for families on $\mathrm{U}(3,1)$. As before we just take a basis of $\mathcal{O}_L$-dual space $(\theta'_1,\cdots, \theta'_m)$ of the finite-dimensional space $H^0(\mathcal{B},\mathcal{L}(\beta))$. Pairing the $\Lambda''_\mathbf{D}$-adic Fourier-Jacobi coefficient of $F$ with these $\theta'_i$ we get a $\Lambda''_\mathbf{D}$-adic family. But our $l'_{\theta^\star}$ is in the $L$-linear combination of the $\theta'_i$'s. Thus we get a $\Lambda''_\mathbf{D}\otimes_{\mathbb{Z}_p}\mathbb{Q}_p$-adic family on $\mathrm{U}(2)$.
\end{proof}
We also define
\begin{equation}\label{116}
\mathbf{B}_1:=\mathbf{B}_{K^{(2,0)}}\langle e_{ord}l_{\theta^\star}(\mathbf{E}_{\mathbf{D},Kling}),\pi(g_2)\mathbf{h}\rangle.
\end{equation}
\index{$\mathbf{B}_1$}
\begin{definition}\label{sslow}
For any nearly ordinary form $f$ or family $\mathbf{f}$ (we use the same notations for $\mathbf{h}$, $\boldsymbol{\theta}$, etcetera) we define $f_{low}(g):=f(g\begin{pmatrix}&1\\1&\end{pmatrix}_p)$(under the identification $D_p^\times\simeq \mathrm{GL}_2(\mathbb{Q}_p)$ given by the $v_0$ projection). Also, if $\boldsymbol{\chi}$ is the (family of) central characters of $\mathbf{f}$ we define $\tilde{\mathbf{f}}=\mathbf{f}.(\boldsymbol{\chi}^{-1}\circ\mathrm{Nm})$. We also define $f^{ss}(g)=f(g\begin{pmatrix}&1\\p^t&\end{pmatrix}_p)$.
\end{definition}
We will compute the specializations $$\mathsf{z}(\mathbf{B}_1)=\langle l_{\theta^\star}(\mathbf{E}_{\mathbf{D},Kling,\mathsf{z}}),\pi(g_2)h_{\mathsf{z},low}\rangle\mathrm{Vol}(K_\mathsf{z})^{-1}.$$
(The $K_\mathsf{z}$ is the level group at the arithmetic point $\mathsf{z}$. Note that the tame level group is fixed at the end of Section \ref{section 5} throughout).

We can evaluate this expression by Ichino's formula for triple products.

We have the following
\begin{proposition}\label{Proposition 8.21}
We use the notations of Section \ref{theta} and Definition \ref{ftheta}. Then there is a $\mathcal{C}$, the product of a constant in $\bar{\mathbb{Q}}_p^\times$ which is fixed along the family and a unit in $\hat{\mathbb{I}}^{ur}[[\Gamma''_\mathcal{K}]]$ (precise definition is given in the following proof), such that for any generic arithmetic point $\mathsf{z}$ of conductor $p^t$,
$$\mathsf{z}(\mathbf{B}_1)=\mathcal{C}_\mathsf{z}p^tC^D_{\mathrm{U}(2)}\mathsf{z}(\mathcal{L}_5)\mathsf{z}(\mathcal{L}_6)
\times(\int_{\mathbb{A}_\mathbb{Q}^\times D^\times(\mathbb{Q})\backslash D^\times(\mathbb{A}_\mathbb{Q})} \pi(g_{2})h^D_\mathsf{z}(g)\pi(g_{1})f_{\mathsf{z},\vartheta}(g)\theta_{\mathsf{z},low}^D(g)dg)$$
\noindent where $\mathcal{L}_5$ and $\mathcal{L}_6$ are Katz $p$-adic $L$-functions in $\hat{\mathbb{I}}^{ur}[[\Gamma''_\mathcal{K}]]$ interpolating the $L$-values $$\Omega^\kappa_p\frac{2\pi i\Gamma(\kappa-1)}{\Omega_\infty^\kappa\mathfrak{g}(\tau_{\mathsf{z},1,p})}L(\chi_{\mathrm{aux}}\tau^{-c}_\mathsf{z},\frac{\kappa-2}{2})\cdot p^{t(\kappa-1)}\cdot (\chi_{aux,1,p}\tau_{\mathsf{z},1,p}(p)p^{-\frac{\kappa}{2}})^t$$
 and 
\begin{align*}
&\Omega_p^{\kappa-2}\frac{\Gamma(\kappa-2)}{\Omega_\infty^{\kappa-2}\mathfrak{g}(\tau_{\mathsf{z},1,p}
\chi_{\theta,\mathsf{z},2}\chi_{\theta,\mathsf{z},1}^{-1})}&\\
&\times L(\lambda^2\chi_{\theta,\mathsf{z}}^{-c}
\chi_{\theta,\mathsf{z}}\chi_{\mathrm{aux}}\tau^{-1}_\mathsf{z},\frac{\kappa-2}{2})p^{t(\kappa-2)}(\tau_{\mathsf{z},1,p}
\chi_{\theta,\mathsf{z},1,p}^{-1}\chi_{\theta,\mathsf{z},2,p}\lambda^2_{1,p}(p)p^{-\frac{\kappa-2}{2}})^t&
\end{align*}
respectively.
\end{proposition}
\begin{remark}
Although our coefficient ring is $\Lambda''_\mathbf{D}$, however in fact the functional does take values in the subring $\mathbb{I}[[\Gamma''_\mathcal{K}]]$. Indeed the additional variable $\Gamma^-_\mathcal{K}$ of $\Lambda''_\mathbf{D}$ corresponds to twisting the Klingen Eisenstein family by $p$-adic anticyclotomic characters.
\end{remark}
\begin{proof}
We first remark that it indeed makes sense to talk about the Katz $p$-adic $L$-functions $\mathcal{L}_5$ and $\mathcal{L}_6$ interpolating those values, since the characters $\tau_\mathsf{z}$ and $\chi_{\theta,\mathsf{z}}$ are indeed interpolated as characters with values in $\mathbb{I}[[\Gamma''_\mathcal{K}]]$.

By Corollary \ref{Corollary 6.29}, we have $\mathsf{z}(B_1)$ equals
\begin{align*}&\frac{p^{2t}}{\Omega_\infty^{2\kappa-1}}&&\times\int_{[\mathrm{U}(2)]\times[\mathrm{U}(2)]}\sum_{n\in\mathbb{Z}_p/p^t\mathbb{Z}_p}E_{sieg,\mathbf{D}_\mathsf{z},2}(u_1, u_2\begin{pmatrix}1&\\n&1\end{pmatrix}_p)(\pi(g_2)h_{\mathsf{z},low})(u_1)\theta_{2,\mathbf{D}_\mathsf{z}}(u_2\begin{pmatrix}1&\\n&1\end{pmatrix}_p)&\\
&&&\times(\pi(g_1)f_{\mathsf{z},\vartheta})(u_2)du_1du_2.
\end{align*}

We integrate over $u_1$ first and factorize the $E_{sieg,\mathcal{D}_\mathsf{z},2}$ via the embedding $\mathrm{U}(2)\times \mathrm{U}(2)\hookrightarrow \mathrm{U}(2,2)$ of $E_{sieg,\mathcal{D}_\mathsf{z},2}$. By the local pullback formulas we computed in Section \ref{5.9},
we get
$$\mathcal{C}_\mathsf{z}p^tC^D_{\mathrm{U}(2)}\mathsf{z}(\mathcal{L}_5)\mathsf{z}(\mathcal{L}_6)\times
(\int_{\mathbb{A}_\mathbb{Q}^\times D^\times(\mathbb{Q})\backslash D^\times(\mathbb{A}_\mathbb{Q})} \pi(g_{2})h^D_\mathsf{z}(g)\pi(g_{1})f_{\mathsf{z},\vartheta}(g)\theta^D_{\phi_2}(g)\bar{\lambda}(\det g) dg)$$
where $\mathcal{C}_\mathsf{z}$ is the product of the local pullback integrals in Section \ref{5.9} (and also the local pullback computations in Sections \ref{sectionar} through \ref{5.7}) at primes outside $p$, local Euler factors at $p$ for $\mathcal{L}_5$, $\mathcal{L}_6$ (See \cite[Equation (4.16)]{Hsi14} for the Euler factor at $p$), Euler factor for $\mathcal{L}_5$ at primes in $\Sigma^2$ and Euler factors for $\mathcal{L}_6$ at $p$. The first is a fixed constant which does not move with $\mathsf{z}$, the rest are interpolated by units in the Iwasawa algebra by our choice of characters in Section \ref{section 7.2}. Thus the $\mathcal{C}_\mathsf{z}$ are clearly interpolated by an element $\mathcal{C}$ mentioned in the proposition. Note also that the Euler factors at other places are trivial.

We are now ready to deduce the expression in the proposition. Although the $\theta_\mathsf{z}$ part appearing above is $\theta_{\phi_2,\mathsf{z}}\otimes\bar{\lambda}$ constructed in Section 6 (not an eigenform), however, in view of the central character of $h_\mathsf{z}$ and $f_\mathsf{z}$, only the eigen-component $\theta_\mathsf{z}$ of $\theta_{\phi_2,\mathsf{z}}\otimes\bar{\lambda}$ with the correct central character matters. Also by the construction this $\theta_\mathsf{z}$ part is a multiple of $\theta^{ss}_\mathsf{z}$ which, after applying the operator $\sum_{n\in\mathbb{Z}_p/p^t\mathbb{Z}_p}\pi_{\theta,\mathsf{z}}(\begin{pmatrix}1&\\n&1\end{pmatrix}_p)$ is $\theta_{\mathsf{z},low}$ by construction (see Proposition \ref{CMTheta}). So by considering the local pairing between $\pi_{\theta_\mathsf{z},p}$ and $\pi^\vee_{\theta_\mathsf{z},p}$, it is easy to see that if we replace this multiple of $\theta^{ss}_\mathsf{z}$ (see Definition \ref{sslow}) by $\theta_{\mathsf{z},low}$ we do not change the whole integral. Thus we finally arrived at
$$\mathcal{C}_\mathsf{z}p^tC^D_{\mathrm{U}(2)}\mathsf{z}(\mathcal{L}_5)\mathsf{z}(\mathcal{L}_6)
\times(\int_{\mathbb{A}_\mathbb{Q}^\times D^\times(\mathbb{Q})\backslash D^\times(\mathbb{A}_\mathbb{Q})} \pi(g_{2})h^D_\mathsf{z}(g)\pi(g_{1})f_{\mathsf{z},\vartheta}(g)\theta_{\mathsf{z},low}^D(g)dg).$$
\end{proof}

By our choices for characters, the corresponding non-$\Sigma$-primitive $p$-adic $L$-functions $\mathcal{L}_5$ and $\mathcal{L}_6$ are units in $\Lambda''_\mathbf{D}$.

In the following, we often omit the superscript $D$ for simplicity. Up to a constant in $\bar{\mathbb{Q}}_p^\times$ (which does not change along the family) we have
\begin{align}\label{AA'}
&&&\mathsf{z}(\mathbf{B}_1)\cdot p^t(\int(\pi(g_4)\tilde{h}_{3,\mathsf{z}})(g)(\pi_{f_\mathsf{z}}(g_3)\tilde{f}_{\mathsf{z},\tilde{\vartheta}})
(g)\tilde{\theta}_{3,\mathsf{z},low}(g)dg)&\notag\\
&=&&(\lambda_{p,2}\chi_{\theta,\mathsf{z},2})^{-t}(p)(\chi_{\theta,\mathsf{z}1}\lambda_{p,1})^{t}(p)p^{3t}
(\int(\pi(g_{2})h_\mathsf{z})(g)
(\pi_{f_\mathsf{z}}(g_1)f_{\mathsf{z},\vartheta})(g)\theta^{ss}_\mathsf{z}(g)dg)&\notag\\
&&&\times(\int (\pi_{\tilde{h}_\mathsf{z}}(g_4)\tilde{h}_{3,\mathsf{z}})(g)
(\pi_{\tilde{f}_\mathsf{z}}(g_{3})\tilde{f}_{\mathsf{z},\tilde{\vartheta}})(g)\tilde{\theta}^{ss}_{3,\mathsf{z}}(g)dg)
\times\mathsf{z}(\mathcal{L}_5^\Sigma\mathcal{L}^\Sigma_6)
&\notag\\
&=&&\lambda^{-2t}_{p,2}(p)\chi_{\theta,\mathsf{z},2}^{-2t}(p)p^{3t}(\int(\pi(g_{2})h_\mathsf{z})(g)
(\pi_{f_\mathsf{z}}(g_1)f_{\mathsf{z},\vartheta})(g)\theta^{ss}_\mathsf{z}(g)dg)\notag&\\
&&&\times(\int (\pi(g_4)\tilde{h}_{3,\mathsf{z}})^{ss}(g)(\pi_{\tilde{f}_\mathsf{z}}(g_{3})
\tilde{f}_{\mathsf{z},\tilde{\vartheta}})^{ss}(g)\tilde{\theta}_{3,\mathsf{z}}(g)dg)
\times\mathsf{z}(\mathcal{L}^\Sigma_5\mathcal{L}^{\Sigma}_6).
&
\end{align}
(Note that by our discussion in Subsection \ref{IPIP}, we know up to multiplying by an element in $\bar{\mathbb{Q}}_p^\times$, the
\begin{equation}\label{10000}
(p^t\int(\pi(g_4)\tilde{h}_{3,\mathsf{z}})(g)(\pi_{f_\mathsf{z}}(g_3)\tilde{f}_{\mathsf{z},\tilde{\vartheta}})(g)
\tilde{\theta}_{3,\mathsf{z},low}(g)dg)
\end{equation}
 is interpolated by an element 
$$\mathbf{B}_{K^{(2,0)}}\langle\tilde{\pi(g_4)\mathbf{h}}_3\cdot\pi_{f_\mathsf{z}}(g_3)\tilde{\mathbf{f}}_{\vartheta}, \tilde{\boldsymbol{\theta}}_3\rangle $$ 
in $\hat{\mathbb{I}}^{ur}[[\Gamma''_\mathcal{K}]]$.)

\begin{definition}\label{7.12}
We define $\chi_{f_\mathsf{z},s}$ and $\chi_{f_\mathsf{z},o}$ so that $\pi_{f_\mathsf{z},p}$ (the $p$-component of the automorphic representation $\pi_f$ associated to $f$) is the principal series representation $\pi(\chi_{f_\mathsf{z},s},\chi_{f_\mathsf{z},o})$ where $\mathrm{val}_p(\chi_{f_\mathsf{z},s})=\frac{1}{2}$ and $\mathrm{val}_p(\chi_{f_\mathsf{z},o})=-\frac{1}{2}$, respectively.
\end{definition}

\noindent We write the $p$-component of $\lambda^2\chi_{\theta,\mathsf{z}}\chi_{h,\mathsf{z}}$ as $((\lambda^2\chi_{\theta,\mathsf{z}}\chi_{h,\mathsf{z}})_1,(\lambda^2\chi_{\theta,\mathsf{z}}\chi_{h,\mathsf{z}})_2)$ and $f_\mathsf{z}$ the normalized $\mathrm{GL}_2$ ordinary form new outside $p$. We also define the following: the $p$-adic $L$-function $\mathcal{L}_1$ such that for any generic $\mathsf{z}$
$$\mathcal{L}_1(\mathsf{z})=\frac{2\pi i\mathfrak{g}(\chi_{\theta,\mathsf{z},2})\mathfrak{g}(\chi_{h,\mathsf{z},1}^{-1})
L(f_\mathsf{z},\lambda^2\chi_{\theta,\mathsf{z}}\chi_{h,\mathsf{z}},\frac{1}{2})(\chi_{f_\mathsf{z},s}(p)
\chi_{f_\mathsf{z},o}(p))^{-t}((\lambda^2
\chi_{\theta,\mathsf{z}}\chi_{h,\mathsf{z}})_2(p).p)^{-2t}p^t}{\Omega_\infty^4}\Omega^4_p;$$
the $p$-adic $L$-function $\mathcal{L}_2$ such that for any generic $\mathsf{z}$
\begin{align*}
&\mathcal{L}_2(\mathsf{z})=\frac{(1-a_p(f_\mathsf{z})^{-1}\chi_{\theta,\mathsf{z},p,1}\chi_{h,\mathsf{z},p,2}(p^{-1}))
}
{(1-a_p(f_\mathsf{z})\chi_{\theta,\mathsf{z},p,1}\chi_{h,\mathsf{z},p,2}(p)p^{-1})}
\times\frac{\mathfrak{g}(\chi_{f_\mathsf{z}}^{-1})L^{(p)}(f_\mathsf{z},\chi_{\theta,\mathsf{z}}^c\chi_{h,\mathsf{z}},
\frac{1}{2})(\chi_{f_\mathsf{z},o}(p)p^{\frac{1}{2}}(\chi_{h_\mathsf{z}}\chi_{\theta,\mathsf{z}}^c)_1(p))^{-t}}
{(2\pi i)^2\langle f_\mathsf{z},f^c_\mathsf{z}|
\begin{pmatrix}&-1\\N&\end{pmatrix}\rangle_{\Gamma_0(N)}};&
\end{align*}

the $p$-adic $L$-function $\mathcal{L}_3$ such that for any generic $\mathsf{z}$,
$$\mathcal{L}_3(\mathsf{z})=\frac{\zeta_{\chi_\mathcal{K}}(1)}{\pi}\frac{\mathfrak{g}(\chi_{\theta,\mathsf{z}})
L(\lambda^2\chi_{\theta,\mathsf{z}}\chi_{\theta,\mathsf{z}}^{-c},1)((\lambda^2\chi_{\theta,\mathsf{z}}
\chi_{\theta,\mathsf{z}}^{-c})_2(p)p)
^{-t}p^t}{\Omega^2_\infty}\Omega^2_p;$$

the $p$-adic $L$-function $\mathcal{L}_4$ such that for any generic $\mathsf{z}$,
$$\mathcal{L}_4(\mathsf{z})=\frac{\zeta_{\chi_\mathcal{K}}(1)}{\pi}.\frac{\mathfrak{g}(\chi_{h,\mathsf{z},1}^{-1})
L(\lambda^2\chi_{h,\mathsf{z}}\chi_{h,\mathsf{z}}^{-c},1)((\lambda^2\chi_{h,\mathsf{z}}\chi_{h,\mathsf{z}}^{-c})_2(p).p)
^{-t}.p^t}
{\Omega^2_\infty}\Omega^2_p.$$

\noindent We refer to \cite{Hsi12}, \cite{Hsi14} for the justification of their interpolation formulas. These values are interpolated by some $p$-adic $L$-functions in $\hat{\mathbb{I}}^{ur}_\mathcal{K}$. Note that by \cite[Lemma 5.3 (vi)]{Hida91} and \cite[Theorem 7.1]{HT93},
$$\frac{\mathfrak{g}(\chi_f^{-1})L(\mathrm{ad}, f_\mathsf{z},1)M\varphi(M)(p-1)(\chi_{f_\mathsf{z},s}(p))^t}{2^3\pi^3W'(f_\mathsf{z})^{-1}\langle f_\mathsf{z},f_\mathsf{z}|\begin{pmatrix}&-1\\N&\end{pmatrix}\rangle_{\Gamma_0(N)}}=1$$
where $W'(f_\mathsf{z})$ is the prime-to-$p$ part of the root number of $f_\mathsf{z}$, which is an element in $\mathbb{I}^\times$ (i.e. a unit). 

Note also that $\mathcal{L}_2$ is in fact a non-zero element in $\mathrm{Frac}(\mathbb{I})$ (i.e. does not depend on the variable $\Gamma_\mathcal{K}$) by checking the $p$-components of $\chi_{\theta,\mathsf{z}}^c\chi_{h,\mathsf{z}}$ and non-zero by our choice of characters in Subsection \ref{section 7.2}. It can actually be written as the ratio of two elements whose specializations to all but finitely many generic arithmetic points are non-zero. By our choices for $\chi_\theta$ and $\chi_h$ we know that $\mathcal{L}_1$ is in $\hat{\mathbb{I}}^{ur}[[\Gamma''_\mathcal{K}]]^\times$.
\noindent We consider the expression:
\begin{align*}
&&&C_{\mathrm{U}(2)}^D\frac{\langle f_\mathsf{z},\tilde{f}^{ss}_\mathsf{z}\rangle\langle h_\mathsf{z},\tilde{h}^{ss}_\mathsf{z}\rangle \langle\tilde{\theta}_\mathsf{z},{\theta}^{ss}_\mathsf{z}\rangle\chi_{\theta,\mathsf{z},1}^{-1}(p^t)\mathsf{z}(\mathcal{L}_1
\mathcal{L}_2\mathcal{L}_5\mathcal{L}_6)}{\mathsf{z}(\mathcal{L}_3\mathcal{L}_4)}
(\chi_{f,\mathsf{z},s}(p)\lambda_{p,1}^2(p)\chi_{\theta,\mathsf{z},1}(p)\chi_{h,\mathsf{z},1}(p)
p^{\frac{3}{2}})^tp^{3t}&\\
&=&&C_{\mathrm{U}(2)}^D\frac{\langle f_\mathsf{z},\tilde{f}_{\mathsf{z},low}\rangle\langle h_\mathsf{z},\tilde{h}_{\mathsf{z},low}\rangle\langle\theta_\mathsf{z},\tilde{\theta}_{\mathsf{z},low}
\rangle\mathsf{z}(\mathcal{L}_1
\mathcal{L}_2\mathcal{L}_5\mathcal{L}_6)}{\mathsf{z}(\mathcal{L}_3\mathcal{L}_4)}p^{3t}.&\
\end{align*}
The above element is clearly interpolated by an element
\begin{equation}\label{(4)}
\mathcal{G}\in\Lambda''_\mathbf{D}
\end{equation}
We first give the following lemma for the local triple product integral at $p$.
\begin{lemma}
At a generic point $\mathsf{z}$, the local triple product integral for the expression in (\ref{AA'}) at $p$ is given by:
$$\frac{I_p(\phi_p\otimes\tilde{\phi}_p)}{\langle\phi_p,\tilde{\phi}_p\rangle}=\frac{p^{-t}(1-p)}{1+p}
\frac{1}{1-a_p(f_\mathsf{z})\chi_{\theta,\mathsf{z},p,1}\chi_{h,\mathsf{z},p,2}
(p)p^{-1}}\cdot\frac{1}{1-a_p(f_\mathsf{z})^{-1}
\chi_{\theta,\mathsf{z},p,2}\chi_{h,\mathsf{z},p,1}(p)}.$$
\end{lemma}
\begin{proof}
This follows from Lemma \ref{7.2}.
\end{proof}
We observe that by definition this expression is a nonzero element in $\mathrm{Frac}(\mathbb{I})$.

\noindent Thus we have the following proposition:
\begin{proposition}\label{p-adic}
Any height $1$ prime of $\hat{\mathbb{I}}^{ur}[[\Gamma_\mathcal{K}]]$ containing $\mathbf{B}_{K^{(2,0)}}\langle l_{\theta^\star}(\mathbf{E}_{Kling,\mathbf{D}}), \pi(g_2')\mathbf{h}\rangle$ must be the pullback of a height $1$ prime of $\hat{\mathbb{I}}^{ur}$.
\end{proposition}

\begin{proof}
Let us recall what we have achieved so far. We have computed the Fourier-Jacobi coefficients of the Siegel-Eisenstein series in Proposition \ref{Proposition 6.31}. Using pullback formulas and computations on theta functions, we further computed the $\theta^\star$-part of the $\beta$-th Fourier-Jacobi coefficient of our ordinary Klingen-Eisenstein series in Corollary \ref{Corollary 6.29}. This is a function on the definite unitary group $\mathrm{U}(2)$. Moreover, we constructed an ordinary family $\mathbf{h}$ of CM forms on $\mathrm{U}(2)$ in Definition \ref{444}. Then we form the pairing of this $\theta^\star$ Fourier-Jacobi coefficient with $\mathbf{h}$ (see the beginning of this subsection), and resulted an element $\mathbf{B}_1$ which is an element in $\Lambda''_\mathbf{D}\otimes_{\mathbb{Z}_p}\mathbb{Q}_p$ by its construction and Lemma \ref{lemma8.19}. In Proposition \ref{Proposition 8.21} we used the doubling method for $\mathrm{U}(2)\times\mathrm{U}(2)\hookrightarrow\mathrm{U}(2,2)$ to obtain an expression for the specializations of $\mathbf{B}_1$ at arithmetic points $\mathsf{z}$.

To prove the proposition, it is enough to show that the product of $\mathbf{B}_1$ and the element (\ref{10000}) in $\hat{\mathbb{I}}^{ur}[[\Gamma''_\mathcal{K}]]$ satisfies the property stated in the proposition.
\noindent We examine the ratio between (\ref{(4)}) and the expression for the element interpolating
\begin{equation}\label{(8)}
\frac{1}{C}p^t\mathsf{z}(\mathbf{B}_1)(\int(\pi(g_4)\tilde{h}_{3,\mathsf{z}})(g)(\pi_{f_\mathsf{z}}(g_3)\tilde{f}_\mathsf{z})(g)
\tilde{\theta}_{3,\mathsf{z},low}(g)dg)
\end{equation}
using Ichino's formula. (Recall $C$ here is the Tamagawa number for $D^\times$ appearing in Ichino's formula.)
By our calculations for the local triple product integrals, and the Petersson inner product of $\langle\theta,\tilde{\theta}_3\rangle$ and $\langle h,\tilde{h}_3\rangle$, the ratio is a product of:
\begin{itemize}
\item Euler factors for $\zeta_{\chi_\mathcal{K}}(1)$ at $\Sigma$;
\item the local Euler factors for $L(ad,f_\mathsf{z},1)$ at $\Sigma\backslash\{p\}$;
\item $p^{t}$ times the local triple product integral for $v=p$, which are nonzero elements in $\mathrm{Frac}(\mathbb{I})$;
\item $\langle f^D_{\mathsf{z},\vartheta},\tilde{f}_{\vartheta,\mathsf{z},low}^D\rangle/\langle f^D_\mathsf{z}, \tilde{f}_{\mathsf{z},low}^D\rangle$ which is interpolated by a non-zero element in $\mathbb{I}$;
\item The local Euler factors of $\mathsf{z}(\mathcal{L}_5)$ and $\mathsf{z}(\mathcal{L}_6)$ at $\Sigma^2$ and $p$ which are units by our choices;
\item The local Euler factors at $\Sigma^2$ of $L(f_\mathsf{z},\chi_{\theta,\mathsf{z}}^c\chi_{h,\mathsf{z}},\frac{1}{2})$ which are non-zero elements in $\mathbb{I}$.
\item The local triple product integrals for $v\nmid p$.
\end{itemize}
The first two items are clearly interpolated by non-zero elements in $\mathrm{Frac}(\hat{\mathbb{I}})$. The last item we listed above has two parts: at split primes and non-split primes. The integrals at split primes are non-zero numbers in $\bar{\mathbb{Q}}_p^\times$ which are fixed throughout the family. At non-split primes, we do not know much about it. We only know that at a generic arithmetic point $\mathsf{z}$, this integral is not zero, and it only depends on $\mathsf{z}|_\mathbb{I}$ at generic points, as observed in Remark \ref{Remark 8.13}. We may assume that at this $\mathsf{z}$ the expression (\ref{(4)}) is non-zero and not a pole (in fact just need $\mathcal{L}_2$ to be a non-zero finite number here). Thus the expression (\ref{(8)}) is not identically zero. So the ratio of (\ref{(8)}) over (\ref{(4)}) is a non-zero element of $\mathrm{Frac}(\hat{\mathbb{I}}^{ur}[[\Gamma''_\mathcal{K}]])$. If we evaluate this ratio at the generic arithmetic points, it depends only on $\mathsf{z}|_{\mathbb{I}}$. And also it is non-zero somewhere. From this, it is not difficult to prove that (say using the following lemma) the ratio is a non-zero element of $\mathrm{Frac}(\hat{\mathbb{I}}^{ur})$. (In fact we apply the following lemma to $(\ref{(8)})/(\ref{(4)})$. Recall that by our choices for characters, the $\mathcal{L}_1$, $\mathcal{L}_5$, $\mathcal{L}_6$ are units in $\Lambda_\mathcal{D}$, and $\mathcal{L}_2$ is a non-zero element in $\mathrm{Frac}(\hat{\mathbb{I}}^{ur})$. Moreover, the local integrals showing up in the Rallis inner product formula for $\theta_\mathsf{z}$ and $h_\mathsf{z}$ at $\Sigma$, which are non-zero and fixed along the family. This gives the ratio between $\langle\theta_\mathsf{z},\tilde{\theta}_{\mathsf{z},low}\rangle\cdot \langle h_\mathsf{z}, \tilde{h}_{\mathsf{z},low}\rangle$ and $\mathsf{z}(\mathcal{L}_3\cdot\mathcal{L}_4)$.
So the proposition follows clearly.
\end{proof}
\begin{lemma}
Suppose $A$ is an element in $\hat{\mathbb{I}}^{ur}[[\Gamma''_\mathcal{K}]]\otimes_{\mathbb{Z}_p}\mathbb{Q}_p$. If for any generic arithmetic points $\mathsf{z},\mathsf{z}'\in\hat{\mathbb{I}}^{ur}[[\Gamma_\mathcal{K}]]$ such that $\mathsf{z}|_{\hat{\mathbb{I}}^{ur}}=\mathsf{z}'|_{\hat{\mathbb{I}}^{ur}}$, we have $\mathsf{z}(A)=\mathsf{z}'(A)$. Then $A\in\hat{\mathbb{I}}^{ur}$.
\end{lemma}
\begin{proof}
This lemma is easily proved by observing that if $\zeta_1, \zeta_2$ are $p^t$-roots of unity and $\phi$ is a generic arithmetic point with conductors being $p^{t'}$ such that $t'>t$, then the composition $\phi'$ of $\phi$ with the ring automorphism $\iota_{\zeta_1,\zeta_2}:\hat{\mathbb{I}}^{ur}[[\Gamma''_\mathcal{K}]]\rightarrow \hat{\mathbb{I}}^{ur}[[\Gamma''_\mathcal{K}]]$ given by identity on $\hat{\mathbb{I}}^{ur}$ and $\gamma^+\mapsto \gamma^+\zeta_1$, $\gamma^-\mapsto \gamma^-\zeta_2$ is still a generic arithmetic point. Let $A$ be the element considered in the lemma. Then $A-A\circ\iota_{\zeta_1,\zeta_2}$ is $0$ at a Zariski dense set of points, and is thus identically zero. The arbitrariness of $\zeta_1, \zeta_2$ implies the lemma.
\end{proof}
\section{Proof of the Theorems}\label{Section 9}

\subsection{Eisenstein Ideals}
Let $K_\mathbf{D}$ be an open compact subgroup of $\mathrm{GU}(3,1)(\mathbb{A}_\mathbb{Q})$ maximal at $p$ and all primes outside $\Sigma$ such that the Klingen-Eisenstein series we construct is invariant under $K^{(p)}_\mathbf{D}$. We consider the ring $\mathbb{T}_\mathbf{D}$ of reduced Hecke algebras acting on the space of ${\Lambda}''_\mathbf{D}$-adic nearly ordinary cuspidal forms with level group $K_\mathbf{D}$. It is generated by the Hecke operators $Z_{v,0}$, $Z_{v,0}^{(i)}$, $T_{i,v}$, $T_{i,v}^{(j)}$ defined before, together with the $\mathrm{U}_p$-operator and then taking the maximal reduced quotient. It is well known that one can interpolate the pseudo Galois characters attached to nearly ordinary cusp forms to get a pseudo-character $R_\mathbf{D}$ of $G_\mathcal{K}$ with values in $\mathbb{T}_\mathbf{D}$ (see \cite[Section 7.2]{SU} for details). We define the ideal $I_\mathbf{D}$ of $\mathbb{T}_\mathbf{D}$ to be generated by $\{t-\lambda(t)\}_t$ for $t$'s in the abstract Hecke algebra and $\lambda(t)$ is the Hecke eigenvalue of $t$ on $\mathbf{E}_{\mathbf{D},Kling}$. Then it is easy to see that the structure map ${\Lambda}''_\mathbf{D}\rightarrow \mathbb{T}_\mathbf{D}/I_\mathbf{D}$ is surjective. Suppose the inverse image of $I_\mathbf{D}$ in ${\Lambda}''_\mathbf{D}$ is $\mathcal{E}_\mathbf{D}$. We call it the Eisenstein ideal. It measures the congruences between the Hecke eigenvalues of cusp forms and Klingen-Eisenstein series. We have:
$$R_\mathbf{D}(\mathrm{mod}I_\mathbf{D})\equiv \mathrm{tr}\rho_{\mathbf{E}_{\mathbf{D},Kling}}(\mathrm{mod}\mathcal{E}_\mathbf{D}).$$
Now we prove the following lemma:
\begin{lemma}\label{8.1}
Let $P$ be a height $1$ prime of $\hat{\mathbb{I}}^{ur}[[\Gamma''_\mathcal{K}]]$ which is not the pullback of a height $1$ prime of $\hat{\mathbb{I}}^{ur}$. Then
$$\mathrm{ord}_P(\mathcal{L}_{\mathbf{f},\xi,\mathcal{K}}^\Sigma)\leq \mathrm{ord}_P(\mathcal{E}_\mathbf{D}).$$
\end{lemma}
\begin{proof}
Suppose $t:=\mathrm{ord}_P(\mathcal{L}^{\Sigma}_{\mathbf{f},\mathcal{K},\xi})>0$. By the fundamental exact sequence Theorem \ref{exact} there is an $\mathbf{H}=\mathbf{E}_{\mathbf{D},Kling}-\mathcal{L}^{\Sigma}_{\mathbf{f},\xi,\mathcal{K}}F$ for some $\Lambda''_\mathbf{D}$-adic form $F$ such that $\mathbf{H}$ is a cuspidal family. We write $\ell$ for the $\Lambda''_\mathbf{D}$-adic functional $\ell(G)=\langle l_{\theta^\star}(G),\pi(g_2')\mathbf{h}\rangle$ constructed in Subsection \ref{7.5} on the space of $\Lambda''_\mathbf{D}$-adic forms. By our assumption on $P$ we have proved that $\ell(\mathbf{H})\not\equiv 0(\mathrm{mod}P)$. Consider the $\Lambda''_\mathbf{D}$-linear map:
$$\mu: \mathbb{T}_\mathbf{D}\rightarrow \Lambda''_{\mathbf{D},P}/P^t\Lambda''_{\mathbf{D},P}$$
given by:
$\mu(t)=\ell(t.\mathbf{H})/\ell(\mathbf{H})$ for $t$ in the Hecke algebra. Then:
$$\ell(t.\mathbf{H})\equiv \ell(t\mathbf{E}_\mathbf{D})\equiv \lambda(t)\ell(\mathbf{E}_\mathbf{D})\equiv \lambda(t)\ell(\mathbf{H})(\mathrm{mod}P^t)$$
so $I_\mathbf{D}$ is contained in the kernel of $\mu$. Thus it induces:
$\Lambda''_{\mathbf{D},P}/\mathcal{E}_\mathbf{D}\Lambda''_{\mathbf{D},P}\twoheadrightarrow
\Lambda''_{\mathbf{D},P}/P^t\Lambda''_{\mathbf{D},P}$
which proves the lemma.
\end{proof}

\subsection{Galois Theoretic Argument}
In this section, for ease of reference, we repeat the set-up and certain results from \cite[Chapter 4]{SU} with some modifications, which are used to construct elements in the Selmer group. \\\\
Let $G$ be a group and $C$ a ring. Let $r: G\rightarrow \mathrm{Aut}_C(V)$ be a representation of $G$ with $V\simeq C^n$. This can be extended to
$r: C[G]\rightarrow \mathrm{End}_C(V)$. For any $x\in C[G]$, define: $\mathrm{Ch}(r,x,T):=\det(id-r(x)T)\in C[T]$.\\\\
Let $(V_1,\sigma_1)$ and $(V_2,\sigma_2)$ be two $C$ representations of $G$. Assume both are defined over a local henselian subring $B\subseteq C$,
we say $\sigma_1$ and $\sigma_2$ are residually disjoint modulo the maximal ideal $\mathfrak{m}_B$ if there exists $x\in B[G]$ such that
$\mathrm{Ch}(\sigma_1,x,T)\ \mathrm{mod}\ \mathfrak{m}_B$ and $\mathrm{Ch}(\sigma_2,x,,T)\ \mathrm{mod}\ \mathfrak{m}_B$ are relatively prime in $\kappa_B[T]$, where
$\kappa_B:=B/\mathfrak{m}_B$.\\\\
Let $H$ be a group with a decomposition $H=G\rtimes\{1,c\}$ with $c\in H$ an element of order two normalizing $G$. For any $C$ representations
$(V,r)$ of $G$ we write $r^c$ for the representation defined by $r^c(g)=r(cgc)$ for all $g\in G$.\\\\
\underline{Polarizations}:\\
\noindent Let $\theta:G\rightarrow \mathrm{GL}_L(V)$ be a representation of $G$ on a vector space $V$ over field $L$ and let
$\psi:H\rightarrow L^\times$ be a character. We assume that $\theta$ satisfies the $\psi$-polarization condition:
$$\theta^c\simeq \psi\otimes\theta^\vee.$$
By a $\psi$-polarization of $\theta$ we mean an $L$-bilinear pairing $\Phi_\theta:V\times V\rightarrow L$ such that
$$\Phi_\theta(\theta(g)v,v')=\psi(g)\Phi_\theta(v,\theta^c(g)^{-1}v').$$
Let $\Phi_\theta^t(v,v'):=\Phi_\theta(v',v)$, which is another $\psi$-polarization. We say that $\psi$ is compatible with the polarization
$\Phi_\theta$ if
$$\Phi_\theta^t=-\psi(c)\Phi_\theta.$$

\noindent Suppose that:\\
\noindent (1) $A_0$ is a pro-finite $\mathbb{Z}_p$ algebra and a Krull domain; \\
\noindent (2) $P\subset A_0$ is a height one prime and $A=\hat{A}_{0,P}$ is the completion of the localization of $A_0$ at $P$. This is a discrete valuation ring. \\
\noindent (3) $R_0$ is  local reduced finite $A_0$-algebra;\\
\noindent (4) $Q\subset R_0$ is prime such that $Q\cap A_0=P$ and $R=\hat{R}_{0,Q}$;\\
\noindent (5) there exist ideals $J_0\subset A_0$ and $I_0\subset R_0$ such that $I_0\cap A_0=J_0,A_0/J_0=R_0/I_0, J=J_0A, I=I_0R, J_0=J\cap A_0$ and $I_0=I\cap R_0$;\\
\noindent (6) $G$ and $H$ are pro-finite groups; we have a subgroup $G_{\bar{v}_0}\subset G$.\\

\noindent \underline{Set Up}:\\
\noindent Suppose we have the following data:\\
(1) a continuous character $\nu:H\rightarrow A_0^\times$;\\
(2) a continuous character $\xi:G\rightarrow A_0^\times$ such that $\bar\chi\not=\bar{\nu}\bar{\chi}^{-c}$; Let $\chi':=\nu\chi^{-c}$;\\
(3) a representation $\rho:G\rightarrow Aut_A(V),V\simeq A^n$, which is a base change from a representation over $A_0$, such that:
\begin{equation*}
\begin{split}
&a.\rho^c\simeq \rho^\vee\otimes \nu,\\
&\bar{\rho} \mbox{ is absolutely irreducible },\\
&\rho \mbox{ is residually disjoint from $\chi$ and $\chi'$};
\end{split}
\end{equation*}
(4) a representation $\sigma:G \rightarrow \mathrm{Aut}_{R\otimes_AF}(M),M\simeq (R\otimes_AF)^m$ with $m=n+2$, which is defined over the image of $R_0$ in $R$,  such that:
\begin{equation*}
\begin{split}
&a.\ \sigma^c\simeq \sigma^\vee\otimes\nu,\\
&b.\ \mathrm{tr}\sigma(g)\in R \mbox{ for all $g\in G$ },\\
&c.\ \mbox{ for any $v\in M$,$\sigma(R[G])v$ is a finitely-generated $R$-module }
\end{split}
\end{equation*}
(5) a proper ideal $I\subset R$ such that $J:=A\cap I\not=0$, the natural map $A/J\rightarrow R/I$ is an isomorphism, and
$$\mathrm{tr}\sigma(g)\equiv \chi'(g)+\mathrm{tr}\rho(g)+\chi(g)\ mod\ I$$
for all $g\in G$.\\\\
(6) $\rho$ is irreducible and $\nu$ is compatible with $\rho$.\\\\
(7) (local conditions for $\sigma$) There is a $G_{\bar{v}_0}$-stable sub-$R\otimes_AF$-module $M^+_{\bar{v}_0}\subseteq M$ such that $M^+_{\bar{v}_0}$ and $M^-_{\bar{v}_0}:=M/M^+_{\bar{v}_0}$ are free $R\otimes_AF$ modules. We also require that $M^+_{\bar{v}_0}$ and $M^-_{\bar{v}_0}$ are disjoint modulo $I$. (In our applications this is always satisfied although the $\bar{\mathbb{F}}_p$-representations of $\bar{\rho}_f$, $\bar{\chi}'$, $\bar{\chi}$ are not necessarily mutually disjoint when restricting to $G_{\bar{v}_0}$.)\\\\
(8) (compatibility with the congruence condition) Assume that for all $x\in R[G_{\bar{v}}]$, we have congruence relation:
$$\mathrm{Ch}(M^+_{\bar{v}_0},x,T)\equiv(1-T\chi(x))\ mod\ I$$
(then we automatically have:
$$\mathrm{Ch}(M^-_{\bar{v}_0},x,T)\equiv \mathrm{Ch}(V_{\bar{v}_0},x,T)(1-T\chi'(x))\ mod\ I)$$
(9) For each $F$-algebra homomorphism $\lambda: R\otimes_AF\rightarrow K$, $K$ a finite field extension of $F$, the representation $\sigma_\lambda:G\rightarrow
\mathrm{GL}_m(M\otimes_{R\otimes F}K)$ obtained from $\sigma$ via $\lambda$ is either absolutely  irreducible or contains an absolutely irreducible two-dimensional sub $K$-representation $\sigma_\lambda'$ such that $\mathrm{tr}\sigma_\lambda'(g)\equiv \chi(g)+\chi'(g)\ mod\ I$.\\\\
One defines the Selmer groups $\mathbf{X}_H(\chi'/\chi):=\mathrm{ker}\{H^1(H,A_0^*(\chi'/\chi))\rightarrow H^1(G_{\bar{v}_0},A_0^*(\chi'\chi))\}^*.$ and $\mathbf{X}_G(\rho_0\otimes\chi^{-1}):=\mathrm{ker}\{H^1(G,V_0\otimes_{A_0}A^*_0(\chi^{-1}))\rightarrow H^1(G_{\bar{v}_0},V_0^-\otimes_{A_0}A_0^*(\chi^{-1}))\}^*$. Let $\mathrm{Ch}_H(\chi'/\chi)$ and $\mathrm{Ch}_G(\rho_0\otimes\chi^{-1})$ be their characteristic ideals as $A_0$-modules.

\begin{proposition}\label{8.2.1}
Under the above assumptions, if $ord_P(\mathrm{Ch}_H(\chi'/\chi))=0$ then
$$ord_P(\mathrm{Ch}_G(\rho_0\otimes\chi^{-1}))\geq ord_P(J).$$
\end{proposition}
\begin{proof}
This can be proved in the same way as \cite[Corollary 4.16]{SU}. The only difference is the Selmer condition at $p$, which we use the description of Section \ref{2.24} to guarantee. Note that the part corresponding to $\rho_0$ corresponds to the upper-left two by two block here while in \cite{SU} the $\rho_f$ contains the highest and the lowest Hodge-Tate weights.
\end{proof}

\noindent Before proving the main theorem we first prove a useful lemma, which appears in an earlier version of \cite{SU} .
\begin{lemma}\label{lemma9.3}
Let $Q\subset \mathbb{I}[[\Gamma''_\mathcal{K}]]$ be a height one prime such that $\mathrm{ord}_Q(\mathcal{L}^\Sigma_{\mathbf{f},\mathcal{K},\xi})\geq 1$ and $\mathrm{ord}_Q(\mathcal{L}_{\chi_\mathbf{f}\bar{\xi}'})=0$, then $\mathrm{ord}_Q(\mathcal{L}^\Sigma_{\chi_\mathbf{f}\bar{\xi}'})=0$.
\end{lemma}
\begin{proof}
Let $\theta=\chi_{f_0}\bar{\xi}'$. If $\mathrm{ord}_Q(\mathcal{L}^\Sigma_{\chi_\mathbf{f}\bar{\xi}'})\geq 1$, then for some $\ell\in \Sigma\backslash\{p\}$,
$$\prod_{\ell\in\Sigma\backslash\{p\}}(1-\theta^{-1}(\gamma_+^{-1}(1+W))^e\ell^{2-\kappa})\in Q,$$
where $e\in\mathbb{Z}_p$ be such that $\ell=\omega^{-1}(\ell)(1+p)^e$. Thus
$$\theta(\ell)\equiv \gamma_+^{-e}\omega(\ell)^{\kappa-2}(1+p)^{e(2-\kappa)}(\mathrm{mod}Q).$$
Thus there is some integer $f$ such that
$$1\equiv (\gamma_+(1+W)^{-1}(1+p)^{\kappa-2})^{-fe}(\mathrm{mod}Q)$$
which implies that for some $p$-power root of unity $\zeta_+$, $Q$ is contained in the kernel of any $\phi'$ such that $\phi'(\gamma_+(1+W)^{-1})=\zeta_+(1+p)^{2-\kappa}$. This implies, by \cite[Theorem I]{Hida91} for the interpolation formula, that at the central critical point $L_\mathcal{K}(f_\phi,\theta_1,1)=0$ where $\theta_1$ is some fixed CM character of infinity type $(\frac{\kappa}{2},-\frac{\kappa}{2})$ and $\phi$ any arithmetic point. But then we can specialize $\mathbf{f}$ to some point $\phi''$ of weight $4$ (this is not an arithmetic point in our definition, but is an interpolation point, by \emph{loc.cit.}). By temperedness for $f_{\phi''}$, the specialization is not $0$. This is a contradiction.
\end{proof}

Now we apply the above result to prove the theorem.
\begin{itemize}
\item $H:=G_{\mathbb{Q},\Sigma}$, $G=G_{\mathcal{K},\Sigma}$, $c$ is the complex conjugation.
\item $A_0={\Lambda}''_\mathbf{D}, A:={\Lambda}''_{\mathbf{D},P}$.
\item $J_0:=\mathcal{E}_\mathbf{D}, J:=\mathcal{E}_\mathbf{D}A$.
\item $R_0:=\mathbb{T}_\mathbf{D}, I_0:=I_\mathbf{D}$.
\item $Q\subset R_0$ is the inverse image of $P$ modulo $\mathcal{E}_\mathbf{D}$ under $\mathbb{T}_\mathbf{D}\rightarrow \mathbb{T}_\mathbf{D}/I_\mathbf{D}=\Lambda_\mathbf{D}/\mathcal{E}_\mathbf{D}$.
\item $R:=T_\mathbf{f}\otimes_\mathbb{I}\Lambda''_\mathbf{D}$, $\rho_0:=\rho_{\pi_\mathbf{f}}\sigma_{(\boldsymbol{\psi}/\boldsymbol{\xi})^c}\epsilon^{-\frac{\kappa+3}{2}}$.
\item $V=V_0\otimes_{A_0}A$, $\rho=\rho_0\otimes_{A_0}A$.
\item $\chi=\sigma_{\boldsymbol{\psi}^c}$, $\chi'=\sigma_{\boldsymbol{\psi}^c}\sigma_{(\boldsymbol{\psi}/\boldsymbol{\xi})'}\epsilon^{-\kappa}$. $\nu=\chi^c\chi'$.
\item $M:=(R\otimes_AF_A)^4$, $F_A$ is the fraction field of $A$.
\item $\sigma$ is the representation on $M$ obtained as the pseudo-representation associated to $\mathbb{T}_\mathbf{D}$, as in \cite[Proposition 7.2.1]{SU}.
\end{itemize}
Now we are ready to prove the main theorem in the introduction.
\begin{proof}
We first note that we need only to prove the corresponding inclusion for the $\Sigma$-primitive Selmer groups and $L$-functions since locally the sizes of the unramified extensions at primes outside $p$ are controlled by the local Euler factors of the $p$-adic $L$-functions since $\mathbb{Q}_\infty\subseteq \mathcal{K}_\infty$. (See \cite[Proposition 2.4]{GV}.)\\

\noindent Recall that we have enlarged our $\mathbb{I}$ at the beginning of Subsection \ref{theta} and the end of Subsection \ref{Section 8.4} which we denote as $\mathbb{J}$ in this proof. We first prove the main theorem with $\hat{\mathbb{I}}^{ur}$ replaced by $\hat{\mathbb{J}}^{ur}$. Under the assumption of Theorem \ref{Theorem 1}, as in \cite[Proposition 12.9]{SU} we know that by the discussion for the anticyclotomic $\mu$-invariant at the end of Section \ref{6.4},  $\mathcal{L}_{\mathbf{f},\xi,\mathcal{K}}$ is not contained in any height one prime which is the pullback of a prime in $\hat{\mathbb{I}}^{ur}[[\Gamma_\mathcal{K}^+]]$. Note that there might be height one primes dividing the Euler factors at non-split primes which are pullbacks of height one primes of $\hat{\mathbb{I}}^{ur}[[\Gamma_\mathcal{K}^+]]$. Such issue has been overlooked in \cite{SU}. These primes are treated in \cite[Lemma 87, Theorem 101]{XW} and can be treated in the same way here. In the following, we treat the height one primes which are not such pullbacks. By lemma \ref{8.1} for any such height one prime $P$ of $\hat{\mathbb{I}}^{ur}[[\Gamma''_\mathcal{K}]]$,
$$\mathrm{ord}_P(\mathcal{L}^{\Sigma}_{\mathbf{f},\mathcal{K},\xi})=\mathrm{ord}_P(\mathcal{L}^{\Sigma}_{\mathbf{f},\mathcal{K},\xi})\leq \mathrm{ord}_P(\mathcal{E}_\mathbf{D}).$$
Applying Proposition \ref{8.2.1}, we prove the first part of the theorem for $\hat{\mathbb{J}}^{ur}$ in place of $\mathbb{I}$.\\

\noindent We replace $\hat{\mathbb{J}}^{ur}[[\Gamma''_\mathcal{K}]]$ by $\hat{\mathbb{I}}^{ur}[[\Gamma_\mathcal{K}]]$. We write $\mathcal{L}$ for $\mathcal{L}_{\mathbf{f},\xi,\mathcal{K}}^{\Sigma}$. Recall Fitting ideal respects base change. We claim that for any $x\in\mathrm{Fitt}(X), x\mathcal{L}^{-1}\in \hat{\mathbb{I}}^{ur}[[\Gamma_\mathcal{K}]]$. In fact, from what we proved for $\mathrm{Fitt}(V\otimes_{\hat{\mathbb{I}}^{ur}[[\Gamma_\mathcal{K}]]}\hat{\mathbb{J}}^{ur}[[\Gamma''_\mathcal{K}]])$ as ideals of $\hat{\mathbb{J}}^{ur}[[\Gamma''_\mathcal{K}]]$, we have $x\mathcal{L}^{-1}\in\hat{\mathbb{J}}^{ur}[[\Gamma''_\mathcal{K}]]\cap F_{\hat{\mathbb{I}}^{ur}[[\Gamma''_\mathcal{K}]]}$ where $F_{\hat{\mathbb{I}}^{ur}[[\Gamma_\mathcal{K}]]}$ is the fraction field of $\hat{\mathbb{I}}^{ur}[[\Gamma_\mathcal{K}]]$. Since $\hat{\mathbb{I}}^{ur}[[\Gamma_\mathcal{K}]]$ is normal and $\hat{\mathbb{J}}^{ur}[[\Gamma''_\mathcal{K}]]$ is finite over $\hat{\mathbb{I}}^{ur}[[\Gamma_\mathcal{K}]]$, we have $x\mathcal{L}^{-1}\in \hat{\mathbb{I}}^{ur}[[\Gamma_\mathcal{K}]]$. Thus $\mathrm{Fitt}(X)\subseteq (\mathcal{L})$, which in turn implies that $\mathrm{char}(X)\subseteq (\mathcal{L})$. This proves Theorem \ref{Theorem 1}.\\

\noindent Now assume we are under the assumption of Theorem \ref{Theorem 2}. Note that in this case $\mathcal{L}_{\chi\bar{\xi}'}=1$. Thus by the Lemma \ref{lemma9.3} $\mathcal{L}_{\mathbf{f},\xi,\mathcal{K}}^{\Sigma}$ is co-prime to $\mathcal{L}_{\chi\bar{\xi}'}^\Sigma$. Suppose $P_1,...,P_t$ are the height one primes of $\mathcal{L}^{\Sigma}_{\mathbf{f},\mathcal{K},\xi}$ that are pullbacks of height one primes in $\hat{\mathbb{I}}^{ur}$. Note that none of the primes passes through $\phi_0$ since the two-variable $p$-adic $L$-function for $f_0$ is not identically $0$. We consider the ring $\hat{\mathbb{I}}^{ur}_{p,P_1,...,P_t}[[\Gamma_\mathcal{K}]]$ where the subscripts denote localizations. Then the argument as in the proof of Theorem \ref{Theorem 1} proves that there is a number $a$ such that $(\mathcal{L}^{\Sigma}_{\mathbf{f},\mathcal{K},\xi})\supseteq (P_1\cdots P_t)^a\mathrm{Fitt}_{\hat{\mathbb{I}}^{ur}[[\Gamma_\mathcal{K}]]}(X^\Sigma)$ as ideals of $\hat{\mathbb{I}}^{ur}[[\Gamma_\mathcal{K}]]$. Specialize to $\phi_0'$, using Proposition \ref{control}, we find
$$(\mathcal{L}_{f_0,\mathcal{K},\xi}^{\Sigma})\supseteq \mathrm{Fitt}_{\hat{\mathcal{O}}^{ur}_L[[\Gamma_\mathcal{K}]]\otimes L}(X_{f_0}).$$
This proves Theorem \ref{Theorem 2}.
\end{proof}

\textsc{Xin Wan, Morningside Center of Mathematics, Academy of Mathematics and Systems Science, Chinese Academy of Science, Beijing, China, 100190.}\\
\indent \textit{E-mail Address}: xw2295@math.columbia.edu
\printindex
\end{document}